\documentclass[a4paper,reqno]{amsart}

\usepackage{amsthm,amssymb,amsmath}

\numberwithin{equation}{section}

\theoremstyle{plain}
\newtheorem{Theorem}{Theorem}[section]
\newtheorem{Corollary}{Corollary}[section]
\newtheorem{Proposition}{Proposition}[section]
\newtheorem{Lemma}{Lemma}[section]

\theoremstyle{definition}
\newtheorem{Definition}{Definition}[section]

\def\R{\mathbb{R}}
\def\Q{\mathbb{Q}}
\def\N{\mathbb{N}}
\def\Z{\mathbb{Z}}
\def\U{\mathbb{U}}

\def\A{\mathbb{A}}

\def\eps{\varepsilon}
\def\ph{\varphi}

\def\Eta{\mathrm{H}}
\def\Iota{\mathrm{I}}
\def\Mu{\mathrm{M}}
\def\Nu{\mathrm{N}}
\def\Rho{\mathrm{P}}
\def\Upsilon{\mathrm{Y}}
\def\Tau{\mathrm{T}}
\def\ee{\mathrm{e}}

\def\calA{\mathcal{A}}
\def\calB{\mathcal{B}}
\def\calC{\mathcal{C}}

\def\calE{\mathcal{E}}
\def\calF{\mathcal{F}}

\def\calI{\mathcal{I}}
\def\calK{\mathcal{K}}
\def\calN{\mathcal{N}}
\def\calP{\mathcal{P}}
\def\calQ{\mathcal{Q}}
\def\rad01{\mathcal{R}}

\def\calZ{\mathcal{Z}}

\def\frakH{\mathfrak{H}}
\def\frakL{\mathfrak{L}}

\def\frakQ{\mathfrak{Q}}
\def\frakA{\mathfrak{A}}
\def\frakF{\mathfrak{F}}
\def\frakn{\mathfrak{n}}
\def\fraks{\mathfrak{s}}
\def\frakj{\mathfrak{j}}
\def\frakJ{\mathfrak{J}}

\def\rmr{\mathrm{r}}

\def\ind{\mathbf{1}}
\def\zerofunc{\mathbf{0}}

\def\majo{\mathfrak{M}}
\def\mino{\mathfrak{m}}
\def\zeroleb{\calZ}

\def\gauge{\mathfrak{G}}

\def\bor{\calB}
\def\leb{\mathcal{L}}
\def\hau{\mathcal{H}}
\def\netm{\mathcal{M}}
\def\pack{\mathcal{P}}
\def\prepack{P}

\newcommand{\lic}[2]{\mathcal{G}^{#1}(#2)}
\newcommand{\croc}[2]{\langle {#1},{#2}\rangle}

\def\prob{\mathbb{P}}
\def\esp{\mathbb{E}}
\def\as{\mathrm{a.s.}}

\def\Hdim{\dim_{\mathrm{H}}}
\def\Pdim{\dim_{\mathrm{P}}}
\def\UBdim{\operatorname{\overline{dim}_{B}}}

\def\opball{\mathrm{B}}
\def\clball{\overline{\mathrm{B}}}
\def\lad{\underline{\delta}}

\newcommand{\diam}[1]{|#1|}
\newcommand{\gene}[1]{\langle #1\rangle}
\newcommand{\dist}[2]{\mathrm{d}(#1,#2)}

\def\prd{\calP(\R^d)}
\def\cantor{\mathbb{K}}

\newcommand{\distZ}[1]{\left\| #1\right\|}
\newcommand{\closure}[1]{\overline{#1}}
\newcommand{\interior}[1]{\operatorname{int} #1}

\def\dd{\mathrm{d}}
\def\smallo{\mathrm{o}}

\def\bad{\mathrm{Bad}}
\def\well{\mathrm{Well}}

\begin{document}

\title[Describability via ubiquity and eutaxy]{Describability via ubiquity and eutaxy in Diophantine approximation}
\author{Arnaud Durand}\address{Arnaud Durand\\ Universit\'e Paris-Sud\\ Laboratoire de Math\'ematiques d'Orsay -- UMR 8628\\ B\^atiment 425\\ 91405 Orsay Cedex, France}
\email{arnaud.durand@math.u-psud.fr}

\date{\today}

\begin{abstract}
	We present a comprehensive framework for the study of the size and large intersection properties of sets of limsup type that arise naturally in Diophantine approximation and multifractal analysis. This setting encompasses the classical ubiquity techniques, as well as the mass and the large intersection transference principles, thereby leading to a thorough description of the properties in terms of Hausdorff measures and large intersection classes associated with general gauge functions. The sets issued from eutaxic sequences of points and optimal regular systems may naturally be described within this framework. The discussed applications include the classical homogeneous and inhomogeneous approximation, the approximation by algebraic numbers, the approximation by fractional parts, the study of uniform and Poisson random coverings, and the multifractal analysis of L\'evy processes.
\end{abstract}

\maketitle


\section{Introduction}

The aim of these notes is to present a comprehensive framework for the study of the size and large intersection properties of sets of limsup type. Such sets arise naturally in Diophantine approximation and multifractal analysis. A simple example is the set of real numbers approximable at rate $\tau$ by rational numbers, namely,
	\[
		J_\tau=\left\{x\in\R\Biggm|\left|x-\frac{p}{q}\right|<\frac{1}{q^\tau}\quad\text{for i.m.~}(p,q)\in\Z\times\N\right\},
	\]
	where i.m.~stands for ``infinitely many''. It follows from Dirichlet's pigeon-hole principle that the sets $J_\tau$, for $\tau\leq 2$, all coincide with the whole real line. For $\tau>2$ however, the sets are small in the sense that they have Lebesgue measure zero. The notion of Hausdorff dimension enables one to give a more precise description of their size. Specifically, a result due to Jarn\'ik and Besicovitch ensures that each set $J_\tau$ has dimension equal to $2/\tau$. Moreover, the fractal structure of these sets is striking: they satisfy the large intersection property discovered by Falconer.
	
The sets $J_\tau$, though representative, are only a particular instance of a wide category of sets enjoying the same remarkable properties. They are of the form
	\[
		\frakF((x_i,r_i)_{i\in\calI})=\left\{x\in\R^d\bigm||x-x_i|<r_i\quad\text{for i.m.~} i\in\calI\right\},
	\]
	where the countably many points $x_i$ and positive real numbers $r_i$ constitute an {\em approximation system}. Under natural hypotheses on the system, a general construction called {\em ubiquity} will enable us to derive the Hausdorff dimension of the latter set, and also to show that the large intersection property holds. Moreover, through {\em the mass and the large intersection transference principles}, we shall explain how to extend the study to Hausdorff measures and large intersection classes associated with general gauge functions. We shall then present a new setting for the analysis of these sets: we shall show that, in most situations, they are {\em fully describable}, meaning that the description of their size and large intersection properties is as complete and precise as possible.

Full describability arises in particular when the underlying approximation system is issued from a {\em eutaxic sequence of points} or an {\em optimal regular system}. We shall study thoroughly these two situations and illustrate them by many examples. Optimal regular systems will enable us to discuss the homogeneous and inhomogeneous Diophantine approximation problems, as well as the approximation by algebraic numbers. With the help of eutaxy, we shall be able study the approximation by fractional parts of sequences, random coverings problems, and finally the multifractal properties of L\'evy processes.

\medskip
{\em Acknowledgements.} These notes are based on series of lectures given during the 2012 Program on Stochastics, Dimension and Dynamics at Morningside Center of Mathematics in Beijing, the 2013 Arithmetic Geometry Year at Poncelet Laboratory in Moscow, and the 2014 Spring School in Analysis held at Universit\'e Blaise Pascal in Clermont-Ferrand. The author would like to thank the organizers of these events, especially Fr\'ed\'eric Bayart, Ai-Hua Fan, Yanick Heurteaux, Philippe Lebacque, Andrzej St\'os and Alexey Zykin, and all the attendees for their valuable comments and questions.


\section{Elementary Diophantine approximation}

\subsection{Very well approximable numbers}\label{subsec:verywell}

Diophantine approximation is originally concerned with the approximation of real numbers by rational numbers or, more generally, the approximations of points in $\R^d$ by points with integer coordinates. The first result on this topic is due to Dirichlet. Throughout, $|\,\cdot\,|_{\infty}$ denotes the supremum norm on $\R^d$.

\begin{Theorem}[Dirichlet, 1842]\label{thm:Dirichlet}
	Let us consider a point $x\in\R^{d}$. Then, for any integer $Q>1$, the next system admits a solution $(p,q)$ in $\Z^d\times\N$\,:
	\begin{equation}\label{eq:thm:Dirichlet}
		\left\{\begin{array}{l}
			1\leq q<Q^d\\[1mm]
			|qx-p|_{\infty}\leq 1/Q
		\end{array}\right.
	\end{equation}
\end{Theorem}

\begin{proof}
Let us consider the points $0, 1, \{x\},\{2x\},\ldots,\{(Q^{d}-1)x\}$, where $\{\,\cdot\,\}$ denotes the coordinate-wise fractional part, and $1$ is the point with all coordinates equal to one. These points all lie in the unit cube $[0,1]^{d}$, which we may decompose as the disjoint union over $u_{1},\ldots,u_{d}\in\{0,\ldots,Q-1\}$ of the cubes $\prod_{i=1}^{d}[u_i/Q,(u_i+1)/Q\rangle$, where $\rangle$ stands for the symbol $]$ if $u_{i}=Q-1$, and for the symbol $)$ otherwise; in other words, the interval is closed if and only if $u_{i}=Q-1$.

There are $Q^d$ such subcubes, and $Q^d+1$ points. Thus, the pigeon-hole principle ensures that there is at least one subcube that contains two of the points. As a result, there exist either two distinct integers $r_1$ and $r_2$ between zero and $Q^d-1$ such that $\{r_1x\}$ and $\{r_2x\}$ are in the same subcube, or one integer $r_2$ between one and $Q^d-1$ such that $\{r_2x\}$ and $1$ belong to the same subcube. In both cases, we deduce that there exist two integers $r_1$ and $r_2$ satisfying $0\leq r_1<r_2<Q^d$, and two points with integers coordinates $s_1$ and $s_2$ in $\Z^d$ such that
	\[
		\left|(r_{1}x-s_{1})-(r_{2}x-s_{2})\right|_{\infty}\leq\frac{1}{Q}.
	\]
	The result now follows from letting $q=r_2-r_1$ and $p=s_2-s_1$.
\end{proof}

Theorem~\ref{thm:Dirichlet} means that the $d$ real numbers $x_{1},\ldots,x_{d}$ may simultaneously be approximated at a distance at most $1/Q$ by $d$ rational numbers with common denominator an integer less than $Q^d$, namely, the rationals $p_{1}/q,\ldots,p_{d}/q$. In fact, this yields a uniform estimate on the quality with which these real numbers may simultaneously be approximated by a sequence of rationals with common denominator. In the next statement, $\gcd(p,q)$ denotes the greatest common divisor of $q$ and all the coordinates of the integer point $p$.

\begin{Corollary}\label{cor:Dirichlet}
	For any point $x\in\R^d\setminus\Q^d$, there exist infinitely many pairs $(p,q)$ in $\Z^d\times\N$ such that
	\begin{equation}\label{eq:cor:Dirichlet}		
		\left|x-\frac{p}{q}\right|_{\infty}<\frac{1}{q^{1+1/d}}
		\qquad\text{and}\qquad
		\gcd(p,q)=1.
	\end{equation}
\end{Corollary}

\begin{proof}
For any point $x$ in $\R^d\setminus\Q^d$, let $\calE_x$ denote the set of pairs $(p,q)$ in $\Z^d\times\N$ such that~(\ref{eq:cor:Dirichlet}) holds. Moreover, for any integer $Q>1$, let $\calE_x(Q)$ denote the set of pairs $(p,q)$ in $\Z^d\times\N$ satisfying~(\ref{eq:thm:Dirichlet}). Theorem~\ref{thm:Dirichlet} ensures that all the sets $\calE_x(Q)$ are nonempty. Moreover, the mapping $(p,q)\mapsto (p,q)/\gcd(p,q)$ sends the sets $\calE_x(Q)$ into $\calE_x$, and reduces the value of $|qx-p|_\infty$. Thus,
\[
\inf_{(p,q)\in\calE_x}|qx-p|_\infty\leq\inf_{(p,q)\in\calE_x(Q)}|qx-p|_\infty\leq\frac{1}{Q}.
\]
Letting $Q\to\infty$, we deduce that the infimum of $|qx-p|_\infty$ over $(p,q)\in\calE_x$ vanishes. Since $x$ is not in $\Q^d$, this implies that $\calE_x$ is necessarily infinite.
\end{proof}

In order to further study the quality of the approximation by points with rational coordinates, we introduce for any real parameter $\tau$, the set
	\begin{equation}\label{eq:df:Jdtau}
		J_{d,\tau}=\left\{x\in\R^d\Biggm|\left|x-\frac{p}{q}\right|_\infty<\frac{1}{q^\tau}\quad\text{for i.m.~}(p,q)\in\Z^d\times\N\right\}.
	\end{equation}
	A plain consequence of Corollary~\ref{cor:Dirichlet} is that the above sets coincide with $\R^d$ for all $\tau\leq 1+1/d$. Furthermore, the sets $J_{d,\tau}$ are clearly nonincreasing with respect to $\tau$.

In fact, the points in $J_{d,\tau}$ are better and better approximated by points with rational coordinates as $\tau$ becomes larger. The quality of the approximation may thus be measured in terms of membership in $J_{d,\tau}$, specifically, through the {\em irrationality exponent} defined by
	\begin{equation}\label{eq:df:irrexpo}
		\tau(x)=\sup\{\tau\in\R\:|\:x\in J_{d,\tau}\}
	\end{equation}
	for any point $x$ in $\R^d\setminus\Q^d$. This exponent is always bounded below by $1+1/d$. A point $x$ for which $\tau(x)$ is larger than $1+1/d$ is called {\em very well approximable}. The set of very well approximable points is denoted by $\well_d$.

It is clear from the above definition that the irrationality exponent reflects the quality with which the points in $\R^d\setminus\Q^d$ are approximated by those with rational coordinates: the higher the exponent, the better the approximation. Besides, observe that the set of very well approximable points satisfies
	\begin{equation}\label{eq:welldcupJdtau}
		\well_d=(\R^d\setminus\Q^d)\cap\bigcup_{\tau>1+1/d} J_{d,\tau}.
	\end{equation}

The main purpose of the metric theory of Diophantine approximation is then to describe the size properties of sets such as $J_{d,\tau}$, or generalizations thereof, in the case of course where they do not coincide with the whole space $\R^d$. To this purpose, the most basic tool, but also the less precise, is the Lebesgue measure. As regards the specific case of the sets $J_{d,\tau}$, and their companion set $\well_d$, we have the following elementary result. The Lebesgue measure in $\R^d$ is denoted by $\leb^d$ in what follows; its basic properties are recalled in Section~\ref{subsec:outmeas}.

\begin{Proposition}\label{prp:lebJdtau}
	The set $\well_d$ of very well approximable points has Lebesgue measure zero, that is,
	\[
		\leb^d(\well_d)=0.
	\]
	Equivalently, all the sets $J_{d,\tau}$, for $\tau>1+1/d$, have Lebesgue measure zero.
\end{Proposition}

\begin{proof}
	The proof is elementary, and amounts to using an appropriate covering of the set $J_{d,\tau}$. To be specific, for any integer $Q\geq 1$, we have
	\[
	J_{d,\tau}\cap[0,1]^d\subseteq\bigcup_{q\geq Q}\bigcup_{p\in\{0,\ldots,q\}^d}\opball_\infty\left(\frac{p}{q},\frac{1}{q^\tau}\right),
	\]
	where $\opball_\infty(x,r)$ denotes the open ball centered at $x$ with radius $r$, in the sense of the supremum norm. The subadditivity of Lebesgue measure yields
	\[
	\leb^d(J_{d,\tau}\cap[0,1]^d)\leq\sum_{q\geq Q}(q+1)^d\left(\frac{2}{q^\tau}\right)^d
	\]
	The above series clearly converges when $\tau>1+1/d$. Letting $Q\to\infty$, we deduce that the Lebesgue measure of $J_{d,\tau}\cap[0,1]^d$ vanishes. The set $J_{d,\tau}$ being invariant under the action of $\Z^d$, its Lebesgue measure thus vanishes in the whole space.
	
	To establish that the set $\well_d$ has Lebesgue measure zero as well, it suffices to observe that the union in~(\ref{eq:welldcupJdtau}) may be indexed by a countable dense subset of values of $\tau$, because of the monotonicity of the sets $J_{d,\tau}$ with respect to $\tau$. More precisely, letting for instance $\tau_n=(1+1/d)+1/n$, we may write that
	\[
	\leb^d(\well_d)\leq\leb^d\left(\bigcup_{n=1}^\infty J_{d,\tau_n}\right)
	\leq\sum_{n=1}^\infty\leb^d(J_{d,\tau_n})=0.
	\]
	
	Finally, knowing that $\well_d$ has Lebesgue measure zero, we can easily recover the fact that the sets $J_{d,\tau}$, for $\tau>1+1/d$, have Lebesgue measure zero as well. It suffices to use~(\ref{eq:welldcupJdtau}) and the fact that $\Q^d$ is countable and thus Lebesgue null.
\end{proof}

It readily follows from Proposition~\ref{prp:lebJdtau} that, in the sense of Lebesgue measure, the irrationality exponent is minimal almost everywhere, that is,
\begin{equation}\label{eq:irrexpotypical}
\text{for~}\leb^d\text{-a.e.~}x\in\R^d\setminus\Q^d
\qquad\tau(x)=1+\frac{1}{d},
\end{equation}
where a.e.~means ``almost every''. Moreover, as shown by Proposition~\ref{prp:lebJdtau}, describing the size of the sets $J_{d,\tau}$ in terms of Lebesgue measure only is not very precise, as we just have the following dichotomy: the set $J_{d,\tau}$ has full Lebesgue measure in $\R^d$ if $\tau\leq 1+1/d$, and Lebesgue measure zero otherwise. A first way of giving a more precise description is then to compute the Hausdorff dimension of the set $J_{d,\tau}$; this will be performed in Sections~\ref{subsec:upbndhausdorff} and~\ref{subsec:JarnikBesicovitch}. One of the purpose of these notes is in fact to provide a description as comprehensive as possible of the size, and also the large intersection, properties of these sets and natural extensions thereof.

\subsection{Badly approximable points}\label{subsec:bad}

These points play a particular r\^ole in Diophantine approximation. A point $x\in\R^d$ is called {\em badly approximable} if
	\[
		\exists\eps>0 \quad \forall (p,q)\in\Z^d\times\N \qquad \left|x-\frac{p}{q}\right|_\infty\geq\frac{\eps}{q^{1+1/d}}.
	\]
	The set of badly approximable points is denoted by $\bad_d$. In dimension $d=1$, the badly approximable points are called {\em badly approximable numbers}. As the name indicates, the elements of $\bad_d$ are badly approximated by the points with rational coordinates. Indeed, the irrationality exponent satisfies
	\[
		\forall x\in\bad_d \qquad \tau(x)=1+\frac{1}{d}.
	\]
	This means that the points in $\bad_d$ attain the bound imposed by Dirichlet's theorem and its corollary, namely, Theorem~\ref{thm:Dirichlet} and Corollary~\ref{cor:Dirichlet}. In other words,
	\begin{equation}\label{eq:badinclwell}
		\bad_d\subseteq (\R^d\setminus\Q^d)\setminus\well_d.
	\end{equation}
	Due to Proposition~\ref{prp:lebJdtau}, the set in the right-hand side of~(\ref{eq:badinclwell}) has full Lebesgue measure in $\R^d\setminus\Q^d$. The badly approximable points thus supply specific examples of points for which the typical property~(\ref{eq:irrexpotypical}) holds. Turning our attention to the left-hand side of~(\ref{eq:badinclwell}), we now establish the following result.

\begin{Proposition}\label{prp:badlebesgued}
	The set $\bad_d$ of badly approximable points has Lebesgue measure zero, that is,
	\[
		\leb^d(\bad_d)=0.
	\]
\end{Proposition}

\begin{proof}
	For any integer $n\geq 1$, let $\widetilde J_n$ be the set obtained when replacing by $1/(n\,q^{1+1/d})$ the approximation radii $1/q^\tau$ in the definition~(\ref{eq:df:Jdtau}) of $J_{d,\tau}$. Clearly,
	\[
		\R^d\setminus\bad_d\supseteq\bigcap_{n=1}^\infty\widetilde J_n,
	\]
	Due to the subadditivity of Lebesgue measure, the proof reduces to showing that for all $n\geq 1$, the set $\R^d\setminus\widetilde J_n$ has Lebesgue measure zero. The corollary to Dirichlet's theorem, namely, Corollary~\ref{cor:Dirichlet} implies that this holds for $n=1$. To prove that this also holds for higher values of $n$, one may then use Proposition~\ref{prp:homubsyscst} below.
\end{proof}

The above measure theoretic considerations directly imply that the inclusion in~(\ref{eq:badinclwell}) is strict. Actually, Lebesgue-almost every point in the set $\R^d\setminus\Q^d$ is neither very well nor badly approximable. The next step in the description of the size properties of $\bad_d$ was first performed by Schmidt~\cite{Schmidt:1969fk} who showed that the Hausdorff dimension of this set is equal to $d$.

\subsection{Inhomogeneous approximation}\label{subsec:inhomapprox}

Inhomogeneous Diophantine approximation usually refers to the approximation of points in $\R^d$ by the system obtained by the points of the form $(p+\alpha)/q$, where as usual $p$ is an integer point, and $q$ is a positive integer, and where $\alpha$ is a point in $\R^d$ that is fixed in advance. When $\alpha$ is equal to zero, one obviously recovers the situation discussed in Section~\ref{subsec:verywell}, which is referred to as the homogeneous one.

The next result due to Khintchine~\cite{Khintchine:1937aa} complements in some sense Dirichlet's theorem, namely, Theorem~\ref{thm:Dirichlet}. Among Khintchine's works, this result may be regarded as an anticipation of his deep transference principle that relates homogeneous and inhomogeneous problems, see {\em e.g.}~\cite[Chapter~V]{Cassels:1957uq}.

\begin{Theorem}\label{thm:Dirichletinhom}
For any point $x\in\R^d$, the following properties are equivalent:
\begin{enumerate}
	\item there exists a real number $\gamma>0$ such that for any integer $Q>1$, the next system admits no solution $(p,q)$ in $\Z^d\times\N$\,:
	\[
		\left\{\begin{array}{l}
		1\leq q<\gamma Q^d\\[1mm]
		|qx-p|_{\infty}\leq 1/Q\,;
		\end{array}\right.
	\]
	\item there exists a real number $\Gamma>0$ such that for any point $\alpha\in\R^d$ and any integer $Q>1$, the next system 	admits a solution $(p,q)$ in $\Z^d\times\N$\,:
	\[
		\left\{\begin{array}{l}
		1\leq q<\Gamma Q^d\\[1mm]
		|qx-p-\alpha|_{\infty}\leq 1/Q.
		\end{array}\right.
	\]
\end{enumerate}
Moreover, if $\gamma$ exists, then $\Gamma$ depends on $\gamma$ and $d$ only. Likewise, if $\Gamma$ exists, then $\gamma$ depends on $\Gamma$ and $d$ only.
\end{Theorem}

Inspecting the proof of Theorem~\ref{thm:Dirichletinhom} given in~\cite{Khintchine:1937aa}, we may deduce the next complementary result. For any point $x\in\R^d$ and any integer $Q>1$, let us define
\[
q(x,Q)=\inf\left\{q\in\N\Biggm||qx-p|_{\infty}\leq\frac{1}{Q}\text{ for some }p\in\Z^d\right\}.
\]
It follows from Dirichlet's theorem that $q(x,Q)$ is less than $Q^d$.

\begin{Proposition}\label{prp:Dirichletinhomvar}
For any real number $\gamma\in(0,1)$, there exist a real number $\Gamma_\ast>1$ and an integer $Q_\ast\geq 1$, both depending on $\gamma$ and $d$ only, such that the following property holds: for any points $x$ and $\alpha$ in $\R^d$ and for any integer $Q>Q_\ast$,
\[
q(x,Q)\geq\gamma Q^d
\quad\Longrightarrow\quad
\exists(p,q)\in\Z^d\times\N\quad
\left\{\begin{array}{l}
q(x,Q)\leq q<2q(x,Q)\\[1mm]
|qx-p-\alpha|_{\infty}\leq\Gamma_\ast/q(x,Q)^{1/d}.
\end{array}\right.
\]
\end{Proposition}

This result will be called upon in the proof of Theorem~\ref{thm:inhomratoptregsys}. The latter theorem will then enable us to study the metric properties of a natural inhomogeneous analog of the set $J_{d,\tau}$ defined by~(\ref{eq:df:Jdtau}), specifically, the set
	\begin{equation}\label{eq:df:Jdtaualpha}
		J^\alpha_{d,\tau}=\left\{x\in\R^d\Biggm|\left|x-\frac{p+\alpha}{q}\right|_\infty<\frac{1}{q^\tau}\quad\text{for i.m.~}(p,q)\in\Z^d\times\N\right\}.
	\end{equation}
	Note that Proposition~\ref{prp:lebJdtau} may straightforwardly be extended to that case. Specifically, one easily checks that $J^\alpha_{d,\tau}$ has Lebesgue measure zero for any $\tau>1+1/d$. Some more work is required to show that, as in the homogeneous setting, the set $J^\alpha_{d,\tau}$ has full Lebesgue measure in the whole space $\R^d$ in the opposite case; this will be a consequence of Corollary~\ref{cor:inhomJBdesc}, which actually gives a much more precise description of the size of the set $J^\alpha_{d,\tau}$.


\section{Hausdorff measures and dimension}\label{sec:hausdim}

\subsection{Premeasures and outer measures}\label{subsec:outmeas}

Before dealing with Hausdorff measures, we introduce general definitions and state classical results from geometric measure theory. We shall not follow here the standard approach that originates in the work of Radon and consists in defining measures on prespecified $\sigma$-fields. Instead, our viewpoint is that initiated by Carath\'eodory: considering {\em outer measures} on all the subsets of the space $\R^d$, and then discussing further {\em measurability} properties of the subsets. Our treatment will be rather brief and we refer to~\cite{Rogers:1970wb} for missing proofs and details. Throughout, we restrict our attention to the space $\R^d$, even if the discussed notions may be defined in general metric spaces.

The collection of all subsets of $\R^d$ is denoted by $\prd$. We recall that a function $\mu:\prd\to [0,\infty]$ is an {\em outer measure} if the next conditions are fulfilled: $\mu(\emptyset)=0$\,; for any sets $E_1$ and $E_2$ in $\prd$ such that $E_1\subseteq E_2$, we have $\mu(E_1)\leq\mu(E_2)$\,; for any sequence $(E_n)_{n\geq 1}$ in $\prd$,
	\[
		\mu\left(\bigcup_{n=1}^\infty E_n\right)\leq\sum_{n=1}^\infty\mu(E_n).
	\]
	Hence, an outer measure $\mu$ is defined on the whole collection $\prd$. However, it enjoys further properties when restricted to the {\em $\mu$-measurable} sets, namely, the sets $E\in\prd$ such that for all $A$ and $B$ in $\prd$,
\[
\left\{\begin{array}{l}
A\subseteq E\\[1mm]
B\subseteq \R^d\setminus E
\end{array}\right.
\qquad\Longrightarrow\qquad
\mu(A\sqcup B)=\mu(A)+\mu(B).
\]
The collection of all $\mu$-measurable sets is denoted by $\calF_\mu$. The connection with the standard approach of measures on $\sigma$-fields is then given by the following result. In its statement, we say that a set $N\in\prd$ is {\em $\mu$-negligible} if $\mu(N)=0$.

\begin{Theorem}\label{thm:measfield}
For any outer measure $\mu$, the following properties hold:
\begin{enumerate}
\item the collection $\calF_\mu$ is a $\sigma$-field of $\R^d$;
\item every $\mu$-negligible set in $\prd$ belongs to $\calF_\mu$;
\item for any sequence $(E_n)_{n\geq 1}$ of {\em disjoint} sets in $\calF_\mu$, we have
\[
\mu\left(\bigsqcup_{n=1}^\infty E_n\right)=\sum_{n=1}^\infty\mu(E_n).
\]
\end{enumerate}
\end{Theorem}

In addition, we also have the following useful properties concerning monotonic sequences of $\mu$-measurable sets: for any nondecreasing sequence $(F_n)_{n\geq 1}$ of $\mu$-measurable subsets of $\R^d$ and for any subset $E$ of $\R^d$,
	\begin{equation}\label{eq:measincunion}
		\mu\left(E\cap\bigcup_{n=1}^\infty\uparrow F_n\right)
		=\lim_{n\to\infty}\uparrow\mu(E\cap F_n)\,;
	\end{equation}
	for any nonincreasing sequence $(F_n)_{n\geq 1}$ of $\mu$-measurable sets and for any subset $E$ of $\R^d$ such that $\mu(E\cap F_n)<\infty$ for some integer $n\geq 1$,
	\begin{equation}\label{eq:measdecint}
		\mu\left(E\cap\bigcap_{n=1}^\infty\downarrow F_n\right)
		=\lim_{n\to\infty}\downarrow\mu(E\cap F_n).	
	\end{equation}

Theorem~\ref{thm:measfield} ensures that the restriction of an outer measure $\mu$ to the $\sigma$-field $\calF_\mu$ is a measure in the usual sense. Conversely, let us consider a measure $\nu$ defined on some $\sigma$-field $\calF$ of subsets of $\R^d$. We may extend $\nu$ to the whole $\prd$ by letting
	\[
		\nu^\ast(E)=\inf_{F\in\calF\atop F\supseteq E} \nu(F)
	\]
	for any set $E\in\prd$. We obtain an outer measure $\nu^\ast$ whose restriction to $\calF$ coincides with $\nu$, and such that the $\sigma$-field of all $\nu^\ast$-measurable sets contains $\calF$. This is a particular case of the following general construction where, rather than just being the extension of a usual measure, an outer measure is derived from a function defined on a class of subsets of $\R^d$.

\begin{Definition}
	A {\em premeasure} is a function of the form $\zeta:\calC\to [0,\infty]$, where $\calC$ is a collection of subsets of $\R^d$ containing the empty set, that satisfies $\zeta(\emptyset)=0$.
\end{Definition}

The construction makes use of the standard notion of covering. Given a set $E$ in $\prd$ and a collection $\calC$ of subsets of $\R^d$ containing the empty set, recall that a sequence of sets $(C_n)_{n\geq 1}$ in $\calC$ is called a {\em covering} of $E$ if
	\[
		E\subseteq\bigcup_{n=1}^\infty C_n.
	\]
	Note that this definition encompasses the case of coverings by finitely many sets, as we can choose the sets $C_n$ to be empty when $n$ is large enough. The next result gives a general method to build an outer measure starting from a premeasure.

\begin{Theorem}\label{thm:absmeas}
	Let $\calC$ be a collection of subsets of $\R^d$ containing the empty set, and let $\zeta$ be a premeasure defined on $\calC$. Then, the function $\zeta^\ast$ defined on $\prd$ by
	\begin{equation}\label{eq:df:zetaasta}
		\zeta^\ast(E)=\inf_{E\subseteq\bigcup_n C_n \atop C_n\in\calC}\sum_{n=1}^\infty\zeta(C_n)
	\end{equation}
	is an outer measure. Here, the infimum is taken over all coverings of the set $E$ by sequences $(C_n)_{n\geq 1}$ of sets that belong to $\calC$.
\end{Theorem}

Obviously, the above procedure is ``closed'', in the sense that if $\mu$ denotes an outer measure, then $\mu$ may be seen as a premeasure on $\prd$ and the outer measure $\mu^\ast$ defined via~(\ref{eq:df:zetaasta}) coincides with $\mu$. Let us now present another way of extending a premeasure into an outer measure, by taking additionally into account the metric structure of $\R^d$. The diameter of a set $E\in\prd$ is denoted by $\diam{E}$.

\begin{Theorem}\label{thm:metmeas}
	Let $\calC$ be a collection of subsets of $\R^d$ containing the empty set, and let $\zeta$ be a premeasure defined on $\calC$. Then, the function $\zeta_\ast$ defined on $\prd$ by
		\begin{equation}\label{eq:df:zetaastm}
			\zeta_\ast(E)=\lim_{\delta\downarrow 0}\uparrow \zeta_\delta(E)
			\qquad\text{with}\qquad
			\zeta_\delta(E)=\inf_{E\subseteq\bigcup_n C_n \atop C_n\in\calC, \diam{C_n}\leq\delta}\sum_{n=1}^\infty\zeta(C_n)
		\end{equation}
		is an outer measure. Here, the infimum is taken over all coverings of the set $E$ by sequences $(C_n)_{n\geq 1}$ of sets belonging to $\calC$ with diameter at most $\delta$.
\end{Theorem}

Let us mention that it is obvious from~(\ref{eq:df:zetaasta}) and~(\ref{eq:df:zetaastm}) that for any premeasure $\zeta$ and any subset $E$ of $\R^d$, we have $\zeta^\ast(E)\leq\zeta_\delta(E)$ for all $\delta>0$; thus, taking the limit as $\delta\to 0$, we deduce that $\zeta^\ast(E)\leq\zeta_\ast(E)$. The main advantage of the above construction over that given by Theorem~\ref{thm:absmeas} is that the outer measure $\zeta_\ast$ is {\em metric}, namely, for all sets $A$ and $B$ in $\prd\setminus\{\emptyset\}$,
	\[
		\dist{A}{B}>0
		\qquad\Longrightarrow\qquad
		\mu(A\sqcup B)=\mu(A)+\mu(B).
	\]
Here, $\dist{A}{B}$ is the distance between $A$ and $B$, that is, the infimum of $|a-b|$ over all $a\in A$ and $b\in B$. This property implies that the Borel sets are measurable with respect to the outer measures $\zeta_\ast$. The Borel $\sigma$-field is denoted by $\bor$.

\begin{Theorem}\label{thm:bormeasm}
	Let $\zeta_\ast$ be the outer measure obtained from a given premeasure $\zeta$ through~(\ref{eq:df:zetaastm}). Then, the Borel subsets of $\R^d$ are $\zeta_\ast$-measurable, that is, $\bor\subseteq\calF_{\zeta_\ast}$.
\end{Theorem}

The general theory discussed above may be applied to define the important example of Lebesgue measure and recover its main properties. The starting point is the premeasure $\upsilon$ defined on the open rectangles of $\R^d$ by
	\begin{equation}\label{eq:df:premeasLeb}
		\upsilon\left(\prod_{i=1}^d (a_i,b_i)\right)=\prod_{i=1}^d (b_i-a_i)
	\end{equation}
	for any points $(a_1,\ldots,a_d)$ and $(b_1,\ldots,b_d)$ in the space $\R^d$ such that the condition $a_i\leq b_i$ holds for any $i\in\{1,\ldots,d\}$. The {\em $d$-dimensional Lebesgue outer measure} $\leb^d$ is then defined as the outer measure on $\prd$ defined with the help of~(\ref{eq:df:zetaasta}) from the premeasure $\upsilon$, namely, $\leb^d=\upsilon^\ast$. The {\em $d$-dimensional Lebesgue measure}, still denoted by $\leb^d$, is the restriction of this outer measure to the $\sigma$-field of its measurable sets. It follows this definition that the Lebesgue outer measure is translation invariant and homogeneous of degree $d$ under dilations. Actually, the Lebesgue outer measure coincides with the outer measure defined on $\prd$ from the premeasure $\upsilon$ with the help of~(\ref{eq:df:zetaastm}), that is, $\leb^d=\upsilon_\ast$. It thus satisfy additional metric properties, and in particular the Borel sets of $\R^d$ are measurable. Finally, though this does not follow from the general theory presented above, one may of course show that $\leb^d(R)=\upsilon(R)$ for any open rectangle $R$ of $\R^d$. Similarly, the right-hand side of~(\ref{eq:df:premeasLeb}) gives the Lebesgue measure of any closed, or half-open, rectangle determined by the points $(a_1,\ldots,a_d)$ and $(b_1,\ldots,b_d)$.

\subsection{Hausdorff measures}\label{subsec:Hausdorff}

\subsubsection{Definition and main properties}

The first definitions and properties of Hausdorff measures were established by Carath\'eodory and Hausdorff. They are obtained by applying Theorem~\ref{thm:metmeas} to premeasures defined in terms of {\em gauge functions}.

\begin{Definition}
A {\em gauge function} is a function $g$ defined on $[0,\infty]$ which is nondecreasing in a neighborhood of zero and satisfies the conditions
	\[
		\lim_{r\to 0}g(r)=g(0)=0
		\qquad\text{and}\qquad
		g(\infty)=\infty.
	\]
	 The collection of all gauge functions is denoted by $\gauge$.
\end{Definition}

The convention that gauge functions take an infinite value at infinity has very little importance and is only aimed at lightening some of the statements below. Note in addition that we do not exclude {\em a priori} the possibility that a gauge function assigns an infinite value to some positive real numbers.

\begin{Definition}
	For any gauge function $g$, the {\em Hausdorff $g$-measure} $\hau^g$ is the outer measure on $\prd$ defined with the help of~(\ref{eq:df:zetaastm}) from the premeasure $g\circ\diam{\cdot}$, namely,
	\[
		\hau^g=(g\circ\diam{\cdot})_\ast,
	\]
	where $g\circ\diam{\cdot}$ is a shorthand for the premeasure defined on $\prd$ by $E\mapsto g(\diam{E})$.
\end{Definition}

It follows that the Hausdorff measures are translation invariant. Moreover, Theorem~\ref{thm:bormeasm} ensures that the Borel subsets of $\R^d$ are measurable with respect to the Hausdorff measures. Besides, for any $\delta>0$, we shall also use the outer measures
\[
\hau^g_\delta=(g\circ\diam{\cdot})_\delta
\]
defined by~(\ref{eq:df:zetaastm}) in terms of the premeasure $g\circ\diam{\cdot}$. Note that they are indeed outer measures as a result of Theorem~\ref{thm:absmeas}.

We may derive from the relative behavior at zero of two given gauge functions a comparison between the corresponding Hausdorff measures. This is the purpose of the next result; its proof is simple and therefore omitted.

\begin{Proposition}\label{prp:compgauge0}
For any gauge functions $g$ and $h$ and for any set $E\subseteq\R^d$,
\[
\left(\liminf_{r\to 0}\frac{g(r)}{h(r)}\right)\hau^h(E)\leq\hau^g(E)\leq\left(\limsup_{r\to 0}\frac{g(r)}{h(r)}\right)\hau^h(E),
\]
except if the lower or upper bound is of the indeterminate form $0\cdot\infty$, in which case the corresponding inequality has no meaning.
\end{Proposition}

Let us now explain how Hausdorff measures behave when taking the image of the set of interest under a mapping that satisfies a form of Lipschitz condition.

\begin{Proposition}\label{prp:haulip}
Let $V$ be a nonempty open subset of $\R^d$ and let $f$ be a mapping defined on $V$ with values in $\R^{d'}$. Let us assume that there exists a continuous increasing function $\ph$ defined on the interval $[0,\infty)$ such that $\ph(0)=0$ and
\[
\forall x,y\in V \qquad |f(x)-f(y)|\leq\ph(|x-y|).
\]
Then, for any gauge function $g$, the function $g\circ\ph^{-1}$ may be extended to a gauge function, and for any subset $E$ of $V$,
\[
\hau^{g\circ\ph^{-1}}(f(E))\leq\hau^g(E).
\]
\end{Proposition}

We omit the proof because it is elementary. Proposition~\ref{prp:haulip} is typically applied to mappings $f$ that are Lipschitz, or uniform H\"older; $\ph$ is then of the form $r\mapsto c\, r^\alpha$.

\subsubsection{Normalized gauge functions and net measures}\label{subsubsec:normgaugenetm}

We shall hardly be interested in the precise value of the Hausdorff $g$-measure of a set, but only in its finiteness or its positiveness. Thus, it will be useful to compare the Hausdorff $g$-measures with simpler objects obtained for instance by making further assumptions on the gauge function $g$ or the form of the coverings. This is the purpose of the next two results. The first statement calls upon the following notion of normalized gauge functions.

\begin{Definition}\label{df:normgauge}
	For any gauge function $g$, we consider the function $g_d$ defined for all real numbers $r>0$ by
	\begin{equation}\label{eq:df:normgauge}
		g_d(r)=r^d\inf_{0<\rho\leq r}\frac{g(\rho)}{\rho^d},
	\end{equation}
	along with $g_d(0)=0$ and $g_d(\infty)=\infty$\,; the function $g_d$ is then called the {\em $d$-normalization of $g$}. Moreover, we say that a gauge function is {\em $d$-normalized} if it coincides with its $d$-normalization in a neighborhood of zero.
\end{Definition}

The next result shows that the Hausdorff measure associated with some gauge function is comparable with the measure associated with its {\em $d$-normalization}. We refer to~\cite[Proposition~3.2]{Durand:2007lr} or to~\cite[Lemma~2.2]{Olsen:2006fk} for the proof.

\begin{Proposition}\label{prp:compnormalgauge}
For any gauge function $g$ in $\gauge$, the function $g_d$ defined above is a gauge function for which the mapping $r\mapsto g_d(r)/r^d$ is nonincreasing on $(0,\infty)$. Moreover, there is a real number $\kappa\geq 1$ such that for any $g\in\gauge$ and any $E\in\prd$,
\[
\hau^{g_d}(E)\leq\hau^g(E)\leq\kappa\,\hau^{g_d}(E).
\]
\end{Proposition}

The second statement shows that we may restrict our attention to coverings with dyadic cubes when estimating Hausdorff measures of sets. The main advantage of working with coverings by dyadic cubes is that they may easily be reduced to coverings by disjoint cubes; this is due to the fact that two dyadic cubes are either disjoint or contained in one another. Recall that a dyadic cube is a set of the form
\[
\lambda=2^{-j}(k+[0,1)^d),
\]
with $j\in\Z$ and $k\in\Z^d$. We also adopt the convention that the empty set is a dyadic cube. The collection of all dyadic cubes, including the empty set, is denoted by $\Lambda$. Given a gauge function $g$, let us consider the premeasure that maps each set $\lambda$ in $\Lambda$ to $g(\diam{\lambda})$, and which is denoted by $g\circ\diam{\cdot}_\Lambda$ for brevity. Then, Theorem~\ref{thm:metmeas} enables us to introduce the outer measure
	\begin{equation}\label{eq:df:netmgauge}
		\netm^g=(g\circ\diam{\cdot}_\Lambda)_\ast,
	\end{equation}
	and Theorem~\ref{thm:bormeasm} shows that the Borel sets are measurable with respect to $\netm^g$; this outer measure is usually termed as a {\em net measure}. We refer to~\cite[Theorem~49]{Rogers:1970wb} for the proof of the following result.

\begin{Proposition}\label{prp:compnetm}
There exists a real $\kappa'\geq 1$ such that for any gauge function $g$ and any subset $E$ of $\R^d$,
\[
\hau^g(E)\leq\netm^g(E)\leq\kappa'\,\hau^g(E).
\]
\end{Proposition}

\subsubsection{Connection with Lebesgue measure}\label{subsubsec:linkHauLeb}

Finally, it is important to observe that the Lebesgue measure $\leb^d$ is a particular example of Hausdorff measure. We have indeed the next statement; we refer to~\cite[Theorem~30]{Rogers:1970wb} for a proof.

\begin{Proposition}\label{prp:comphauleb}
	There exists a real number $\kappa''>0$ such that for any $B\in\bor$,
		\[
			\hau^{r\mapsto r^d}(B)=\kappa''\leb^d(B).
		\]
\end{Proposition}

If the space $\R^d$ is endowed with the Euclidean norm, it can be shown that the constant $\kappa''$ arising in the statement of Proposition~\ref{prp:comphauleb} is given by
	\[
		\kappa''=\left(\frac{4}{\pi}\right)^{d/2}\gamma_d
		\qquad\text{with}\qquad
		\gamma_d=\Gamma\left(\frac{d}{2}+1\right),
	\]
where $\Gamma$ denotes the gamma function, see~\cite[pp.~56--58]{Rogers:1970wb} for a detailed proof.

Furthermore, we shall often use the following noteworthy result for general Hausdorff measures. For any gauge function $g$ in $\gauge$, we introduce the parameter
	\begin{equation}\label{eq:df:liminfgrd}
		\ell_g=\liminf_{r\to 0}\frac{g(r)}{r^d}\in[0,\infty],
	\end{equation}
	which enables us to define the subsets of gauge functions
	\begin{equation}\label{eq:df:gaugeinftyast}
		\gauge^\infty=\{g\in\gauge\:|\:\ell_g=\infty\}
		\qquad\text{and}\qquad
		\gauge^\ast=\{g\in\gauge\:|\:\ell_g\in(0,\infty]\}
	\end{equation}

\begin{Proposition}\label{prp:dichogauge}
	For any gauge function $g$ in $\gauge$, depending on the value of $\ell_g$, one of the three following situations occurs:
	\begin{enumerate}
		\item\label{item:prp:dichogauge1} if $\ell_g=\infty$, then for any $B\in\bor$,
			\[
				\leb^d(B)>0
				\qquad\Longrightarrow\qquad
				\hau^g(B)=\infty\,;
			\]
		\item\label{item:prp:dichogauge2} if $\ell_g\in(0,\infty)$, then there is a real number $\kappa_g>0$ such that for any $B\in\bor$,
			\[
				\hau^g(B)=\kappa_g\,\leb^d(B)\,;
			\]
		\item\label{item:prp:dichogauge3} if $\ell_g=0$, then the outer measure $\hau^g$ is equal to zero.
	\end{enumerate}
\end{Proposition}

\begin{proof}
	Let $g_d$ denote the $d$-normalization of $g$. Since $g_d(r)/r^d$ tends to $\ell_g$ when $r$ goes to zero, we deduce from Propositions~\ref{prp:compgauge0} and~\ref{prp:comphauleb} that $\hau^{g_d}(B)$ is equal to $\kappa''\ell_g\leb^d(B)$ for any Borel set $B\in\bor$, except if latter value is of the indeterminate form $0\cdot\infty$. Along with Proposition~\ref{prp:compnormalgauge}, this directly yields~(\ref{item:prp:dichogauge1}).
	
	Let us now assume that $\ell_g$ is finite. It follows from Proposition~\ref{prp:compnormalgauge} that
		\[
			\kappa_g=\hau^g([0,1)^d)\leq\kappa\,\hau^{g_d}([0,1)^d)=\kappa\kappa''\ell_g\leb^d([0,1)^d)=\kappa\kappa''\ell_g<\infty.
		\]
		In particular, $\kappa_g$ vanishes when $\ell_g$ does. The countable subadditivity and the translation invariance of $\hau^g$ then lead to~(\ref{item:prp:dichogauge3}). Finally, note that $\kappa_g$ is both positive and finite when $\ell_g$ is. Indeed, in addition to the previous bound, we have
		\[
			\kappa_g=\hau^g([0,1)^d)\geq\hau^{g_d}([0,1)^d)=\kappa''\ell_g\leb^d([0,1)^d)=\kappa''\ell_g>0.
		\]
		The measurability of the dyadic cubes with respect to $\hau^g$ and the translation invariance of that outer measure imply that $\hau^g(\lambda)$ equals $\kappa_g\leb^d(\lambda)$ for any dyadic cube $\lambda$. These cubes generate the Borel $\sigma$-field, so we deduce~(\ref{item:prp:dichogauge2}) as in the proof of~\cite[Theorem~30]{Rogers:1970wb}, or from the uniqueness of extension lemma, see {\em e.g.}~\cite[Lemma~1.6(a)]{Williams:1991hb}.
\end{proof}

\subsection{Hausdorff dimension}

The Hausdorff measures associated with gauge functions enable to give a precise description of the size of a subset of $\R^d$. However, it is arguably more intuitive, and often sufficient, to restrict to a specific class of gauge functions, namely, the power functions $r\mapsto r^s$, for $s>0$. This approach gives rise to the notion of Hausdorff dimension.

For these particular gauge functions, we use the notation $\hau^s$ instead of $\hau^{r\mapsto r^s}$, for brevity, and we call this outer measure the $s$-dimensional Hausdorff measure. It is clear that the gauge function $r\mapsto r^s$ is normalized if and only if $s\leq d$; when $s>d$, the corresponding $d$-normalization is the zero function and, on account of Proposition~\ref{prp:compnormalgauge}, the $s$-dimensional Hausdorff measure is constant equal to zero.

Specializing Proposition~\ref{prp:compgauge0} to the power gauge functions, we see that for any nonempty set $E\subseteq\R^d$, there is a critical value $s_0\in [0,d]$ such that the $s$-dimensional Hausdorff measure of $E$ is infinite when $s\in(0,s_0)$, and zero when $s>s_0$. This observation yields the notion of Hausdorff dimension.

\begin{Definition}\label{df:Hausdorffdim}
	The {\em Hausdorff dimension} of a nonempty set $E\subseteq\R^d$ is defined by
	\[
		\Hdim E=
		\sup\{s\in(0,d]\:|\:\hau^s(E)=\infty\}=\inf\{s\in(0,d]\:|\:\hau^s(E)=0\}.
	\]
\end{Definition}

We adopt the convention that the supremum and the infimum are equal to zero and $d$, respectively, if the inner sets are empty. Moreover, the Hausdorff dimension of the empty set is $-\infty$. Specializing the results of Section~\ref{subsec:Hausdorff} to the power gauge functions leads to the following proposition. We recall that a mapping $f$ defined on an open set $V\subseteq\R^d$ and valued in $\R^{d'}$ is {\em bi-Lipschitz with constant $c_f$} if
	\begin{equation}\label{eq:df:biLip}
		\forall x,y\in V \qquad \frac{|x-y|}{c_f}\leq|f(x)-f(y)|\leq c_f|x-y|.
	\end{equation}

\begin{Proposition}\label{prp:Hausdorff}
	Hausdorff dimension satisfies the following properties.
	\begin{enumerate}
		\item\label{item:prp:Hausdorff1} {\em Monotonicity}: for any subsets $E_1$ and $E_2$ of $\R^d$,
		\[
		E_1\subseteq E_2 \qquad\Longrightarrow\qquad \Hdim E_1\leq\Hdim E_2.
		\]
		\item\label{item:prp:Hausdorff2} {\em Countable stability}: for any sequence $(E_n)_{n\geq 1}$ of subsets of $\R^d$,
		\[
		\Hdim\bigcup_{n=1}^\infty E_n=\sup_{n\geq 1}\Hdim E_n.
		\]
		\item {\em Countable sets}: if $E\subseteq\R^d$ is nonempty and countable, then $\Hdim E=0$.
		\item {\em Sets with positive Lebesgue measure}: if a subset $E$ of $\R^d$ has positive Lebesgue measure, then $\Hdim E=d$.
		\item {\em Invariance under bi-Lipschitz mappings}: if $V$ is an open subset of $\R^d$ and $f:V\to\R^{d'}$ is a bi-Lipschitz mapping, then for any subset $E$ of $V$,
		\[
		\Hdim f(E)=\Hdim E.
		\]
	\end{enumerate}
\end{Proposition}

Finally, we directly infer from Proposition~\ref{prp:compnetm} that the formula in Definition~\ref{df:Hausdorffdim} is still valid when $\hau^s$ is replaced by $\netm^s$, where $\netm^s$ is a shorthand for the net measure $\netm^{r\mapsto r^s}$ introduced in Section~\ref{subsubsec:normgaugenetm}.

\subsection{Upper bounds for limsup sets}\label{subsec:upbndhausdorff}

Deriving upper bounds on Hausdorff dimensions or, more generally, obtaining an upper bound on the Hausdorff measure of a set is usually elementary: it suffices to make use of an appropriate covering of the set. There is a situation that we shall often encounter where the covering is natural: when the set under study is a limsup of simpler sets, such as balls for instance. Let us recall that the limsup of a sequence $(E_n)_{n\geq 1}$ of subsets of $\R^d$ is
	\[
		\limsup_{n\to\infty} E_n=\bigcap_{m=1}^\infty\bigcup_{n=m}^\infty E_n,
	\]
	and consists of the points that belong to $E_n$ for infinitely many values of $n$. We then have the following elementary result.

\begin{Lemma}\label{lem:upbndlimsup}
	For any sequence $(E_n)_{n\geq 1}$ of subsets of $\R^d$ and for any gauge function $g$, the following implication holds:
	\[
		\sum_{n=1}^\infty g(\diam{E_n})<\infty
		\qquad\Longrightarrow\qquad
		\hau^g\left(\limsup_{n\to\infty} E_n\right)=0.
	\]
\end{Lemma}

\begin{proof}
	Let us consider a real $\delta>0$ and a gauge function $g$ such that the series $\sum_{n} g(\diam{E_n})$ converges. In particular, $g(\diam{E_n})$ tends to zero as $n\to\infty$; thus, unless $g$ is the zero function in a neighborhood of the origin, in which case the result is trivial, we deduce that $\diam{E_n}\leq\delta$ for all $n$ larger than some integer $n_0\geq 1$. We then choose an integer $m>n_0$ and cover $E$ by the sets $E_n$, for $n\geq m$, thereby obtaining
	\[
		\hau^g_\delta(E)\leq\sum_{n=m}^\infty g(\diam{E_n}).
	\]
	The series being convergent, the right-hand side tends to zero as $m\to\infty$, and the result follows from letting $\delta$ tend to zero.
\end{proof}

Specializing the above result to power gauge functions, we directly deduce an upper bound on the Hausdorff dimension of the limsup of the sets $E_n$, namely,
	\[
		\Hdim\left(\limsup_{n\to\infty}E_n\right)
		\leq\inf\left\{s>0\Biggm|\sum_{n=1}^\infty \diam{E_n}^s<\infty\right\}.
	\]

A typical application of Lemma~\ref{lem:upbndlimsup} is the derivation of an upper bound on the Hausdorff dimension of the homogeneous approximation set $J_{d,\tau}$ defined by~(\ref{eq:df:Jdtau}). Recall that this set is equal to the whole space $\R^d$ when $\tau\leq 1+1/d$, so that we may suppose that we are in the opposite case. It is then clear that
	\[
		J_{d,\tau}=\bigcup_{k\in\Z^d}(k+J'_{d,\tau})
		\qquad\text{with}\qquad
		J'_{d,\tau}=\bigcap_{Q=1}^\infty\bigcup_{q=Q}^\infty\bigcup_{p\in\{0,\ldots,q\}^d} \opball_\infty\left(\frac{p}{q},\frac{1}{q^\tau}\right).
	\]
	The set $J'_{d,\tau}$ is the limsup of the balls $\opball_\infty(p/q,q^{-\tau})$, for $p\in\{0,\ldots,q\}^d$ and $q\geq 1$. The previous bound prompts us to examine the convergence of $\sum_q (q+1)^d (2/q^\tau)^s$, for $s>0$. Obviously this series converges if $s$ is larger than $(d+1)/\tau$. We deduce that the set $J'_{d,\tau}$ has Hausdorff dimension bounded above by this value. The countable stability of Hausdorff dimension, namely, Proposition~\ref{prp:Hausdorff}(\ref{item:prp:Hausdorff2}) finally yields
	\begin{equation}\label{eq:upbndJdtau}
		\Hdim J_{d,\tau}\leq\frac{d+1}{\tau}.
	\end{equation}

Likewise, we may also compute an upper bound on the Hausdorff dimension of a very classical fractal set: the middle-third Cantor set, denoted by $\cantor$. There are several ways of writing this set; the most suitable for dimension estimates is
	\[
		\cantor=\bigcap_{j=0}^\infty\downarrow\bigsqcup_{u\in\{0,1\}^j}I_u.
	\]
	Here, $I_u$ denotes the closed interval with left endpoint $2u_1/3+\ldots+2u_j/3^j$ and length $3^{-j}$, if $u$ is the word $u_1\ldots u_j$ in $\{0,1\}^j$. For consistency, we adopt the convention that the set $\{0,1\}^0$ contains only one element, the empty word $\varnothing$, and that the set $I_\varnothing$ is equal to the whole interval $[0,1]$. Note that every point of the Cantor set $\cantor$ belongs to one of the intervals $I_u$ with $u\in\{0,1\}^j$, for every integer $j\geq 0$. In particular, $\cantor$ is the limsup of the intervals $I_u$. In the spirit of the previous general bound, we need to inspect the convergence of $\sum_j 2^j(3^{-j})^s$. Accordingly,
	\begin{equation}\label{eq:upbndcantor}
		\Hdim\cantor\leq\frac{\log 2}{\log 3}.
	\end{equation}

\subsection{Lower bounds: the mass distribution principle}\label{subsec:lobndhausdorff}

Whereas deriving upper bounds on Hausdorff dimensions often amounts to finding appropriate coverings, a standard way of establishing lower bounds is to build a clever outer measure on the set under study. This remark is embodied by the next simple, but crucial, result.

\begin{Lemma}[mass distribution principle]\label{lem:massdist}
	Let $E$ be a subset of $\R^d$, let $\mu$ be an outer measure on $\R^d$ such that $\mu(E)>0$, and let $g$ be a gauge function. If there exists a real number $\delta_0>0$ such that for any subset $C$ of $\R^d$,
	\[
		\diam{C}\leq\delta_0
		\qquad\Longrightarrow\qquad
		\mu(C)\leq g(\diam{C}),
	\]
	then the set $E$ has positive Hausdorff $g$-measure, specifically,
	\[
		\hau^g(E)\geq\mu(E)>0.
	\]
\end{Lemma}

\begin{proof}
Let us consider a real number $\delta\in(0,\delta_0]$ and a sequence $(C_n)_{n\geq 1}$ of subsets of $\R^d$ with diameter at most $\delta$ such that $E\subseteq\bigcup_n C_n$. Then,
	\[
		\mu(E)\leq\mu\left(\bigcup_{n=1}^\infty C_n\right)
		\leq\sum_{n=1}^\infty\mu(C_n)\leq\sum_{n=1}^\infty g(\diam{C_n}).
	\]
	The result follows from taking the infimum over all $(C_n)_{n\geq 1}$ and letting $\delta\to 0$.
\end{proof}

In particular, if the gauge function is of the form $r\mapsto c\,r^s$ for some $c>0$ and $s\in(0,d]$, the set $E$ has positive $s$-dimensional Hausdorff measure, and thus Hausdorff dimension at least $s$. Applying this idea to the Cantor set yields
	\begin{equation}\label{eq:lobndcantor}
		\Hdim\cantor\geq\frac{\log 2}{\log 3},
	\end{equation}
	a bound that matches~(\ref{eq:upbndcantor}). Indeed, let $\calC$ denote the collection formed by the empty set and all the intervals $I_u$, for $u\in\{0,1\}^j$ and $j\geq 0$. We define a premeasure $\zeta$ on $\calC$ by letting $\zeta(\emptyset)=0$, and $\zeta(I_u)=2^{-j}$ if the word $u$ has length $j$. Theorem~\ref{thm:absmeas} enables us to extend via the formula~(\ref{eq:df:zetaasta}) the premeasure $\zeta$ to an outer measure $\zeta^\ast$ on all the subsets of $\R$. One then easily checks that the function $\mu$ that maps a subset $E$ of $\R$ to the value $\zeta^\ast(E\cap\cantor)$ is also an outer measure. Now, given a subset $C$ of $\R$ with diameter at most one, let us derive an appropriate upper bound on $\mu(C)$. We may clearly assume that $C\cap\cantor$ is nonempty, as $\mu(C)$ vanishes otherwise. Moreover, if $C$ has positive diameter, there is a unique integer $j\geq 0$ such that $3^{-(j+1)}\leq\diam{C}<3^{-j}$. The intervals $I_u$, for $u\in\{0,1\}^j$, are separated by a distance at least $3^{-j}$. Hence, the set $C$ intersects only one of these intervals, which is denoted by $I(C)$. Therefore, $C\cap\cantor$ is included in $I(C)$, so that
	\[
		\mu(C)=\zeta^\ast(C\cap\cantor)\leq\zeta(I(C))=2^{-j}=(3^{-j})^s\leq 3^s\diam{C}^s=2\diam{C}^s,
	\]
	where $s$ is equal to $\log 2/\log 3$. The same bound holds when $C$ has diameter zero. Actually, in that case, $C$ is reduced to a single point in $\cantor$. For each integer $j\geq 0$, there is a unique $u\in\{0,1\}^j$ such that this point belongs to $I_u$, so that
	\[
		\mu(C)=\zeta^\ast(C\cap\cantor)\leq\zeta(I_u)=2^{-j}\xrightarrow[j\to\infty]{}0.
	\]
	We conclude by observing that $\mu(\cantor)\geq 1$,; this is in fact a consequence of Lemma~\ref{lem:basictree} below. By Lemma~\ref{lem:massdist}, this implies that $\hau^s(\cantor)\geq 1/2$, and hence~(\ref{eq:lobndcantor}).

\subsection{The general Cantor construction}\label{subsec:genCantor}

This is an extension of the above approach. This construction is indexed by a {\em tree}, that is, a subset $T$ of the set
	\[
		\U=\bigcup_{j=0}^\infty\N^j
	\]
	such that the three following properties hold:
	\begin{itemize}
		\item The empty word $\varnothing$ belongs to $T$.
		\item If the word $u=u_1\ldots u_j$ is not empty and belongs to $T$, then the word $\pi(u)=u_1\ldots u_{j-1}$ also belongs to $T$; this word is the {\em parent} of $u$.
		\item For every word $u$ in $T$, there exists an integer $k_u(T)\geq 0$ such that the word $uk$ belongs to $T$ if and only if $1\leq k\leq k_u(T)$; the number of {\em children} of $u$ in $T$ is then equal to $k_u(T)$.
	\end{itemize}
	Let us recall here that, in accordance with a convention adopted previously, the set $\N^0$ arising in the definition of $\U$ is reduced to the singleton $\{\varnothing\}$; the empty word $\varnothing$ clearly corresponds to the {\em root} of the tree.

To each element $u$ of the tree $T$, we may then associate a compact subset $I_u$ of $\R^d$, and a possibly infinite nonnegative value $\zeta(I_u)$. Defining in addition $\zeta(\emptyset)=0$, we thus obtain a premeasure $\zeta$ on the collection $\calC$ formed by the empty set together with all the sets $I_u$. We assume these objects are compatible with the tree structure, in the sense that for every $u\in T$,
	\begin{equation}\label{eq:compattree}
		I_u\supseteq\bigsqcup_{k=1}^{k_u(T)} I_{uk}
		\qquad\text{and}\qquad
		\zeta(I_u)\leq\sum_{k=1}^{k_u(T)}\zeta(I_{uk}).
	\end{equation}
	In particular, nodes $u\in T$ such that $k_u(T)$ vanishes, {\em i.e.}~childless nodes, are not excluded {\em a priori} but the corresponding sets necessarily satisfy $\zeta(I_u)=0$. More generally, $\zeta(I_u)$ surely vanishes when the subtree of $T$ formed by the descendants of $u$ is finite; this is easily seen by induction on the height of this subtree. Besides, the second condition in~(\ref{eq:compattree}) may easily be replaced by an equality. Indeed, it suffices to replace $\zeta$ by the premeasure $\xi$ defined on $\calC$ by $\xi(I_\varnothing)=\zeta(I_\varnothing)$ and
	\[
		\xi(I_{uk})=\frac{\zeta(I_{uk})}{\sum\limits_{l=1}^{k_u(T)}\zeta(I_{ul})}\xi(I_u),
	\]
	for $u\in T$ and $k\in\{1,\ldots,k_u(T)\}$. When the denominator vanishes, the numerator vanishes as well, and we adopt the convention that the quotient is zero. Note that the premeasure thus obtained bounds $\zeta$ from below.

Thanks to Theorem~\ref{thm:absmeas}, we may then extend the premeasure $\zeta$ to an outer measure $\zeta^\ast$ on all the subsets of $\R^d$ through the formula~(\ref{eq:df:zetaasta}). This finally enables us to consider the limiting set
	\begin{equation}\label{eq:limittree}
		K=\bigcap_{j=0}^\infty\downarrow\bigsqcup_{u\in T\cap\N^j}I_u,
	\end{equation}
	together with the outer measure $\mu$ that maps a set $E\subseteq\R^d$ to the value $\zeta^\ast(E\cap K)$. If the tree $T$ is finite, it is clear that $K$ is empty and $\mu$ is the zero measure, and so the construction is pointless. The next result discusses the basic properties of $K$ and $\mu$ in the opposite situation; we leave its proof to the reader.

\begin{Lemma}\label{lem:basictree}
	Let us assume that the tree $T$ is infinite. Then, $K$ is a nonempty compact subset of $I_\varnothing$. Moreover, the outer measure $\mu$ has total mass $\mu(K)=\zeta(I_\varnothing)$.
\end{Lemma}

The next result shows that, under further conditions on the compact sets $I_u$, we may derive a lower bound on the Hausdorff dimension of the limiting set $K$. In its statement, $(\eps_j)_{j\geq 1}$ and $(m_j)_{j\geq 1}$ are the sequences defined by
	\[
		\eps_j=\min_{u,v\in T\cap\N^j\atop u\neq v}\dist{I_u}{I_v}
		\qquad\text{and}\qquad
		m_j=\min_{u\in T\cap\N^{j-1}}k_u(T).
	\]

\begin{Lemma}\label{lem:mindimgenCantor}
	Let us assume that the sequence $(\eps_j)_{j\geq 1}$ is decreasing and that the sequence $(m_j)_{j\geq 1}$ is positive. Then,
	\[
		\Hdim K\geq\liminf_{j\to\infty}\frac{\log(m_1\ldots m_{j-1})}{-\log(m_j^{1/d}\eps_j)}.
	\]
\end{Lemma}

We omit the proof from these notes, and we content ourselves with mentioning that it relies crucially on Lemma~\ref{lem:massdist}, namely, the mass distribution principle. We refer to~\cite[Example~4.6]{Falconer:2003oj} for a discussion of the one-dimensional case.


\section{Homogeneous ubiquity and dimensional results}\label{sec:firstubiq}

The purpose of this section is twofold. First, we present an abstract framework which encompasses and naturally extends the set of points that are approximable at rate at least $\tau$ by points with rational coordinates; recall that this set is denoted by $J_{d,\tau}$ and defined by~(\ref{eq:df:Jdtau}), and arises in the classical homogeneous approximation problem. Second, we present a general method to determine the Hausdorff dimension of the sets that fit into this framework. This will enable us in particular to determine the Hausdorff dimension of the set $J_{d,\tau}$, thereby recovering a famous result due to Jarn\'ik~\cite{Jarnik:1929mf} and Besicovitch~\cite{Besicovitch:1934ly}, see Section~\ref{subsec:JarnikBesicovitch}.

We also recall from Section~\ref{subsec:verywell} that the set $J_{d,\tau}$ coincides with the whole space $\R^d$ when the parameter $\tau$ is not larger than $1+1/d$. This follows from Dirichlet's theorem, through Corollary~\ref{cor:Dirichlet}. Hence, every point in $\R^d$ is approximated by points with rational coordinates at rate at least $1+1/d$. The first step is then to identify an appropriate notion of {\em approximation system} to generalize the combination of the points $p/q$ with the radii $1/q^{1+1/d}$ for which the uniform approximation is achieved. The second step is to introduce natural generalizations of the smaller sets $J_{d,\tau}$, for $\tau>1+1/d$. The third step is finally to provide optimal upper and lower bounds on the Hausdorff dimension of these generalized sets.

As explained hereunder, through the remarkable notion of ubiquity, an {\em a priori} lower bound on the Hausdorff dimension can be derived from the sole knowledge that one of the sets has full Lebesgue measure. Thanks to ubiquity, the arguably difficult lower bound in the Jarn\'ik-Besicovitch theorem will in fact quite amazingly be a straightforward consequence of a simple result, namely, Dirichlet's theorem.

\subsection{Approximation system}

This notion is modeled on the emblematic example that consists of the pairs $(p/q,1/q^{1+1/d})$, for $p\in\Z^d$ and $q\in\N$, on which the sets $J_{d,\tau}$ are based. We shall present many other examples throughout these notes.

\begin{Definition}\label{df:approxsys}
Let $\calI$ be a countably infinite index set. We say that a family $(x_i,r_i)_{i\in\calI}$ of elements of $\R^d\times(0,\infty)$ is an {\em approximation system} if
	\[
		\sup_{i\in\calI} r_i<\infty
		\qquad\text{and}\qquad
		\forall m\in\N \quad \#\left\{i\in\calI\Biggm| |x_i|<m \text{ and } r_i>\frac{1}{m} \right\}<\infty.
	\]
\end{Definition}

Let us point out that we do not need to specify the norm $|\,\cdot\,|$ the space $\R^d$ is endowed with. In fact, Proposition~\ref{prp:homubsyscst} below implies that the notions considered in this section do not depend on the chosen norm. Now, replacing the system supplied by the rational points by an arbitrary approximation system $(x_i,r_i)_{i\in\calI}$, we may naturally rewrite the sets $J_{d,\tau}$, for $\tau\geq 1+1/d$, in the form
	\begin{equation}\label{eq:df:Ft}
		F_t=\left\{x\in\R^d\bigm||x-x_i|<r_i^t\quad\text{for i.m.~} i\in\calI\right\},
	\end{equation}
	where $t\geq 1$. If $x$ is in $F_t$, then there exists an injective sequence $(i_n)_{n\geq 1}$ of indices in $\calI$ such that $|x-x_{i_n}|<r_{i_n}^t$ for all $n\geq 1$. Since $(x_i,r_i)_{i\in\calI}$ is an approximation system, for any $\eps>0$ and any $n\geq 1$ such that $r_{i_n}>\eps$, we have
	\[
		|x_{i_n}|\leq |x|+|x-x_{i_n}|<|x|+\sup_{i\in\calI} r_i^t.
	\]
	Thus, letting $m$ denote an integer larger than both $1/\eps$ and the right-hand side above, we deduce that $|x_{i_n}|<m$ and $r_{i_n}>1/m$, which means that there are only finitely many possible values of the integer $n$ when $\eps$ is given. We readily deduce that $r_{i_n}$ tends to zero and $x_{i_n}$ tends to $x$ as $n\to\infty$. The point $x$ is thus approximated by the sequence $(x_{i_n})_{n\geq 1}$ at a rate given by the sequence $(r_{i_n}^t)_{n\geq 1}$\,; this justifies the terminology of the previous definition. Moreover, it is obvious and useful to remark that, up to extracting, we may suppose that the latter sequence is decreasing without losing the approximation property.

Our main purpose is now to give an upper and a lower bound on the Hausdorff dimension of the sets $F_t$ under appropriate assumptions on the approximation system $(x_i,r_i)_{i\in\calI}$. In Sections~\ref{sec:largeint} and~\ref{sec:transference}, we subsequently extend these bounds toward large intersection properties and general Hausdorff measures.

\subsection{Upper bound on the Hausdorff dimension}

The preceding discussion shows that the sets $F_t$ are essentially of limsup type, thereby falling in the framework dealt with in Section~\ref{subsec:upbndhausdorff}. More precisely, for any bounded open set $U\subseteq\R^d$, let
	\[
		\calI_U=\{i\in\calI\:|\: x_i\in U\}.
	\]
	If a given point $x$ belongs to $F_t\cap U$, the above remark ensures that there exists a sequence $(i_n(x))_{n\geq 1}$ of indices in $\calI$ such that $x_{i_n(x)}$ tends to $x$ as $n\to\infty$. As the set $U$ is open, the indices $i_n(x)$ thus belong to $\calI_U$ for $n$ sufficiently large. On top of that, for any real number $\eps>0$, we have
	\[
		\#\{i\in\calI_U\:|\: r_i>\eps\}
		\leq\#\left\{i\in\calI\Biggm| |x_i|<m \text{ and } r_i>\frac{1}{m} \right\}<\infty
	\]
	for $m$ large enough. We may thus find an enumeration $(i_n)_{n\geq 1}$ of the set $\calI_U$ such that the sequence $(r_{i_n})_{n\geq 1}$ is nonincreasing and tends to zero at infinity. We finally end up with an approximate local expression of $F_t$ as a limsup set, namely,
	\begin{equation}\label{eq:sandwichlimsup}
		F_t\cap U\subseteq
		\limsup_{n\to\infty}\opball(x_{i_n},r_{i_n}^t)
		\subseteq F_t\cap\closure{U},
	\end{equation}
where $\closure{U}$ stands for the closure of the open set $U$.

In view of Section~\ref{subsec:upbndhausdorff}, it is thus natural to examine the convergence of the series $\sum_n\diam{\opball(x_{i_n},r_{i_n}^t)}^s$, for $s>0$. To be more specific, making a convenient change of variable, this amounts to considering the infimum of all $s$ such that $\sum_{i\in\calI_U} r_i^s$ converges. Note that this infimum is clearly a nondecreasing function of $U$. In order to cover the case where $U$ is unbounded, and maybe also obtain a better value in the bounded case, we finally introduce the exponent
	\begin{equation}\label{eq:df:sU}
		s_U=
		\inf_{U=\bigcup_\ell U_\ell}
		\sup_{\ell\geq 1}
		\inf\left\{s>0\Biggm|\sum_{i\in\calI_{U_\ell}} r_i^s<\infty\right\},
	\end{equation}
	where the infimum is taken over all sequences $(U_\ell)_{\ell\geq 1}$ of bounded open sets whose union is equal to $U$. Our approach thus leads to the following statement.

\begin{Proposition}\label{prp:upbndapproxsys}
For any approximation system $(x_i,r_i)_{i\in\calI}$, any open subset $U$ of $\R^d$ and any real number $t\geq 1$,
\[
\Hdim(F_t\cap U)\leq\frac{s_U}{t}.
\]
\end{Proposition}

\begin{proof}
	Let $(U_\ell)_{\ell\geq 1}$ denote a sequence of bounded open sets whose union is equal to $U$. For any $\ell\geq 1$, the open set $U_\ell$ is bounded, so the inclusions~(\ref{eq:sandwichlimsup}) are valid. As a consequence, if $s$ is a positive real number such that $\sum_{i\in\calI_{U_\ell}} r_i^s$ converges, we may apply Lemma~\ref{lem:upbndlimsup} with the gauge function $r\mapsto r^{s/t}$, thereby deducing that the set $F_t\cap U_\ell$ has dimension at most $s/t$. We conclude thanks to the countable stability of Hausdorff dimension, namely, Proposition~\ref{prp:Hausdorff}(\ref{item:prp:Hausdorff2}).
\end{proof}

In most situations, the na\"ive bound supplied by Proposition~\ref{prp:upbndapproxsys} gives the exact value of Hausdorff dimension, and moreover the parameter $s_U$ does not depend on the choice of the open set $U$. This happens for instance when the approximation system are derived from eutaxic sequences or optimal regular systems; these two notions are discussed in Sections~\ref{sec:eutaxy} and~\ref{sec:optregsys}, respectively.

\subsection{Lower bound on the Hausdorff dimension}

Our goal is now to establish a lower bound on the Hausdorff dimension of the set $F_t$ defined by~(\ref{eq:df:Ft}) under the following simple assumption on the underlying approximation system $(x_i,r_i)_{i\in\calI}$.

\begin{Definition}\label{df:homubiqsys}
	Let $\calI$ be a countably infinite index set, let $(x_i,r_i)_{i\in\calI}$ be an approximation system in $\R^d\times(0,\infty)$ and let $U$ be a nonempty open subset of $\R^d$. We call $(x_i,r_i)_{i\in\calI}$ a {\em homogeneous ubiquitous system} in $U$ if the set $F_1$ has full Lebesgue measure in $U$, {\em i.e.}
	\[
		\text{for~}\leb^d\text{-a.e.~}x\in U
		\quad \exists\text{~i.m.~}i\in\calI
		\qquad |x-x_i|<r_i.
	\]
\end{Definition}

Note that we do not impose that all the points $x_i$ belong to the open set $U$. Actually, the approximation system is usually fixed at the beginning, and the open set is then allowed to change so that one can examine local approximation properties. Moreover, the fact that a given approximation system $(x_i,r_i)_{i\in\calI}$ is homogeneously ubiquitous ensures that the approximating points $x_i$ are well spread, in accordance with the corresponding approximation radii $r_i$. The following remarkable result, established to Jaffard~\cite{Jaffard:2000mk}, shows that this assumption suffices to obtain an {\em a priori} lower bound on the Hausdorff dimension of the sets $F_t$.

\begin{Theorem}\label{thm:lobndhomubsys}
	Let $(x_i,r_i)_{i\in\calI}$ denote a homogeneous ubiquitous system in some nonempty open subset $U$ of $\R^d$. Then, for any real number $t>1$,
	\[
		\Hdim(F_t\cap U)\geq\frac{d}{t}.
	\]
	More precisely, the set $F_t\cap U$ has positive Hausdorff measure with respect to the gauge function $r\mapsto r^{d/t}|\log r|$.
\end{Theorem}

Combining Theorem~\ref{thm:lobndhomubsys} with Proposition~\ref{prp:upbndapproxsys} above, we remark that if $(x_i,r_i)_{i\in\calI}$ is a homogeneous ubiquitous system in $U$, then the parameter $s_U$ defined by~(\ref{eq:df:sU}) is necessarily bounded below by $d$. We also readily deduce the following result.

\begin{Corollary}\label{cor:eqhomubsys}
	Let $(x_i,r_i)_{i\in\calI}$ denote a homogeneous ubiquitous system in some nonempty open subset $U$ of $\R^d$. Let us assume that $s_U\leq d$. Then, for any $t>1$,
	\[
		\Hdim(F_t\cap U)=\frac{d}{t}.
	\]
\end{Corollary}

Again, an emblematic situation where this holds is when the approximation system are issued from eutaxic sequences or optimal regular systems, see Sections~\ref{sec:eutaxy} and~\ref{sec:optregsys}. We shall prove an extension of Theorem~\ref{thm:lobndhomubsys} to the framework of sets with large intersection in Section~\ref{sec:largeint}, see Theorem~\ref{thm:lihomubsys} for a precise statement. In addition, the transference principles discussed in Section~\ref{sec:transference} natural generalize these results to Hausdorff measures and large intersection classes associated with arbitrary gauge functions. Besides, let us mention that a heterogeneous and a localized version of Theorem~\ref{thm:lobndhomubsys} are established in~\cite{Barral:2004ae} and~\cite{Barral:2011aa}, respectively.

The remainder of this section is devoted to the proof of Theorem~\ref{thm:lobndhomubsys}. We thus fix a homogeneous ubiquitous system $(x_i,r_i)_{i\in\calI}$ and a nonempty open subset $U$ of $\R^d$. We may obviously assume that $U$ has diameter at most one. Consequently, the index set $\calI_U$ admits an enumeration $(i_n)_{n\geq 1}$ such that the sequence $(r_{i_n})_{n\geq 1}$ is nonincreasing and tends to zero at infinity.

\subsubsection{A covering lemma}

The proof of Theorem~\ref{thm:lobndhomubsys} calls upon a simple result in the spirit of Vitali's covering lemma but with a measure theoretic flavor.

\begin{Lemma}\label{lem:lobndhomubsys}
	For any nonempty open set $V\subseteq U$ and any $\rho>0$, there is a finite set $\calI(V,\rho)\subseteq\calI_U$ such that $r_i\leq\rho$ for all $i\in\calI(V,\rho)$, and
	\[
		\bigsqcup_{i\in\calI(V,\rho)}\clball(x_i,r_i)\subseteq V
		\qquad\text{and}\qquad
		\sum_{i\in\calI(V,\rho)}\leb^d(\clball(x_i,r_i))\geq\frac{\leb^d(V)}{2\cdot 3^d}.
	\]
\end{Lemma}

\begin{proof}
	For any $\rho>0$, there exists an integer $n_\rho\geq 1$ such that $r_{i_n}\leq\rho$ for all $n\geq n_\rho$. We observe that $(x_{i_n},r_{i_n})_{n\geq n_\rho}$ is a homogeneous ubiquitous system in $U$. Hence, every nonempty open set $V\subseteq U$ necessarily contains a closed ball of the form $\clball(x_{i_n},r_{i_n})$, for $n\geq n_\rho$. Indeed, any such open set $V$ contains an open ball of the form $\opball(x_0,r_0)$, and the smaller ball $\opball(x_0,r_0/2)$ contains a point $x$ that belongs to infinitely many open balls of the form $\opball(x_{i_n},r_{i_n})$ with $n\geq n_\rho$\,; choosing $n$ so large that $r_{i_n}$ is smaller than $r_0/4$, we may use the point $x$ to ensure that
	\[
		\clball(x_{i_n},r_{i_n})\subseteq\opball(x_0,r_0)\subseteq V.
	\]
	
	Therefore, if $V$ denotes a nonempty open subset of $U$, we can define
	\[
		n_1=\min\left\{n\geq n_\rho\bigm|\clball(x_{i_n},r_{i_n})\subseteq V\right\}.
	\]
	For any integer $K\geq 1$, the same argument allows us to define in a recursive manner
	\[
		n_{K+1}=\min\left\{n>n_K\Biggm|
		\clball(x_{i_n},r_{i_n})\subseteq V\setminus\bigcup_{k=1}^K \clball(x_{i_{n_k}},r_{i_{n_k}})
		\right\}.
	\]
	We thus obtain a increasing sequence of positive integers $(n_K)_{K\geq 1}$. Then, recalling that the radii $r_{i_n}$ monotonically tend to zero as $n\to\infty$, we infer that
	\begin{equation}\label{eq:incUlimsup}
		V\cap\limsup_{n\to\infty} \opball(x_{i_n},r_{i_n})
		\subseteq
		\bigcup_{k=1}^\infty \clball(x_{i_{n_k}},3r_{i_{n_k}}).
	\end{equation}
	Indeed, if $x$ belongs to the set in the left-hand side of~(\ref{eq:incUlimsup}), we necessarily have $x\in\clball(x_{i_n},r_{i_n})\subseteq V$ for some sufficiently large integer $n\geq n_1$. Letting $K$ denote the unique integer such that $n_K\leq n<n_{K+1}$, we deduce from the mere definition of $n_{K+1}$ that the ball $\clball(x_{i_n},r_{i_n})$ meets at least one of the balls $\clball(x_{i_{n_k}},r_{i_{n_k}})$, for $k\in\{1,\ldots,K\}$, at some point denoted by $y$. Hence,
	\[
		|x-x_{i_{n_k}}|\leq|x-x_{i_n}|+|x_{i_n}-y|+|y-x_{i_{n_k}}|\leq r_{i_n}+r_{i_n}+r_{i_{n_k}}\leq 3r_{i_{n_k}},
	\]
	where the latter bound results from the fact that $n\geq n_K\geq n_k$ and that the radii are nonincreasing. We deduce that $x$ belongs to the right-hand side of~(\ref{eq:incUlimsup})
	
	Finally, since $(x_{i_n},r_{i_n})_{n\geq 1}$ is a homogeneous ubiquitous system in $U$, the left-hand side of~(\ref{eq:incUlimsup}) has Lebesgue measure equal to $\leb^d(V)$. Consequently, along with~(\ref{eq:incUlimsup}), the subadditivity and dilation behavior of Lebesgue measure imply that
	\[
		\leb^d(V)\leq\leb^d\left(\bigcup_{k=1}^\infty\clball(x_{i_{n_k}},3r_{i_{n_k}})\right)
		\leq 3^d\sum_{k=1}^\infty\leb^d(\clball(x_{i_{n_k}},r_{i_{n_k}})).
	\]
	To conclude, we define $K$ as the first integer such that the $K$-th partial sum of the series appearing in the right-hand side exceeds $\leb^d(V)/(2\cdot 3^d)$, and then $\calI(V,\rho)$ as the set of all indices $i_{n_k}$, for $k\in\{1,\ldots,K\}$.
\end{proof}

\subsubsection{The ubiquity construction}

After fixing a real number $t>1$, the proof of Theorem~\ref{thm:lobndhomubsys} now consists in applying Lemma~\ref{lem:lobndhomubsys} repeatedly in order to build a generalized Cantor set that is embedded in the set $F_t\cap U$, together with an appropriate outer measure thereon. We shall ultimately apply the mass distribution principle, namely, Lemma~\ref{lem:massdist} to this outer measure. To this end, we shall need an estimate on the mass of balls, {\em i.e.}~on the scaling properties of the outer measure.

The construction is modeled on that presented in Section~\ref{subsec:genCantor}; recall that it is indexed by a tree $T$ and consists of a collection of compact sets $(I_u)_{u\in T}$ and a companion premeasure $\zeta$ such that the compatibility conditions~(\ref{eq:compattree}) hold. However, we need to be more precise in the present construction, and we actually require the following more specific conditions:
\begin{enumerate}\setcounter{enumi}{-1}
	\item\label{item:condubiqconstruc0} every node in the indexing tree $T$ has at least one child, that is,
		\[
			\min_{u\in T}k_u(T)\geq 1\,;
		\]
	\item\label{item:condubiqconstruc1} the compact set $I_\varnothing$ indexed by the root of the tree is a closed ball contained in $U$ with diameter in $(0,1)$ and
		\begin{equation}\label{eq:df:zetaIvarnothing}
			\zeta(I_\varnothing)=\diam{I_\varnothing}^{d/t}\log\frac{1}{\diam{I_\varnothing}}\,;
		\end{equation}
	\item\label{item:condubiqconstruc2} for every node $u\in T\setminus\{\varnothing\}$, there exists an index $i_u\in\calI_U$ such that
		\[
			I_u=B_u^t\subset B_u\subseteq I_{\pi(u)}\,;
		\]
	\item\label{item:condubiqconstruc3} for every node $u\in T\setminus\{\varnothing\}$, we have simultaneously
		\[
			\diam{B_u}\leq 2\exp\left(-\frac{2\cdot 6^d}{t}\diam{I_{\pi(u)}}^{d(1/t-1)-1}\right),
		\]
	in addition to both
		\[
			\bigsqcup_{v\in S_u} B_{v}\subseteq I_{\pi(u)}
			\qquad\text{and}\qquad
			\sum_{v\in S_u}\leb^d(B_v)\geq\frac{\leb^d(I_{\pi(u)})}{2\cdot 3^d}\,;
		\]
	\item\label{item:condubiqconstruc4} for every node $u\in T\setminus\{\varnothing\}$, the premeasure $\zeta$ satisfies
		\[
			\zeta(I_u)=\frac{\leb^d(B_u)}{\sum\limits_{v\in S_u}\leb^d(B_v)}\,\zeta(I_{\pi(u)}).
		\]
\end{enumerate}
In the above conditions, $S_u$ denotes the set formed by a given node $u$ and its siblings, namely, the nodes $v\in T$ such that $\pi(v)=\pi(u)$. Moreover, the sets $B_u$ and $B_u^t$ are the closed balls defined by
	\begin{equation}\label{eq:df:Bu}
		B_u=\clball(x_{i_u},r_{i_u})
		\qquad\text{and}\qquad
		B_u^t=\clball\left(x_{i_u},\frac{r_{i_u}^t}{2}\right).
	\end{equation}
	In addition, let us recall that $\pi(u)$ denotes the parent of a given node $u$, and $k_u(T)$ is the size of its progeny. Also, note that the compatibility conditions~(\ref{eq:compattree}) easily result from~(\ref{item:condubiqconstruc0}--\ref{item:condubiqconstruc4}) above; we even have equality in the compatibility condition that concerns the premeasure $\zeta$. Lastly, it is useful to remark that the ball $B_u$ involved in the construction all have diameter at most one, since they are included in $U$.

The construction is performed inductively on the generation of the indexing tree. In order to guarantee~(\ref{item:condubiqconstruc1}), we begin the construction by considering an arbitrary closed ball with diameter in $(0,1)$ that is contained in the nonempty open set $U$\,; this ball is the compact set $I_\varnothing$ indexed by the root of the tree. We also define $\zeta(I_\varnothing)$ by~(\ref{eq:df:zetaIvarnothing}), in addition to the compulsory condition $\zeta(\emptyset)=0$.

Furthermore, let us assume that the tree, the compact sets and the companion premeasure have been defined up to a given generation $j$ in such a way that the conditions~(\ref{item:condubiqconstruc0}--\ref{item:condubiqconstruc4}) above hold; we now build the tree, the compacts and the premeasure at the next generation $j+1$ in the following manner. For each node $u$ of the $j$-th generation, we apply Lemma~\ref{lem:lobndhomubsys} to the interior of $I_u$ and the real number
	\[
		\rho_u=\exp\left(-\frac{2\cdot 6^d}{t}\diam{I_u}^{d(1/t-1)-1}\right)\,;
	\]
	the resulting finite subset of $\calI_U$ is denoted by $\calI(\interior{I_u},\rho_u)$. We then decide that the progeny of the node $u$ in the tree $T$ has cardinality $k_u(T)$ equal to that of $\calI(\interior{I_u},\rho_u)$. Furthermore, we let $i_{uk}$, for $k\in\{1,\ldots,k_u(T)\}$, denote the elements of $\calI(\interior{I_u},\rho_u)$. Making use of the notation~(\ref{eq:df:Bu}), we therefore have
	\[
		\bigsqcup_{k=1}^{k_u(T)}B_{uk}\subseteq\interior{I_u}\subseteq I_u
		\qquad\text{and}\qquad
		\sum_{k=1}^{k_u(T)}\leb^d(B_{uk})\geq\frac{\leb^d(I_u)}{2\cdot 3^d}.
	\]
	On top of that, the radii of the balls $B_{uk}$ are bounded above by $\rho_u$. Using the notation~(\ref{eq:df:Bu}) again, we also define the compact sets $I_{uk}$ as being equal to the closed balls $B_{uk}^t$, for $k\in\{1,\ldots,k_u(T)\}$. This way, the condition~(\ref{item:condubiqconstruc0}) is satisfied by the nodes of the $j$-th generation, and the conditions~(\ref{item:condubiqconstruc2}--\ref{item:condubiqconstruc3}) hold for those of the $(j+1)$-th generation. Finally, for $k\in\{1,\ldots,k_u(T)\}$, we define
	\[
		\zeta(I_{uk})=\frac{\leb^d(B_{uk})}{\sum\limits_{l=1}^{k_u(T)}\leb^d(B_{ul})}\,\zeta(I_{u}),
	\]
	so that~(\ref{item:condubiqconstruc4}) holds for the nodes of the $(j+1)$-th generation. Finally, the above procedure clearly implies that every node of the tree has at least one child, {\em i.e.}~ the condition~(\ref{item:condubiqconstruc0}) holds.

\subsubsection{Scaling properties of the premeasure}

The next result gives an upper bound on the premeasure $\zeta$ in terms of the diameters of sets.

\begin{Lemma}\label{lem:bndmuIu}
	For any node $u\in T$,
	\begin{equation}\label{eq:bndmuIu}
		\zeta(I_u)\leq\diam{I_u}^{d/t}\log\frac{1}{\diam{I_u}}.
	\end{equation}
\end{Lemma}

\begin{proof}
	Let us prove~(\ref{eq:bndmuIu}) by induction on the length of the word $u\in T$. First, equality holds when $u$ is the empty word, due to the mere value of $\zeta(I_\varnothing)$ determined by~(\ref{eq:df:zetaIvarnothing}). Moreover, if we consider a node $u\in T\setminus\{\varnothing\}$ and if we assume that~(\ref{eq:bndmuIu}) holds for its parent node $\pi(u)$, then the conditions~(\ref{item:condubiqconstruc2}--\ref{item:condubiqconstruc4}) yield
	\begin{align*}
		\zeta(I_u)&\leq 2\cdot 3^d\leb^d(B_u)\frac{\zeta(I_{\pi(u)})}{\leb^d(I_{\pi(u)})}
		=2\cdot 6^d\diam{I_u}^{d/t}\frac{\zeta(I_{\pi(u)})}{\diam{I_{\pi(u)}}^d}\\
		&\leq 2\cdot 6^d\diam{I_u}^{d/t}\diam{I_{\pi(u)}}^{d(1/t-1)}\log\frac{1}{\diam{I_{\pi(u)}}}.
	\end{align*}
	Finally, in view of the restriction on the diameter of the ball $B_u$ imposed by the condition~(\ref{item:condubiqconstruc3}) and the obvious fact that $\log(1/r)\leq 1/r$ for all $r>0$, we have
	\[
		\diam{I_{\pi(u)}}^{d(1/t-1)}\log\frac{1}{\diam{I_{\pi(u)}}}
		\leq\diam{I_{\pi(u)}}^{d(1/t-1)-1}\leq\frac{t}{2\cdot 6^d}\log\frac{2}{\diam{B_u}}
		=\frac{1}{2\cdot 6^d}\log\frac{1}{\diam{I_u}},
	\]
	which leads to~(\ref{eq:bndmuIu}) for the node $u$ itself.
\end{proof}

\subsubsection{The limiting outer measure and its scaling properties}

With the help of Theorem~\ref{thm:absmeas}, we may extend as usual the premeasure $\zeta$ to an outer measure $\zeta^\ast$ on all the subsets of $\R^d$ {\em via}~(\ref{eq:df:zetaasta}). We may also consider the limiting compact set $K$ defined by~(\ref{eq:limittree}), in addition to the outer measure $\mu$ that maps a set $E\subseteq\R^d$ to the value $\zeta^\ast(E\cap K)$. The tree $T$ considered here is infinite, so Lemma~\ref{lem:basictree} shows that $K$ is a nonempty compact subset of $I_\varnothing$. Moreover, the outer measure $\mu$ has total mass $\mu(K)=\zeta(I_\varnothing)$. The next result shows that $K\subseteq F_t\cap U$ as required.

\begin{Proposition}\label{prp:KinclFtcapU}
	The compact set $K$ is contained in $F_t\cap U$. As a consequence,
	\[
		\mu(F_t\cap U)=\mu(K)=\zeta(I_\varnothing)=\diam{I_\varnothing}^{d/t}\log\frac{1}{\diam{I_\varnothing}}.
	\]
\end{Proposition}

\begin{proof}
	On the one hand, we already mentioned that $K\subseteq I_\varnothing\subseteq U$. On the other hand, if a point $x$ belongs to $K$, then there exists a sequence $(\xi_j)_{j\geq 1}$ of positive integers such that $x\in I_{\xi_1\ldots\xi_j}$ for all $j\geq 1$. Hence, the point $x$ belong to the infinitely many nested balls $B_{\xi_1\ldots\xi_j}^t\subseteq\opball(x_{i_{\xi_1\ldots\xi_j}},r_{i_{\xi_1\ldots\xi_j}}^t)$, and so belongs to $F_t$.
\end{proof}

Thanks to Lemma~\ref{lem:bndmuIu}, we may now derive an upper bound on the $\mu$-mass of sufficiently small closed balls $\R^d$.

\begin{Proposition}\label{prp:bndmuB}
	For any closed ball $B$ of $\R^d$ with diameter less than $\ee^{-d/t}$,
	\[
		\mu(B)\leq 2\cdot 12^d\diam{B}^{d/t}\log\frac{1}{\diam{B}}.
	\]
\end{Proposition}

\begin{proof}
	We may obviously assume that the ball $B$ intersects the compact set $K$, as otherwise $\mu(B)$ clearly vanishes. Besides, if the ball $B$ intersects only one compact set $I_u$ at each generation, then there exists a sequence $(\xi_j)_{j\geq 1}$ of positive integers such that $B\cap K\subseteq I_{\xi_1\ldots\xi_j}$ for all $j\geq 1$, so that
	\[
		\mu(B)=\zeta^\ast(B\cap K)\leq\zeta(I_{\xi_1\ldots\xi_j})
		\leq\diam{I_{\xi_1\ldots\xi_j}}^{d/t}\log\frac{1}{\diam{I_{\xi_1\ldots\xi_j}}}
		\xrightarrow[j\to\infty]{}0,
	\]
	thanks to Lemma~\ref{lem:bndmuIu}. The upshot is that we may suppose in what follows that there exists a node $u\in T$ such that the ball $B$ intersects the compact set $I_u$, and at least two compacts indexed by the children of $u$. We further assume that $u$ has minimal length, which in fact ensures its uniqueness.
	
	The easy case is when the diameter of the ball $B$ exceeds that of the compact set $I_u$\,; indeed, as the intersection set $B\cap K$ is covered by the sole $I_u$, we may then deduce from Lemma~\ref{lem:bndmuIu} that
	\[
		\mu(B)=\zeta^\ast(B\cap K)\leq\zeta(I_u)\leq\diam{I_u}^{d/t}\log\frac{1}{\diam{I_u}}
		\leq\diam{B}^{d/t}\log\frac{1}{\diam{B}}.
	\]
	Note that the last inequality holds because $\diam{B}$ is small enough to ensure that the considered function of the diameter is nondecreasing.
	
	Let us now deal with the opposite case in which $\diam{B}$ is smaller than $\diam{I_u}$. Let $\calK$ denote the set of all integers $k\in\{1,\ldots,k_u(T)\}$ such that $I_{uk}$ intersects $B$. We shall use the next simple volume estimate, whose proof is left to the reader:
		\begin{equation}\label{eq:volestimubiq}
			\forall k\in\calK \qquad \leb^d(B\cap B_{uk})\geq 4^{-d}\,\leb^d(B_{uk}).
		\end{equation}
	
	The previous lemma enables us to estimate the $\mu$-mass of the ball $B$. Indeed, the ball intersects the compact set $K$ inside the compact sets $I_{uk}$, for $k\in\calK$, so the conditions~(\ref{item:condubiqconstruc3}) and~(\ref{item:condubiqconstruc4}) yield
	\begin{align*}
		\mu(B)&=\zeta^\ast(B\cap K)\\
		&\leq\sum_{k\in\calK}\zeta(I_{uk})
		=\sum_{k\in\calK}\frac{\leb^d(B_{uk})}{\sum\limits_{l=1}^{k_u(T)}\leb^d(B_{ul})}\,\zeta(I_u)
		\leq 2\cdot 3^d\frac{\zeta(I_u)}{\leb^d(I_u)}\sum_{k\in\calK}\leb^d(B_{uk}).
	\end{align*}
	Now, making use of~(\ref{eq:volestimubiq}) and the disjointness of the balls $B_{uk}$, we infer that
	\[
		\mu(B)\leq 2\cdot 12^d\frac{\zeta(I_u)}{\leb^d(I_u)}\sum_{k\in\calK}\leb^d(B\cap B_{uk})
		\leq 2\cdot 12^d\frac{\zeta(I_u)}{\leb^d(I_u)}\leb^d(B).
	\]
	Combining the condition~(\ref{item:condubiqconstruc2}), the definition~(\ref{eq:df:Bu}) of the balls $B_u^t$ and the bound on the $\zeta$-mass of $I_u$ given by Lemma~\ref{lem:bndmuIu}, we deduce that
	\[
		\mu(B)\leq 2\cdot 12^d\diam{B}^d\diam{I_u}^{d(1/t-1)}\log\frac{1}{\diam{I_u}}
		\leq 2\cdot 12^d\diam{B}^{d/t}\log\frac{1}{\diam{B}}.
	\]
	For the latter bound, we used the fact that $t>1$ and $\diam{I_u}>\diam{B}$. We conclude by combining this bound with the one obtained in the previous easier case.
\end{proof}

To finish the proof of Theorem~\ref{thm:lobndhomubsys}, it remains to apply the mass distribution principle, namely, Lemma~\ref{lem:massdist}. In fact, any bounded subset $C$ of $\R^d$ may be embedded in a closed ball $B$ with radius equal to $\diam{C}$. If we assume in addition that $\diam{C}<\ee^{-d/t}/2$, the ball $B$ has diameter less than $\ee^{-d/t}$, and Proposition~\ref{prp:bndmuB} gives
	\[
		\mu(C)\leq\mu(B)\leq 2\cdot 12^d\diam{B}^{d/t}\log\frac{1}{\diam{B}}
		\leq 2\cdot 12^d 2^{d/t}\diam{C}^{d/t}\log\frac{1}{\diam{C}}.
	\]
	Letting $g$ denote the gauge function $r\mapsto r^{d/t}|\log r|$, the mass distribution principle and Proposition~\ref{prp:KinclFtcapU} finally ensure that
	\[
		\hau^g(F_t\cap U)\geq\frac{\mu(F_t\cap U)}{2\cdot 12^d 2^{d/t}}=\frac{g(\diam{I_\varnothing})}{2\cdot 12^d 2^{d/t}}>0,
	\]
	from which we deduce that the set $F_t\cap U$ has Hausdorff dimension at least $d/t$.

\subsection{Behavior under uniform dilations}\label{subsec:unifdil}

The covering lemma coming into play in the above proof, namely, Lemma~\ref{lem:lobndhomubsys} also entails that multiplying all the approximation radii by a common positive factor does not alter the property of being a homogeneous ubiquitous system. We have indeed the next useful statement, which implies in particular that homogeneous ubiquity is independent on the norm.

\begin{Proposition}\label{prp:homubsyscst}
	Let $(x_i,r_i)_{i\in\calI}$ be a homogeneous ubiquitous system in some nonempty open subset $U$ of $\R^d$. Then, for any real number $c>0$, the family $(x_i,c\,r_i)_{i\in\calI}$ is also a homogeneous ubiquitous system in $U$.
\end{Proposition}

\begin{proof}
	The family $(x_i,c\,r_i)_{i\in\calI}$ is clearly an approximation system, so it remains to show that the set $R_c$ of all points $x\in\R^d$ such that $|x-x_i|<c\,r_i$ for infinitely many indices $i\in\calI$ has full Lebesgue measure in $U$. This is obvious if $c\geq 1$, because $R_c$ contains $R_1$, which has full Lebesgue measure in $U$. We may thus restrict our attention to the case in which $c<1$.
	
	Let $V$ be a nonempty bounded open subset of $U$ and let $j$ be a positive integer. By Lemma~\ref{lem:lobndhomubsys}, there is a finite subset $\calI_j=\calI(V,2^{-j})$ of $\calI$ such that the balls $\clball(x_i,r_i)$ are disjoint, contained in $V$, with radius at most $2^{-j}$, and a total Lebesgue measure at least $\leb^d(V)/(2\cdot 3^d)$. In particular,
		\[
			R_c\cap V\supseteq\limsup_{j\to\infty}\bigsqcup_{i\in\calI_j} \opball(x_i,c\,r_i)
			=\bigcap_{j=1}^\infty\downarrow\bigcup_{j'=j}^\infty\bigsqcup_{i\in\calI_{j'}} \opball(x_i,c\,r_i).
		\]
		The set $V$ is bounded, thereby having finite measure. Hence,~(\ref{eq:measdecint}) yields
		\[
			\leb^d(R_c\cap V)
			\geq\limsup_{j\to\infty}\sum_{i\in\calI_j}\leb^d(\opball(x_i,c\,r_i))
			\geq\frac{c^d\,\leb^d(V)}{2\cdot 3^d}.
		\]
	
	Let us assume that $\leb^d(U\setminus R_c)$ is positive. Then $\leb^d(U_m\setminus R_c)$ is positive for $m$ large enough, where $U_m$ denotes the set $U\cap (-m,m)^d$. Furthermore, we have
		\[
			\leb^d((R_c\cap U_m)\setminus K)<\frac{c^d\,\leb^d(U_m\setminus R_c)}{2\cdot 3^d}
		\]
		for some compact subset $K$ of $R_c\cap U_m$, see {\em e.g.}~\cite[Theorem~1.10]{Mattila:1995fk}. Applying what precedes to the bounded open set $V=U_m\setminus K$, we obtain
		\[
			\leb^d(R_c\cap (U_m\setminus K))
			\geq\frac{c^d\,\leb^d(U_m\setminus K)}{2\cdot 3^d}
			\geq\frac{c^d\,\leb^d(U_m\setminus R_c)}{2\cdot 3^d},
		\]
		and we end up with a contradiction. Hence, $R_c$ has full Lebesgue measure in $U$.
\end{proof}

\subsection{The Jarn\'ik-Besicovitch theorem}\label{subsec:JarnikBesicovitch}

Let us recall from~(\ref{eq:df:Jdtau}) that $J_{d,\tau}$ is the set of points that are approximable at rate at least $\tau$ by the points with rational coordinates. The set $J_{d,\tau}$ coincides with the whole space $\R^d$ when $\tau\leq 1+1/d$, as a consequence of Dirichlet's theorem, see Corollary~\ref{cor:Dirichlet}. We thus assume throughout that $\tau$ is larger than $1+1/d$. In that case, we know from Proposition~\ref{prp:lebJdtau} that $J_{d,\tau}$ has Lebesgue measure zero. Moreover, using elementary methods, we established in Section~\ref{subsec:upbndhausdorff} that $J_{d,\tau}$ has Hausdorff dimension at most $(d+1)/\tau$, see~(\ref{eq:upbndJdtau}). Actually, the latter bound is the exact value of the dimension. The above theory of homogeneous ubiquitous systems will indeed enable us to recover the next result, obtained independently by Jarn\'ik~\cite{Jarnik:1929mf} and Besicovitch~\cite{Besicovitch:1934ly}.

\begin{Theorem}[Jarn\'ik, Besicovitch]\label{thm:JarnikBesicovitch}
	For any real number $\tau>1+1/d$ and any nonempty open subset $U$ of $\R^d$,
	\[
		\Hdim(J_{d,\tau}\cap U)=\frac{d+1}{\tau}.
	\]
\end{Theorem}

\begin{proof}
	The set $J_{d,1+1/d}$ coincides with the whole $\R^d$, so it obviously has full measure therein, namely, for Lebesgue-almost every $x\in\R^d$, there are infinitely many pairs $(p,q)\in\Z^d\times\N$ such that $|x-p/q|_\infty<q^{-1-1/d}$. This means that the family $(p/q,q^{-1-1/d})_{(p,q)\in\Z^d\times\N}$ is a homogeneous ubiquitous system in $\R^d$. Besides, for any integer $M\geq 1$ and any real number $s>0$, note that
	\[
		\sum_{(p,q)\in\Z^d\times\N\atop p/q\in\opball_\infty(0,M)} (q^{-1-1/d})^s
		=\sum_{q=1}^\infty q^{-(1+1/d)s}\#(\Z^d\cap\opball_\infty(0,qM)).
	\]
	The cardinality appearing in the sum is of the order of $(qM)^d$, up to numerical constants. Hence, the critical value $s$ for the convergence of the series is that for which $(1+1/d)s-d$ is equal to one. We deduce that for any open subset $U$ of $\R^d$, the parameter $s_U$ defined by~(\ref{eq:df:sU}) is bounded above by $d$.

	We are now in position to apply Corollary~\ref{cor:eqhomubsys}. For any $\tau>1+1/d$, the approximation radii $q^{-\tau}$ in the definition of $J_{d,\tau}$ may be written in the form $(q^{-1-1/d})^t$ with $t=\tau d/(d+1)>1$. We deduce that for any nonempty open set $U\subseteq\R^d$, the set $J_{d,\tau}\cap U$ has Hausdorff dimension equal to $d/t$, {\em i.e.}~equal to $(d+1)/\tau$.
\end{proof}

We can relate this result with the notion of irrationality exponent, defined by~(\ref{eq:df:irrexpo}). In fact, for any $\tau\geq 1+1/d$, let $\Tau^\geq_{d,\tau}$ denote the set of points $x$ in $\R^d\setminus\Q^d$ with irrationality exponent satisfying $\tau(x)\geq\tau$. It is then clear that
	\[
		J_{d,\tau}\setminus\Q^d
		\subseteq\Tau^\geq_{d,\tau}=
		\bigcap_{\tau'<\tau}\downarrow J_{d,\tau'}\setminus\Q^d.
	\]
	Due to Theorem~\ref{thm:JarnikBesicovitch} and the fact that $\Q^d$ is countable, we deduce that $\Tau^\geq_{d,\tau}$ has Hausdorff dimension equal to $(d+1)/\tau$ in every nonempty open subset of $\R^d$. Theorem~\ref{thm:lobndhomubsys} actually gives a slightly more precise result: letting $g_\tau$ denote the gauge function $r\mapsto r^{(d+1)/\tau}|\log r|$ and $U$ be such an open set, we know that the set $J_{d,\tau}\cap U$ has positive Hausdorff $g_\tau$-measure, so the set $\Tau^\geq_{d,\tau}\cap U$ satisfies the same property. This allows us to determine the Hausdorff dimension of the set $\Tau^=_{d,\tau}$ of points $x$ in $\R^d\setminus\Q^d$ with irrationality exponent equal to $\tau$. As a matter of fact,
	\begin{equation}\label{eq:settauexact}
		\Tau^=_{d,\tau}
		=\Tau^\geq_{d,\tau}\setminus\bigcup_{\tau'>\tau}\uparrow J_{d,\tau'}.
	\end{equation}
	Moreover, thanks to Proposition~\ref{prp:compgauge0}, we easily infer that the sets $J_{d,\tau'}$, for $\tau'>\tau$, have Hausdorff $g_\tau$-measure zero. Finally, the mapping $\tau'\mapsto J_{d,\tau'}$ is nonincreasing, so the union in~(\ref{eq:settauexact}) may be written as a countable one, and the subadditivity of Hausdorff measures entails that its $g_\tau$-measure vanishes. This yields
	\[
		\Hdim(\Tau^=_{d,\tau}\cap U)=\frac{d+1}{\tau},
	\]
	as indeed the set in the left-hand side of~(\ref{eq:settauexact}) has positive $g_\tau$-measure in $U$.


\section{Large intersection properties}\label{sec:largeint}

We present in this section the classes of sets with large intersection introduced by Falconer~\cite{Falconer:1985fc,Falconer:1994hx}, and we enlighten their remarkable interplay with the general theory of homogeneous ubiquitous systems discussed in Section~\ref{sec:firstubiq}.

\subsection{The large intersection classes}\label{subsec:largeintclass}

These classes are composed of subsets of $\R^d$ with Hausdorff dimension at least a given $s$ satisfying the striking counterintuitive property that countable intersections of the sets also have Hausdorff dimension at least $s$. This is in stark contrast with, for instance, the case of two affine subspaces with dimension $s_1$ and $s_2$, respectively, where the intersection is generically expected to have dimension $s_1+s_2-d$. The aforementioned classes are formally defined as follows. Recall that a $G_\delta$-set is one that may be written as the intersection of a countable sequence of open sets.

\begin{Definition}\label{df:lic}
	For any $s\in(0,d]$, the class $\lic{s}{\R^d}$ of {\em sets with large intersection in $\R^d$ with dimension at least $s$} is the collection of all $G_\delta$-subsets $F$ of $\R^d$ such that
	\[
		\Hdim\bigcap_{n=1}^\infty\varsigma_n(F)\geq s
	\]
	for any sequence $(\varsigma_n)_{n\geq 1}$ of similarity transformations of $\R^d$.
\end{Definition}

As shown later in these notes, numerous examples of sets with large intersection arise in metric number theory. Let us point out that the middle-third Cantor set $\cantor$ gives a typical example of set that is {\em not} with large intersection. Indeed, letting $\varsigma$ denote the mapping that sends a real number $x$ to $(x+1)/3$, we readily observe that $\cantor\cap\varsigma(\cantor)$ is reduced to the points $1/3$ and $2/3$, thereby having Hausdorff dimension zero, whereas $\cantor$ itself has dimension $\log 2/\log 3$, see~(\ref{eq:upbndcantor}) and~(\ref{eq:lobndcantor}).

As mentioned above, the main property of the large intersection classes $\lic{s}{\R^d}$ are their stability under countable intersections; remarkably, they are also stable under bi-Lipischitz transformations, {\em i.e.}~mappings satisfying~(\ref{eq:df:biLip}). This is the purpose of the next statement.

\begin{Theorem}\label{thm:stablic}
	For any real number $s\in(0,d]$, the class $\lic{s}{\R^d}$ is closed under countable intersections and bi-Lipschitz transformations of $\R^d$.
\end{Theorem}

The proof of Theorem~\ref{thm:stablic} is postponed to Section~\ref{subsec:licproofs} so as not to interrupt the flow of the presentation. Combined with the definition of the classes $\lic{s}{\R^d}$ given above, Theorem~\ref{thm:stablic} directly yields the following maximality property with respect to countable intersections and similarities.

\begin{Corollary}
	For any real number $s\in(0,d]$, the class $\lic{s}{\R^d}$ is the maximal class of $G_\delta$-subsets of $\R^d$ with Hausdorff dimension at least $s$ that is closed under countable intersections and similarity transformations.
\end{Corollary}

We now give several characterizations of the classes $\lic{s}{\R^d}$. Some of them are expressed in terms of outer net measures that are obtained by restricting to coverings by dyadic cubes. For any $s\in(0,d]$, let us consider the premeasure, denoted by $\diam{\cdot}^s_\Lambda$, that maps a given cube $\lambda\in\Lambda$ to $\diam{\lambda}^s$. Then, as in Section~\ref{subsubsec:normgaugenetm}, Theorem~\ref{thm:metmeas} allows us to consider the net measure
	\[
		\netm^s=(\diam{\cdot}^s_\Lambda)_\ast
	\]
	defined by~(\ref{eq:df:zetaastm}). In view of Proposition~\ref{prp:compnetm}, this outer measure is comparable with the $s$-dimensional Hausdorff measure. In addition, Theorem~\ref{thm:absmeas} enables us to introduce the outer measure
	\begin{equation}\label{eq:df:netmsinfty}
		\netm^s_\infty=(\diam{\cdot}^s_\Lambda)^\ast
	\end{equation}
	that is defined by~(\ref{eq:df:zetaasta}), and thus corresponds to coverings by dyadic cubes of arbitrary diameter. It is clear that the outer measures $\netm^s_\infty$ bound the net measures $\netm^s$ from below. Hence, for any subset $E$ of $\R^d$,
	\begin{equation}\label{eq:netmHdim}
		\netm^s_\infty(E)>0 \qquad\Longrightarrow\qquad \Hdim E\geq s.
	\end{equation}
	Moreover, the next lemma shows that the $\netm^s_\infty$-mass of dyadic cubes may easily be expressed in terms of their diameters. The proof is omitted because this result actually follows from the more general formula~(\ref{eq:netmlambdagauge}) that will be established later.

\begin{Lemma}\label{lem:netmlambda}
	For any real number $s\in(0,d]$ and any dyadic cube $\lambda\in\Lambda$,
	\[
		\netm^s_\infty(\lambda)=\netm^s_\infty(\interior{\lambda})=\diam{\lambda}^s.
	\]
\end{Lemma}

We can now enumerate the properties that characterize the large intersection classes; note that the formulations given by Falconer~\cite{Falconer:1994hx} are slightly erroneous and one has to consider the corrected versions below, where $s\in(0,d]$ and $F\subseteq\R^d$\,:
\begin{enumerate}
	\item\label{item:lic1} for any nonempty open subset $U$ of $\R^d$ and any sequence $(f_n)_{n\geq 1}$ of bi-Lipschitz transformations from $U$ to $\R^d$,
		\[
			\Hdim\bigcap_{n=1}^\infty f_n^{-1}(F)\geq s\,;
		\]
	\item\label{item:lic2} for any sequence $(\varsigma_n)_{n\geq 1}$ of similarity transformations of $\R^d$,
		\[
			\Hdim\bigcap_{n=1}^\infty \varsigma_n(F)\geq s\,;
		\]
	\item\label{item:lic3} for any positive real number $t<s$ and any dyadic cube $\lambda\in\Lambda$,
		\[
			\netm^t_\infty(F\cap\lambda)=\netm^t_\infty(\lambda)\,;
		\]
	\item\label{item:lic4} for any positive real number $t<s$ and any open subset $V$ of $\R^d$,
		\[
			\netm^t_\infty(F\cap V)=\netm^t_\infty(V)\,;
		\]
	\item\label{item:lic5} for any positive real number $t<s$, there exists a real number $c\in(0,1]$ such that for any dyadic cube $\lambda\in\Lambda$,
		\[
			\netm^t_\infty(F\cap\lambda)\geq c\,\netm^t_\infty(\lambda)\,;
		\]
	\item\label{item:lic6} for any positive real number $t<s$, there exists a real number $c\in(0,1]$ such that any open subset $V$ of $\R^d$,
		\[
			\netm^t_\infty(F\cap V)\geq c\,\netm^t_\infty(V).
		\]
\end{enumerate}

Note that the property~(\ref{item:lic2}) coincides with the definition of the large intersection class $\lic{s}{\R^d}$ under the assumption that $F$ is a $G_\delta$-set. The next result details the logical relationships between the previous properties, and in fact implies that they give equivalent characterizations of the large intersection classes.

\begin{Theorem}\label{thm:caraclic}
	Let us consider a real number $s\in(0,d]$ and a subset $F$ of $\R^d$.
	\begin{itemize}
		\item The following implications hold:
			\[
				(\ref{item:lic1})
				\ \Longrightarrow\ 
				(\ref{item:lic2})
				\ \Longrightarrow\ 
				(\ref{item:lic3})
				\ \Longleftrightarrow\ 
				(\ref{item:lic4})
				\ \Longrightarrow\ 
				(\ref{item:lic5})
				\ \Longleftrightarrow\ 
				(\ref{item:lic6}).
			\]
		\item If $F$ is a $G_\delta$-set, then the properties~(\ref{item:lic1}--\ref{item:lic6}) are all equivalent, and characterize the class $\lic{s}{\R^d}$.
	\end{itemize}
\end{Theorem}

The proof of Theorem~\ref{thm:caraclic} is discussed in Section~\ref{subsec:licproofs}. Note that the characterizations~(\ref{item:lic5}) and~(\ref{item:lic6}) still hold when changing the norm on $\R^d$\,; the large intersection classes are thus independent on the choice of the norm. Hereunder are several other noteworthy properties of these classes. They all readily follow from Definition~\ref{df:lic}, except the last one for which we refer to~\cite[Theorem~C(f)]{Falconer:1994hx}.

\begin{Proposition}\label{prp:morelic}
	The large intersection classes $\lic{s}{\R^d}$, for $s\in(0,d]$, satisfy all the following properties.
	\begin{enumerate}
		\item\label{item:prp:morelic1} Any $G_\delta$-subset of $\R^d$ that contains a set in the class $\lic{s}{\R^d}$ also belongs to the class $\lic{s}{\R^d}$.
		\item\label{item:prp:morelic2} The mapping $s\mapsto\lic{s}{\R^d}$ is nonincreasing.
		\item The class $\lic{s}{\R^d}$ is the intersection over $t<s$ of the classes $\lic{t}{\R^d}$.
		\item For any sets $F\in\lic{s}{\R^d}$ and $F'\in\lic{s'}{\R^{d'}}$, the product set $F\times F'$ belongs to the class $\lic{s+s'}{\R^{d+d'}}$.
	\end{enumerate}
\end{Proposition}

Finally, note that a set with large intersection is necessarily dense in the whole space $\R^d$. This is easily seen for instance by considering the characterization~(\ref{item:lic3}) of the large intersection classes given by Theorem~\ref{thm:caraclic}, and by making use of Lemma~\ref{lem:netmlambda}. However, in some applications, the considered sets are thought of satisfying a large intersection property in some nonempty open subset $U$ of $\R^d$, but fail to be dense in the whole space $\R^d$ itself. We therefore need to introduce localized versions of the large intersection classes. In that situation, the use of similarity transformations is not suitable anymore; a convenient way of proceeding is thus to localize the characterization~(\ref{item:lic4}) of the large intersection classes given by Theorem~\ref{thm:caraclic} in the following manner.

\begin{Definition}\label{df:licloc}
	For any real number $s\in(0,d]$ and any nonempty open subset $U$ of $\R^d$, the class $\lic{s}{U}$ of {\em sets with large intersection in $U$ with dimension at least $s$} is the collection of all $G_\delta$-subsets $F$ of $\R^d$ such that
		\[
			\netm^t_\infty(F\cap V)=\netm^t_\infty(V)
		\]
		for any positive real number $t<s$ and any open subset $V$ of $U$.
\end{Definition}

Obviously, thanks to Theorem~\ref{thm:caraclic}, each class $\lic{s}{U}$ coincides with the class $\lic{s}{\R^d}$ introduced in Definition~\ref{df:lic} when $U=\R^d$. Moreover, similarly to $\lic{s}{\R^d}$, the class $\lic{s}{U}$ does not depend on the norm, either. This is again due to the fact that the properties~(\ref{item:lic5}) and~(\ref{item:lic6}) arising in the statement of Theorem~\ref{thm:caraclic} are invariant under a change of norm and, up to a localization in the vein of Definition~\ref{df:licloc}, still characterize the newly introduced classes. We also have the next result.

\begin{Theorem}\label{thm:licloc}
	Let $s\in(0,d]$ and let $U$ be a nonempty open subset of $\R^d$. Then:
	\begin{enumerate}
		\item\label{item:thm:licloc1} the class $\lic{s}{U}$ is closed under countable intersections;
		\item\label{item:thm:licloc2} for any set $F\in\lic{s}{U}$ and any nonempty open set $V\subseteq U$,
			\[
				\Hdim(F\cap V)\geq s.
			\]
	\end{enumerate}
\end{Theorem}

The second property in Theorem~\ref{thm:licloc} follows from~(\ref{eq:netmHdim}), whereas the first is proven in Section~\ref{subsec:licproofs}. This result suggests that the large intersection property is a combination of a density property with a measure theoretic aspect. Theorem~\ref{thm:stablic} may thus be seen as a Hausdorff dimensional analog of the Baire category theorem.

\subsection{Packing dimension}

The sets with large intersection also display a remarkable behavior with respect to packing dimension. Let us explain how this notion of dimension, due to Tricot~\cite{Tricot:1982fk}, is defined. First, given a gauge function $g$, we define on the collection of all subsets $F$ of $\R^d$ the {\em packing $g$-premeasure} by
	\[
		\prepack^g(F)=\lim_{\delta\downarrow 0}\downarrow \prepack^g_\delta(F)
		\qquad\text{with}\qquad
		\prepack^g_\delta(F)=\sup\sum_{n=1}^\infty g(\diam{B_{n}}),
	\]
	where the supremum is taken over all sequences $(B_{n})_{n\geq 1}$ of disjoint closed balls of $\R^d$ centered in the set $F$ and with diameter less than $\delta$. The premeasures $\prepack^g$ are only finitely subadditive; it is thus more convenient to work with the corresponding {\em packing $g$-measure}, defined by
	\[
		\pack^g=(\prepack^g)^\ast
	\]
	as in the formula~(\ref{eq:df:zetaasta}), which is an outer measure on $\R^d$, as a consequence of Theorem~\ref{thm:absmeas}. It is actually possible to show that the Borel subsets of $\R^d$ are $\pack^g$-measurable, see~\cite[Chapter~5]{Mattila:1995fk} for details.

The definition of packing dimension is then very similar to that of Hausdorff dimension, namely, Definition~\ref{df:Hausdorffdim}. Specifically, when the gauge function $g$ is of the form $r\mapsto r^{s}$ with $s>0$, it is customary to use $\pack^s$ as a shorthand for $\pack^g$, and the packing dimension of a nonempty set $F\subseteq\R^d$ is defined by
	\[
		\Pdim F=\sup\{ s\in (0,d) \:|\: \pack^s(F)=\infty \}=\inf\{ s\in (0,d) \:|\: \pack^s(F)=0 \},
	\]
	with the convention that $\sup\emptyset=0$ and $\inf\emptyset=d$. When the set $F$ is empty, we adopt the convention that the packing dimension is equal to $-\infty$. Moreover, one recovers the upper box-counting dimension $\UBdim E$ by considering the premeasures $\prepack^s$ instead of $\pack^s$ in the latter formula.

The packing dimension of sets with large intersection is discussed in the next statement, which may be seen as a counterpart to Theorem~\ref{thm:licloc}(\ref{item:thm:licloc2}).

\begin{Proposition}\label{prp:licPdim}
	Let $s\in(0,d]$ and let $U$ be a nonempty open subset of $\R^d$. Then, for any set $F\in\lic{s}{U}$ and for any nonempty open set $V\subseteq U$,
	\[
		\Pdim(F\cap V)=d.
	\]
\end{Proposition}

In other words, a set with large intersection has maximal packing dimension in any nonempty open set; the same property obviously holds for box-counting dimensions as well, because sets with large intersection are dense. Again, for the sake of clarity, the proof of Proposition~\ref{prp:licPdim} is postponed to Section~\ref{subsec:licproofs}.

\subsection{Proof of the main results}\label{subsec:licproofs}

\subsubsection{Ancillary lemmas}\label{subsubsec:anclemlic}

The above results are established with the help of several technical lemmas concerning the outer measures $\netm^s_\infty$ that we now state.

\begin{Lemma}\label{lem:netmlambdaV}
	Let us consider two real numbers $s\in(0,d]$ and $c\in(0,1]$, a subset $F$ of $\R^d$, and an open subset $V$ of $\R^d$. Suppose that there is a $\delta>0$ such that
	\[
		\netm^s_\infty(F\cap\lambda)\geq c\,\netm^s_\infty(\lambda)
	\]
	for all dyadic cubes $\lambda\in\Lambda$ with diameter at most $\delta$ that are contained in $V$. Then,
	\[
		\netm^s_\infty(F\cap V)\geq c\,\netm^s_\infty(V).
	\]
\end{Lemma}

\begin{proof}
	This is a direct extension of~\cite[Lemma~1]{Falconer:1994hx}.
\end{proof}

\begin{Lemma}\label{lem:netmst}
	Let us consider two real numbers $s\in(0,d]$ and $c\in(0,1]$, a subset $F$ of $\R^d$, and an open subset $V$ of $\R^d$. Let us suppose that
	\[
		\netm^s_\infty(F\cap\lambda)\geq c\,\netm^s_\infty(\lambda)
	\]
	for all dyadic cubes $\lambda\in\Lambda$ that are contained in $V$. Then,
	\[
		\netm^t_\infty(F\cap\lambda)=\netm^t_\infty(\lambda)
	\]
	for all dyadic cubes $\lambda\in\Lambda$ that are contained in $V$ and all real numbers $t\in(0,s)$.
\end{Lemma}

\begin{proof}
	This is a simple refinement of~\cite[Lemma~2]{Falconer:1994hx}.
\end{proof}

\begin{Lemma}\label{lem:netmbilip}
	Let $U$ be a nonempty open subset of $\R^d$ and let $f$ be a bi-Lipschitz mapping from $U$ to $\R^d$ with constant $c_f\geq 1$, see~(\ref{eq:df:biLip}). Let us consider two real numbers $s\in(0,d]$ and $c\in(0,1]$ and a subset $F$ of $\R^d$, and suppose that
	\[
		\netm^s_\infty(F\cap V)\geq c\,\netm^s_\infty(V)
	\]
	for any open subset $V$ of $\R^d$. Then, for any open subset $V$ of $U$,
	\[
		\netm^s_\infty(f^{-1}(F)\cap V)\geq \frac{c}{(3c_f)^{2d}}\netm^s_\infty(V).
	\]
\end{Lemma}

\begin{proof}
	The statement is clearly invariant under a change of norm, so we may assume throughout the proof that $\R^d$ is endowed with the supremum norm $|\,\cdot|_\infty$. Let us observe that a Lipschitz mapping $g:U\to\R^d$ with constant $k\geq 1$ satisfies
	\begin{equation}\label{eq:netmlip}
		\netm^s_\infty(g(A))\leq (3 k)^d\netm^s_\infty(A)	
	\end{equation}
	for any subset $A$ of $U$. Indeed, if $(\lambda_n)_{n\geq 1}$ denotes a covering of the set $A$, then $g(A)$ is covered by the image sets $g(\lambda_n)$, and each of these sets is itself covered by $(\lceil k\rceil+1)^d$ dyadic cubes with diameter equal to that of the initial cube $\lambda_n$. Hence, if $V$ is an open subset of $U$, then $f(V)$ is an open subset of $\R^d$ and
	\begin{align*}
		\netm^s_\infty(V)
		&\leq (3c_f)^d\netm^s_\infty(f(V))\\
		&\leq \frac{(3c_f)^d}{c}\netm^s_\infty(F\cap f(V))
		\leq \frac{(3c_f)^{2d}}{c}\netm^s_\infty(f^{-1}(F)\cap V),
	\end{align*}
	which gives the required estimate.
\end{proof}

\begin{Lemma}\label{lem:netminterGdelta}
	Let $U$ be a nonempty subset of $\R^d$ and let $s\in(0,d]$. Let us consider a sequence $(F_k)_{k\geq 1}$ of $G_\delta$-subsets of $\R^d$ such that
	\[
		\netm^s_\infty(F_k\cap V)=\netm^s_\infty(V)
	\]
	for any $k\geq 1$ and any open subset $V$ of $U$. Then, for any open subset $V$ of $U$,
	\[
		\netm^s_\infty\left(\bigcap_{k=1}^\infty F_k\cap V\right)
		\geq 3^{-d}\netm^s_\infty(V).
	\]
\end{Lemma}

\begin{proof}
	See the proof of~\cite[Lemma~4]{Falconer:1994hx}.
\end{proof}

\subsubsection{Proof of Theorem~\ref{thm:caraclic}}

We may now establish the various relationships between the properties~(\ref{item:lic1}--\ref{item:lic6}) involved in the statement of Theorem~\ref{thm:caraclic}.

First, the proof that (\ref{item:lic1}) $\Rightarrow$ (\ref{item:lic2}) follows from the observation that the inverse of a similarity transformation of $\R^d$ is a bi-Lipschitz mapping.

For the proof that (\ref{item:lic2}) $\Rightarrow$ (\ref{item:lic3}), we refer to~\cite{Falconer:1994hx} even though there is a slight mistake at this point of the paper. However, the properties~(\ref{item:lic1}--\ref{item:lic6}) above are exact; comparing them with those in~\cite{Falconer:1994hx}, it is easy to spot the error and correct the proof.

The proof that (\ref{item:lic3}) $\Leftrightarrow$ (\ref{item:lic4}) $\Rightarrow$ (\ref{item:lic5}) $\Leftrightarrow$ (\ref{item:lic6}) follows straightforwardly from Lemma~\ref{lem:netmlambda}, together with the observation that the interior of a dyadic cube $\lambda$ is an open set with the same $\netm^t_\infty$-mass than $\lambda$ itself, by virtue of Lemma~\ref{lem:netmlambdaV}.

Finally, to establish that (\ref{item:lic6}) $\Rightarrow$ (\ref{item:lic1}) for $G_\delta$-sets, let us assume that $F$ is a $G_\delta$-set satisfying~(\ref{item:lic6}), and let $(f_n)_{n\geq 1}$ denote a sequence of bi-Lipschitz transformations defined on a nonempty open set $U$. For each $n\geq 1$, let $c_n$ denote a constant such that $f_n$ satisfies~(\ref{eq:df:biLip}). Given $t\in(0,s)$, Lemma~\ref{lem:netmbilip} ensures that for any $t'\in(t,s)$, there is a real number $c\in(0,1]$ such that for any open subset $V$ of $U$,
	\[
		\netm^{t'}_\infty(f_n^{-1}(F)\cap V)\geq \frac{c}{(3c_n)^{2d}}\netm^{t'}_\infty(V).
	\]
	Applying this estimate to the interior of dyadic cubes and making use of Lemma~\ref{lem:netmlambda}, we get for every dyadic cube $\lambda$ contained in $U$,
	\begin{align*}
		\netm^{t'}_\infty(f_n^{-1}(F)\cap\lambda)
		&\geq\netm^{t'}_\infty(f_n^{-1}(F)\cap\interior{\lambda})\\
		&\geq\frac{c}{(3c_n)^{2d}}\netm^{t'}_\infty(\interior{\lambda})
		=\frac{c}{(3c_n)^{2d}}\netm^{t'}_\infty(\lambda).
	\end{align*}
	Then, it follows from Lemma~\ref{lem:netmst} that for every dyadic cube $\lambda\subseteq U$,
	\[
		\netm^t_\infty(f_n^{-1}(F)\cap\lambda)=\netm^t_\infty(\lambda),
	\]
	and Lemma~\ref{lem:netmlambdaV} now ensures that this also holds when $\lambda$ is replaced by an arbitrary open subset of $U$. Finally, Lemma~\ref{lem:netminterGdelta} ensures that
	\[
		\netm^t_\infty\left(\bigcap_{n=1}^\infty f_n^{-1}(F)\cap U\right)
		\geq 3^{-d}\netm^t_\infty(U)>0.
	\]
	To conclude, it remains to use~(\ref{eq:netmHdim}) to deduce that the intersection of the sets $f_n^{-1}(F)$ has Hausdorff dimension at least $t$, and to let $t$ tend to $s$.

\subsubsection{Proof of Theorem~\ref{thm:stablic}}\label{subsubsec:proof:thm:stablic}

This is a direct consequence of Theorem~\ref{thm:caraclic}. To prove the stability under countable intersections, let us consider a sequence $(F_n)_{n\geq 1}$ of sets in $\lic{s}{\R^d}$. When $t\in(0,s)$, the characterization~(\ref{item:lic4}) ensures that all the sets $F_n$ have maximal $\netm^t_\infty$-mass in all the open subsets of $\R^d$. Lemma~\ref{lem:netminterGdelta} implies that
	\[
		\netm^t_\infty\left(\bigcap_{n=1}^\infty F_n\cap V\right)
		\geq 3^{-d}\netm^t_\infty(V)
	\]
	for any open subset $V$ of $\R^d$, and the characterization~(\ref{item:lic6}) shows that the intersection of the sets $F_n$ belongs to the class $\lic{s}{\R^d}$.

To establish the stability under bi-Lipschitz mappings, let us consider a set $F$ in $\lic{s}{\R^d}$ and a bi-Lipschitz mapping $f$ defined on $\R^d$. Again, when $t\in(0,s)$, the characterization~(\ref{item:lic4}) of this class ensures that the set $F$ has maximal $\netm^t_\infty$-mass in all the open subsets of $\R^d$. Lemma~\ref{lem:netmbilip} then shows that for any open set $V\subseteq\R^d$,
	\[
		\netm^t_\infty(f^{-1}(F)\cap V)\geq\frac{\netm^t_\infty(V)}{(3c_f)^{2d}},
	\]
	where $c_f$ is a constant associated with $f$ as in~(\ref{eq:df:biLip}). We conclude that $f^{-1}(F)$ is in $\lic{s}{\R^d}$ thanks to the characterization~(\ref{item:lic6}) of this class.

\subsubsection{Proof of Theorem~\ref{thm:licloc}(\ref{item:thm:licloc1})}

The proof is parallel to that of the stability under countable intersections of the classes $\lic{s}{\R^d}$ given in Section~\ref{subsubsec:proof:thm:stablic}. It suffices to replace the characterization~(\ref{item:lic4}) of the class $\lic{s}{\R^d}$ by the definition of the generalized classes $\lic{s}{U}$, namely, Definition~\ref{df:licloc}. As above, we then apply Lemma~\ref{lem:netminterGdelta}. Finally, we obtain an analog of the characterization~(\ref{item:lic6}) of the large intersection classes by applying Lemma~\ref{lem:netmst}.

\subsubsection{Proof of Proposition~\ref{prp:licPdim}}

When $U=\R^d$, the result is~\cite[Theorem~D(b)]{Falconer:1994hx}. We thus refer to that paper for the proof in that case, and we content ourselves here with extending the result to arbitrary nonempty open sets $U$. Let us consider a set $F\in\lic{s}{U}$, a nonempty open set $V\subseteq U$, and an arbitrary nonempty dyadic cube $\lambda_0$ contained in $V$. We write $\lambda_0$ in the form $2^{-j_0}(k_0+[0,1)^d)$ with $j_0\in\Z$ and $k_0\in\Z^d$, and we define
	\[
		\widetilde F=\bigsqcup_{k\in\Z^d}(k 2^{-j_0}+(F\cap\interior{\lambda_0})).
	\]
	The fact that $F$ is a $G_\delta$-subset of $\R^d$ implies that $\widetilde F$ is a $G_\delta$-set as well. Furthermore, for any dyadic cube $\lambda$ with diameter at most that of $\lambda_0$, there exists a unique integer point $k\in\Z^d$ such that $\lambda$ is contained in $k 2^{-j_0}+\lambda_0$, so that
	\[
		\widetilde F\cap\lambda=(k 2^{-j_0}+(F\cap\interior{\lambda_0}))\cap\lambda.
	\]
	With the help of~(\ref{eq:netmlip}), we deduce that for any $t\in(0,s)$,
	\begin{align*}
		\netm^t_\infty(\widetilde F\cap\lambda)
		&\geq 3^{-d}\netm^t_\infty(F\cap\interior{\lambda_0}\cap(-k 2^{-j_0}+\lambda))\\
		&\geq 3^{-d}\netm^t_\infty(F\cap\interior{(-k 2^{-j_0}+\lambda)})\\
		&=3^{-d}\netm^t_\infty(\interior{(-k 2^{-j_0}+\lambda)})
		=3^{-d}\netm^t_\infty(\lambda).
	\end{align*}
	The last equality is due to Lemma~\ref{lem:netmlambda}. The previous one holds because the interior of $-k 2^{-j_0}+\lambda$ is an open subset of $U$, and the set $F$ is in $\lic{s}{U}$. Finally, Lemmas~\ref{lem:netmlambdaV} and~\ref{lem:netmst} enable us to deduce that $\widetilde F\in\lic{s}{\R^d}$, from which it follows that
	\[
		\Pdim(F\cap V)\geq\Pdim(F\cap\lambda_0)
		\geq\Pdim(F\cap\interior{\lambda_0})=\Pdim(\widetilde F\cap\interior{\lambda_0})=d.
	\]
	This results from applying~\cite[Theorem~D(b)]{Falconer:1994hx} to the set with large intersection $\widetilde F$ and the open set $\interior{\lambda_0}$, and from the packing counterpart of the monotonicity property satisfied by Hausdorff dimension, see Proposition~\ref{prp:Hausdorff}(\ref{item:prp:Hausdorff1}).

\subsection{Link with ubiquitous systems}

We showed in Section~\ref{sec:firstubiq} that if $(x_i,r_i)_{i\in\calI}$ denotes a homogeneous ubiquitous system in some nonempty open subset $U$ of $\R^d$, then for any real number $t>1$, the set $F_t$ defined by~(\ref{eq:df:Ft}) satisfies
	\[
		\Hdim(F_t\cap U)\geq\frac{d}{t},
	\]
	see Theorem~\ref{thm:lobndhomubsys}. The next result shows that the sets $F_t$ actually belong to the large intersection classes given by Definition~\ref{df:licloc}. Let us mention in passing that, similarly to Theorem~\ref{thm:lobndhomubsys}, heterogeneous versions of that result are proven in~\cite{Barral:2006fk,Durand:2006uq}.

\begin{Theorem}\label{thm:lihomubsys}
	Let $(x_i,r_i)_{i\in\calI}$ denote a homogeneous ubiquitous system in some nonempty open subset $U$ of $\R^d$. Then, for any real number $t>1$,
	\[
		F_t\in\lic{d/t}{U}.
	\]
\end{Theorem}

\begin{proof}
	As mentioned in Sections~\ref{subsec:unifdil} and~\ref{subsec:largeintclass}, neither the notion of homogeneous ubiquitous system nor the large intersection classes depend on the choice of the norm. For convenience, we assume throughout the proof that the space $\R^d$ is endowed with the supremum norm; the diameter of a set $E$ is denoted by $\diam{E}_\infty$.
	
	Let us consider two real numbers $\alpha\in(0,1)$ and $s\in(0,d/t)$, and a nonempty dyadic cube $\lambda\subseteq U$ with diameter at most one. Dilating the closure of $\lambda$ around its center, we obtain a closed ball $B$ with diameter $\alpha\diam{\lambda}_\infty$ that is contained in the interior of $\lambda$. We can reproduce the proof of Theorem~\ref{thm:lobndhomubsys} with $U$ being the interior of $\lambda$ and $I_\varnothing$ being the ball $B$. We thus obtain an outer measure $\mu$ supported in $F_t\cap\interior{\lambda}$ with total mass given by~(\ref{eq:df:zetaIvarnothing}) and such that Proposition~\ref{prp:bndmuB} holds.
	
	Moreover, let $(\lambda_n)_{n\geq 1}$ denote a covering of the set $F_t\cap\interior{\lambda}$ by dyadic cubes. As already observed multiple times, there exists a subset $\calN$ of $\N$ such that the cubes $\lambda_n$, for $n\in\calN$, are disjoint and contained in $\lambda$, and still cover $\interior{\lambda}$. If we assume in addition that the latter set has diameter less than $\ee^{-d/t}/2$, we see that every cube $\lambda_n$ with $n\in\calN$ is included in a closed ball $B_n$ with radius equal to $\diam{\lambda_n}_\infty$, and thus diameter smaller than $\ee^{-d/t}$. Applying Proposition~\ref{prp:bndmuB}, we get
	\[
		\mu(\lambda_n)\leq\mu(B_n)\leq 2\cdot 12^d\diam{B_n}_\infty^{d/t}\log\frac{1}{\diam{B_n}_\infty}
		\leq 2\cdot 12^d 2^{d/t}\diam{\lambda_n}_\infty^{d/t}\log\frac{1}{\diam{\lambda_n}_\infty}.
	\]
	Arguing as in the proof of the mass distribution principle, {\em i.e.}~Lemma~\ref{lem:massdist}, we get
	\begin{align*}
		(\alpha\diam{\lambda}_\infty)^{d/t}\log\frac{1}{\alpha\diam{\lambda}_\infty}
		&=\diam{I_\varnothing}_\infty^{d/t}\log\frac{1}{\diam{I_\varnothing}_\infty}=\mu(F_t\cap\interior{\lambda})\\
		&\leq 2\cdot 12^d 2^{d/t}\sum_{n=1}^\infty\diam{\lambda_n}_\infty^{d/t}\log\frac{1}{\diam{\lambda_n}_\infty}.
	\end{align*}
	We then use the fact that the function $r\mapsto r^{d/t-s}\log(1/r)$ is nondecreasing near zero. Specifically, if the diameter of $\lambda$ is less than $\ee^{-t/(d-st)}$, we have
	\[
		\diam{\lambda_n}_\infty^{d/t}\log\frac{1}{\diam{\lambda_n}_\infty}
		=\diam{\lambda_n}_\infty^s\diam{\lambda_n}_\infty^{d/t-s}\log\frac{1}{\diam{\lambda_n}_\infty}
		\leq\diam{\lambda_n}_\infty^s\diam{\lambda}_\infty^{d/t-s}\log\frac{1}{\diam{\lambda}_\infty}
	\]
	for all $n\geq 1$. Combining this observation with the previous bound, and then taking the infimum over all dyadic coverings, we obtain
	\[
		\netm^s_\infty(F_t\cap\lambda)\geq\netm^s_\infty(F_t\cap\interior{\lambda})
		\geq\frac{\alpha^{d/t}\log(\alpha\diam{\lambda}_\infty)}{2\cdot 12^d 2^{d/t}\log\diam{\lambda}_\infty}\diam{\lambda}_\infty^s,
	\]
	with the proviso that the diameter of $\lambda$ is smaller than $\delta_{s,t}$, defined as the minimum of $\ee^{-d/t}/2$ and $\ee^{-t/(d-st)}$. Now, thanks to Lemma~\ref{lem:netmlambda}, we may replace $\diam{\lambda}_\infty^s$ by $\netm^s_\infty(\lambda)$. Hence, letting $\alpha$ tend to one, we end up with
	\[
		\netm^s_\infty(F_t\cap\lambda)\geq\frac{\netm^s_\infty(\lambda)}{2\cdot 12^d 2^{d/t}}
	\]
	for any dyadic cube $\lambda\subseteq U$ with diameter smaller than $\delta_{s,t}$. The restriction on the diameter may easily be removed. Indeed, if $\lambda$ is an arbitrary dyadic cube contained in $U$, applying Lemma~\ref{lem:netmlambdaV} to its interior, and then Lemma~\ref{lem:netmlambda} again, we get
	\[
		\netm^s_\infty(F_t\cap\lambda)\geq\netm^s_\infty(F_t\cap\interior{\lambda})\geq\frac{\netm^s_\infty(\interior{\lambda})}{2\cdot 12^d 2^{d/t}}=\frac{\netm^s_\infty(\lambda)}{2\cdot 12^d 2^{d/t}}
	\]
	for all real numbers $s\in(0,d/t)$ and all dyadic cubes $\lambda\subseteq U$. Finally, Lemma~\ref{lem:netmst} implies that for all such $s$ and $\lambda$, we have in fact
	\[
		\netm^s_\infty(F_t\cap\lambda)=\netm^s_\infty(\lambda).
	\]
	The result follows from another utilization of Lemma~\ref{lem:netmlambdaV}.
\end{proof}

\subsection{The Jarn\'ik-Besicovitch theorem revisited}

As an immediate application, let us show that the set $J_{d,\tau}$ defined by~(\ref{eq:df:Jdtau}) is a set with large intersection. Recall that $J_{d,\tau}$ is formed by the points that are approximable at rate at least $\tau$ by those with rational coordinates. We already know that this set coincides with the whole space $\R^d$ when $\tau\leq 1+1/d$ and has Hausdorff dimension equal to $(d+1)/\tau$ in every nonempty open subset of $\R^d$ otherwise, see Corollary~\ref{cor:Dirichlet} and Theorem~\ref{thm:JarnikBesicovitch}.

We obtained the latter dimensional result, known as the Jarn\'ik-Besicovitch theorem, in Section~\ref{subsec:JarnikBesicovitch} above. We started from the following two observations: the family $(p/q,q^{-1-1/d})_{(p,q)\in\Z^d\times\N}$ is a homogeneous ubiquitous system in the whole space $\R^d$\,; for this system, the sets $F_t$ defined by~(\ref{eq:df:Ft}) coincide with the sets $J_{d,\tau}$, with the proviso that the parameters are such that $t=\tau d/(d+1)$. Thanks to Theorem~\ref{thm:lihomubsys}, the same observations lead to the next statement.

\begin{Corollary}\label{cor:JarnikBesicovitchsli}
	For any real paramter $\tau>1+1/d$, the set $J_{d,\tau}$ is a set with large intersection in the whole space $\R^d$ with dimension at least $(d+1)/\tau$, namely,
	\[
		J_{d,\tau}\in\lic{(d+1)/\tau}{\R^d}.
	\]
\end{Corollary}

This result was already obtained in~\cite{Falconer:1994hx}. Combined with Proposition~\ref{prp:licPdim}, this shows in particular that the set $J_{d,\tau}$ has packing dimension equal to $d$ in every nonempty open subset of $\R^d$. For the sake of completeness, let us point out that in the opposite case where $\tau\leq 1+1/d$, the set $J_{d,\tau}$ clearly belongs to the class $\lic{d}{\R^d}$ because it coincides with the whole space $\R^d$ itself.


\section{Transference principles}\label{sec:transference}

The purpose of this section is to extend the above theory of homogeneous ubiquitous systems, especially Theorems~\ref{thm:lobndhomubsys} and~\ref{thm:lihomubsys}, toward Hausdorff measures and large intersection classes associated with arbitrary gauge functions, thereby aiming at a complete and precise description of the size and large intersection properties of associated limsup sets.

\subsection{Homogeneous $g$-ubiquitous system}

We begin by recalling the main results of Section~\ref{sec:firstubiq}, and shedding new light thereon. Let $\calI$ be a countably infinite index set, let $(x_i,r_i)_{i\in\calI}$ be an approximation system in the sense of Definition~\ref{df:approxsys}, and let $F_t$ be the sets defined by~(\ref{eq:df:Ft}), namely,
	\[
		F_t=\left\{x\in\R^d\bigm||x-x_i|<r_i^t\quad\text{for i.m.~} i\in\calI\right\}.
	\]
	Moreover, let $U$ denote a nonempty open subset of $\R^d$. According to Definition~\ref{df:homubiqsys}, the family is a homogeneous ubiquitous system in $U$ if the set $F_1$ has full Lebesgue measure in $U$. In that situation, Theorem~\ref{thm:lobndhomubsys} shows that for any real number $t>1$,
	\[
		\Hdim(F_t\cap U)\geq\frac{d}{t}.
	\]
	In fact, the set $F_t\cap U$ has positive Hausdorff measure with respect to the gauge function $r\mapsto r^{d/t}|\log r|$. Thus, the mere fact that the set $F_1$ has full Lebesgue measure in $U$ yields an {\em a priori} lower bound on the Hausdorff dimension of the sets $F_t$, which are smaller than $F_1$ when $t$ is larger than one.
	
Let us adopt a new perspective: we consider from now on that the set
	\begin{equation}\label{eq:df:Fxiri}
		\frakF((x_i,r_i)_{i\in\calI})=\left\{x\in\R^d\bigm||x-x_i|<r_i\quad\text{for i.m.~} i\in\calI\right\}
	\end{equation}
	is that of which we seek to estimate the size. In the above notations, this set coincides with the set $F_1$. The trick however is to observe that for any $t\geq 1$, this set also coincides with the set $F_t$ associated with the underlying family $(x_i,r_i^{1/t})_{i\in\calI}$, which is an approximation system as well. In that new situation, Theorem~\ref{thm:lobndhomubsys} ensures that if $(x_i,r_i^{1/t})_{i\in\calI}$ is a homogeneous ubiquitous system in $U$, that is, if
	\begin{equation}\label{eq:condFxiritmod}
		\text{for~}\leb^d\text{-a.e.~}x\in U
		\quad \exists\text{~i.m.~}i\in\calI
		\qquad |x-x_i|<r_i^{1/t},
	\end{equation}
	then the set $\frakF((x_i,r_i)_{i\in\calI})$ has positive Hausdorff measure in the open set $U$ with respect to the gauge function $r\mapsto r^{d/t}|\log r|$, so in particular
	\[
		\Hdim(\frakF((x_i,r_i)_{i\in\calI})\cap U)\geq\frac{d}{t}.
	\]
	
A further way to recast this result is to let $g$ denote the gauge function $r\mapsto r^{d/t}$, to rewrite the assumption~(\ref{eq:condFxiritmod}) in the form
	\begin{equation}\label{eq:condFxirigauge}
		\leb^d(U\setminus\frakF((x_i,g(r_i)^{1/d})_{i\in\calI}))=0,		
	\end{equation}
	where the involved set is defined as in~(\ref{eq:df:Fxiri}), and to reinterpret the conclusion as the fact that the set $\frakF((x_i,r_i)_{i\in\calI})$ has positive Hausdorff measure in $U$ with respect to the gauge function $r\mapsto g(r)|\log r|$. Note that the gauge function $g$ is $d$-normalized in the sense of Definition~\ref{df:normgauge}, because $g$ coincides on the interval $(0,\infty)$ with its $d$-normalization $g_d$, defined by~(\ref{eq:df:normgauge}). Thus, the condition~(\ref{eq:condFxirigauge}) still holds when $g$ is replaced by $g_d$. In that situation, the approximation system $(x_i,r_i)_{i\in\calI}$ will be called {\em $g$-ubiquitous}, in accordance with the following definition.

\begin{Definition}\label{df:homubiqsysgauge}
	Let $\calI$ be a countably infinite index set, let $(x_i,r_i)_{i\in\calI}$ be an approximation system in $\R^d\times(0,\infty)$, let $g$ be a gauge function and let $U$ be a nonempty open subset of $\R^d$. We say that $(x_i,r_i)_{i\in\calI}$ is a {\em homogeneous $g$-ubiquitous system} in $U$ if the following condition holds:
		\[
			\leb^d(U\setminus\frakF((x_i,g_d(r_i)^{1/d})_{i\in\calI}))=0.
		\]
\end{Definition}

The latter condition means that for Lebesgue-almost every point $x$ in the open set $U$, the inequality $|x-x_i|<g_d(r_i)^{1/d}$ holds for infinitely many indices $i$ in $\calI$. Hence, the previous definition may be seen as an extension of that of a homogeneous ubiquitous system. In fact, according to Definitions~\ref{df:homubiqsys} and~\ref{df:homubiqsysgauge}, respectively, an approximation system is a homogeneous ubiquitous system in some nonempty open set $U$ if and only if it is homogeneously ubiquitous in $U$ with respect to any gauge function whose $d$-normalization is $r\mapsto r^d$.

\subsection{Mass transference principle}

Remarkably, Beresnevich and Velani~\cite{Beresnevich:2005vn} managed to extend the above approach to any gauge function $g$, and also improved the above conclusion. Specifically, they established the following {\em mass transference principle} for the sets defined by~(\ref{eq:df:Fxiri}).

\begin{Theorem}[Beresnevich and Velani]\label{thm:masstransprinc}
	Let $\calI$ be a countably infinite index set, let $(x_i,r_i)_{i\in\calI}$ be an approximation system in $\R^d\times(0,\infty)$, let $g$ be a gauge function and let $U$ be a nonempty open subset of $\R^d$. If $(x_i,r_i)_{i\in\calI}$ is a homogeneous $g$-ubiquitous system in $U$, then for every nonempty open subset $V$ of $U$,
	\[
		\hau^g(\frakF((x_i,r_i)_{i\in\calI})\cap V)=\hau^g(V).
	\]
\end{Theorem}

Some of the ideas supporting Theorem~\ref{thm:masstransprinc} are similar to those developed in the proof of Theorem~\ref{thm:lobndhomubsys} above. However, Theorem~\ref{thm:lobndhomubsys} being essentially concerned with Hausdorff dimension only, its proof does not require as much accuracy as in the proof of Theorem~\ref{thm:masstransprinc}, where Hausdorff measures associated with arbitrary gauge functions are considered. The proof of Theorem~\ref{thm:masstransprinc} is therefore somewhat technically involved. Consequently, we omit it from these notes, and we refer the reader to Beresnevich and Velani's paper~\cite{Beresnevich:2005vn}. We just mention that Theorem~\ref{thm:masstransprinc} is a straightforward consequence of Theorem~2 in~\cite{Beresnevich:2005vn}, except that Beresnevich and Velani only considered $d$-normalized functions. However, this assumption may easily be removed with the help of Propositions~\ref{prp:compnormalgauge} and~\ref{prp:dichogauge}.

Theorem~\ref{thm:masstransprinc} is remarkable because of its universality. The general philosophy behind this result is that it enables one to automatically convert a property concerning the Lebesgue measure of a limsup of balls to a property concerning the Hausdorff measures of similar sets obtained by dilating the balls. This leads in particular to a full description of the size properties of limsup of balls for which the description of the Lebesgue measure is known. This principle may be applied to many approximation systems arising in metric number theory and probability, especially those coming from eutaxic sequences and optimal regular systems, see Sections~\ref{subsec:eutaxyapprox} and~\ref{subsec:optregsysapprox}, as well as the applications discussed in Sections~\ref{sec:inhom}--\ref{sec:Poisson}. However, our approach below relies on the notion of describability introduced in Section~\ref{sec:desc} and, at heart, on the large intersection transference principle. Hence, the mass transference will never be used {\em per se} in what follows.

\subsection{Large intersection transference principle}\label{subsec:largeintprinc}

The purpose of this section is to give an analog of the mass transference principle for large intersection properties. In the spirit of Theorem~\ref{thm:masstransprinc}, this result leads to a very precise description of the large intersection properties of a limsup of balls in terms of arbitrary gauge functions. Accordingly, we first need to introduce large intersection classes that are associated with arbitrary gauge functions, thereby generalizing the original classes introduced by Falconer and presented in Section~\ref{subsec:largeintclass}. We adopt the same viewpoint as in the definition of the localized classes $\lic{s}{U}$, namely, Definition~\ref{df:licloc}. In particular, the generalized classes are defined with the help of outer net measures; these are built in terms of general gauge functions and coverings by dyadic cubes.

\subsubsection{Net measures revisited}\label{subsubsec:netmrev}

The net measures were first discussed in Section~\ref{subsubsec:normgaugenetm}. We restrict ourselves here to gauge functions that are $d$-normalized in the sense of Definition~\ref{df:normgauge}. The resulting outer net measures then satisfy additional properties that are in fact necessary to an appropriate definition of the generalized classes.

If $g$ denotes a $d$-normalized gauge function, the set of all real numbers $\eps>0$ such that $g$ is nondecreasing on $[0,\eps]$ and $r\mapsto g(r)/r^d$ is nonincreasing on $(0,\eps]$ is nonempty. We may thus define $\eps_g$ as the supremum of this set, and next $\Lambda_g$ as the collection of all dyadic cubes with diameter less than $\eps_g$. We then consider the premeasure $g\circ\diam{\cdot}_{\Lambda_g}$ that sends each set $\lambda$ in $\Lambda_g$ to $g(\diam{\lambda})$, and Theorem~\ref{thm:absmeas} allows us to define similarly to~(\ref{eq:df:zetaasta}) the outer measure
	\[
		\netm^g_\infty=(g\circ\diam{\cdot}_{\Lambda_g})^\ast
	\]
	resulting from coverings by dyadic cubes with diameter less than $\eps_g$.
	
The outer measure $\netm^g_\infty$ provides a lower bound on the corresponding net measure $\netm^g$, which is defined by~(\ref{eq:df:netmgauge}) and is comparable with the Hausdorff measure $\hau^g$, see Proposition~\ref{prp:compnetm}. As a consequence, there is a real number $\kappa\geq 1$ independent on $g$ such that for any set $E\subseteq\R^d$,
	\begin{equation}\label{eq:comphaunetmgauge}
		\kappa\,\hau^g(E)\geq\netm^g_\infty(E).
	\end{equation}

Recall that the outer net measures $\netm^s_\infty$, defined by~(\ref{eq:df:netmsinfty}) for $s\in(0,d]$, played a crucial r\^ole in the characterization of Falconer's classes and the definition of their localized counterparts $\lic{s}{U}$, see Theorem~\ref{thm:caraclic} and Definition~\ref{df:licloc}, respectively. These outer measures are actually an instance of the above construction. Specifically, for any $s\in(0,d]$, the gauge function $r\mapsto r^s$ is clearly $d$-normalized and the parameter $\eps_{r\mapsto r^s}$ is infinite. Hence, the collection $\Lambda_{r\mapsto r^s}$ coincides with the whole $\Lambda$, from which it follows that $\netm^{r\mapsto r^s}_\infty$ is merely equal to $\netm^s_\infty$. The outer measures $\netm^g_\infty$ thus extend naturally those used in Section~\ref{sec:largeint}\,; this hints at why they will play a key r\^ole in the definition of the generalized large intersection classes.

Finally, it is useful to point out that the value in each dyadic cube of the $\netm^g_\infty$-mass of Lebesgue full sets has a very simple expression.

\begin{Lemma}\label{lem:netmfullLeb}
	For any $d$-normalized gauge function $g$, any dyadic cube $\lambda$ in $\Lambda_g$, and any subset $F$ of $\R^d$, the following implication holds:
	\[
		\leb^d(\lambda\setminus F)=0
		\qquad\Longrightarrow\qquad
		\netm^g_\infty(F\cap\lambda)=g(\diam{\lambda}).
	\]	
\end{Lemma}

\begin{proof}
	The intersection set $F\cap\lambda$ is obviously covered by the sole cube $\lambda$, so that
	\[
		\netm^g_\infty(F\cap\lambda)\leq g(\diam{\lambda}).
	\]
	In order to prove the reverse inequality, let us consider a covering $(\lambda_n)_{n\geq 1}$ of the intersection set $F\cap\lambda$ by dyadic cubes with diameter less than $\eps_g$. If $\lambda$ is contained in some cube $\lambda_{n_0}$, the fact that $g$ is nondecreasing on $[0,\eps_g)$ implies that
	\[
		g(\diam{\lambda})\leq g(\diam{\lambda_{n_0}})\leq\sum_{n=1}^\infty g(\diam{\lambda_n}).
	\]
	Otherwise, we observe that the cubes $\lambda_n\subset\lambda$ suffice to cover the set $F\cap\lambda$. Along with the fact that the mapping $r\mapsto g(r)/r^d$ is nonincreasing on $(0,\eps_g)$, this yields
	\begin{align*}
		\sum_{n=1}^\infty g(\diam{\lambda_n})
		&\geq\sum_{n\geq 1\atop\lambda_n\subset\lambda}\frac{g(\diam{\lambda_n})}{\diam{\lambda_n}^d}\diam{\lambda_n}^d
		\geq\frac{g(\diam{\lambda})}{\diam{\lambda}^d}\sum_{n\geq 1\atop\lambda_n\subset\lambda}\diam{\lambda_n}^d
		=\frac{g(\diam{\lambda})}{\diam{\lambda}^d}\kappa'^d\sum_{n\geq 1\atop\lambda_n\subset\lambda}\leb^d(\lambda_n)\\
		&\geq\frac{g(\diam{\lambda})}{\diam{\lambda}^d}\kappa'^d\leb^d(F\cap\lambda)=\frac{g(\diam{\lambda})}{\diam{\lambda}^d}\kappa'^d\leb^d(\lambda)=g(\diam{\lambda}).
	\end{align*}
	Here, $\kappa'$ stands for the diameter of the unit cube of $\R^d$, which depends on the choice of the norm. We conclude by taking the infimum over all coverings $(\lambda_n)_{n\geq 1}$.
\end{proof}

The previous result may be used to express the $\netm^g_\infty$-mass of dyadic cubes in terms of their diameters. As a matter of fact, using the notations of Lemma~\ref{lem:netmfullLeb}, if the set $F$ is chosen to be the cube $\lambda$ itself, or its interior, we get
	\begin{equation}\label{eq:netmlambdagauge}
		\netm^g_\infty(\lambda)=\netm^g_\infty(\interior{\lambda})=g(\diam{\lambda}),
	\end{equation}
	a formula which extends Lemma~\ref{lem:netmlambda} to any $d$-normalized gauge function. Likewise, all the ancillary lemmas from Section~\ref{subsubsec:anclemlic} may be extended to such gauge functions; we refer to~\cite{Durand:2007uq} for precise statements, see in particular Lemmas~10 and~12 therein.

\subsubsection{Generalized large intersection classes}

We are now in position to define the large intersection classes that are associated with general gauge functions. We defined those classes in~\cite{Durand:2007uq}, and we refer to that paper for all the proofs and details that are missing in the presentation below. As mentioned above, there is a lineage with the definition of the localized classes $\lic{s}{U}$, see Definition~\ref{df:licloc}.

We write $h\prec g$ to indicate that two $d$-normalized gauge functions $g$ and $h$ are such that the quotient $h/g$ monotonically tends to infinity at zero, that is,
	\[
		h\prec g
		\qquad\Longleftrightarrow\qquad
		\lim_{r\downarrow 0}\uparrow\frac{h(r)}{g(r)}=\infty.
	\]
	This means essentially that $h$ increases faster than $g$ near the origin. Note that $g$ may vanish in a neighborhood of zero; in that situation, we adopt the convention that $h\prec g$ for any choice of $h$, even if $h$ also vanishes near zero.

\begin{Definition}\label{df:liclocgauge}
	For any gauge function $g$ and any nonempty open subset $U$ of $\R^d$, the class $\lic{g}{U}$ of {\em sets with large intersection in $U$ with respect to $g$} is the collection of all $G_\delta$-subsets $F$ of $\R^d$ such that
	\begin{equation}\label{eq:df:liclocgauge}
		\netm^h_\infty(F\cap V)=\netm^h_\infty(V)
	\end{equation}
	for any $d$-normalized gauge function $h$ satisfying $h\prec g_d$, where $g_d$ denotes the $d$-normalization of $g$ defined by~(\ref{eq:df:normgauge}), and for any open subset $V$ of $U$.
\end{Definition}

Note that the class $\lic{g}{U}$ associated with a given gauge function $g$ coincides with that associated with its $d$-normalization, namely, the class $\lic{g_d}{U}$. One may therefore restrict oneself to $d$-normalized gauge functions when studying large intersection properties. Moreover, if two gauge functions are such that their respective $d$-normalizations match near the origin, the corresponding classes coincide. Besides, let us point out that, similarly to the classes $\lic{s}{U}$, the generalized classes defined above do not depend on the choice of the norm on $\R^d$, either.

With a view to detailing the connection with the localized classes $\lic{s}{U}$, we associate with any gauge function $g$ the following dimensional parameter $s_g$.

\begin{Definition}
	Let $g$ be a gauge function with $d$-normalization denoted by $g_d$. The {\em lower dimension} of the gauge function $g$ is the parameter defined by
	\[
		s_g=\sup\left\{s\in(0,d]\:|\:(r\mapsto r^s)\prec g_d\right\},
	\]
	with the convention that the supremum is equal to zero if the inner set is empty.
\end{Definition}

Obviously, we have $s_g=\min\{s,d\}$ if the gauge function $g$ is of the form $r\mapsto r^s$, with $s>0$. The relationship between the generalized classes $\lic{g}{U}$ and the original classes $\lic{s}{U}$ is now detailed in the next statement.

\begin{Proposition}\label{prp:linkFalcgauge}
	For any gauge function $g$ with lower dimension satisfying $s_g>0$ and for any nonempty open subset $U$ of $\R^d$, the following inclusion holds:
	\[
		\lic{g}{U}\subseteq\lic{s_g}{U}.
	\]
	In particular, for any set $F$ in $\lic{g}{U}$ and for any nonempty open set $V\subseteq U$,
	\[
	\Hdim(F\cap V)\geq s_g
	\qquad\text{and}\qquad
	\Pdim(F\cap V)=d.
	\]
	Moreover, the left-hand inequality above still holds if $s_g$ vanishes.
\end{Proposition}

\begin{proof}
	Let us assume that $s_g$ is positive and let us consider a set $F$ in $\lic{g}{U}$. First, $F$ is a $G_\delta$-subset of $\R^d$. Then, for any $s\in(0,s_g)$, we have $(r\mapsto r^s)\prec g_d$, and Definition~\ref{df:liclocgauge} implies that
		\[
			\netm^{r\mapsto r^s}_\infty(F\cap V)=\netm^{r\mapsto r^s}_\infty(V)
		\]
		for any open subset $V$ of $U$. Recalling that the outer measure $\netm^{r\mapsto r^s}_\infty$ is identical to the outer measure $\netm^s_\infty$ defined by~(\ref{eq:df:netmsinfty}), we deduce from Definition~\ref{df:licloc} that the set $F$ belongs to the original localized class $\lic{s_g}{U}$.
	
	Moreover, applying Theorem~\ref{thm:licloc} and Proposition~\ref{prp:licPdim}, we deduce that the set $F$ has Hausdorff dimension at least $s_g$ and packing dimension equal to $d$ in every nonempty open subset $V$ of $U$. Finally, in view of Definition~\ref{df:liclocgauge}, any set in the class $\lic{g}{U}$ has to be dense in $U$. Therefore, the Hausdorff dimension of $F\cap V$ is necessarily bounded below by zero, that is, by $s_g$ when this value vanishes.
\end{proof}

Choosing $U$ equal to the whole space $\R^d$, we clearly deduce from Proposition~\ref{prp:linkFalcgauge} a statement bearing on Falconer's original classes $\lic{s}{\R^d}$. In addition, as easily seen for instance musing on the examples discussed in Sections~\ref{sec:inhom}--\ref{sec:Poisson}, the inclusion appearing in the statement of Proposition~\ref{prp:linkFalcgauge} is strict.

Let us now briefly discuss the case in which the gauge function $g$ has a $d$-normalization $g_d$ that vanishes in a neighborhood of zero. The $d$-normalized gauge function that is constant equal to zero is denoted by $\zerofunc$\,; let us mention in passing that its lower dimension clearly satisfies $s_\zerofunc=d$.

\begin{Proposition}\label{prp:lic0fullLeb}
	For any nonempty open set $U\subseteq\R^d$, the large intersection class $\lic{\zerofunc}{U}$ is formed by the $G_\delta$-subsets of $\R^d$ with full Lebesgue measure in $U$.
\end{Proposition}

\begin{proof}
	Let us consider a $G_\delta$-subset $F$ of $\R^d$ with full Lebesgue measure in $U$. Lemma~\ref{lem:netmfullLeb}, combined with~(\ref{eq:netmlambdagauge}), ensures that for any $d$-normalized gauge function $g$ and any dyadic cube $\lambda$ in $\Lambda_g$ that is contained in $U$,
		\[
			\netm^g_\infty(F\cap\lambda)=g(\diam{\lambda})=\netm^g_\infty(\lambda).
		\]
		We finally conclude that $F$ belongs to the class $\lic{\zerofunc}{U}$ thanks to the extension of Lemma~\ref{lem:netmlambdaV} to arbitrary $d$-normalized gauge functions, see~\cite[Lemma~10]{Durand:2007uq}.

	Conversely, let us consider a set $F$ in the class $\lic{\zerofunc}{U}$. First, $F$ is necessarily a $G_\delta$-set. Moreover, we know that~(\ref{eq:df:liclocgauge}) holds in particular for the $d$-normalized gauge function $r\mapsto r^d$ and for all open balls $\opball(x,r)$ contained in $U$. Using~(\ref{eq:comphaunetmgauge}) and~(\ref{eq:df:liclocgauge}), and letting $\kappa''$ be the constant appearing in Proposition~\ref{prp:comphauleb}, we get
		\[
			\kappa\kappa''\leb^d(F\cap\opball(x,r))=\kappa\hau^d(F\cap\opball(x,r))\geq\netm^d_\infty(F\cap\opball(x,r))=\netm^d_\infty(\opball(x,r)).
		\]
		Letting $\lambda$ be a nonempty dyadic cube contained in $\opball(x,r)$ with minimal generation, we have $\diam{\lambda}\geq c\,r$ for some $c>0$ depending on the norm only, and Lemma~\ref{lem:netmlambda} yields
		\[
			\netm^d_\infty(\opball(x,r))\geq\netm^d_\infty(\lambda)=\diam{\lambda}^d
			\geq c^d r^d=\frac{c^d}{\leb^d(\opball(0,1))}\leb^d(\opball(x,r)),
		\]
		where the last equality follows from fact that the Lebesgue measure is translation invariant and homogeneous with degree $d$ with respect to dilations. Hence,
		\[
			\frac{\leb^d(F\cap\opball(x,r))}{\leb^d(\opball(x,r))}
			\geq\frac{c^d}{\kappa\kappa''\leb^d(\opball(0,1))}>0
		\]
		for any open ball $\opball(x,r)$ contained in $U$. It follows from the Lebesgue density theorem that $F$ has full Lebesgue measure in $U$, see~\cite[Corollary~2.14]{Mattila:1995fk}.
\end{proof}

The various remarkable properties of the large intersection classes $\lic{g}{U}$ naturally extend those satisfied by Falconer's classes, see Section~\ref{subsec:largeintclass}. We begin by stating the properties that follow immediately from the definition. The next result may be seen as a partial analog of Proposition~\ref{prp:morelic}.

\begin{Proposition}\label{prp:morelicgauge}
	Let $g$ be in the set $\gauge$ of gauge functions, let $g_d$ denote its $d$-normalization, and let $U$ be a nonempty open subset of $\R^d$.
	\begin{enumerate}
		\item\label{item:prp:morelicgauge1} Any $G_\delta$-subset of $\R^d$ that contains a set in $\lic{g}{U}$ also belongs to $\lic{g}{U}$.
		\item\label{item:prp:morelicgauge2} The following equalities hold:
		\[
			\lic{g}{U}=\bigcap_{V\text{ open}\atop\emptyset\neq V\subseteq U}\lic{g}{V}
			\qquad\text{and}\qquad
			\lic{g}{U}=\bigcap_{h\in\gauge\atop h_d\prec g_d}\lic{h}{U}.
		\]
	\end{enumerate}
\end{Proposition}

All the properties are essentially immediate from the definition of the generalized large intersection classes, and the proof is therefore omitted here. The next result extends Theorem~\ref{thm:stablic} to the large intersection classes $\lic{g}{U}$, thereby showing that they enjoy the same stability properties as Falconer's classes.

\begin{Theorem}\label{thm:stablicgauge}
	Let $g$ be a gauge function with $d$-normalization denoted by $g_d$ and with lower dimension denoted by $s_g$, and let $U$ be a nonempty open subset of $\R^d$. The following properties hold:
	\begin{enumerate}
		\item\label{item:thm:stablicgauge1} the class $\lic{g}{U}$ is closed under countable intersections;
		\item\label{item:thm:stablicgauge2} for any bi-Lipschitz transformation $f:U\to\R^d$ and any set $F\subseteq\R^d$,
		\[
			F\in\lic{g}{f(U)}
			\qquad\Longrightarrow\qquad
			f^{-1}(F)\in\lic{g}{U}\,;
		\]
		\item\label{item:thm:stablicgauge3} for any set $F$ in the class $\lic{g}{U}$ and for every gauge function $h$,
		\[
			h_d\prec g_d
			\qquad\Longrightarrow\qquad
			\hau^h(F\cap U)=\hau^h(U).
		\]
	\end{enumerate}
\end{Theorem}

\begin{proof}[A few words on the proof]
The result corresponds to Theorem~1 in~\cite{Durand:2007uq}, so we refer to that paper for the whole proof. Let us just mention that the statement in~\cite{Durand:2007uq} only addresses the $d$-normalized gauge functions $g$ for which the parameter $\ell_g$ defined by~(\ref{eq:df:liminfgrd}) is positive. In that situation, note that the Hausdorff $h$-measure of the set $F\cap U$ that appears in~(\ref{item:thm:stablicgauge3}) is actually infinite, as a consequence of Propositions~\ref{prp:compgauge0} and~\ref{prp:dichogauge}. Furthermore, the normalization assumption made in~\cite{Durand:2007uq} may easily be dropped with the help of Proposition~\ref{prp:compnormalgauge}. In addition, Theorem~\ref{thm:stablicgauge} clearly holds for $\ell_g=0$. Indeed, in that situation, the gauge function $g_d$ vanishes near zero and Proposition~\ref{prp:lic0fullLeb} ensures that the class $\lic{g}{U}$ is formed by the $G_\delta$-sets with full Lebesgue measure in $U$. All the properties are thus satisfied, even~(\ref{item:thm:stablicgauge3}) which may be obtained with the help of Propositions~\ref{prp:compgauge0} and~\ref{prp:dichogauge}.
\end{proof}

A plain consequence of Theorem~\ref{thm:stablicgauge} is that if $(F_n)_{n\geq 1}$ is a sequence of sets in $\lic{g}{U}$ and if $h$ is a gauge function, then
	\[
		h_d\prec g_d
		\qquad\Longrightarrow\qquad
		\hau^h\left(\bigcap_{n=1}^\infty F_n\cap U\right)=\hau^h(U).
	\]
	Thanks to Proposition~\ref{prp:dichogauge}, the latter equality may be rewritten in various alternate forms depending on the value of the parameter $\ell_h$ defined as in~(\ref{eq:df:liminfgrd}). In addition, this equality implies that the intersection of all the sets $F_n$ has Hausdorff dimension bounded below by $s_g$, and this bound is clearly attained if one of the sets has Hausdorff dimension at most $s_g$.

\subsubsection{The transference principle}

The classes associated with arbitrary gauge functions being defined, we may state the large intersection analog counterpart of the mass transference principle. While the latter result discusses the Hausdorff measures for the set $\frakF((x_i,r_i)_{i\in\calI})$ defined by~(\ref{eq:df:Fxiri}), the next statement concerns its large intersection properties.

\begin{Theorem}\label{thm:largeintprinc}
	Let $\calI$ be a countably infinite set, let $(x_i,r_i)_{i\in\calI}$ be an approximation system in $\R^d\times(0,\infty)$, let $g$ be a gauge function and let $U$ be a nonempty open subset of $\R^d$. If $(x_i,r_i)_{i\in\calI}$ is a homogeneous $g$-ubiquitous system in $U$, then
	\[
		\frakF((x_i,r_i)_{i\in\calI})\in\lic{g}{U}.
	\]
\end{Theorem}

The result is a straightforward consequence of Theorem~2 in~\cite{Durand:2007uq}\,; we refer to that paper for a comprehensive proof. Similarly to the mass transference principle, some ideas supporting Theorem~\ref{thm:largeintprinc} are analogous to those developed in the proof of Theorem~\ref{thm:lobndhomubsys} above, and also that of Theorem~\ref{thm:lihomubsys} which is more specifically concerned with large intersection properties. Just as the mass transference principle extends Theorem~\ref{thm:lobndhomubsys} to arbitrary Hausdorff measures, the above {\em large intersection transference principle} may indeed be seen as an extension of Theorem~\ref{thm:lihomubsys}. To be specific, let $(x_i,r_i)_{i\in\calI}$ denote a homogeneous ubiquitous system in $U$ in the sense of Definition~\ref{df:homubiqsys}. Thus, for any $t>1$, the family $(x_i,r_i^t)_{i\in\calI}$ is homogeneously ubiquitous in $U$ with respect to the gauge function $r\mapsto r^{d/t}$. Theorem~\ref{thm:largeintprinc} then ensures that the set $F_t$ defined by~(\ref{eq:df:Ft}) is a set with large intersection in $U$ with respect to the same gauge function. This gauge function clearly has lower dimension equal to $d/t$, so we deduce with the help of Corollary~\ref{prp:linkFalcgauge} that the set $F_t$ belongs to Falconer's class $\lic{d/t}{U}$, which is exactly the conclusion of Theorem~\ref{thm:lihomubsys}.

Furthermore, the large intersection transference principle nicely complements the mass transference principle: under similar hypotheses, it shows that the size properties of sets of the form~(\ref{eq:df:Fxiri}) are in fact stable under countable intersections and bi-Lipschitz mappings. Also, due to Proposition~\ref{prp:morelicgauge}(\ref{item:prp:morelicgauge2}) and Theorem~\ref{thm:stablicgauge}(\ref{item:thm:stablicgauge3}), it implies that for any gauge function $h$ and any nonempty open set $V\subseteq U$,
	\[
		h_d\prec g_d
		\qquad\Longrightarrow\qquad
		\hau^h(\frakF((x_i,r_i)_{i\in\calI})\cap V)=\infty=\hau^h(V).
	\]
	Note that the last equality follows from Proposition~\ref{prp:dichogauge}(\ref{item:prp:dichogauge1}), because $h(r)/r^d$ necessarily tends to infinity as $r$ goes to zero. Unfortunately, we may not apply this with $h$ being equal to $g$, thereby failing narrowly to recover the conclusion of the mass transference principle, specifically,
	\[
		\hau^g(\frakF((x_i,r_i)_{i\in\calI})\cap V)=\hau^g(V).
	\]
	However, we may often in practice circumvent this problem and, through the notion of describability introduced in Section~\ref{sec:desc}, the large intersection transference principle will be sufficient to describe both size {\em and} large intersection properties of limsup of balls for which the description of the Lebesgue measure is known. We shall apply this principle to the limsup sets issued from eutaxic sequences and optimal regular systems, see Sections~\ref{subsec:eutaxyapprox} and~\ref{subsec:optregsysapprox}, and the examples discussed in Sections~\ref{sec:inhom}--\ref{sec:Poisson}.


\section{Describable sets}\label{sec:desc}

The purpose of this section is to combine the conclusions of the mass and the large intersection principles discussed in Section~\ref{sec:transference} and place them in a wider setting. This new framework aims at describing in a complete and precise manner the size and large intersection properties of various subsets of $\R^d$ that are derived from the eutaxic sequences and optimal regular systems discussed in Sections~\ref{sec:eutaxy} and~\ref{sec:optregsys}.

Note that the size and large intersection properties of Lebesgue full sets are easily described as follows. Let $E$ be a Borel subset of $\R^d$ and let $U$ be a nonempty open subset of $\R^d$. If $E$ has full Lebesgue measure in $U$, then Proposition~\ref{prp:dichogauge} ensures that for any gauge function $g$ and any nonempty open set $V\subseteq U$,
	\[
		\hau^g(E\cap V)=\hau^g(V).
	\]
	In particular, the set $E$ has Hausdorff dimension equal to $d$ in $V$. Furthermore, under the stronger assumption that $E$ admits a $G_\delta$-subset with full Lebesgue measure in $U$, Propositions~\ref{prp:lic0fullLeb} and~\ref{prp:morelicgauge}(\ref{item:prp:morelicgauge2}) imply that for any gauge function $g$ and any nonempty open set $V\subseteq U$,
	\[
		\exists F\in\lic{g}{V} \qquad F\subseteq E.
	\]
	By Proposition~\ref{prp:linkFalcgauge}, the set $E$ thus admits a subset in the class $\lic{d}{V}$, thereby having packing dimension equal to $d$ in $V$. The above description of the size and large intersection properties of Lebesgue full sets being both precise and complete, we shall exclude such sets from our analysis.
	
	Our framework will enable us to achieve a similar description for some Lebesgue null sets. The collection of all Borel subsets of $\R^d$ that are Lebesgue null in the open set $U$ is denoted by $\zeroleb(U)$, specifically,
	\[
		\zeroleb(U)=\{E\in\bor\:|\:\leb^d(E\cap U)=0\},
	\]
	where $\bor$ is the Borel $\sigma$-field, in accordance with the notation initiated in Section~\ref{subsec:outmeas}. The starting point is the notion of majorizing and minorizing collections of gauge functions that we now introduce.

\subsection{Majorizing and minorizing gauge functions}

Let $E$ be a set in $\zeroleb(U)$. On the one hand, Proposition~\ref{prp:dichogauge} ensures that for any gauge function $g$,
	\[
		\ell_g<\infty
		\qquad\Longrightarrow\qquad
		\hau^g(E\cap U)=0,
	\]
	where $\ell_g$ is defined by~(\ref{eq:df:liminfgrd}). Studying what happens for the other gauge functions, namely, those belonging to the set $\gauge^\infty$ defined by~(\ref{eq:df:gaugeinftyast}) gives rise to the following notion of majorizing gauge function.

\begin{Definition}\label{df:majo}
	Let $U$ be a nonempty open subset of $\R^d$ and let $E$ be a set in $\zeroleb(U)$. We say that a gauge function $g\in\gauge^\infty$ is a {\em majorizing} for $E$ in $U$ if
	\[
		\hau^g(E\cap U)=0.
	\]
	Such gauge functions form the {\em majorizing collection} of $E$ in $U$, denoted by $\majo(E,U)$.
\end{Definition}

It is plain from Proposition~\ref{prp:compnormalgauge} that a gauge function $g\in\gauge^\infty$ is majorizing for $E$ in $U$ if and only if its $d$-normalization $g_d$ satisfies the same property. Also, as a simple example, let us point out that
	\begin{equation}\label{eq:countablemajo}
		E\cap U\text{ countable}
		\qquad\Longrightarrow\qquad
		\majo(E,U)=\gauge^\infty,
	\end{equation}
	because a countable set has Hausdorff $g$-measure zero for any gauge function $g$.

On the other hand, Proposition~\ref{prp:lic0fullLeb} shows that a $G_\delta$-subset of $\R^d$ with Lebesgue measure zero in $U$ cannot belong to the large intersection class $\lic{\zerofunc}{U}$, and therefore cannot belong to any of the classes $\lic{g}{U}$ for which $\ell_g=0$. Similarly to the previous definition, looking at the other gauge functions, specifically, those in the set $\gauge^\ast$ defined by~(\ref{eq:df:gaugeinftyast}) results in the following notion of minorizing gauge function.

\begin{Definition}\label{df:mino}
	Let $U$ be a nonempty open subset of $\R^d$ and let $E$ be a set in $\zeroleb(U)$. We say that a gauge function $g\in\gauge^\ast$ is a {\em minorizing} for $E$ in $U$ if
	\[
		\exists F\in\lic{g}{U} \qquad F\subseteq E.
	\]
	Such gauge functions form the {\em minorizing collection} of $E$ in $U$, denoted by $\mino(E,U)$.
\end{Definition}

Similarly to what happens for majorizing gauge functions, a gauge function $g\in\gauge^\ast$ is minorizing for $E$ in $U$ if and only if $g_d$ is; this follows from Definition~\ref{df:liclocgauge}. Moreover, if $E$ is a $G_\delta$-set for which $g$ is minorizing in $U$, it follows from Proposition~\ref{prp:morelicgauge}(\ref{item:prp:morelicgauge1}) that $E$ belongs to the class $\lic{g}{U}$. Finally, we now have
	\begin{equation}\label{eq:countablemino}
		E\cap U\text{ countable}
		\qquad\Longrightarrow\qquad
		\mino(E,U)=\emptyset,
	\end{equation}
	because the existence of a minorizing gauge function requires that $E$ is dense in $U$.

The next result enlightens the monotonicity properties of $\majo(E,U)$ and $\mino(E,U)$ when regarded as functions defined on the set of pairs $(E,U)$ such that $U$ is a nonempty open subset of $\R^d$ and $E$ is in $\zeroleb(U)$. The proof is omitted, as it just relies on Proposition~\ref{prp:morelicgauge}(\ref{item:prp:morelicgauge2}) and the fact that Hausdorff measures are outer measures.

\begin{Proposition}\label{prp:monomajomino}
	The majorizing and minorizing collections satisfy the following monotonicity properties:
	\begin{enumerate}
		\item the mappings $E\mapsto\majo(E,U)$ and $U\mapsto\majo(E,U)$ are both nonincreasing;
		\item the mappings $E\mapsto\mino(E,U)$ and $U\mapsto\mino(E,U)$ are nondecreasing and nonincreasing, respectively.
	\end{enumerate}
\end{Proposition}

Let us now turn our attention to the behavior under countable unions and intersections of the two collections. The next result is a plain consequence of Theorem~\ref{thm:stablicgauge}(\ref{item:thm:stablicgauge1}) and the fact that Hausdorff measures are outer measures.

\begin{Proposition}\label{prp:cupcapmajomino}
	Let us consider a nonempty open subset $U$ of $\R^d$. Then, for any sequence $(E_n)_{n\geq 1}$ in the collection $\zeroleb(U)$,
			\[
				\majo\left(\bigcup_{n=1}^\infty E_n,U\right)=\bigcap_{n=1}^\infty\majo(E_n,U)
				\quad\text{and}\quad
				\mino\left(\bigcap_{n=1}^\infty E_n,U\right)=\bigcap_{n=1}^\infty\mino(E_n,U).
			\]
\end{Proposition}

The structure of the majorizing and minorizing collections is reminiscent of that of two intervals of the real line whose intersection is at most a singleton. We have indeed the next result; it may easily be proven with the help of Propositions~\ref{prp:compgauge0},~\ref{prp:compnormalgauge},~\ref{prp:dichogauge}(\ref{item:prp:dichogauge1}) and~\ref{prp:morelicgauge}(\ref{item:prp:morelicgauge2}), along with Theorem~\ref{thm:stablicgauge}(\ref{item:thm:stablicgauge3}).

\begin{Proposition}\label{prp:majominoint}
	Consider a nonempty open set $U\subseteq\R^d$, a set $E$ in $\zeroleb(U)$, and two gauge functions $g$ and $h$ with $d$-normalizations such that $g_d\prec h_d$. Then,
	\[
		\left\{\begin{array}{lcl}
			g\in\majo(E,U) & \qquad\Longrightarrow\qquad & h\in\majo(E,U)\setminus\mino(E,U)\\[2mm]
			h\in\mino(E,U) & \qquad\Longrightarrow\qquad & g\in\mino(E,U)\setminus\majo(E,U).
		\end{array}\right.
	\]
\end{Proposition}

The above analogy with intervals of the real line can in fact be pursued, so as to introduce natural definitions concerning sets of gauge functions. We begin by remarking that for any gauge function $g\in\gauge^\ast$, we may obtain a gauge function $\overline g\in\gauge^\ast$ satisfying $\overline g_d\prec g_d$ by considering $\overline g=g_d^{1/2}$. Likewise, for any gauge function $g\in\gauge^\infty$, we get a gauge function $\underline g\in\gauge^\infty$ with $g_d\prec\underline g_d$ by considering $\underline g(r)=r^{d/2}g_d(r)^{1/2}$. Studying whether these properties still hold for given subsets of $\gauge^\ast$ and $\gauge^\infty$ yields the notions of left-openness and right-openness, respectively.

\begin{Definition}
	Let $\frakH$ denote a subset of $\gauge^\ast$. We say that the collection $\frakH$ is:
	\begin{itemize}
		\item {\em $d$-normalized} if for any $g\in\gauge^\ast$, the gauge function $g$ belongs to the set $\frakH$ if and only if its $d$-normalization $g_d$ does;
		\item {\em left-open} if it is $d$-normalized and for any gauge function $g\in\frakH$, there exists a gauge function $\overline g\in\frakH$ such that $\overline g_d\prec g_d$\,;
		\item {\em right-open} if it is $d$-normalized, contained in $\gauge^\infty$, and for any gauge function $g\in\frakH$, there exists a gauge function $\underline g\in\frakH$ such that $g_d\prec\underline g_d$.
	\end{itemize}
\end{Definition}

The above observations ensure that the whole collection $\gauge^\ast$ is left-open, and the collection $\gauge^\infty$ is both left-open and right-open. Along with Propositions~\ref{prp:monomajomino} and~\ref{prp:majominoint}, they also straightforwardly lead to the following  link between the majorizing and minorizing collections, and their openness properties.

\begin{Corollary}\label{cor:linkopen}
	Let us consider a nonempty open subset $U$ of $\R^d$ and a set $E$ belonging to the collection $\zeroleb(U)$.
	\begin{enumerate}
		\item The collection $\majo(E,U)$ is right-open. If it is also left-open, then
			\[
				\majo(E,U)\subseteq\gauge^\infty\setminus\bigcup_{V\text{ open}\atop\emptyset\neq V\subseteq U}\mino(E,V).
			\]
		\item The collections $\mino(E,U)$ and $\mino(E,U)\cap\gauge^\infty$ are left-open. If the latter is also right-open, then
			\[
				\mino(E,U)\cap\gauge^\infty\subseteq\gauge^\infty\setminus\bigcup_{V\text{ open}\atop\emptyset\neq V\subseteq U}\majo(E,U).
			\]
	\end{enumerate}
\end{Corollary}

In particular, if either of the collections $\majo(E,U)$ and $\mino(E,U)\cap\gauge^\infty$ is simultaneously left-open and right-open, then these two collections are necessarily disjoint, so no gauge function can be majorizing and minorizing at the same time. Under the stronger assumption that {\em both} collections are left-open and right-open simultaneously, we have the next implications for size and large intersection properties.

\begin{Corollary}\label{cor:linkopenmet}
	Consider a nonempty open set $U\subseteq\R^d$ and a set $E\in\zeroleb(U)$, and assume that $\majo(E,U)$ and $\mino(E,U)\cap\gauge^\infty$ are both left-open and right-open. Then, for any gauge function $g\in\gauge^\ast$ and any nonempty open set $V\subseteq U$,
		\[
			\left\{\begin{array}{lcl}
				g\in\majo(E,U)\cup(\gauge^\ast\setminus\gauge^\infty) & \quad\Longrightarrow\quad & \hau^g(E\cap V)=0\\[2mm]
				g\in\mino(E,U)\cap\gauge^\infty & \quad\Longrightarrow\quad & \hau^g(E\cap V)=\infty
			\end{array}\right.
		\]
		and
		\[
			\left\{\begin{array}{lcl}
				g\in\majo(E,U) & \quad\Longrightarrow\quad & \forall F\in\lic{g}{V} \quad F\not\subseteq E\\[2mm]
				g\in\mino(E,U) & \quad\Longrightarrow\quad & \exists F\in\lic{g}{V} \quad F\subseteq E.
			\end{array}\right.
		\]
\end{Corollary}

\begin{proof}
	All the properties result directly from Definitions~\ref{df:majo} and~\ref{df:mino}, Proposition~\ref{prp:monomajomino} and Corollary~\ref{cor:linkopen}, except the two following ones. First, the set $E\cap V$ has Hausdorff $g$-measure zero for any $g$ in $\gauge^\ast\setminus\gauge^\infty$, because of Proposition~\ref{prp:dichogauge}(\ref{item:prp:dichogauge2}) and the assumption that $E\in\zeroleb(U)$. Second, if $g$ is in $\mino(E,U)\cap\gauge^\infty$, the right-openness property entails that this set contains a gauge function $\underline g$ with $g_d\prec\underline g_d$. By Proposition~\ref{prp:monomajomino}, the gauge function $\underline g$ is also in $\mino(E,V)$. Now, the gauge function $(g_d\underline g_d)^{1/2}$ satisfies $g_d\prec(g_d\underline g_d)^{1/2}\prec\underline g_d$, and thus cannot be majorizing for $E$ in $V$, due to Proposition~\ref{prp:majominoint}. We conclude with Proposition~\ref{prp:compgauge0} that $E\cap V$ has infinite $g$-measure.
\end{proof}

\subsection{Describability}

In light of Corollary~\ref{cor:linkopenmet}, in the ideal situation where we know that every gauge function is either majorizing or minorizing, the description of the size and large intersection properties of a set will be both precise and complete; we shall then say that the set if fully describable. A further question is to establish a criterion to determine whether a given gauge function is majorizing or minorizing; this will lead to the notions of $\frakn$-describable and $\fraks$-describable sets that are detailed afterward. These notions are naturally connected with those of eutaxic sequence and optimal regular system discussed in Sections~\ref{subsec:eutaxyapprox} and~\ref{subsec:optregsysapprox}, respectively, thereby being particularly relevant to the many applications discussed in Sections~\ref{sec:inhom}--\ref{sec:Poisson}.

\subsubsection{Fully describable sets}

To be more specific, we define the notion of fully describable set in the following manner.

\begin{Definition}
	Let $U$ be a nonempty open subset $U$ of $\R^d$ and let $E$ be a set in $\zeroleb(U)$. We say that the set $E$ is {\em fully describable in $U$} if
	\[
		\gauge^\infty\subseteq\majo(E,U)\cup\mino(E,U).
	\]
\end{Definition}

Obviously, the notion of fully describable set is only relevant to the setting of sets with large intersection. For instance, the middle-third Cantor set $\cantor$ has positive Hausdorff measure in the dimension $s=\log 2/\log3$, see the derivation of~(\ref{eq:lobndcantor}). Thus, the gauge function $r\mapsto r^s$ cannot be majorizing for $\cantor$ in $(0,1)$. Furthermore, as already observed in Section~\ref{subsec:largeintclass}, the set $\cantor$ cannot contain any set with large intersection. In particular, the previous gauge function cannot be minorizing either. Hence, the Cantor set $\cantor$ is not fully describable in $(0,1)$.

Besides, if $U$ denotes again an arbitrary nonempty open subset of $\R^d$, we already discussed a trivial example of fully describable set in $U$, namely, the Borel subsets $E$ of $\R^d$ for which the intersection $E\cap U$ is a countable set. We have indeed
	\[
		\gauge^\infty=\majo(E,U)\cup\mino(E,U),
	\]
	as an immediate consequence of~(\ref{eq:countablemajo}) and~(\ref{eq:countablemino}). Another situation where a set $E$ is fully describable in $U$ is when it admits a minorizing gauge function $g$ such that $\ell_g$ is finite. This corresponds to the next statement, whose proof is left to the reader.

\begin{Proposition}
	Let	$U$ be a nonempty open subset of $\R^d$ and let $E$ be a set in $\zeroleb(U)$. Then, the following implication holds:
		\[
		\mino(E,U)\setminus\gauge^\infty\neq\emptyset
		\qquad\Longrightarrow\qquad
		\left\{\begin{array}{l}
			\majo(E,U)=\emptyset\\[2mm]
			\mino(E,U)=\gauge^\ast.
		\end{array}\right.
		\]
\end{Proposition}

Lastly, we will often be able to prove that a set $E$ is fully describable in $U$ because the collection $\mino(E,U)$ contains some set $\frakH\subseteq\gauge^\infty$ and the collection $\majo(E,U)$ contains its complement $\frakH^\complement$. Under an openness assumption on $\frakH$ or $\frakH^\complement$, the next proposition yields the exact expression of the majorizing and minorizing collections. We omit the proof, because it is an elementary consequence of Proposition~\ref{prp:majominoint}.

\begin{Proposition}\label{prp:descincl}
	Let $U$ be a nonempty open subset of $\R^d$, let $E$ be a set in $\zeroleb(U)$, and let $\frakH$ be a subset of $\gauge^\infty$ with complement $\frakH^\complement=\gauge^\infty\setminus\frakH$. Let us assume that:
	\begin{itemize}
		\item the collections $\mino(E,U)$ and $\majo(E,U)$ contain $\frakH$ and $\frakH^\complement$, respectively;
		\item the collection $\frakH$ is right-open, or the collection $\frakH^\complement$ is left-open.
	\end{itemize}	
	Then, the following equalities hold:
	\[
		\majo(E,U)=\frakH^\complement
		\qquad\text{and}\qquad
		\mino(E,U)\cap\gauge^\infty=\frakH,
	\]
\end{Proposition}

\subsubsection{$\frakn$-describable sets}\label{subsubsec:frakndesc}

We now single out an important category of fully describable sets; they are characterized by the existence of a simple criterion to decide whether a given gauge function is majorizing or minorizing. This criterion is expressed in terms of integrability properties with respect to a given measure that belongs to the collection, denoted by $\rad01$, of all positive Borel measures $\frakn$ on the interval $(0,1]$ such that $\frakn$ has infinite total mass and
	\begin{equation}\label{eq:condrad01}
		\forall\rho\in(0,1] \qquad \Phi_\frakn(\rho)=\frakn([\rho,1])<\infty.
	\end{equation}
	The function $\Phi_\frakn$ is then clearly nonincreasing on $(0,1]$. Moreover, at any given $\rho$, it is left-continuous with a finite right-limit, namely,
	\[
		\Phi_\frakn(\rho+)=\frakn((\rho,1]).
	\]
	Extending this notation to the case where $\rho$ vanishes, we get that $\Phi_\frakn(0+)$ is infinite because $\frakn$ has infinite total mass. It is worth pointing out here that the $d$-normalization $g_d$ of an arbitrary gauge function $g$ is always Borel measurable and bounded on $(0,1]$. This enables us to introduce the notation
	\[
		\croc{\frakn}{g_d}=\int_{(0,1]} g_d(r)\,\frakn(\dd r).
	\]
	We shall in fact restrict our attention to certain measures in $\rad01$ only, namely, those belonging to the subcollection
	\begin{equation}\label{eq:df:rad01d}
		\rad01_d=\{\frakn\in\rad01\:|\:\croc{\frakn}{r\mapsto r^d}<\infty\}.
	\end{equation}

For any $\frakn$ in $\rad01$, the gauge functions $g\not\in\gauge^\ast$ clearly satisfy $\croc{\frakn}{g_d}<\infty$. If $\frakn$ is in $\rad01_d$, this property actually holds for all $g\not\in\gauge^\infty$. Indeed, the parameter $\ell_g$ is then finite, so that $g_d(r)\leq\ell_g\,r^d$ for all $r\in(0,1]$. The finiteness of $\croc{\frakn}{g_d}$ therefore remains undecided only if $g$ is in $\gauge^\infty$\,; this motivates the introduction of the set
	\[
		\gauge(\frakn)=\{g\in\gauge^\infty\:|\:\croc{\frakn}{g_d}=\infty\},
	\]
	along with its complement in $\gauge^\infty$, which is denoted by $\gauge(\frakn)^\complement$.

\begin{Definition}\label{df:frakndesc}
	Let $U$ be a nonempty open subset of $\R^d$, let $E$ be a set in $\zeroleb(U)$, and let $\frakn$ be a measure in $\rad01_d$. We say that the set $E$ is {\em $\frakn$-describable in $U$} if
	\[
		\majo(E,U)=\gauge(\frakn)^\complement
		\qquad\text{and}\qquad
		\mino(E,U)\cap\gauge^\infty=\gauge(\frakn).
	\]
\end{Definition}

It is clear from the definition that if $E$ denotes $\frakn$-describable set in $U$, then $E$ is fully describable in $U$ and the majorizing and minorizing collections are disjoint. We know that this situation occurs when either of the collections $\majo(E,U)$ and $\mino(E,U)\cap\gauge^\infty$ is simultaneously left-open and right-open. The following lemma actually implies that both collections are left-open and right-open at the same time, which will enable us to subsequently apply Corollary~\ref{cor:linkopenmet}. It also entails that $\mino(E,U)\cap\gauge^\infty$ is nonempty, meaning that $E$ contains a set with large intersection.

\begin{Lemma}\label{lem:gaugemeasopen}
	For any measure $\frakn$ in $\rad01_d$, the following properties hold:
	\begin{enumerate}
		\item\label{item:lem:gaugemeasopen1} the set $\gauge^\ast\setminus\gauge(\frakn)$ is left-open;
		\item\label{item:lem:gaugemeasopen2} the set $\gauge(\frakn)$ is right-open and nonempty.
	\end{enumerate}
\end{Lemma}

\begin{proof}
	In order to prove~(\ref{item:lem:gaugemeasopen1}), let us consider a $d$-normalized gauge function $g\in\gauge^\ast$ such that $\croc{\frakn}{g_d}<\infty$. We recall that the parameter $\eps_g$ is defined in Section~\ref{subsubsec:netmrev}. We may build a decreasing sequence $(r_n)_{n\geq 1}$ of real numbers in $(0,\eps_g)$ that converges to zero and such that for all $n\geq 2$,
	\[
		g(r_n)\leq g(r_{n-1})\,\ee^{-1/n}
		\qquad\text{and}\qquad
		\int_{0<r\leq r_{n-1}} g(r)\,\frakn(\dd r)\leq\frac{1}{(n+1)^3}.
	\]
	Then, for any $r\in(0,r_1]$, we define $\overline g(r)=g(r)\xi(r)$, where the function $\xi$ satisfies
	\[
	\xi(r)=n+\frac{\log g(r_{n-1})-\log g(r)}{\log g(r_{n-1})-\log g(r_{n})}
	\]
	for any $n\geq 2$ and any $r\in(r_n,r_{n-1}]$. It is straightforward to check that $\overline g$ may be extended to a gauge function such that $\overline g_d\prec g_d$ and $\croc{\frakn}{\overline g_d}<\infty$.
	
	To prove the right-openness in~(\ref{item:lem:gaugemeasopen2}), let us suppose that $g\in\gauge^\infty$ and $\croc{\frakn}{g_d}=\infty$. Let us define $r_1=\eps_{g}/2$, and also $\theta(r)=g(r)/r^d$ for all $r\in(0,r_1]$. For any $n\geq 2$, there exists a real number $r_n\in(0,r_{n-1})$ such that
	\[
		\theta(r_n)\geq\theta(r_{n-1})\,\ee \qquad\text{and}\qquad \int_{r_n<r\leq r_{n-1}} g(r)\,\frakn(\dd r)\geq 1.
	\]
	The sequence $(r_n)_{n\geq 1}$ is decreasing and converges to zero. Then, for any $r\in(0,r_1]$, we consider $\underline g(r)=g(r)/\xi(r)$, where the function $\xi$ is given by
	\[
		\xi(r)=n+\frac{\log\theta(r)-\log\theta(r_{n-1})}{\log\theta(r_{n})-\log\theta(r_{n-1})}
	\]
	for any $n\geq 2$ and any $r\in(r_n,r_{n-1}]$. One can easily check that $\underline g$ may be extended to a gauge function in $\gauge^\infty$ such that $g\prec\underline g$ and $\croc{\frakn}{\underline g_d}=\infty$. Finally, the nonemptyness in~(\ref{item:lem:gaugemeasopen2}) may be established by formally replacing the gauge function $g$ above by the indicator function of the interval $(0,1]$.
\end{proof}

As mentioned above, Lemma~\ref{lem:gaugemeasopen} enables us to apply Corollary~\ref{cor:linkopenmet} to the $\frakn$-describable sets. This boils down to the next statement, which gives a complete and precise description of the size and large intersection properties of those sets.

\begin{Theorem}\label{thm:frakndescslip}
	Let $U$ be a nonempty open subset of $\R^d$, let $E$ be a set in $\zeroleb(U)$, and let $\frakn$ be a measure in $\rad01_d$. Let us assume that $E$ is $\frakn$-describable in $U$. For any nonempty open set $V\subseteq U$, the following properties hold:
	\begin{enumerate}
		\item\label{item:thm:frakndescslip1} for any gauge function $g\in\gauge\setminus\gauge(\frakn)$,
			\[
				\left\{\begin{array}{l}
					\hau^g(E\cap V)=0\\[2mm]
					\forall F\in\lic{g}{V} \quad F\not\subseteq E
				\end{array}\right.
			\]
		\item\label{item:thm:frakndescslip2} for any gauge function $g\in\gauge(\frakn)$,
			\[
				\left\{\begin{array}{l}
					\hau^g(E\cap V)=\infty\\[2mm]
					\exists F\in\lic{g}{V} \quad F\subseteq E.
				\end{array}\right.
			\]
	\end{enumerate}
\end{Theorem}

\begin{proof}
	The property~(\ref{item:thm:frakndescslip2}) results directly from combining of Lemma~\ref{lem:gaugemeasopen} and Corollary~\ref{cor:linkopenmet}. This is also the case of~(\ref{item:thm:frakndescslip1}) when the gauge function $g$ is in $\gauge^\infty$. It remains us to prove~(\ref{item:thm:frakndescslip1}) when $g$ is not in $\gauge^\infty$. Given that $E\in\zeroleb(U)$, Proposition~\ref{prp:dichogauge} leads to the first part of~(\ref{item:thm:frakndescslip1}), and Proposition~\ref{prp:lic0fullLeb} implies the second part in the situation where $\ell_g$ vanishes. Finally, if $g$ is in $\gauge^\ast$, Lemma~\ref{lem:gaugemeasopen} ensures that there is a $d$-normalized gauge function $\overline{g}\prec g_d$ for which $\croc{\frakn}{\overline{g}}<\infty$. Necessarily, $\overline{g}$ is in $\gauge^\infty$, thus verifying~(\ref{item:thm:frakndescslip1}). Hence, $E\cap V$ has Hausdorff $\overline{g}$-measure zero, and we deduce from Theorem~\ref{thm:stablicgauge}(\ref{item:thm:stablicgauge3}) the second part of~(\ref{item:thm:frakndescslip1}) for the gauge function $g$.
\end{proof}

We may associate with every measure $\frakn$ in $\rad01_d$ an exponent that characterizes its integrability properties at the origin, specifically,
	\begin{equation}\label{eq:df:sfraknrad01}
		s_\frakn=\sup\{s\in(0,d]\:|\:(r\mapsto r^s)\in\gauge(\frakn)\}=\inf\{s\in(0,d]\:|\:(r\mapsto r^s)\not\in\gauge(\frakn)\}.
	\end{equation}
	Note that the rightmost set contains $d$, so that its infimum is well defined. The leftmost set may however be empty and, in that situation, we adopt the convention that its supremum is equal to zero. Restricting Theorem~\ref{thm:frakndescslip} to the gauge functions $r\mapsto r^s$, we directly obtain the following dimensional statement.

\begin{Corollary}\label{cor:frakndescslip}
Let $U$ be a nonempty open subset of $\R^d$, let $E$ be a set in $\zeroleb(U)$, and let $\frakn$ be a measure in $\rad01_d$. Let us assume that $E$ is $\frakn$-describable in $U$. Then, for any nonempty open set $V\subseteq U$,
	\[
		\Hdim (E\cap V)=s_\frakn.
	\]
	Let us assume that $s_\frakn>0$. Then, for any nonempty open set $V\subseteq U$,
	\[
		\Pdim (E\cap V)=d,
	\]
	and, if $E$ is a $G_\delta$-set, it belongs to the large intersection class $\lic{s_\frakn}{V}$.
\end{Corollary}

\begin{proof}
	Let us assume that $s_\frakn<d$. We deduce from Theorem~\ref{thm:frakndescslip}(\ref{item:thm:frakndescslip1}) that $E\cap V$ has Hausdorff $s$-dimensional measure zero, for any $s\in (s_\frakn,d]$. Hence, this set has Hausdorff dimension at most $s_\frakn$. Obviously, this bound still holds if $s_\frakn=d$.
	
	If the parameter $s_\frakn$ is positive, Theorem~\ref{thm:frakndescslip}(\ref{item:thm:frakndescslip2}) implies that for any $s\in (0,s_\frakn)$, there exists a subset $F_s$ of $E$ that belongs to the generalized class $\lic{r\mapsto r^s}{V}$. Proposition~\ref{prp:linkFalcgauge} then ensures that each set $F_s$ belongs to the original class $\lic{s}{V}$ and that its intersection with the open set $V$ has Hausdorff dimension at least $s$ and packing dimension equal to $d$. It follows that $E\cap V$ has Hausdorff dimension at least $s_\frakn$ and packing dimension equal to $d$. Furthermore, if $E$ is a $G_\delta$-set itself, we deduce from Proposition~\ref{prp:morelicgauge}(\ref{item:prp:morelicgauge1}) that the set $E$ belongs to all the classes $\lic{s}{V}$, for $s\in (0,s_\frakn)$. In view of Definition~\ref{df:licloc}, this implies that $E\in\lic{s_\frakn}{V}$.
	
	Finally, note that the lower bound on the Hausdorff dimension of $E\cap V$ still holds when $s_\frakn$ vanishes. Indeed, by Lemma~\ref{lem:gaugemeasopen}(\ref{item:lem:gaugemeasopen2}), there is a gauge function in $\gauge(\frakn)$. Applying Theorem~\ref{thm:frakndescslip}(\ref{item:thm:frakndescslip2}) with such a gauge function, we infer that $E\cap V$ is nonempty, thus having nonnegative Hausdorff dimension.
\end{proof}

\subsubsection{$\fraks$-describable sets}\label{subsubsec:fraksdesc}

This section is parallel to previous one. We consider another category of fully describable sets where we have at hand a criterion to decide whether a gauge function is majorizing or minorizing. This criterion is now expressed in terms of growth rates at the origin. To be specific, for any real parameter $\fraks\in[0,d)$, let $\gauge(\fraks)$ denote the subset of $\gauge^\infty$ defined by the condition
	\[
		g\in\gauge(\fraks)
		\quad\qquad\Longleftrightarrow\qquad\quad
		\forall s>\fraks \qquad g_d(r)\neq\smallo(r^s)\quad\text{as}\quad r\to 0.
	\]
	Note that $g_d(r)\neq\smallo(r^d)$ for any $g\in\gauge^\infty$. Hence, in the previous condition, the only relevant values of $s$ are those in $(\fraks,d)$. Moreover, the mapping $\fraks\mapsto\gauge(\fraks)$ is clearly nondecreasing. Finally, the complement in $\gauge^\infty$ of $\gauge(\fraks)$ is denoted by $\gauge(\fraks)^\complement$.

\begin{Definition}
	Let $U$ be a nonempty open subset of $\R^d$, let $E$ be a set in $\zeroleb(U)$, and let $\fraks$ be in $[0,d)$. We say that the set $E$ is {\em $\fraks$-describable in $U$} if
	\[
		\majo(E,U)=\gauge(\fraks)^\complement
		\qquad\text{and}\qquad
		\mino(E,U)\cap\gauge^\infty=\gauge(\fraks).
	\]
\end{Definition}

Similarly to $\frakn$-describable sets, it is clear that $\fraks$-describable sets are fully describable, with disjoint majorizing and minorizing collections. The connection with $\frakn$-describability is even deeper. In fact, for any $s\in[0,d)$, let us consider
	\begin{equation}\label{eq:df:frakns}
		\frakn_s(\dd r)=\frac{\ind_{(0,1]}(r)}{r^{s+1}}\,\dd r.
	\end{equation}
	It is elementary to check that each measure $\frakn_s$ belongs to the collection $\rad01_d$, and that the associated exponent given by~(\ref{eq:df:sfraknrad01}) is equal to $s$. In particular, in view of Corollary~\ref{cor:frakndescslip}, every $\frakn_s$-describable set has Hausdorff dimension equal to $s$. Moreover, the sets $\gauge(\frakn_s)$ are nondecreasing with respect to $s$, and taking monotonic intersections yield the newly introduced sets $\gauge(\fraks)$, specifically, for any $\fraks\in[0,d)$,
	\begin{equation}\label{eq:linkfraknfraksdesc}
		\gauge(\fraks)=\bigcap_{s\in(\fraks,d)}\downarrow\gauge(\frakn_s).
	\end{equation}
	The above link between $\fraks$-describability and $\frakn$-describability leads to the following analog of Lemma~\ref{lem:gaugemeasopen}.

\begin{Lemma}\label{lem:gaugerealopen}
	For any real number $\fraks\in [0,d)$, the following properties hold:
	\begin{enumerate}
		\item\label{item:lem:gaugerealopen1} the set $\gauge^\ast\setminus\gauge(\fraks)$ is left-open;
		\item\label{item:lem:gaugerealopen2} the set $\gauge(\fraks)$ is right-open and nonempty.
	\end{enumerate}
\end{Lemma}

\begin{proof}
	The left-openness of the set $\gauge^\ast\setminus\gauge(\fraks)$ is inherited from that of the sets $\gauge^\ast\setminus\gauge(\frakn)$, for $\frakn\in\rad01_d$. Indeed, if $g$ is $d$-normalized gauge function in $\gauge^\ast\setminus\gauge(\fraks)$, then~(\ref{eq:linkfraknfraksdesc}) ensures that $g\not\in\gauge(\frakn_s)$ for some $s\in(\fraks,d)$. By Lemma~\ref{lem:gaugemeasopen}(\ref{item:lem:gaugemeasopen1}), there is a $d$-normalized gauge function $\overline{g}$ in $\gauge^\ast\setminus\gauge(\frakn_s)$ such that $\overline{g}\prec g$. By~(\ref{eq:linkfraknfraksdesc}) again, $\overline{g}$ does not belong to $\gauge(\fraks)$, and we end up with~(\ref{item:lem:gaugerealopen1}).
	
	Furthermore, let us recall that the mapping $s\mapsto\gauge(\frakn_s)$ is nondecreasing. Thanks to~(\ref{eq:linkfraknfraksdesc}), we deduce that $\gauge(\fraks)$ contains $\gauge(\frakn_\fraks)$. Lemma~\ref{lem:gaugemeasopen}(\ref{item:lem:gaugemeasopen2}) shows that the latter set is nonempty, so the former is nonempty as well.
	
	Finally, the right-openness property in~(\ref{item:lem:gaugerealopen2}) follows from the fact that, if $g$ is a $d$-normalized gauge function in $\gauge(\fraks)$, letting $\underline{g}(r)=g(r)/\log(g(r)/r^d)$ yields as required a $d$-normalized gauge function in $\gauge(\fraks)$ such that $g\prec\underline{g}$.
\end{proof}

Owing to Lemma~\ref{lem:gaugerealopen}, if a set $E$ is $\fraks$-describable in $U$, then $\mino(E,U)\cap\gauge^\infty$ is nonempty, so $E$ necessarily contains a set with large intersection. Furthermore, both $\majo(E,U)$ and $\mino(E,U)\cap\gauge^\infty$ are left-open and right-open at the same time. We may thus apply Corollary~\ref{cor:linkopenmet}, and deduce the following complete and precise description of the size and large intersection properties of the set $E$.

\begin{Theorem}\label{thm:fraksdescslip}
	Let $U$ be a nonempty open subset of $\R^d$, let $E$ be a set in $\zeroleb(U)$, and let $\fraks$ be in $[0,d)$. Let us assume that $E$ is $\fraks$-describable in $U$. For any nonempty open set $V\subseteq U$, the following properties hold:
	\begin{enumerate}
		\item\label{item:thm:fraksdescslip1} for any gauge function $g\in\gauge\setminus\gauge(\fraks)$,
			\[
				\left\{\begin{array}{l}
					\hau^g(E\cap V)=0\\[2mm]
					\forall F\in\lic{g}{V} \quad F\not\subseteq E
				\end{array}\right.
			\]
		\item\label{item:thm:fraksdescslip2} for any gauge function $g\in\gauge(\fraks)$,
			\[
				\left\{\begin{array}{l}
					\hau^g(E\cap V)=\infty\\[2mm]
					\exists F\in\lic{g}{V} \quad F\subseteq E\,;
				\end{array}\right.
			\]
	\end{enumerate}
\end{Theorem}

Theorem~\ref{thm:fraksdescslip} above may be regarded as an analog of Theorem~\ref{thm:frakndescslip}, and may be established by easily adapting the proof of the latter. The proof is therefore omitted here. We just mention that one needs to use Lemma~\ref{lem:gaugerealopen} instead of Lemma~\ref{lem:gaugemeasopen} whenever necessary, and that Corollary~\ref{cor:linkopenmet} is crucial in that proof too.

For all $s\in(0,d]$ and $\fraks\in[0,d)$, one easily checks that the gauge function $r\mapsto r^s$ belongs to the set $\gauge(\fraks)$ if and only if $s\leq\fraks$. Therefore, restricting Theorem~\ref{thm:fraksdescslip} to these specific gauge functions leads to the following dimensional statement which is parallel to Corollary~\ref{cor:frakndescslip}. Again, the proof is very similar to that of the latter result; for that reason, it is left to the reader.

\begin{Corollary}\label{cor:fraksdescslip}
	Let $U$ be a nonempty open subset of $\R^d$, let $E$ be a set in $\zeroleb(U)$, and let $\fraks$ be in $[0,d)$. Let us assume that $E$ is $\fraks$-describable in $U$. Then, for any nonempty open set $V\subseteq U$,
		\[
			\Hdim (E\cap V)=\fraks
			\qquad\text{with}\qquad
			\hau^\fraks(E\cap V)=\infty.
		\]
		Let us assume that $\fraks>0$. Then, for any nonempty open set $V\subseteq U$,
		\[
			\Pdim (E\cap V)=d.
		\]
		Moreover, there exists a subset of $E$ in the large intersection class $\lic{\fraks}{U}$. In particular, if $E$ is a $G_\delta$-set itself, it belongs to the latter class.
\end{Corollary}

Lastly, especially in view of the applications in Diophantine approximation, it is important to observe that countable intersections of $\frakn_s$-describable sets can lead to $\frakn_s$-describable sets, but also to $\fraks$-describable sets.

\begin{Proposition}\label{prp:dichofraksdesc}
	Let $U$ be a nonempty open subset of $\R^d$ and, for each $n\geq 1$, let $E_n$ be a set in $\zeroleb(U)$ that is $\frakn_{s_n}$-describable in $U$ for some $s_n\in[0,d)$. Letting
	\[
		E=\bigcap_{n=1}^\infty E_n
		\qquad\text{and}\qquad
		\fraks=\inf_{n\geq 1}s_n,
	\]
	we then have the following dichotomy:
	\begin{itemize}
		\item if the infimum is attained at some $n_0$, then $E$ is $\frakn_{s_{n_0}}$-describable in $U$\,;
		\item if the infimum is not attained, then $E$ is $\fraks$-describable in $U$.
	\end{itemize}
\end{Proposition}

\begin{proof}
	To begin with, Propositions~\ref{prp:monomajomino} and~\ref{prp:cupcapmajomino} show  that the minorizing and majorizing collections of $E$ in $U$ satisfy
	\begin{equation}\label{eq:majominoEnfraknsn}
		\mino(E,U)\cap\gauge^\infty=\bigcap_{n=1}^\infty\gauge(\frakn_{s_n})
		\qquad\text{and}\qquad
		\majo(E,U)\supseteq\gauge^\infty\setminus\bigcap_{n=1}^\infty\gauge(\frakn_{s_n}).
	\end{equation}
	If the infimum is attained at a given integer $n_0$, the intersection over all $n\geq 1$ of the sets $\gauge(\frakn_{s_n})$ coincides with the sole $\gauge(\frakn_{s_{n_0}})$, so that
	\[
		\mino(E,U)\cap\gauge^\infty=\gauge(\frakn_{s_{n_0}})
		\qquad\text{and}\qquad
		\majo(E,U)\supseteq\gauge(\frakn_{s_{n_0}})^\complement.
	\]
	It follows from Proposition~\ref{prp:descincl} and Lemma~\ref{lem:gaugemeasopen}(\ref{item:lem:gaugemeasopen2}) that $E$ is $\frakn_{s_{n_0}}$-describable in $U$.
	
	The proof is very similar when the infimum is not attained. Indeed, using the monotonicity of $s\mapsto\gauge(\frakn_s)$ and combining~(\ref{eq:linkfraknfraksdesc}) with~(\ref{eq:majominoEnfraknsn}), we now get
	\[
		\mino(E,U)\cap\gauge^\infty=\gauge(\fraks)
		\qquad\text{and}\qquad
		\majo(E,U)\supseteq\gauge(\fraks)^\complement.
	\]
	By Proposition~\ref{prp:descincl} and Lemma~\ref{lem:gaugerealopen}(\ref{item:lem:gaugerealopen2}), the set $E$ is thus $\fraks$-describable in $U$.
\end{proof}	

Slightly modifying this approach leads to another situation where $\fraks$-describable sets arise naturally. Given a real number $\fraks\in[0,d)$ and a nonempty open set $U$, we consider a sequence $(E_s)_{s\in(\fraks,d)}$ of sets in $\zeroleb(U)$, and we assume that the mapping $s\mapsto E_s$ is increasing and that each set $E_s$ is $\frakn_s$-describable in $U$. We then choose in the interval $(\fraks,d)$ an arbitrary decreasing sequence $(s_n)_{n\geq 1}$ that converges to $\fraks$. The monotonicity of the sets $E_s$ with respect to $s$ implies that their intersection is equal to that of the sets $E_{s_n}$. Moreover, the latter sets fall into the above setting because the infimum of the real numbers $s_n$ is not attained. Hence, the intersection over all $s\in(\fraks,d)$ of the sets $E_s$ is $\fraks$-describable in $U$.


\section{Eutaxic sequences}\label{sec:eutaxy}

The notion of eutaxic sequence was introduced by Lesca~\cite{Lesca:1968rm} and later studied by Reversat~\cite{Reversat:1976yq}. It provides a nice setting to the study of Diophantine approximation properties, and we shall indeed use it in Sections~\ref{sec:fracpart},~\ref{sec:Dvoretzky} and~\ref{sec:Poisson} to analyze the approximation by fractional parts of sequences and by random sequences of points. As shown in Section~\ref{subsec:eutaxyapprox}, the limsup sets issued from eutaxic sequences fall into the framework of describable sets; we shall thus be able to completely describe size and large intersection properties for the sets arising in all these examples.

\subsection{Sequencewise eutaxy}\label{subsec:eutaxyseq}

With the notion of eutaxy, the emphasis is put on the sequence $(x_n)_{n\geq 1}$ of approximating points in $\R^d$, and one is ultimately interested in its uniform approximation behavior with respect to {\em all} possible sequences $(r_n)_{n\geq 1}$ of approximation radii. Let us assume that the series $\sum_n r_n^d$ converges. It is clear that the set of all points $x\in\R^d$ for which
	\begin{equation}\label{eq:ximnxnrn}
		\exists\text{~i.m.~}n\geq 1
		\qquad |x-x_n|<r_n
	\end{equation}
	has Lebesgue measure zero; this may indeed be deduced from applying Lemma~\ref{lem:upbndlimsup} with the gauge function $r\mapsto r^d$, which essentially yields the Lebesgue measure, in view of Proposition~\ref{prp:comphauleb}. In that situation, we may rearrange the points in such a way that the sequence $(r_n)_{n\geq 1}$ is nonincreasing and converges to zero. Now, eutaxy comes into play when one assumes that the series $\sum_n r_n^d$ is divergent, or equivalently that $(r_n)_{n\geq 1}$ belongs to the collection $\Rho_d$ of real sequences defined by
	\[
		(r_n)_{n\geq 1}\in\Rho_d
		\qquad\Longleftrightarrow\qquad
		\left\{\begin{array}{l}
		\forall n\geq 1 \quad r_{n+1}\leq r_n\\[2mm]
		\lim\limits_{n\to\infty}r_n=0\\[2mm]
		\sum\limits_{n=1}^\infty r_n^d=\infty.
		\end{array}
		\right.
	\]

The simplest notion of eutaxy is obtained when specifying a sequence $(r_n)_{n\geq 1}$ in $\Rho_d$ and deciding on whether or not~(\ref{eq:ximnxnrn}) is satisfied by Lebesgue-almost every point of some open set under consideration.

\begin{Definition}\label{df:eutaxicseqwise}
Let $U$ be a nonempty open subset of $\R^d$, and let $(r_n)_{n\geq 1}$ be a sequence in $\Rho_d$. A sequence $(x_n)_{n\geq 1}$ of points in $\R^d$ is called {\em eutaxic in $U$ with respect to $(r_n)_{n\geq 1}$} if the following condition holds:
	\[
		\text{for~}\leb^d\text{-a.e.~}x\in U
		\quad \exists\text{~i.m.~}n\geq 1
		\qquad |x-x_n|<r_n.
	\]
\end{Definition}

The link with approximation problems will be discussed in Section~\ref{subsec:eutaxyapprox}. However, let us point out now that for any sequence $(r_n)_{n\geq 1}$ in $\Rho_d$ and any sequence $(x_n)_{n\geq 1}$ of points in $\R^d$, the family $(x_n,r_n)_{n\geq 1}$ is necessarily an approximation system. Moreover, if $U$ denotes a nonempty open subset of $\R^d$, this family is a homogeneous ubiquitous system in $U$ if and only if $(x_n)_{n\geq 1}$ is eutaxic in $U$ with respect to $(r_n)_{n\geq 1}$. This readily follows from the respective definitions of the various involved notions, namely, Definitions~\ref{df:approxsys},~\ref{df:homubiqsys} and~\ref{df:eutaxicseqwise}. In light of Proposition~\ref{prp:homubsyscst}, we deduce that $(x_n)_{n\geq 1}$ is eutaxic with respect to $(r_n)_{n\geq 1}$ if and only if it is eutaxic with respect to $(c\,r_n)_{n\geq 1}$, for any fixed real number $c>0$. In particular, the fact that a sequence is eutaxic does not depend on the norm the space $\R^d$ is endowed with.

\subsection{Uniform eutaxy}\label{subsec:eutaxyunif}

Rather than the sequencewise, the notion of uniform eutaxy is the one studied by Lesca~\cite{Lesca:1968rm} and Reversat~\cite{Reversat:1976yq}. It is obtained when sequencewise eutaxy holds regardless of the choice of the sequence $(r_n)_{n\geq 1}$ in $\Rho_d$. In view of the remark following Definition~\ref{df:eutaxicseqwise}, it is clear that this notion does not depend on the choice of the norm, either.

\begin{Definition}\label{df:eutaxicunif}
	Let $U$ be a nonempty open subset of $\R^d$. A sequence $(x_n)_{n\geq 1}$ of points in $\R^d$ is called {\em uniformly eutaxic in $U$} if the following condition holds:
	\[
		\forall (r_n)_{n\geq 1}\in\Rho_d
		\quad\text{for~}\leb^d\text{-a.e.~}x\in U
		\quad \exists\text{~i.m.~}n\geq 1
		\qquad |x-x_n|<r_n.
	\]
\end{Definition}

We shall establish a sufficient and a necessary condition for a sequence of points to be uniformly eutaxic. They are expressed in terms of the dyadic cubes introduced in Section~\ref{subsubsec:normgaugenetm}. When a cube $\lambda$ is of the form $2^{-j}(k+[0,1)^d)$, the integer $j$ is its {\em generation} and is denoted by $\gene{\lambda}$. For any $x\in\R^d$ and any $j\in\Z$, there is a unique dyadic cube with generation $j$ that contains $x$\,; this cube is denoted by $\lambda_j(x)$. Let us now fix a sequence $(x_n)_{n\geq 1}$ of points in $\R^d$. For any nonempty dyadic cube $\lambda\in\Lambda$ and any integer $j\geq 0$, let us define a collection $\Mu((x_n)_{n\geq 1};\lambda,j)$ of dyadic cubes by the following condition:
	\[
		\lambda'\in\Mu((x_n)_{n\geq 1};\lambda,j)
		\qquad\Longleftrightarrow\qquad
		\left\{\begin{array}{l}
			\lambda'\subseteq\lambda\\[1mm]
			\gene{\lambda'}=\gene{\lambda}+j\\[1mm]
			x_n\in\lambda'\text{ for some }n\leq 2^{d\gene{\lambda'}}.
		\end{array}\right.
	\]
It will be clear from the context what sequence $(x_n)_{n\geq 1}$ is considered, and there should be no confusion if we write $\Mu(\lambda,j)$ as a shorthand for $\Mu((x_n)_{n\geq 1};\lambda,j)$.

\subsubsection{A sufficient condition for uniform eutaxy}\label{subsubsec:eutaxyunifCS}

It is obvious that the cardinality of the set $\Mu(\lambda,j)$ is bounded above by $2^{dj}$. When it is bounded below by a fraction of $2^{dj}$, the sequence $(x_n)_{n\geq 1}$ is uniformly eutaxic, as shown by the following criterion.

\begin{Theorem}\label{thm:CSeutaxy}
Let $U$ be a nonempty open subset of $\R^d$ and let $(x_n)_{n\geq 1}$ be a sequence of points in $\R^d$. Let us assume that
\begin{equation}\label{eq:liminfpos}
\text{for~}\leb^d\text{-a.e.~}x\in U
\qquad
\liminf_{j_0,j\to\infty} 2^{-dj}\#\Mu((x_n)_{n\geq 1};\lambda_{j_0}(x),j)>0.
\end{equation}
Then, the sequence $(x_n)_{n\geq 1}$ is uniformly eutaxic in $U$.
\end{Theorem}

The proof of Theorem~\ref{thm:CSeutaxy} relies on the next useful measure theoretic result excerpted from Sprind\v{z}uk's book, see~\cite[Lemma~5]{Sprindzuk:1979eu}.

\begin{Lemma}\label{lem:2ndmoment}
	Let $\mu$ be an outer measure on $\R^d$ such that $\mu(\R^d)$ is finite, and let $(E_n)_{n\geq 1}$ be a sequence of $\mu$-measurable sets such that $\sum_n\mu(E_n)$ diverges. Then,
	\[
		\mu\left(\limsup_{n\to\infty}E_n\right)
		\geq\limsup_{N\to\infty}
		\frac{\left(\sum\limits_{n=1}^N\mu(E_n)\right)^{2}}
		{\sum\limits_{m=1}^N\sum\limits_{n=1}^N\mu(E_m\cap E_n)}.
	\]
\end{Lemma}

\begin{proof}[Proof of Theorem~\ref{thm:CSeutaxy}]
	We work with the supremum norm, which does not alter the notion of eutaxy. Let us consider a nonempty open subset $U$ of $\R^d$ and a sequence $(x_n)_{n\geq 1}$ of points in $\R^d$ such that~(\ref{eq:liminfpos}) holds for Lebesgue-almost every $x\in U$. Our goal is to establish that for any sequence $(r_n)_{n\geq 1}$ in $\Rho_d$, the set
		\[
			F=\left\{x\in\R^d\bigm||x-x_n|_\infty<r_n\quad\text{for i.m.~}n\geq 1\right\}
		\]
		has full Lebesgue measure in $U$. To proceed, let $U_\ast$ denote the set of all points $x$ in $U$ such that~(\ref{eq:liminfpos}) holds and none of the coordinates of $x$ is a dyadic number. Then, $U_\ast$ has full Lebesgue measure in $U$. Furthermore, for any $x\in U_\ast$, there exist a real number $\alpha(x)>0$ and an integer $\underline{j}(x)\geq 0$ such that
		\[
			\forall j_0,j\geq\underline{j}(x) \qquad \#\Mu(\lambda_{j_0}(x),j)\geq\alpha(x)\, 2^{dj}.
		\]
		The proof now reduces to showing that there is a real number $\kappa>0$ such that
		\begin{equation}\label{eq:minlebF1}
			\forall j_0\geq\underline{j}(x) \qquad \leb^d(F\cap\lambda_{j_0}(x))\geq\kappa\,\alpha(x)^2\leb^d(\lambda_{j_0}(x)).
		\end{equation}
		Indeed,~(\ref{eq:minlebF1}) implies that the density of the set $F$ at the point $x$ is positive. Therefore, if this holds for any $x$ in $U_\ast$, then the Lebesgue density theorem shows that Lebesgue-almost every point of $U_\ast$ belongs to $F$, see~\cite[Corollary~2.14]{Mattila:1995fk}. As a result, $F$ has full Lebesgue measure in $U$.
	
	It now remains to show that any point $x$ in $U_\ast$ satisfies~(\ref{eq:minlebF1}). For any fixed integer $j_0\geq\underline{j}(x)$, we begin by observing that for any integer $j\geq\underline{j}(x)$, there exists a set $S_j(x,j_0)\subseteq\{1,\ldots,2^{d(j_0+j)}\}$ with:
		\begin{itemize}
			\item $\# S_j(x,j_0)\geq\alpha(x)\,2^{d(j-1)}$\,;
			\item $x_n\in\lambda_{j_0}(x)$ for any $n\in S_j(x,j_0)$\,;
			\item $|x_n-x_{n'}|_\infty\geq 2^{-(j_0+j)}$ for any distinct $n,n'\in S_j(x,j_0)$.
		\end{itemize}
		Indeed, for each $\beta\in\{0,1\}^d$, let us consider the cubes in $\Mu(\lambda_{j_0}(x),j)$ of the form $2^{-(j_0+j)}(k+[0,1)^d)$, where the coordinates of $k$ are equal to those of $\beta$ modulo two. For a suitable $\beta$, there are at least $2^{-d}\,\#\Mu(\lambda_{j_0}(x),j)$ such cubes. The result then follows from the observation that these cubes are at a distance at least $2^{-(j_0+j)}$ of each other and that each cube contains at least a point $x_n$ with $n\leq 2^{d(j_0+j)}$.
	
	Then, let us define $\widetilde r_n=\min\{r_n,1/(2n^{1/d})\}$ for each $n\geq 1$. We thereby obtain another sequence $(\widetilde r_n)_{n\geq 1}$ in $\Rho_d$. Indeed, otherwise, the sequence $(\widetilde r_n^d)_{n\geq 1}$ would be nonincreasing and have a finite sum, so that $n\widetilde r_n^d$ would tend to zero as $n$ goes to infinity. Thus, $\widetilde r_n$ would be equal to $r_n$ for $n$ large enough and the series $\sum_n r_n^d$ would converge, contradicting the assumption that $(r_n)_{n\geq 1}$ belongs to $\Rho_d$. Now, for any integer $j\geq\underline{j}(x)$, let us consider the set
		\[
			V_j(x,j_0)=\bigcup_{n\in S_j(x,j_0)} \opball_\infty(x_n,\rho_{j_0+j}),
		\]
		where $\rho_j$ is a shorthand for $\widetilde r_{2^{dj}}$. Since the sequence $(\widetilde r_n)_{n\geq 1}$ is nonincreasing and converges to zero, all the points in the limsup of these sets, except maybe those forming the sequence $(x_n)_{n\geq 1}$, belong to both the closure of $\lambda_{j_0}(x)$ and the set $\widetilde F$ obtained by replacing $r_n$ by $\widetilde r_n$ in the definition of $F$. Therefore,
		\[
			\leb^d\left(\limsup_{j\to\infty} V_j(x,j_0)\right)\leq\leb^d\left(\widetilde F\cap\closure{\lambda_{j_0}(x)}\right)\leq\leb^d(F\cap\lambda_{j_0}(x)).
		\]
		Hence, to obtain~(\ref{eq:minlebF1}), it suffices to provide an appropriate lower bound on the Lebesgue measure of the limsup of the sets $V_j(x,j_0)$. This may be done with the help of Lemma~\ref{lem:2ndmoment}. In fact, the sets $V_j(x,j_0)$ are all contained in the closure of the cube $\lambda_{j_0}(x)$, so that we may apply this lemma with the restriction of the Lebesgue measure to this closed cube. The resulting lower bound yields
		\begin{equation}\label{eq:minVjxj0}
			\leb^d(F\cap\lambda_{j_0}(x))
			\geq\limsup_{J\to\infty}
			\frac{\left(\sum\limits_{j=\underline{j}(x)}^{J}\leb^d(V_j(x,j_0))\right)^{2}}
			{\sum\limits_{j=\underline{j}(x)}^{J}\sum\limits_{j'=\underline{j}(x)}^{J}\leb^d(V_j(x,j_0)\cap V_{j'}(x,j_0))}.
		\end{equation}
	
	However, to apply Lemma~\ref{lem:2ndmoment}, we need to check the divergence condition
		\begin{equation}\label{eq:sumVjxj0div}
			\sum_{j=\underline{j}(x)}^\infty\leb^d(V_j(x,j_0))=\infty.
		\end{equation}
		To this end, we observe that for any $j\geq\underline{j}(x)$ and any distinct $n$ and $n'$ in $S_j(x,j_0)$, the two open balls with common radius $\rho_{j_0+j}$ and center $x_n$ and $x_{n'}$, respectively, are disjoint. Otherwise, any point $y$ in their intersection would satisfy
		\[
			|x_n-x_{n'}|_\infty\leq|y-x_n|_\infty+|y-x_{n'}|_\infty<2\rho_{j_0+j}\leq 2^{-(j_0+j)},
		\]
		which would contradict the third property of the set $S_j(x,j_0)$ given above. As a result, the balls forming the set $V_j(x,j_0)$ are disjoint, so that
		\begin{equation}\label{eq:minlebVjxj0}
			\leb^d\left(V_j(x,j_0)\right)=(2\rho_{j_0+j})^d\,\# S_j(x,j_0)\geq\alpha(x)\,2^{dj}\,\rho_{j_0+j}^d.
		\end{equation}
		In order to derive~(\ref{eq:sumVjxj0div}), we finally use the fact that the sequence $(\widetilde r_n)_{n\geq 1}$ is nonincreasing, as this enables us to write that
		\begin{equation}\label{eq:sum2drhogr1}
			2^{d j_0}(2^d-1)\sum_{j=\underline{j}(x)}^\infty 2^{dj}\,\rho_{j_0+j}^d\geq\sum_{j=j_0+\underline{j}(x)}^\infty\sum_{n=2^{dj}}^{2^{d(j+1)}-1}\widetilde r_n^d=\infty.
		\end{equation}
	
	To obtain~(\ref{eq:minlebF1}), and thus complete the proof, it suffices to combine the lower bound~(\ref{eq:minVjxj0}) with the following inequality that holds for any integer $J$ sufficiently large and that we now establish:
		\begin{equation}\label{eq:sumlebcapVj}
			\sum_{j=\underline{j}(x)}^{J}\sum_{j'=\underline{j}(x)}^{J}\leb^d\left(V_j(x,j_0)\cap V_{j'}(x,j_0)\right)
			\leq\frac{2^{d(j_0+4)}}{\alpha(x)^2}\left(\sum_{j=\underline{j}(x)}^{J}\leb^d\left(V_j(x,j_0)\right)\right)^{2}.
		\end{equation}
		Let us consider two integers $j$ and $j'$ such that $\underline{j}(x)\leq j<j'$. With a view to giving an upper bound on the Lebesgue measure of the intersection of the sets $V_j(x,j_0)$ and $V_{j'}(x,j_0)$, let us observe that for any integer $n\in S_j(x,j_0)$,
		\[
			\opball_\infty(x_n,\rho_{j_0+j})\cap V_{j'}(x,j_0)=\bigcup_{n'\in S_{j'}(x,j_0)} \left( \opball_\infty(x_n,\rho_{j_0+j})\cap \opball_\infty(x_{n'},\rho_{j_0+j'}) \right).
		\]
		The points $x_{n'}$, with $n'\in S_{j'}(x,j_0)$ such that this last intersection is nonempty, all lie in the open ball with center $x_n$ and radius $2\rho_{j_0+j}$. Moreover, there are at most $(2^{j_0+j'+2}\rho_{j_0+j}+2)^d$ cubes with generation $j_0+j'$ that intersect this ball and each of them contains at most one of the points $x_{n'}$. Thus,
		\begin{align*}
			\leb^d(\opball_\infty(x_n,\rho_{j_0+j})\cap V_{j'}(x,j_0))
			&\leq (2^{j_0+j'+2}\rho_{j_0+j}+2)^d (2\rho_{j_0+j'})^d\\
			&\leq 2^{3d-1}\rho_{j_0+j'}^d (1+2^{d(j_0+j'+1)}\rho_{j_0+j}^d).
		\end{align*}
		Along with the fact that there are at most $2^{dj}$ integers in $S_j(x,j_0)$, this yields
		\[
			\leb^d\left(V_j(x,j_0)\cap V_{j'}(x,j_0)\right)\leq 2^{d(j+3)-1}\rho_{j_0+j'}^d  (1+2^{d(j_0+j'+1)}\rho_{j_0+j}^d).
		\]
		As a consequence, for any integer $J\geq\underline{j}(x)$, the left-hand side of~(\ref{eq:sumlebcapVj}) is at most
		\[
			2^d\sum_{j=\underline{j}(x)}^{J} 2^{dj}\rho_{j_0+j}^d + 2^{3d} \sum_{j,j'} 2^{dj}\rho_{j_0+j'}^d + 2^{4d} \sum_{j,j'} 2^{d(j_0+j+j')}\rho_{j_0+j}^d\rho_{j_0+j'}^d,
		\]
		where the second and third sums are both over the integers $j$ and $j'$ that satisfy $\underline{j}(x)\leq j<j'\leq J$. Note that the second sum is equal to
		\[
			\sum_{j'=\underline{j}(x)+1}^{J} 2^{dj'}\rho_{j_0+j'}^d \sum_{j=\underline{j}(x)}^{j'-1} 2^{d(j-j')}
			\leq\frac{1}{2^d-1}\sum_{j'=\underline{j}(x)+1}^{J} 2^{dj'}\rho_{j_0+j'}^d,
		\]
		and the third sum is obviously smaller than half the sum bearing on all the integers $j$ and $j'$ between $\underline{j}(x)$ and $J$. Thus, the left-hand side of~(\ref{eq:sumlebcapVj}) is at most
		\[
			\left(2^d+\frac{2^{3d}}{2^d-1}\right)2^{-dj_0}\sum_{j=\underline{j}(x)}^{J} 2^{d(j_0+j)}\rho_{j_0+j}^d+2^{4d-1} 2^{-d j_0}\left(\sum_{j=\underline{j}(x)}^{J} 2^{d(j_0+j)}\rho_{j_0+j}^d\right)^{2}.
		\]
		In view of~(\ref{eq:sum2drhogr1}), the first sum tends to infinity as $J\to\infty$, thereby being larger than one, and thus smaller than its square, for $J$ large enough. The left-hand side of~(\ref{eq:sumlebcapVj}) is therefore bounded above by
		\[
			2^{-d(j_0-4)} \left(\sum_{j=\underline{j}(x)}^{J} 2^{d(j_0+j)}\rho_{j_0+j}^d \right)^{2},
		\]
		for any integer $J$ sufficiently large, and this bound leads to the right-hand side of~(\ref{eq:sumlebcapVj}) with the help of~(\ref{eq:minlebVjxj0}).
\end{proof}

\subsubsection{A necessary condition for uniform eutaxy}

It is not known whether Theorem~\ref{thm:CSeutaxy} also yields a necessary condition for uniform eutaxy. However, note that the sufficient condition~(\ref{eq:liminfpos}) clearly holds if
	\begin{equation}\label{eq:infliminfpos}
		\inf_{\lambda\in\Lambda\setminus\{\emptyset\}\atop\lambda\subseteq U}\liminf_{j\to\infty}2^{-dj}\#\Mu((x_n)_{n\geq 1};\lambda,j)>0.
	\end{equation}
	Moreover, it is plain that this stronger assumption fails when the liminf vanishes for some nonempty dyadic cube $\lambda$. The next result shows that, in this situation, the sequence under consideration cannot be uniformly eutaxic.

\begin{Theorem}\label{thm:CNeutaxy}
	Let $U$ be a nonempty open subset of $\R^d$ and let $(x_n)_{n\geq 1}$ be a sequence of points in $\R^d$. Let us assume that
		\[
			\exists\lambda\in\Lambda\setminus\{\emptyset\}
			\qquad
			\left\{\begin{array}{l}
				\lambda\subseteq U\\[1mm]
				\liminf\limits_{j\to\infty} 2^{-dj}\#\Mu((x_n)_{n\geq 1};\lambda,j)=0.
			\end{array}\right.
		\]
		Then, the sequence $(x_n)_{n\geq 1}$ is not uniformly eutaxic in $U$.
\end{Theorem}

\begin{proof}
	We still work with the supremum norm. Let us consider an integer $j\geq 0$ and, on the one hand, let us define the set
		\begin{equation}\label{eq:df:UjCNeutaxy}
			U_j=\bigcup_{n\leq 2^{d(\gene{\lambda}+j)}\atop x_n\in\lambda}\opball_\infty(x_n,2^{-(\gene{\lambda}+j)}).
		\end{equation}
		If $\lambda'$ is a nonempty dyadic subcube of $\lambda$, let $\widetilde\lambda'$ stand for the open cube concentric with $\lambda'$ with triple sidelength. If moreover $\lambda'$ has generation $\gene{\lambda}+j$ and contains some point $x_n$, then $\widetilde\lambda'$ contains the ball in~(\ref{eq:df:UjCNeutaxy}) that is centered at this $x_n$. Hence,
		\[
			U_j=\bigcup_{\lambda'\subseteq\lambda\atop\gene{\lambda'}=\gene{\lambda}+j}
			\bigcup_{n\leq 2^{d\gene{\lambda'}}\atop x_n\in\lambda'}\opball_\infty(x_n,2^{-\gene{\lambda'}})
			\subseteq\bigcup_{\lambda'\in\Mu(\lambda,j)}\widetilde\lambda',
		\]
		from which it directly follows that
		\[
			\leb^d(U_j)\leq 3^d 2^{-d(\gene{\lambda}+j)}\#\Mu(\lambda,j).
		\]
		On the other hand, let us consider the set $U'_j$ obtained by replacing in~(\ref{eq:df:UjCNeutaxy}) the condition $x_n\in\lambda$ by the conjunction of the fact that $x_n\not\in\lambda$ and that the open ball with center $x_n$ and radius $2^{-(\gene{\lambda}+j)}$ meets the cube $\lambda$. In that case, the ball actually meets the boundary of the cube $\lambda$. This means that each point of $U'_j$ is within distance $2^{1-(\gene{\lambda}+j)}$ from this boundary, and thus
		\begin{align*}
			\leb^d(U'_j)
			&\leq (2^{-\gene{\lambda}}+2^{2-(\gene{\lambda}+j)})^d-(2^{-\gene{\lambda}}-2^{2-(\gene{\lambda}+j)})^d\\
			&\leq 2^{3-d\gene{\lambda}-j}\sum_{\ell=0}^{d-1}(1+2^{2-j})^{d-1-\ell}(1-2^{2-j})^\ell
			\leq 5^d 2^{3-d\gene{\lambda}-j},
		\end{align*}
		with the proviso that $j\geq 2$. As a consequence, summing the two above upper bounds and letting $j$ go to infinity, we deduce that
		\[
			\liminf_{j\to\infty}\leb^d\left(\lambda\cap\bigcup_{n=1}^{2^{d(\gene{\lambda}+j)}}\opball_\infty(x_n,2^{-(\gene{\lambda}+j)})\right)
			\leq
			3^d 2^{-d\gene{\lambda}}\liminf_{j\to\infty} 2^{-dj}\#\Mu(\lambda,j),
		\]
		because the set in the left-hand side is contained in the union of $U_j$ and $U'_j$.
	
	We now make use of the assumption bearing on the cube $\lambda$, namely, that the lower limit in the right-hand side vanishes. Thus, we may find an increasing sequence $(j_m)_{m\geq 1}$ of nonnegative integers such that $j_1=0$ and for all $m\geq 1$,
		\[
			\leb^d\left(\lambda\cap\bigcup_{n=2^{d(\gene{\lambda}+j_m)}+1}^{2^{d(\gene{\lambda}+j_{m+1})}}\opball_\infty(x_n,2^{-(\gene{\lambda}+j_{m+1})})\right)
			\leq 2^{-m}.
		\]
		For simplicity, we define $n_m=2^{d(\gene{\lambda}+j_m)}$ for all $m\geq 1$, and also $n_0=0$. We then consider the unique sequence $(r_n)_{n\geq 1}$ such that
		\[
			\forall m\geq 0 \quad
			\forall n\in\{n_m+1,\ldots,n_{m+1}\} \qquad
			r_n=n_{m+1}^{-1/d}.
		\]
		Clearly, this sequence is nonincreasing and converges to zero. Moreover, for any integer $m\geq 0$,
		\[
			\sum_{n=n_m+1}^{n_{m+1}} r_n^d
			=1-\frac{n_m}{n_{m+1}}
			\geq 1-2^{-d},
		\]
		so that the series $\sum_n r_n^d$ is divergent. We may therefore conclude that the sequence $(r_n)_{n\geq 1}$ belongs to the collection $\Rho_d$.
	
	On top of that, for any integer $\underline{m}\geq 1$, we have
		\[
			\leb^d\left(\lambda\cap\bigcup_{n=n_{\underline{m}}+1}^\infty\opball_\infty(x_n,r_n)\right)
			\leq\sum_{m=\underline{m}}^\infty\leb^d\left(\lambda\cap\bigcup_{n=n_m+1}^{n_{m+1}}\opball_\infty(x_n,n_{m+1}^{-1/d})\right)
			.
		\]
		By definition of the integers $n_m$, the summand in the right-hand side is bounded above by $2^{-m}$, so that the whole sum is bounded by $2^{-\underline{m}+1}$. The left-hand side thus converges to zero when $\underline{m}$ tends to infinity. We deduce that
		\[
			\leb^d\left(\lambda\cap\limsup_{n\to\infty}\opball_\infty(x_n,r_n)\right)
			\leq\inf_{m\geq 1}\leb^d\left(\lambda\cap\bigcup_{n=m}^\infty\opball_\infty(x_n,r_n)\right)=0,
		\]
		which implies that the sequence $(x_n)_{n\geq 1}$ cannot be uniformly eutaxic in $U$.
\end{proof}

\subsection{Approximation by eutaxic sequences}\label{subsec:eutaxyapprox}

The notion of eutaxic sequence is naturally connected with those of approximation system and homogeneous ubiquitous system introduced in Section~\ref{sec:firstubiq}. In fact, as mentioned in Section~\ref{subsec:eutaxyseq}, for any sequence $(r_n)_{n\geq 1}$ in $\Rho_d$ and any sequence $(x_n)_{n\geq 1}$ of points in $\R^d$, the family $(x_n,r_n)_{n\geq 1}$ is necessarily an approximation system. This prompts us to consider the problem of the approximation within distances $r_n$ by the points $x_n$. Accordingly, the sets $F_t$ defined by~(\ref{eq:df:Ft}) in the general setting are now given by
	\begin{equation}\label{eq:df:Ftn}
		F_t=\left\{x\in\R^d\bigm||x-x_n|<r_n^t\quad\text{for i.m.~}n\geq 1\right\},
	\end{equation}
	and their size and large intersection properties may be studied by specializing the results of Sections~\ref{sec:firstubiq} and~\ref{sec:largeint}. This yields the next statement.

\begin{Theorem}\label{thm:approxeutaxic}
	Let $(x_n)_{n\geq 1}$ be a sequence of points in $\R^d$ that is eutaxic in some nonempty open subset $U$ of $\R^d$, with respect to some sequence $(r_n)_{n\geq 1}$ in $\Rho_d$. We assume that the series $\sum_n r_n^s$ is convergent for all $s>d$. Then, for any $t>1$,
	\[
		\Hdim(F_t\cap U)=\frac{d}{t}
		\qquad\text{and}\qquad
		F_t\in\lic{d/t}{U}.
	\]
\end{Theorem}

\begin{proof}
The convergence assumption on the series $\sum_n r_n^s$ implies that the parameter $s_U$ defined by~(\ref{eq:df:sU}) is bounded above by $d$ regardless of the choice of the open set $U$. Moreover, since $(x_n)_{n\geq 1}$ is eutaxic in $U$ with respect to $(r_n)_{n\geq 1}$, the family $(x_n,r_n)_{n\geq 1}$ is a homogeneous ubiquitous system in $U$. Therefore, we may apply Corollary~\ref{cor:eqhomubsys}, and deduce that the set $F_t\cap U$ has Hausdorff dimension equal to $d/t$ for any real number $t>1$. For the same reason, due to Theorem~\ref{thm:lihomubsys}, the set $F_t$ belongs to the large intersection class $\lic{d/t}{U}$.
\end{proof}

When the underlying sequence is {\em uniformly} eutaxic, we may dramatically improve on the above approach with the help of the large intersection transference principle presented in Section~\ref{subsec:largeintprinc}. This enables us to show that the limsup sets that are naturally associated with such sequences fall into the category of $\frakn$-describable sets introduced in Section~\ref{subsubsec:frakndesc}. Let $(x_n)_{n\geq 1}$ be a sequence of points in $\R^d$ and let $\rmr=(r_n)_{n\geq 1}$ be a nonincreasing sequence of positive real numbers. Instead of using the notation~(\ref{eq:df:Ftn}), we rather opt for~(\ref{eq:df:Fxiri})\,; specifically, we consider the set
	\begin{equation}\label{eq:df:Fxnrn}
		\frakF((x_n,r_n)_{n\geq 1})
		=\left\{x\in\R^d\bigm||x-x_n|<r_n\quad\text{for i.m.~} n\geq 1\right\}.
	\end{equation}
	Our main result is the following complete description of the size and large intersection properties of this set. Below, $\frakn_\rmr$ is the Borel measure on $(0,1]$ characterized by the fact that for any nonnegative measurable function $f$ supported in $(0,1]$,
	\begin{equation}\label{eq:df:fraknrmr}
		\int_{(0,1]} f(r)\,\frakn_\rmr(\dd r)=\sum_{n=1}^\infty f(r_n).
	\end{equation}
	Equivalently, $\frakn_\rmr$ is the sum of all Dirac point masses located at some $r_n\in(0,1]$.

\begin{Theorem}\label{thm:desceutaxy}
	Let $U$ be a nonempty open subset of $\R^d$, let $(x_n)_{n\geq 1}$ be a sequence in $\R^d$ that is uniformly eutaxic in $U$, and let $\rmr=(r_n)_{n\geq 1}$ be a nonincreasing sequence of positive real numbers. Then, the following properties hold:
	\begin{itemize}
		\item if $\sum_n r_n^d$ diverges, then $\frakF((x_n,r_n)_{n\geq 1})$ has full Lebesgue measure in $U$\,;
		\item if $\sum_n r_n^d$ converges, then $\frakF((x_n,r_n)_{n\geq 1})$ is $\frakn_\rmr$-describable in $U$.
	\end{itemize}
\end{Theorem}

\begin{proof}
	In the divergence case, we cannot use the uniform eutaxy directly because the sequence $(r_n)_{n\geq 1}$ need not converge to zero, and thus need not be in $\Rho_d$. Yet, as seen in the proof of Theorem~\ref{thm:CSeutaxy}, the sequence defined by $\widetilde r_n=\min\{r_n,1/(2n^{1/d})\}$ for each $n\geq 1$ is in $\Rho_d$. We deduce that the smaller set $\frakF((x_n,\widetilde r_n)_{n\geq 1})$ has full Lebesgue measure in $U$, and thus $\frakF((x_n,r_n)_{n\geq 1})$ as well.
	
	Let us suppose from now on that $\sum_n r_n^d$ converges. In particular, the real numbers $r_n$ go to zero as $n\to\infty$, thereby being bounded above by one for $n$ large enough. The set $\frakF((x_n,r_n)_{n\geq 1})$ is unchanged when removing finitely many terms $x_n$ and $r_n$, so there is no loss in generality in assuming that $r_n\in(0,1]$ for all $n$.
	
	For any gauge function $g$ in $\gauge(\frakn_\rmr)^\complement$, the series $\sum_n g_d(r_n)$ is convergent, so Lemma~\ref{lem:upbndlimsup} entails that the set $\frakF((x_n,r_n)_{n\geq 1})$ has Hausdorff $g_d$-measure equal to zero. Proposition~\ref{prp:compnormalgauge} allows us to transfer the previous property to the Hausdorff $g$-measure itself. The gauge function $g$ is thus majorizing for $\frakF((x_n,r_n)_{n\geq 1})$ in $U$. Considering in particular the gauge function $r\mapsto r^d$, we infer from Proposition~\ref{prp:comphauleb} that this set has Lebesgue measure zero in $U$.

	Now, if $g$ is a gauge function in $\gauge(\frakn_\rmr)$, the series $\sum_n g_d(r_n)$ is divergent. Since the real numbers $r_n$ are nonincreasing and tend to zero, the real numbers $g_d(r_n)^{1/d}$ tend to zero as well and, at least for $n$ sufficiently large, are also nonincreasing. The limsup set $\frakF((x_n,r_n)_{n\geq 1})$ being unchanged when removing initial terms, we may assume that this is the case for all $n$. This means that $(g_d(r_n)^{1/d})_{n\geq 1}$ belongs to $\Rho_d$, and we deduce from the uniform eutaxy property that
		\[
			\leb^d(U\setminus\frakF((x_n,g_d(r_n)^{1/d})_{n\geq 1}))=0.
		\]
		Hence, $(x_n,r_n)_{n\geq 1}$ is not only an approximation system, but also a homogeneously $g$-ubiquitous in $U$ in the sense of Definition~\ref{df:homubiqsysgauge}. We are now in position to apply the large intersection transference principle, namely, Theorem~\ref{thm:largeintprinc}. Accordingly, we deduce that $\frakF((x_n,r_n)_{n\geq 1})$ belongs to $\lic{g}{U}$, which means that $g$ is minorizing.
	
	Finally, the majorizing and the minorizing collections of the set $\frakF((x_n,r_n)_{n\geq 1})$ in $U$ contain $\gauge(\frakn_\rmr)^\complement$ and $\gauge(\frakn_\rmr)$, respectively. This means in particular that this set is fully describable in $U$. We conclude using Proposition~\ref{prp:descincl} and Lemma~\ref{lem:gaugemeasopen}(\ref{item:lem:gaugemeasopen2}).
\end{proof}

The complete description of the size and large intersection properties may indeed be obtained by further applying Theorem~\ref{thm:frakndescslip} to the set $\frakF((x_n,r_n)_{n\geq 1})$. We may also apply Corollary~\ref{cor:frakndescslip} if only a dimensional result is needed. We suppose that the series $\sum_n r_n^d$ converges; otherwise, the set has full Lebesgue measure in $U$ and, as explained at the beginning of Section~\ref{sec:desc}, its dimensional properties are trivial. Accordingly, letting $s_\rmr$ be the exponent associated with $\frakn_\rmr$ through~(\ref{eq:df:sfraknrad01}), we have
	\begin{equation}\label{eq:condsumrn}
		\left\{\begin{array}{lll}
			s<s_\rmr & \Longrightarrow & \sum_n r_n^s=\infty\\[2mm]
			s>s_\rmr & \Longrightarrow & \sum_n r_n^s<\infty,
		\end{array}
		\right.
	\end{equation}
	and we end up with the next statement.

\begin{Corollary}\label{cor:desceutaxy}
	Let $(x_n)_{n\geq 1}$ be a sequence in $\R^d$ that is uniformly eutaxic in some nonempty open set $U\subseteq\R^d$, and let $(r_n)_{n\geq 1}$ be a nonincreasing sequence of positive real numbers such that $\sum_n r_n^d$ converges. Then, for any nonempty open set $V\subseteq U$,
	\[
		\left\{\begin{array}{l}
			\Hdim (\frakF((x_n,r_n)_{n\geq 1})\cap V)=s_\rmr\\[1mm]
			\Pdim (\frakF((x_n,r_n)_{n\geq 1})\cap V)=d\\[1mm]
			F_\rmr\in\lic{s_\rmr}{V},
		\end{array}\right.
	\]
	where the last two properties are valid under the assumption that $s_\rmr$ is positive.
\end{Corollary}

If the underlying sequence of points is uniformly eutaxic, the previous result is an improvement on Theorem~\ref{thm:approxeutaxic}. In fact, the sets $F_t$ defined by~(\ref{eq:df:Ftn}) are obtained when considering the sequence $(r_n^t)_{n\geq 1}$ instead of $(r_n)_{n\geq 1}$ in the above statement. Moreover, if a sequence $(r_n)_{n\geq 1}$ satisfies the assumptions of Theorem~\ref{thm:approxeutaxic}, one easily checks that the exponent associated with $(r_n^t)_{n\geq 1}$ {\em via}~(\ref{eq:condsumrn}) is equal to $d/t$.

This will enable us to examine in Sections~\ref{sec:fracpart},~\ref{sec:Dvoretzky} and~\ref{sec:Poisson} prominent examples of eutaxic sequences based on fractional parts and on random points. Consequently, through the framework of describable sets, we shall be able to shed light on the size and large intersection properties of the associated limsup sets.


\section{Optimal regular systems}\label{sec:optregsys}

The notion of optimal regular system was introduced by Baker and Schmidt~\cite{Baker:1970jf}, and subsequently refined by Beresnevich~\cite{Beresnevich:2000fk}. As we shall illustrate in Sections~\ref{sec:inhom} and~\ref{sec:alg}, they encompass many important examples arising in the metric theory of Diophantine approximation, such as the points with rational coordinates and the real algebraic numbers with bounded degree. On top of that, they naturally give rise to uniformly eutaxic sequences. In light of Section~\ref{subsec:eutaxyapprox}, we shall thus be able to describe thoroughly the size and large intersection properties of limsup sets issued from optimal regular systems; this will be performed in Section~\ref{subsec:optregsysapprox}.

\subsection{Definition and link with eutaxy}\label{subsec:defoptregsys}

Our purpose now is to define the notion of optimal regular system, and to discuss the connection with eutaxic sequences.

\begin{Definition}\label{df:optregsys}
	Let $\calA$ be a countably infinite subset of $\R^d$, let $H:\calA\to (0,\infty)$ be a {\em height} function, and let $U$ be a nonempty open subset of $\R^d$. The pair $(\calA,H)$ is:
	\begin{enumerate}
		\item {\em admissible} if for any integer $m\geq 1$,
		\begin{equation}\label{eq:df:admissiblepair}
			\#\left\{a\in\calA\bigm| |a|<m \text{ and } H(a)\leq m \right\}<\infty\,;
		\end{equation}
		\item a {\em regular system in $U$} if it is admissible and if one may find a real number $\kappa>0$ such that for any open ball $B\subseteq U$, there is a real number $h_B>0$ such that for all $h>h_B$, there exists a subset $\calA_{B,h}$ of $\calA\cap B$ with
		\[
			\left\{\begin{array}{l}
				\#\calA_{B,h}\geq\kappa\diam{B}^d h^d \\[2mm]
				\forall a\in\calA_{B,h} \qquad H(a)\leq h \\[2mm]
				\forall a,a'\in\calA_{B,h} \qquad a\neq a' \quad\Longrightarrow\quad |a-a'|\geq 1/h\,;
			\end{array}\right.
		\]
		\item an {\em optimal system in $U$} if it is admissible and if for any open ball $B$, there exist two real numbers $\kappa'_B>0$ and $h'_B>0$ such that for all $h>h'_B$,
		\[
			\#\{a\in\calA\cap U\cap B \:|\: H(a)\leq h \}\leq \kappa'_B\, h^d.
		\]
	\end{enumerate}
\end{Definition}

Throughout what follows, we shall freely employ the notations of Definition~\ref{df:optregsys} without necessarily reintroducing them. The admissibility condition~(\ref{eq:df:admissiblepair}) is related with the application to approximation problems, and will be justified in Section~\ref{subsec:optregsysapprox}. Moreover, it is elementary to remark that any regular system in $U$ is also regular in every nonempty open subset of $U$\,; the same observation holds for the optimality property. Now, when the set $U$ is bounded, the next lemma shows that any regular system inside may be enumerated monotonically with respect to the height function. The resulting enumerations will play a key r\^ole in the connection between optimal regular systems and eutaxic sequences.

\begin{Lemma}\label{lem:existsmonoenum}
	Let $U$ be a nonempty bounded open subset of $\R^d$, and let $(\calA,H)$ denote a regular system in $U$. Then, there exists an enumeration $(a_n)_{n\geq 1}$ of the set $\calA\cap U$ such that $H(a_n)$ monotonically tends to infinity as $n\to\infty$.
\end{Lemma}

\begin{proof}
	On the one hand, the regularity property of the system $(\calA,H)$ ensures that the set $\calA\cap U$ is countably infinite. On the other hand, as the set $U$ is bounded, it is contained in the open ball $\opball(0,m)$, for $m$ sufficiently large, and the admissibility condition~(\ref{eq:df:admissiblepair}) implies that for any $h>0$, only finitely many points in $\calA\cap U$ have height bounded above by $h$. We deduce the existence of an increasing sequence $(h_j)_{j\geq 1}$ of nonnegative integers with initial term zero and such that all the sets
	\[
		A_j=\{a\in\calA\cap U \:|\: h_j<H(a)\leq h_{j+1}\}
	\]
	are both nonempty and finite. For each integer $j\geq 1$, we write the elements of the set $A_j$ in the form $a^{(j)}_1,\ldots,a^{(j)}_{\#A_j}$, in such a way that
	\[
		H(a^{(j)}_1)\leq\ldots\leq H(a^{(j)}_{\#A_j}).
	\]
	It is clear that for any integer $n\geq 1$, there is a unique pair of integers $(j,k)$, with $j\geq 1$ and $k\in\{1,\ldots,\#A_j\}$, such that
	\[
		n=\#A_1+\ldots+\#A_{j-1}+k.
	\]
	We then define $a_n$ as being equal to $a^{(j)}_k$, and it is elementary to check that the sequence $(a_n)_{n\geq 1}$ fulfills the conditions of the lemma.
\end{proof}

Any sequence $(a_n)_{n\geq 1}$ resulting from Lemma~\ref{lem:existsmonoenum} will be called a {\em monotonic enumeration} of the regular system $(\calA,H)$ in the set $U$. We now present the first part of the connection between optimal regular systems and eutaxic sequences.

\begin{Proposition}\label{prp:optregsyseutaxic}
	Let $U$ be a nonempty bounded open subset of $\R^d$, let $(\calA,H)$ be an optimal regular system in $U$, and let $(a_n)_{n\geq 1}$ denote a monotonic enumeration of $(\calA,H)$ in $U$. Then, the sequence $(a_n)_{n\geq 1}$ is uniformly eutaxic in $U$. In fact,
	\begin{equation}\label{eq:prp:optregsyseutaxic}
		\inf_{\lambda\in\Lambda\setminus\{\emptyset\}\atop\lambda\subseteq U}\liminf_{j\to\infty}2^{-dj}\#\Mu((a_n)_{n\geq 1};\lambda,j)>0.
	\end{equation}
\end{Proposition}

\begin{proof}
	The set $U$ being bounded, it is contained in some open ball $B$. We consider a real number $\gamma\in(0,1)$ such that $\kappa'_B\gamma\leq\diam{[0,1)^d}^d$, and a nonempty dyadic cube $\lambda$ contained in $U$. Observe that there exists an open ball $B'\subseteq\lambda$ satisfying $\diam{B'}=\diam{\lambda}$. Then, let $j$ be a nonnegative integer so large that
		\[
			h=\gamma^{1/d}\frac{2^j}{\diam{\lambda}}>\max\{h'_B,h_{B'}\}.
		\]
	
	The choice of $h$ ensures that any dyadic subcube $\lambda'$ of $\lambda$ with generation equal to $\gene{\lambda}+j$ cannot contain more than one point of the set $\calA_{B',h}$. Otherwise, we would have two distinct points in $\calA_{B',h}$ at a distance bounded above by
		\[
			\diam{\lambda'}=2^{-j}\diam{\lambda}=\frac{\gamma^{1/d}}{h}<\frac{1}{h},
		\]
		which would contradict the third property satisfied by $\calA_{B',h}$. Moreover, every point contained in $\calA_{B',h}$ has height bounded above by $h$ and belongs to the set $\calA\cap U$, thereby being of the form $a_n$ for some $n\geq 1$. The monotonicity of the enumeration implies that $n$ is actually bounded above by
		\[
			\#\{a\in\calA\cap U\cap B\:|\:H(a)\leq h\}\leq \kappa'_B\, h^d
			=\kappa'_B\left(\gamma^{1/d}\,\frac{2^j}{\diam{\lambda}}\right)^d
			\leq\left(\diam{[0,1)^d}\frac{2^j}{\diam{\lambda}}\right)^d,
		\]
		so that $n\leq 2^{d(\gene{\lambda}+j)}$. Lastly, all the points of $\calA_{B',h}$ are contained in $B'$, and thus belong to some dyadic subcube of $\lambda$ with generation $\gene{\lambda}+j$. We deduce that
		\[
			\#\Mu((a_n)_{n\geq 1};\lambda,j)\geq\#\calA_{B',h}\geq\kappa\diam{B'}^d h^d
			=\kappa\left(\diam{\lambda}\gamma^{1/d}\frac{2^j}{\diam{\lambda}}\right)^d
			=\kappa\gamma 2^{dj},
		\]
		and we end up with~(\ref{eq:prp:optregsyseutaxic}) by letting $j$ tend to infinity. Hence, the sequence $(a_n)_{n\geq 1}$ satisfies the condition~(\ref{eq:infliminfpos}), and so the weaker condition~(\ref{eq:liminfpos}) holds as well. The uniform eutaxy of the sequence thus follows from Theorem~\ref{thm:CSeutaxy}.
\end{proof}

Further investigating the connection between optimal regular systems and eutaxic sequences, we now give a converse result to Proposition~\ref{prp:optregsyseutaxic}. We start from the property~(\ref{eq:prp:optregsyseutaxic}) that already appeared in the statement of this proposition and is in fact stronger than uniform eutaxy. This means that we assume that the sequence under consideration satisfies a condition of the form~(\ref{eq:infliminfpos}). As already observed, this condition implies the sufficient condition~(\ref{eq:liminfpos}) that guarantees uniform eutaxy.

\begin{Proposition}\label{prp:eutaxicoptregsys}
	Let $U$ be a nonempty open subset of $\R^d$, and let $(a_n)_{n\geq 1}$ denote a sequence of points contained in $U$. We assume that~(\ref{eq:prp:optregsyseutaxic}) holds, so that in particular $(a_n)_{n\geq 1}$ is uniformly eutaxic in $U$. Moreover, let $\calA$ denote the collection of all values $a_n$, for $n\geq 1$. We endow $\calA$ with the height function $H$ defined by
		\[
			H(a)=\inf\{n\geq 1\:|\:a=a_n\}^{1/d}.
		\]
		Then, the pair $(\calA,H)$ is an optimal regular system in the open set $U$.
\end{Proposition}

\begin{proof}
	For any open ball $B$ and any real number $h>0$, it is clear that a point $a\in\calA\cap U\cap B$ satisfying $H(a)\leq h$ is among the points $a_1,\ldots,a_{\lfloor h^d\rfloor}$. This proves that the pair $(\calA,H)$ is admissible, and is in fact an optimal system in $U$.
	
	Let us now establish that $(\calA,H)$ is a also a regular system in $U$. Throughout, $c$ denotes a real number such that $|x|_{\infty}/c\leq |x|\leq c|x|_{\infty}$ for all $x$ in $\R^d$. Let $B$ be a nonempty open ball contained in $U$, and let $\lambda_B$ denote a nonempty dyadic cube contained in $B$ with minimal generation. One easily checks that $\diam{B}\leq 6c\,2^{-\gene{\lambda_B}}$. Moreover, there is an integer $\underline{j}(\lambda_B)\geq 0$ such that
		\[
			\forall j\geq\underline{j}(\lambda_B) \qquad \#\Mu((a_n)_{n\geq 1};\lambda_B,j)\geq\alpha\, 2^{d(j-1)},
		\]
		where $\alpha$ denotes the left-hand side of~(\ref{eq:prp:optregsyseutaxic}). Thus, just as in the proof of Theorem~\ref{thm:CSeutaxy}, detailed in Section~\ref{subsubsec:eutaxyunifCS}, we infer that for any integer $j\geq\underline{j}(\lambda_B)$, there exists a set $S_j(\lambda_B)\subseteq\{1,\ldots,2^{d(\gene{\lambda_B}+j)}\}$ satisfying the following properties:
		\begin{itemize}
			\item $\# S_j(\lambda_B)\geq\alpha\,2^{d(j-2)}$\,;
			\item $a_n\in\lambda_B$ for any $n\in S_j(\lambda_B)$\,;
			\item $|a_n-a_{n'}|_\infty\geq 2^{-(\gene{\lambda_B}+j)}$ for any distinct $n,n'\in S_j(\lambda_B)$.
		\end{itemize}
		For any real number $h$ larger than $c\,2^{\gene{\lambda_B}+\underline{j}(\lambda_B)}$, letting $j$ be equal to the integer $\lfloor\log_2(h/c)\rfloor-\gene{\lambda_B}$, where $\log_2$ is the base two logarithm, we have $j\geq\underline{j}(\lambda_B)$. Hence, we may define $\calA_{B,h}$ as the collection of all points $a_n$, for $n$ in $S_j(\lambda_B)$. It is then straightforward to check that $\calA_{B,h}$ is a subset of $\calA\cap B$ such that
		\[
			\left\{\begin{array}{l}
			\#\calA_{B,h}=\# S_j(\lambda_B)\geq\alpha\,2^{d(j-2)}\geq\alpha\diam{B}^d h^d/(48c^2)^d \\[2mm]
			\forall a\in\calA_{B,h} \qquad H(a)\leq (2^{d(\gene{\lambda_B}+j)})^{1/d}\leq h/c\leq h \\[2mm]
			\forall a,a'\in\calA_{B,h} \qquad a\neq a' \quad\Longrightarrow\quad |a-a'|\geq 2^{-(\gene{\lambda_B}+j)}/c\geq 1/h,
			\end{array}\right.
		\]
		and we deduce that the pair $(\calA,H)$ is a regular system in the set $U$.
\end{proof}

Combining Propositions~\ref{prp:optregsyseutaxic} and~\ref{prp:eutaxicoptregsys}, we may finally deduce that, rather than being equivalent to uniform eutaxy, the notion of optimal regular system is essentially comparable with the stronger condition~(\ref{eq:prp:optregsyseutaxic}).

\subsection{Approximation by optimal regular systems}\label{subsec:optregsysapprox}

In the spirit of Diophantine approximation, given an optimal regular system $(\calA,H)$, we may naturally consider the sets of the form
	\begin{equation}\label{eq:df:FphAH}
		F_\ph=\left\{x\in\R^d\bigm||x-a|<\ph(H(a))\quad\text{for i.m.~} a\in\calA\right\},
	\end{equation}
	where $\ph$ is a positive nonincreasing continuous function defined on $[0,\infty)$. Combining Theorem~\ref{thm:desceutaxy} and Proposition~\ref{prp:optregsyseutaxic}, we shall be able to describe the size and large intersection properties of these sets, because the underlying optimal regular system actually results in a uniformly eutaxic sequence. However, though they will play a prominent r\^ole in the proofs, we shall state our results without explicitly mentioning eutaxic sequences.

We may now justify the admissibility condition~(\ref{eq:df:admissiblepair}) arising in the definition of optimal regular systems. In fact, if the pair $(\calA,H)$ is admissible and the function $\ph$ additionally tends to zero at infinity, then the family $(a,\ph(H(a)))_{a\in\calA}$ of elements of $\R^d\times(0,\infty)$ is an approximation system in the sense of Definition~\ref{df:approxsys}. The set $F_\ph$ thus naturally fits into the frameworks supplied by homogeneous ubiquity, and by the mass and large intersection transference principles, see Sections~\ref{sec:firstubiq} and~\ref{sec:transference}.

The complete description of the size and large intersection properties of the set $F_\ph$ calls upon two objects defined in terms of the function $\ph$, specifically, the integral
	\begin{equation}\label{eq:df:Iph}
			I_\ph=\int_0^\infty \eta^{d-1}\ph(\eta)^d\,\dd\eta,
	\end{equation}
	and the Borel measure $\frakn_\ph$ on $(0,1]$ characterized by the condition that for any nonnegative measurable function $f$ supported in $(0,1]$,
	\begin{equation}\label{eq:df:fraknph}
		\int_{(0,1]} f(r)\,\frakn_\ph(\dd r)=\int_0^\infty \eta^{d-1}f(\ph(\eta))\,\dd\eta.
	\end{equation}
	We may now describe the size and large intersection properties of $F_\ph$.

\begin{Theorem}\label{thm:descoptregsys}
	Let $U$ be a nonempty open subset of $\R^d$, let $(\calA,H)$ be an optimal regular system in $U$, and let $\ph$ be a positive nonincreasing continuous function defined on $[0,\infty)$. Then, the following properties hold:
	\begin{itemize}
		\item if $I_\ph$ diverges, then $F_\ph$ has full Lebesgue measure in $U$\,;
		\item if $I_\ph$ converges, then $F_\ph$ is $\frakn_\ph$-describable in $U$.
	\end{itemize}
\end{Theorem}

\begin{proof}
	We begin with the divergence case. The open set $U$ may clearly be written as a countable union of open balls. Hence, the proof reduces to establishing that $F_\ph$ has full Lebesgue measure in any nonempty open ball contained in $U$. If $B$ is such a ball, the pair $(\calA,H)$ is also an optimal regular system in $B$, so Lemma~\ref{lem:existsmonoenum} enables us to consider a monotonic enumeration of $(\calA,H)$ in $B$, denoted by $(a_n)_{n\geq 1}$. Then, it is clear that $F_\ph$ contains the set $F^B_\ph$ defined by
		\begin{equation}\label{eq:FBphenum}
			F^B_\ph=\left\{x\in\R^d\bigm||x-a_n|<r_n\quad\text{for i.m.~} n\geq 1\right\},
		\end{equation}
		where $r_n=\ph(H(a_n))$ for any $n\geq 1$. By virtue of Proposition~\ref{prp:optregsyseutaxic}, the sequence $(a_n)_{n\geq 1}$ is uniformly eutaxic in $B$. Moreover, the divergence of the integral $I_\ph$ implies that of the series $\sum_n r_n^d$ diverges, see hereunder. The sequence defined by $\widetilde r_n=\min\{r_n,1/(2n^{1/d})\}$ for each $n\geq 1$ is then $\Rho_d$. We deduce that for Lebesgue-almost every $x$ in $B$, there are infinitely many integers $n\geq 1$ such that $|x-a_n|<\widetilde r_n$. Hence, the set $F^B_\ph$ has full Lebesgue measure in $B$, and thus the set $F_\ph$ as well.
	
	Let us prove that the series $\sum_n r_n^d$ diverges when the integral $I_\ph$ does. First, we may clearly assume that the function $\ph$ tends to zero at infinity; the result is elementary otherwise. Let $\zeta$ be the premeasure defined on the intervals  $(h,h')$, with $0<h\leq h'<\infty$, by the formula $\zeta((h,h'))=\ph(h)^d-\ph(h')^d$, and let $\zeta_\ast$ be the outer measure defined by~(\ref{eq:df:zetaastm}). It follows from Theorem~\ref{thm:bormeasm} that the Borel sets contained in $(0,\infty)$ are $\zeta_\ast$-measurable. The resulting Borel measure is called the {\em Lebesgue-Stieltjes measure} associated with the monotonic function $\ph^d$, and we may integrate locally bounded measurable functions with respect to that measure. One may prove that the above outer measure $\zeta_\ast$ coincides with the outer measure $\zeta^\ast$ defined by~(\ref{eq:df:zetaasta}), and also coincides with the premeasure $\zeta$ on the intervals where it is defined. Combining this observation with~(\ref{eq:measincunion}) and the fact that $\ph$ tends to zero at infinity, we deduce in particular that $\zeta_\ast([h,\infty))=\ph(h)^d$ for any $h>0$. Accordingly, using Tonelli's theorem, we have
		\[
			\sum_{n=1}^\infty r_n^d
			=\sum_{n=1}^\infty\int_0^\infty\ind_{\{H(a_n)\leq h\}}\,\zeta_\ast(\dd h)
			=\int_0^\infty\#\{n\geq 1\:|\:H(a_n)\leq h\}\,\zeta_\ast(\dd h).
		\]
		The regularity of the system implies that this is bounded below by
		\[
			\int_0^\infty\kappa\diam{B}^d h^d\,\zeta_\ast(\dd h)
				+\int_0^{h_B}\left(\#\{n\geq 1\:|\:H(a_n)\leq h\}-\kappa\diam{B}^d h^d\right)\,\zeta_\ast(\dd h),
		\]
		and the first integral is equal to $\kappa d\diam{B}^d I_\ph$ by Tonelli's theorem again. This proves that the series $\sum_n r_n^d$ is divergent when the integral $I_\ph$ is.
	
	Let us turn our attention to the convergence case. Note that, since the function $\ph$ is nonincreasing, it necessarily tends to zero at infinity. Let us consider a gauge function $g\in\gauge(\frakn_\ph)$. The idea is to replace in the above arguments the function $\ph$ by the function $h\mapsto g_d(\ph(h))^{1/d}$, denoted by $g_d^{1/d}\circ\ph$ for short. This new function might not be continuous and nonincreasing on the whole interval $[0,\infty)$, but surely satisfies these properties on the closed right-infinite interval of all real numbers $h\geq 0$ such that $\ph(h)\leq\eps_{g_d}/2$, where $\eps_{g_d}$ is defined in Section~\ref{subsubsec:netmrev}. Therefore, letting $\widetilde\ph(h)=g_d(\min\{\ph(h),\eps_{g_d}/2\})^{1/d}$, we get a function that is continuous and nonincreasing on the whole $[0,\infty)$ and matches the function of interest near infinity.
	
	The fact that $g$ is in $\gauge(\frakn_\ph)$ implies that the integral $I_{\widetilde\ph}$ is divergent. We deduce from the previous paragraphs that the set $F_{\widetilde\ph}$ has full Lebesgue measure in $U$, and thus that the larger set $F_{g_d^{1/d}\circ\ph}$ has full Lebesgue measure in $U$ as well. As a consequence, $(a,\ph(H(a)))_{a\in\calA}$ is not only an approximation system, but also a homogeneous $g$-ubiquitous system in $U$. We conclude that $F_\ph$ belongs to $\lic{g}{U}$ by means of the large intersection transference principle, namely, Theorem~\ref{thm:largeintprinc}. The gauge function $g$ is thus minorizing for $F_\ph$ in $U$.
	
	Given a nonempty open ball $B\subseteq U$, let $(a_n)_{n\geq 1}$ denote again a monotonic enumeration of $(\calA,H)$ in $B$. The intersection $F_\ph\cap B$ is contained in the set $F^B_\ph$ defined by~(\ref{eq:FBphenum}). Indeed, if $x$ is in $F_\ph\cap B$, we may find in $B$ a ball $B'$ of the form $\opball(x,r)$ for a sufficiently small $r>0$. Moreover, as $\ph$ tends to zero at infinity, we have $\ph(h)\leq r$ for any real number $h$ larger than some $h_0$. Now, there is an infinite subset $\calA_x$ of $\calA$ formed by points $a$ satisfying $|x-a|<\ph(H(a))$. In particular, all these points belong to the open ball centered at $x$ with radius $\ph(0)$, so that
	\[
		\{a\in\calA_x\:|\:H(a)\leq h_0\}\subseteq\left\{a\in\calA\bigm| |a|<|x|+\ph(0) \text{ and } H(a)\leq h_0\right\}.
	\]
	The latter set is finite in view of the admissibility condition~(\ref{eq:df:admissiblepair}). It follows that infinitely many points $a$ in the set $\calA_x$ have height larger than $h_0$, thereby satisfying $\ph(H(a))\leq r$. All these points thus belong to the ball $B'$, and must then be of the form $a_n$ for some integer $n\geq 1$. We deduce that $x$ belongs to the set $F^B_\ph$.

	Let us now consider a gauge function $g\in\gauge(\frakn_\ph)^\complement$. As shown below, the series $\sum_n g_d(r_n)$ is then convergent. Combining Lemma~\ref{lem:upbndlimsup} and Proposition~\ref{prp:compnormalgauge}, we deduce that the set $F^B_\ph$ has Hausdorff $g$-measure zero. Hence, the set $F_\ph\cap B$ has $g$-measure zero as well, and we may in fact replace the ball $B$ above by the whole open set $U$, because the Hausdorff $g$-measure is an outer measure and every open set may be written as a countable union of inside open balls. The gauge function $g$ is thus majorizing for $F_\ph$ in $U$. Besides, let us remark that when $I_\ph$ is convergent, the gauge function $r\mapsto r^d$ is in $\gauge(\frakn_\ph)^\complement$, and we deduce from Proposition~\ref{prp:comphauleb} that $F_\ph$ has Lebesgue measure zero in $U$.

	Let us justify the convergence of $\sum_n g_d(r_n)$. The gauge function $g_d$ is nondecreasing on the interval $[0,\eps_{g_d})$, so we may consider a function $\widetilde g$ that is nondecreasing on $[0,\infty)$ and coincides with $g_d$ on $[0,\eps_{g_d})$. Still reasoning as above, we define a premeasure $\zeta$ by $\zeta((h,h'))=\widetilde g(\ph(h))-\widetilde g(\ph(h'))$ when $0<h\leq h'<\infty$, and then consider the outer measure $\zeta_\ast$ given by~(\ref{eq:df:zetaastm}). We end up with a Borel measure on $(0,\infty)$ such that $\zeta_\ast([h,\infty))=\widetilde g(\ph(h))$ for any $h>0$. Thanks to Tonelli's theorem,
		\[
			\sum_{n=1}^\infty \widetilde g(r_n)=\int_0^\infty\#\{n\geq 1\:|\:H(a_n)\leq h\}\,\zeta_\ast(\dd h).
		\]
		Due to the optimality of the underlying system, this is bounded above by
		\[
			\int_0^\infty\kappa'_B h^d\,\zeta_\ast(\dd h)
			+\int_0^{h'_B}\left(\#\{n\geq 1\:|\:H(a_n)\leq h\}-\kappa'_B h^d\right)\,\zeta_\ast(\dd h),
		\]
		and another application of Tonelli's theorem shows that the first integral equals
		\[
			\kappa'_B d \int_0^\infty\eta^{d-1}\widetilde g(\ph(\eta))\,\dd\eta.
		\]
		Since $g$ is in $\gauge(\frakn_\ph)^\complement$ and $\ph$ tends to zero at infinity, this integral is convergent. So, the series in the left-hand side is also convergent. We may replace $\widetilde g$ by $g_d$ without altering the convergence of the series, again because $\ph$ vanishes at infinity.

	Finally, we established that the majorizing and the minorizing collections of the set $F_\ph$ in $U$ contain $\gauge(\frakn_\ph)^\complement$ and $\gauge(\frakn_\ph)$, respectively. In particular, this set is fully describable in $U$. We conclude using Proposition~\ref{prp:descincl} and Lemma~\ref{lem:gaugemeasopen}(\ref{item:lem:gaugemeasopen2}).
\end{proof}


\section{Homogeneous and inhomogeneous approximation}\label{sec:inhom}

A simple example of optimal regular system is supplied by the points with rational coordinates; this corresponds to the problem of homogeneous Diophantine approximation. We now detail this example, together with its inhomogeneous counterpart. We shall then state the corresponding metric results obtained by further applying Theorem~\ref{thm:descoptregsys}, thereby recovering  famous theorems due to Besicovitch~\cite{Besicovitch:1934ly}, Jarn\'ik~\cite{Jarnik:1929mf,Jarnik:1931qf} and Khintchine~\cite{Khintchine:1926fk}, as well as their inhomogeneous analogs.

\subsection{Associated optimal regular system}

Let us recall from Section~\ref{subsec:inhomapprox} that the inhomogeneous approximation problem is obtained when shifting the approximating rational points $p/q$ by a chosen value $\alpha$ in $\R^d$. The approximation is then realized by the points that belong to the collection
	\[
		\Q^{d,\alpha}=\left\{\frac{p+\alpha}{q},\ (p,q)\in\Z^d\times\N\right\}
	\]
	Obviously, when $\alpha$ vanishes, we recover the set $\Q^d$ of points with rational coordinates, and the homogeneous approximation problem. The collection $\Q^{d,\alpha}$ is endowed with the height function $H^\alpha_d$ defined by
	\begin{equation}\label{eq:df:heightQalphad}
		H^\alpha_d(a)=\inf\{q\in\N\:|\:qa-\alpha\in\Z^d\}^{1+1/d}.
	\end{equation}
	As shown by the next statement, we thus obtain an optimal regular system in $\R^d$. The proof is essentially due to Bugeaud~\cite{Bugeaud:2004zr} and relies on an inhomogeneous approximation result discussed in Section~\ref{subsec:inhomapprox} above, specifically, Proposition~\ref{prp:Dirichletinhomvar}.

\begin{Theorem}\label{thm:inhomratoptregsys}
	For any point $\alpha$ in $\R^d$, the pair $(\Q^{d,\alpha},H^\alpha_d)$ is an optimal regular system in $\R^d$.
\end{Theorem}

\begin{proof}
	When the open set $U$ is equal to the whole space $\R^d$ in Definition~\ref{df:optregsys}, one easily checks that the notion of optimal regular system does not depend on the choice of the norm. We thus choose to work with the supremum norm.

	Establishing the optimality of the system is rather elementary. Indeed, let $B$ denote the open ball $\opball(x,r)$, and let $a$ be a point in $\Q^{d,\alpha}\cap B$ with height at most $h$. We write $a$ in the form $(p+\alpha)/q$, with $p\in\Z^d$ and $q\in\N$ as small as possible. As a result, $H^\alpha_d(a)=q^{1+1/d}$, which means that $q\leq h^{d/(d+1)}$. Moreover, the number of possible values for the point $p$ is not greater than $(2rq+1)^d$. Hence,
	\begin{align*}
		\#\{a\in\Q^{d,\alpha}\cap B \:|\: H^\alpha_d(a)\leq h \}
		&\leq\sum_{1\leq q\leq h^{d/(d+1)}}(2rq+1)^d\\
		&\leq h^{d/(d+1)}(2rh^{d/(d+1)}+1)^d
		\leq (4r)^d h^d,
	\end{align*}
	where the last bound holds for $h\geq (2r)^{-1-1/d}$.

	Let us prove the regularity of the system. For any point $y$ in $\R^d$, let $q(y)$ denote the minimal value of the integer $q\geq 1$ for which
	\[
		\exists p\in\Z^d \qquad |qy-p|_\infty\leq\frac{1}{\lfloor 2^{-1/d} h^{1/(d+1)}\rfloor}.
	\]
	Dirichlet's theorem then shows that $2q(y)$ is bounded above by $h^{d/(d+1)}$, with the proviso that the following condition holds:
	\begin{equation}\label{eq:proof:thm:inhomratoptregsys1}
	h\geq 2^{(d+1)^2/d}.
	\end{equation}
	We assume from now on that this condition is satisfied. Given two parameters $\gamma$ and $\delta$ in $(0,1)$, let $B'$ be the open ball concentric with $B$ with radius $\delta$ times that of $B$, and let $B''$ be the set of points $y$ in $B'$ such that $2q(y)<\gamma h^{d/(d+1)}$. The set $B''$ is covered by the closed balls with radius $2^{1+1/d}/(q h^{1/(d+1)})$ centered at the rational points $p/q$ within distance $1/q$ of the ball $B'$ and with denominator $q<\gamma h^{d/(d+1)}/2$. For any fixed choice of $q$, there are at most $(2q\delta r+3)^d$ such points.	Hence, the Lebesgue measure of the set $B''$ is at most
	\[
		\sum_{1\leq q<\gamma h^{d/(d+1)}/2}(2q\delta r+3)^d\left(\frac{2^{2+1/d}}{q h^{1/(d+1)}}\right)^d
		=\frac{2^{2d+1}}{h^{d/(d+1)}}\sum_{1\leq q<\gamma h^{d/(d+1)}/2}\left(2\delta r+\frac{3}{q}\right)^d.
	\]
In order to derive an upper bound on the sum in the right-hand side, we first consider the case in which $q<3/(2\delta r)$. In that situation, the summand is clearly bounded by $6^d$. In the opposite case, the summand is bounded by $(4\delta r)^d$. Thus,
	\[
	\leb^d(B'')
	\leq\frac{3\cdot 24^d}{\delta r h^{d/(d+1)}}+(16\delta r)^d\gamma.
	\]
	Finally, we define $\calA_{B,h}$ as any maximal collection of points belonging to the set $\Q^{d,\alpha}\cap B$ with height at most $h$ and separated from each other by a distance at least $(2/\gamma)^{1+1/d}/h$, thus in particular at least $1/h$.
	
	We now search for an appropriate lower bound on the cardinality of $\calA_{B,h}$. Note that each point $y$ in the set $B'\setminus B''$ satisfies
	\[
	q(y)\geq\frac{\gamma}{2}\,h^{d/(d+1)}\geq\gamma\lfloor 2^{-1/d}h^{1/(d+1)}\rfloor^d.
	\]
	Applying Proposition~\ref{prp:Dirichletinhomvar} to the integer $\lfloor 2^{-1/d}h^{1/(d+1)}\rfloor$, the point $\alpha$, and each point $y$ in the set $B'\setminus B''$, we infer the existence of two real numbers $\Gamma_\ast$ and $H_\ast$, both larger than one and depending on $\gamma$ and $d$ only, such that the condition
	\begin{equation}\label{eq:proof:thm:inhomratoptregsys2}
	h>H_\ast
	\end{equation}
	implies that for each point $y$ in the set $B'\setminus B''$, there is a pair $(p,q)$ in $\Z^d\times\N$ with
	\[
	q(y)\leq q<2q(y)
	\qquad\text{and}\qquad
	|qy-p-\alpha|_{\infty}\leq\frac{\Gamma_\ast}{q(y)^{1/d}}.
	\]
	In that situation, we straightforwardly deduce that
	\[
	\left|y-\frac{p+\alpha}{q}\right|_\infty\leq\frac{\Gamma_\ast}{q(y)^{1+1/d}}
	\leq\frac{\Gamma_\ast}{h}\left(\frac{2}{\gamma}\right)^{1+1/d}.
	\]
	Given that the point $y$ is in the ball $B'$, this means in particular that the point $(p+\alpha)/q$ belongs to the set $\Q^{d,\alpha}\cap B$ if the following condition holds:
	\begin{equation}\label{eq:proof:thm:inhomratoptregsys3}
	\frac{\Gamma_\ast}{h}\left(\frac{2}{\gamma}\right)^{1+1/d}+\delta r\leq r.
	\end{equation}
	On top of that, we observed previously that $2q(y)$ is bounded above by $h^{d/(d+1)}$, so we deduce that this point satisfies
	\[
	H^\alpha_d\left(\frac{p+\alpha}{q}\right)\leq q^{1+1/d}<(2q(y))^{1+1/d}\leq h.
	\]
	Since the collection $\calA_{B,h}$ is maximal, it contains a point $(p'+\alpha)/q'$ located at a distance smaller than $(2/\gamma)^{1+1/d}/h$ from $(p+\alpha)/q$, so that
	\[
	\left|y-\frac{p'+\alpha}{q'}\right|_\infty
	\leq\left|y-\frac{p+\alpha}{q}\right|_\infty+\left|\frac{p+\alpha}{q}-\frac{p'+\alpha}{q'}\right|_\infty
	<\frac{\Gamma_\ast+1}{h}\left(\frac{2}{\gamma}\right)^{1+1/d}.
	\]
	Hence, the set $B'\setminus B''$ is covered by the open balls centered at the points in $\calA_{B,h}$ with radius the right-hand side above. Using the fact that $\Gamma_\ast>1$, we obtain
	\[
	(2\delta r)^d-\frac{3\cdot 24^d}{\delta r h^{d/(d+1)}}-(16\delta r)^d\gamma
	\leq\leb^d(B'\setminus B'')
	\leq\left(\frac{4\Gamma_\ast}{h}\right)^d\left(\frac{2}{\gamma}\right)^{d+1}\#\calA_{B,h},
	\]
	from which we deduce that
	\begin{equation}\label{eq:proof:thm:inhomratoptregsys4}
	\frac{\#\calA_{B,h}}{\diam{B}^d h^d}\geq
	\left(\frac{\delta}{4\Gamma_\ast}\right)^d\left(\frac{\gamma}{2}\right)^{d+1}\left(1-8^d\gamma-\frac{3\cdot 12^d}{(\delta r)^{d+1}h^{d/(d+1)}}\right).
	\end{equation}
	To conclude, we choose $\gamma$ smaller than $8^{-d}$, and $\delta$ arbitrarily, and we require that $h$ is large enough to ensure that~(\ref{eq:proof:thm:inhomratoptregsys1}),~(\ref{eq:proof:thm:inhomratoptregsys2}) and~(\ref{eq:proof:thm:inhomratoptregsys3}) all hold, and that~(\ref{eq:proof:thm:inhomratoptregsys4}) holds with a constant that depends on $d$ in the right-hand side.
\end{proof}

Combining Proposition~\ref{prp:optregsyseutaxic} and Theorem~\ref{thm:inhomratoptregsys}, we directly get the following property: for any nonempty bounded open subset $U$ of $\R^d$, any monotonic enumeration of the optimal regular system $(\Q^{d,\alpha},H^\alpha_d)$ in the set $U$ is uniformly eutaxic. In particular, the arguably most natural enumeration of the rational numbers that are strictly between zero and one, namely, the sequence
\[
\frac{1}{2},\frac{1}{3},\frac{2}{3},\frac{1}{4},\frac{3}{4},\frac{1}{5},\frac{2}{5},\frac{3}{5},\frac{4}{5},\frac{1}{6},\frac{5}{6},\frac{1}{7},\frac{2}{7},\frac{3}{7},\frac{4}{7},\frac{5}{7},\frac{6}{7},\ldots
\]
is uniformly eutaxic in the open interval $(0,1)$.

\subsection{General metrical implications}\label{subsec:inhomdesc}

We may use Theorem~\ref{thm:inhomratoptregsys} in conjunction with Theorem~\ref{thm:descoptregsys} to describe the size and large intersection properties of the set 
	\begin{equation}\label{eq:df:Qdpsialpha}
		\frakQ^\alpha_{d,\psi}=\left\{x\in\R^d\Biggm|\left|x-\frac{p+\alpha}{q}\right|_\infty<\psi(q)\quad\text{for i.m.~}(p,q)\in\Z^d\times\N\right\},
	\end{equation}
	where $\psi$ denotes a positive nonincreasing continuous function defined on $[0,\infty)$. Indeed, this set is exactly the set $F_\ph$ defined by~(\ref{eq:df:FphAH}) when $\ph(\eta)=\psi(\eta^{d/(d+1)})$ for all $\eta\geq 0$, and the underlying system $(\calA,H)$ is equal to $(\Q^{d,\alpha},H^\alpha_d)$. This prompts us to introduce the associated integral $I_\ph$ and measure $\frakn_\ph$ defined by~(\ref{eq:df:Iph}) and~(\ref{eq:df:fraknph}), respectively. However, it is more natural to express the results in terms of $\psi$ only, so we preferably consider the integral
	\[
		I_{d,\psi}=\int_0^\infty q^d\psi(q)^d\,\dd q
	\]
	and the Borel measure $\frakn_{d,\psi}$ on $(0,1]$ characterized by the condition that for any nonnegative measurable function $f$ supported in $(0,1]$,
	\[
		\int_{(0,1]} f(r)\,\frakn_{d,\psi}(\dd r)=\int_0^\infty q^d f(\psi(q))\,\dd q.
	\]
	Performing a simple change of variable, one easily checks that the convergence of $I_{d,\psi}$ amounts to that of $I_\ph$, and that the sets of gauge functions $\gauge(\frakn_{d,\psi})$ and $\gauge(\frakn_{\ph})$ coincide. The above discussion finally leads to the next statement.

\begin{Theorem}\label{thm:inhomdesc}
	Let $\alpha$ be a point in $\R^d$ and let $\psi$ denote a positive nonincreasing continuous function defined on $[0,\infty)$. Then, the following properties hold:
	\begin{itemize}
		\item if $I_{d,\psi}$ diverges, then $\frakQ^\alpha_{d,\psi}$ has full Lebesgue measure in $\R^d$\,;
		\item if $I_{d,\psi}$ converges, then $\frakQ^\alpha_{d,\psi}$ is $\frakn_{d,\psi}$-describable in $\R^d$.
	\end{itemize}
\end{Theorem}

Let us now detail some consequences of Theorem~\ref{thm:inhomdesc}. First, this result shows that the set $\frakQ^\alpha_{d,\psi}$ has full Lebesgue measure in $\R^d$ if the integral $I_{d,\psi}$ diverges, and Lebesgue measure zero if it converges. Hence, for any nonempty open set $V\subseteq\R^d$,
	\[
		\leb^d(\frakQ^\alpha_{d,\psi}\cap V)=
		\begin{cases}
			\leb^d(V) & \text{if } \sum_q q^d \psi(q)^d=\infty \\[2mm]
			0 & \text{if } \sum_q q^d \psi(q)^d<\infty.
		\end{cases}
	\]
	Indeed, the monotonicity of $\psi$ implies that the convergence of $I_{d,\psi}$ amounts to that of the above series. We thus recover a result obtained by Schmidt~\cite{Schmidt:1964vn}. In the homogeneous case, this corresponds to a famous theorem due to Khintchine~\cite{Khintchine:1926fk}.

Thanks to Theorem~\ref{thm:frakndescslip}, we may in fact deduce from Theorem~\ref{thm:inhomdesc} a complete description of the size and large intersection properties of the set $\frakQ^\alpha_{d,\psi}$. We restrict our attention to the case where $\frakQ^\alpha_{d,\psi}$ has Lebesgue measure zero; as explained at the beginning of Section~\ref{sec:desc}, these properties are trivial otherwise. In particular, we infer that for any gauge function $g$ and any nonempty open set $V\subseteq\R^d$,
	\[
		\hau^g(\frakQ^\alpha_{d,\psi}\cap V)=
		\begin{cases}
			\infty & \text{if } \sum_q q^d g_d(\psi(q))=\infty \\[2mm]
			0 & \text{if } \sum_q q^d g_d(\psi(q))<\infty.
		\end{cases}
	\]
	Note that we use here the elementary fact that a gauge function $g$ belongs to the set $\gauge(\frakn_{d,\psi})$ if and only if its $d$-normalization $g_d$ is such that the above series diverges; this follows from the monotonicity of $\psi$ and that of $g_d$ near the origin. We thus recover the extension established by Bugeaud~\cite{Bugeaud:2004zr} of a classical statement due to Jarn\'ik~\cite{Jarnik:1931qf}. Likewise, Theorems~\ref{thm:frakndescslip} and~\ref{thm:inhomdesc} allow us to recover the description of the large intersection properties of the set $\frakQ^\alpha_{d,\psi}$ that was obtained in~\cite{Durand:2007uq}.

Using Corollary~\ref{cor:frakndescslip}, we may also give a more concise dimensional statement. In fact, still focusing on the case where $\frakQ^\alpha_{d,\psi}$ has Lebesgue measure zero, we see that the integral $I_{d,\psi}$ converges and that the exponent associated with $\frakn_{d,\psi}$ {\em via}~(\ref{eq:df:sfraknrad01}) is
	\[
		s_{d,\psi}=\limsup_{q\to\infty}\frac{(d+1)\log q}{-\log\psi(q)},
	\]
	so we eventually obtain the next statement.

\begin{Corollary}\label{cor:inhomdesc}
		Let $\alpha$ be a point in $\R^d$ and let $\psi$ denote a positive nonincreasing continuous function defined on $[0,\infty)$ such that $I_{d,\psi}$ converges. Then, for any nonempty open set $V\subseteq\R^d$,
	\[
		\left\{\begin{array}{l}
			\Hdim (\frakQ^\alpha_{d,\psi}\cap V)=s_{d,\psi}\\[1mm]
			\Pdim (\frakQ^\alpha_{d,\psi}\cap V)=d\\[1mm]
			\frakQ^\alpha_{d,\psi}\in\lic{s_{d,\psi}}{V},
		\end{array}\right.
	\]
	where the last two properties are valid under the assumption that $s_{d,\psi}$ is positive.
\end{Corollary}

Another application is to describe the size and large intersection properties of the intersection of countably many sets of the form $\frakQ^\alpha_{d,\psi}$. To be specific, for each integer $n\geq 1$, let us consider a point $\alpha_n$ in $\R^d$ and a positive nonincreasing continuous function $\psi_n$ defined on $[0,\infty)$ such that $I_{d,\psi_n}$ converges. Then, similarly to~(\ref{eq:majominoEnfraknsn}), we may combine Theorem~\ref{thm:inhomdesc} with Propositions~\ref{prp:monomajomino} and~\ref{prp:cupcapmajomino} to infer that
	\begin{equation}\label{eq:capmajominoinhom}
		\left\{\begin{array}{l}
			\displaystyle\mino\left(\bigcap_{n=1}^\infty\frakQ^{\alpha_n}_{d,\psi_n},\R^d\right)\cap\gauge^\infty=\bigcap_{n=1}^\infty\gauge(\frakn_{d,\psi_n})\\[5mm]
			\displaystyle\majo\left(\bigcap_{n=1}^\infty\frakQ^{\alpha_n}_{d,\psi_n},\R^d\right)\supseteq\gauge^\infty\setminus\bigcap_{n=1}^\infty\gauge(\frakn_{d,\psi_n}).
		\end{array}\right.
	\end{equation}
	Hence, the intersection of the sets $\frakQ^{\alpha_n}_{d,\psi_n}$ is fully describable in $\R^d$. Further assumptions on $\psi_n$ can make the intersection of the sets $\gauge(\frakn_{d,\psi_n})$ more explicit, and yield more comprehensive results. For instance, if the measures $\frakn_{d,\psi_n}$ may be written in the form~(\ref{eq:df:frakns}), Proposition~\ref{prp:dichofraksdesc} implies that the intersection of the sets $\frakQ^{\alpha_n}_{d,\psi_n}$ is either $\frakn_s$-describable for some $s\in[0,d)$, or $\fraks$-describable for some $\fraks\in[0,d)$.

\subsection{An inhomogeneous Jarn\'ik-Besicovitch theorem}

Let us focus on the particular case where the function $\psi$ is of the form $q\mapsto q^{-\tau}$ on the interval $[1,\infty)$, for some real number $\tau>0$. Then, $\frakQ^\alpha_{d,\psi}$ reduces to the set defined by~(\ref{eq:df:Jdtaualpha}), namely,
	\[
		J^\alpha_{d,\tau}=\left\{x\in\R^d\Biggm|\left|x-\frac{p+\alpha}{q}\right|_\infty<\frac{1}{q^\tau}\quad\text{for i.m.~}(p,q)\in\Z^d\times\N\right\}.
	\]
	When $\alpha$ vanishes, the above set is the introductory set $J_{d,\tau}$ defined by~(\ref{eq:df:Jdtau}) that corresponds to the homogeneous setting. We complete the definition of the function $\psi$ by assuming that it is constant equal to one on $[0,1]$. Clearly, $\Iota_{d,\psi}$ converges if and only if $\tau>1+1/d$. In that case, the set $\gauge(\frakn_{d,\psi})$ coincides with $\gauge(\frakn_{(d+1)/\tau})$, where $\frakn_{(d+1)/\tau}$ is defined as in~(\ref{eq:df:frakns}). Theorem~\ref{thm:inhomdesc} then leads to the next statement.

\begin{Corollary}\label{cor:inhomJBdesc}
	For any point $\alpha$ in $\R^d$, the following properties hold:
	\begin{itemize}
		\item for any $\tau\leq 1+1/d$, the set $J^\alpha_{d,\tau}$ has full Lebesgue measure in $\R^d$\,;
		\item for any $\tau>1+1/d$, the set $J^\alpha_{d,\tau}$ is $\frakn_{(d+1)/\tau}$-describable in $\R^d$.
	\end{itemize}
\end{Corollary}

In the latter case, making use of Corollary~\ref{cor:frakndescslip} and recalling that the exponent associated through~(\ref{eq:df:sfraknrad01}) to the measure $\frakn_{(d+1)/\tau}$ is equal to $(d+1)/\tau$, we get
	\[
		\left\{\begin{array}{l}
			\Hdim J^\alpha_{d,\tau}=(d+1)/\tau\\[1mm]
			\Pdim J^\alpha_{d,\tau}=d\\[1mm]
			J^\alpha_{d,\tau}\in\lic{(d+1)/\tau}{\R^d}.
		\end{array}\right.
	\]
	This can also be seen as a consequence of Corollary~\ref{cor:inhomdesc}. In the homogeneous case where $\alpha$ vanishes, we thus recover the Jarn\'ik-Besicovitch theorem and the large intersection companion result, see Theorem~\ref{thm:JarnikBesicovitch} and Corollary~\ref{cor:JarnikBesicovitchsli}, respectively.

In light of the end of the previous section, we may also consider countably many values of the parameter $\alpha$ and study the size of the intersection of the corresponding sets $J^\alpha_{d,\tau}$, for possibly different values of the parameter $\tau$.

\begin{Corollary}\label{cor:inhomJBdescint}
	Given a sequence $(\alpha_n)_{n\geq 1}$ of points in $\R^d$ and a sequence $(\tau_n)_{n\geq 1}$ of real numbers, let us consider
	\[
		J_{d,\ast}=\bigcap_{n=1}^\infty J^{\alpha_n}_{d,\tau_n}
		\qquad\text{and}\qquad
		\tau_\ast=\sup_{n\geq 1}\tau_n.
	\]
	Then, the following properties hold:
	\begin{enumerate}
		\item if $\tau_\ast\leq1+1/d$, the set $J_{d,\ast}$ has full Lebesgue measure in $\R^d$\,;
		\item if $\tau_\ast>1+1/d$ and is attained, $J_{d,\ast}$ is $\frakn_{(d+1)/\tau_\ast}$-describable in $\R^d$\,;
		\item if $\tau_\ast>1+1/d$ and is not attained, $J_{d,\ast}$ is $((d+1)/\tau_\ast)$-describable in $\R^d$.
	\end{enumerate}
\end{Corollary}

\begin{proof}
	The first case in an elementary consequence of Corollary~\ref{cor:inhomJBdesc}. We suppose from now on that $\tau_\ast>1+1/d$, so that the set $\calN$ of all integers $n\geq 1$ such that $\tau_n>1+1/d$ is nonempty. Note that $\tau_\ast$ is also the supremum of $\tau_n$ over $n\in\calN$. Now, Proposition~\ref{prp:monomajomino} yields on the one hand
	\[
		\majo(J_{d,\ast},\R^d)
		\supseteq\majo\left(\bigcap_{n\in\calN} J^{\alpha_n}_{d,\tau_n},\R^d\right).
	\]
	On the other hand, let us consider a gauge function $g$ that is minorizing in $\R^d$ for the intersection over $n\in\calN$ of the sets $J^{\alpha_n}_{d,\tau_n}$. Due to Corollary~\ref{cor:inhomJBdesc}, the intersection over $n\in\N\setminus\calN$ of these sets has full Lebesgue measure in $\R^d$. By Propositions~\ref{prp:lic0fullLeb} and~\ref{prp:morelicgauge}(\ref{item:prp:morelicgauge2}), any gauge function is minorizing in $\R^d$ for this set, and so is $g$ in particular. This shows with Theorem~\ref{thm:stablicgauge}(\ref{item:thm:stablicgauge1}) that $g$ is minorizing for $J_{d,\ast}$. Hence,
	\[
		\mino(J_{d,\ast},\R^d)
		\supseteq\mino\left(\bigcap_{n\in\calN} J^{\alpha_n}_{d,\tau_n},\R^d\right).
	\]
	Proposition~\ref{prp:dichofraksdesc} enables us to appropriately express the right-hand side of either of the two above inclusions in terms of either $\gauge(\frakn_{(d+1)/\tau_\ast})$ or $\gauge((d+1)/\tau_\ast)$, depending on whether or not $\tau_\ast$ is attained, respectively. We conclude using Proposition~\ref{prp:descincl}, along with Lemma~\ref{lem:gaugemeasopen}(\ref{item:lem:gaugemeasopen2}) in the first case, and Lemma~\ref{lem:gaugerealopen}(\ref{item:lem:gaugerealopen2}) in the second.
\end{proof}

Subsequently applying Corollary~\ref{cor:frakndescslip} or Corollary~\ref{cor:fraksdescslip} depending on the situation, we readily deduce from Corollary~\ref{cor:inhomJBdescint} that for any sequence $(\alpha_n)_{n\geq 1}$ of points in $\R^d$ and any sequence $(\tau_n)_{n\geq 1}$ of real numbers with supremum denoted by $\tau_\ast$,
	\[
		\Hdim \bigcap_{n=1}^\infty J^{\alpha_n}_{d,\tau_n}=\min\left\{\frac{d+1}{\tau_\ast},d\right\},
	\]
	with the usual convention that the right-hand side vanishes if $\tau_\ast$ is infinite.

\subsection{Inhomogeneous Liouville points}

Note that the mapping $\tau\mapsto J^\alpha_{d,\tau}$ is decreasing. In the spirit of the end of Section~\ref{subsubsec:fraksdesc}, this prompts us to introduce
	\[
		L^\alpha_d=\bigcap_{\tau>1+1/d}\downarrow J^\alpha_{d,\tau}.
	\]
	The monotonicity property satisfied by the sets $J^\alpha_{d,\tau}$ shows that $L^\alpha_d$ coincides for instance with the intersection over all $n\geq 1$ of the sets $J^\alpha_{d,n}$. We are in the setting of Corollary~\ref{cor:inhomJBdescint}, with the supremum being infinite. This yields the next statement.

\begin{Corollary}\label{cor:inhomLiouville0}
	For any point $\alpha$ in $\R^d$, the set $L^\alpha_d$ is $0$-describable in $\R^d$.
\end{Corollary}

The complete description of the size and large intersection properties of the set $L^\alpha_d$ then follows from Theorem~\ref{thm:fraksdescslip}. Moreover, we deduce from Corollary~\ref{cor:fraksdescslip} that this set has Hausdorff dimension equal to zero and packing dimension equal to $d$ in every nonempty open subset of $\R^d$.

Let us now establish a connection between the set $L^\alpha_d$ and a natural extension to the inhomogeneous and multidimensional setting of the notion of Liouville number.

\begin{Definition}
	Let $\alpha$ be a point in $\R^d$. A point $x$ in $\R^d$ is called {\em $\alpha$-Liouville} if $x$ does not belong to $\Q^{d,\alpha}$ and if for any integer $n\geq 1$, there exists an integer $q\geq 1$ and a point $p\in\Z^d$ such that
	\[
		\left|x-\frac{p+\alpha}{q}\right|_\infty<\frac{1}{q^n}.
	\]
\end{Definition}

For $\alpha=0$ and $d=1$, we obviously recover the condition that defines Liouville numbers. Excluding the points in $\Q^{d,\alpha}$ from this definition is analogous to excluding the irrationals from the classical definition of Liouville numbers. In fact, this ensures that for each integer $n\geq 1$, there are infinitely many pairs $(p,q)$ such that the above inequality holds. As a consequence, the set of $\alpha$-Liouville points in $\R^d$ is equal to the set $L^\alpha_d\setminus\Q^{d,\alpha}$. As shown by the next statement, removing the points in $\Q^{d,\alpha}$ does not alter the describability properties of the set $L^\alpha_d$.

\begin{Corollary}\label{cor:inhomLiouville}
	For any point $\alpha$ in $\R^d$, the set $L^\alpha_d\setminus\Q^{d,\alpha}$ of all $\alpha$-Liouville points in $\R^d$ is $0$-describable in $\R^d$.
\end{Corollary}

\begin{proof}
	The set $\R^d\setminus\Q^{d,\alpha}$ is clearly a Lebesgue full $G_\delta$-subset of $\R^d$. Owing to Propositions~\ref{prp:lic0fullLeb} and~\ref{prp:morelicgauge}(\ref{item:prp:morelicgauge2}), it thus belongs to the class $\lic{\zerofunc}{\R^d}$, and in fact to all the classes $\lic{g}{\R^d}$, for $g$ in $\gauge$. Due to Proposition~\ref{prp:cupcapmajomino} and Corollary~\ref{cor:inhomLiouville0}, we get
	\[
		\mino(L^\alpha_d\setminus\Q^{d,\alpha},\R^d)\cap\gauge^\infty
		=\mino(L^\alpha_d,\R^d)\cap\gauge^\infty=\gauge(0).
	\]
	In addition, Proposition~\ref{prp:monomajomino} and Corollary~\ref{cor:inhomLiouville0} show that
	\[
		\majo(L^\alpha_d\setminus\Q^{d,\alpha},\R^d)
		\supseteq\majo(L^\alpha_d,\R^d)=\gauge(0)^\complement.
	\]
	We conclude with the help of Proposition~\ref{prp:descincl} and Lemma~\ref{lem:gaugerealopen}(\ref{item:lem:gaugerealopen2}).
\end{proof}

Let us mention a noteworthy consequence of Corollary~\ref{cor:inhomLiouville}. Let us consider an arbitrary gauge function $g$ in $\gauge(0)$. Then, Theorem~\ref{thm:fraksdescslip} shows that the set of all $\alpha$-Liouville points in $\R^d$, namely, $L^\alpha_d\setminus\Q^{d,\alpha}$ belongs to the large intersection class $\lic{g}{\R^d}$. Now, for any given point $x$ in $\R^d$, the mapping $y\mapsto x-y$ is obviously bi-Lipschitz. We deduce from Theorem~\ref{thm:stablicgauge}(\ref{item:thm:stablicgauge1}--\ref{item:thm:stablicgauge2}) that the set
	\[
		(L^\alpha_d\setminus\Q^{d,\alpha})
		\cap(x-(L^\alpha_d\setminus\Q^{d,\alpha}))
	\]
	also belongs to $\lic{g}{\R^d}$. Hence, there are uncountably many ways of writing a given point as the sum of two $\alpha$-Liouville points. This substantially improves on a result by Erd\H{o}s~\cite{Erdos:1962uq} according to which any real number may be written as a sum of two Liouville numbers. Of course, variations are possible as one may freely replace $y\mapsto x-y$ above by any bi-Lipschitz mapping, or even a countable number thereof.

Finally, let us also point out that the set of Liouville numbers, {\em i.e.}~the set $L^0_1$ in the above notations, also comes into play in the theory of dynamical systems, especially in the study of the homeomorphisms of the circle, see~\cite{Durand:2008jk} for details.


\section{Fractional parts of sequences}\label{sec:fracpart}

We show in this section that the fractional parts of sequences yield emblematic examples of eutaxic sequences, and we detail various implications of this property in metric Diophantine approximation. Recall that $\{x\}$ stands for coordinatewise fractional part of the point $x\in\R^d$, and belongs to $[0,1)^d$.

\subsection{Sequencewise eutaxy}

We begin with the sequencewise version of eutaxy. As defined in Section~\ref{subsec:eutaxyseq}, this notion arises when choosing a sequence $(r_n)_{n\geq 1}$ in $\Rho_d$, and requiring a sequence of points under consideration to approximate within distances $r_n$ Lebesgue-almost every point of a given open set.

\subsubsection{Linear sequences}\label{subsubsec:fracpartlinseqeutaxy}

By this term, we mean sequences of the form $(\{nx\})_{n\geq 1}$ with $x$ in $\R^d$. Our main result is the following. Although intrinsic proofs are available, it is particularly easy to establish this result with the help of Theorem~\ref{thm:inhomdesc}.

\begin{Theorem}\label{thm:nxeutaxicd}
Let $(r_n)_{n\geq 1}$ be a sequence in $\Rho_d$. Then, for $\leb^d$-almost every point $x\in\R^d$, the sequence $(\{nx\})_{n\geq 1}$ is eutaxic in $(0,1)^d$ with respect to $(r_n)_{n\geq 1}$.
\end{Theorem}

\begin{proof}	
	The sequence $(r_n/n)_{n\geq 1}$ is both positive and nonincreasing, so we may find a positive nonincreasing continuous function $\psi$ defined on $[0,\infty)$ that coincides on $\N$ with this sequence. Hence, the integral $I_{d,\psi}$ on which relies Theorem~\ref{thm:inhomdesc} satisfies
	\[
		I_{d,\psi}=\int_0^\infty q^d\psi(q)^d\,\dd q
		\geq\sum_{n=1}^\infty (n-1)^d\psi(n)^d
		=2^{-d}\sum_{n=2}^\infty r_n^d=\infty.
	\]
	It follows that for any $\alpha$ in $(0,1)^d$, the set $\frakQ^\alpha_{d,\psi}$ defined by~(\ref{eq:df:Qdpsialpha}) has full Lebesgue measure in $\R^d$. As a result, $\leb^d$-almost every point $x$ in $\R^d$ satisfies
	\[
		|nx-(p_n+\alpha)|_\infty<n\psi(n)=r_n
	\]
	with some integer point $p_n$, for infinitely many integers $n\geq 1$. For convenience, we work with the supremum norm; this does not alter the notion of eutaxy, see Section~\ref{subsec:eutaxyseq}. Letting $h=(1/2,\ldots,1/2)$, we have
	\[
		|\lfloor nx\rfloor-p_n|_\infty
		\leq|nx-(p_n+\alpha)|_\infty+|\{nx\}-h|_\infty+|\alpha-h|_\infty
		<r_n+\frac{1}{2}+|\alpha-h|_\infty
	\]
	The right-hand side is smaller than one for $n$ sufficiently large, because the sequence $(r_n)_{n\geq 1}$ converges to zero. The point $p_n$ is then necessarily equal to $\lfloor nx\rfloor$. We deduce that for all $\alpha\in(0,1)^d$ and for $\leb^d$-almost all $x\in\R^d$, the inequality
	\[
		|\alpha-\{nx\}|_\infty<r_n
	\]
	holds for infinitely many $n\geq 1$. This holds {\em a fortiori} for $\leb^d$-almost every $\alpha$. To conclude, we exchange the order of $\alpha$ and $x$ with the help of Tonelli's theorem.
\end{proof}

We may now apply Theorem~\ref{thm:approxeutaxic} to the example supplied by Theorem~\ref{thm:nxeutaxicd}. Here, the formula~(\ref{eq:df:Ftn}) for the sets $F_t$ becomes
	\[
		F_t(x)=\left\{y\in\R^d\bigm||y-\{nx\}|<r_n^t\quad\text{for i.m.~}n\geq 1\right\},
	\]
	where $x$ is chosen according to the Lebesgue measure. Due to the aforementioned results, we then know that for any sequence $(r_n)_{n\geq 1}$ in $\Rho_d$ such that $\sum_n r_n^s$ converges for all $s>d$, and for Lebesgue-almost every point $x\in\R^d$, we have both
	\begin{equation}\label{eq:silipFtnx}
		\Hdim(F_t(x)\cap U)=\frac{d}{t}
		\qquad\text{and}\qquad
		F_t(x)\in\lic{d/t}{U}
	\end{equation}
	for any real number $t>1$ and for any nonempty open subset $U$ of $(0,1)^d$. In metric Diophantine approximation, it is customary to recast such a result with the help of the distance to the nearest integer point defined for every $z$ in $\R^d$ by
	\begin{equation}\label{eq:df:distZd}
		\distZ{z}=\inf_{p\in\Z^d}|z-p|_\infty.
	\end{equation}
	This amounts to considering, instead of $F_t(x)$, the companion set
	\[
		F'_t(x)=\left\{y\in\R^d\bigm|\distZ{y-nx}<r_n^t\quad\text{for i.m.~}n\geq 1\right\}.
	\]
	We may now easily deduce the next result from~(\ref{eq:silipFtnx}).

\begin{Corollary}\label{cor:dimfracpart}
Let $(r_n)_{n\geq 1}$ be a sequence in $\Rho_d$ such that $\sum_n r_n^s$ converges for all $s>d$. Then, for Lebesgue-almost every $x\in\R^d$ and for any $t>1$,
	\[
		\Hdim F'_t(x)=\frac{d}{t}.
	\]
\end{Corollary}

\begin{proof}
	For all $x\in\R^d$ and $t>1$, the set $F'_t(x)$ contains the set $F_t(x)\cap(0,1)^d$, so the lower bound on the dimension follows from~(\ref{eq:silipFtnx}). For the upper bound, it suffices to combine Lemma~\ref{lem:upbndlimsup} with the observation that
		\[
			F'_t(x)\cap[0,1)^d\subseteq\limsup_{n\to\infty}\bigcup_{p\in\{-1,0,1\}^d}\opball_\infty(\{nx\}+p,r_n^t)
		\]
		and that the sets $F'_t(x)$ are invariant under the translations by vectors in $\Z^d$.
\end{proof}

An emblematic particular case is obtained by letting the sequence of approximating radii be given by $r_n=n^{-1/d}$. This sequence clearly satisfies the assumptions of Corollary~\ref{cor:dimfracpart} and, up to a simple change of parameter, we deduce that for Lebesgue-almost every point $x\in\R^d$ and for every real number $\sigma>1/d$,
	\begin{equation}\label{eq:dimfracpartsigma}
		\Hdim\left\{y\in\R^d\Biggm|\distZ{y-nx}<\frac{1}{n^\sigma}\quad\text{for i.m.~}n\geq 1\right\}=\frac{1}{\sigma}.
	\end{equation}
	In the one-dimensional setting, this result is well known, and even holds when $x$ is an arbitrary irrational real number, see~\cite{Bugeaud:2003ye}.

\subsubsection{Other sequences}

Theorem~\ref{thm:nxeutaxicd} may be extended to the case in which the underlying sequence is driven by a nonconstant polynomial with integer coefficients. In fact, Schmidt~\cite{Schmidt:1964vn} established the following result.

\begin{Theorem}\label{thm:eutaxicpoly}
	Let $P$ be a nonconstant polynomial with coefficients in $\Z$ and let $(r_n)_{n\geq 1}$ be a sequence in $\Rho_d$. Then, for Lebesgue-almost every point $x\in\R^d$, the sequence $(\{P(n)x\})_{n\geq 1}$ is eutaxic in $(0,1)^d$ with respect to $(r_n)_{n\geq 1}$.
\end{Theorem}

Subsequently, Philipp~\cite{Philipp:1967eu} showed that, in dimension one, the above property still holds when the polynomial is replaced by the exponential function to a given integer base $b\geq 2$\,; this is related with the base $b$ expansion of real numbers.

\begin{Theorem}\label{thm:eutaxicpower}
	Let us consider an integer $b\geq 2$ and a sequence $(r_n)_{n\geq 1}$ in $\Rho_d$. Then, for Lebesgue-almost every point $x\in\R$, the sequence $(\{b^n x\})_{n\geq 1}$ is eutaxic in $(0,1)$ with respect to $(r_n)_{n\geq 1}$.
\end{Theorem}

Philipp showed that this property also holds for $x$ in a Lebesgue full subset of the interval $[0,1)$ when the multiplication by $b^n$ is replaced by the $n$-th iterate of either of the following mappings: the Gauss map $x\mapsto\{1/x\}$ for continued fractions; the $\theta$-adic expansion map $x\mapsto\{\theta x\}$, where $\theta>1$. We refer to~\cite{Philipp:1967eu} for precise statements. In all those cases, we may reproduce the approach developed in Section~\ref{subsubsec:fracpartlinseqeutaxy} so as to obtain dimensional results analogous to Corollary~\ref{cor:dimfracpart}.

\subsection{Uniform eutaxy}

This section is the counterpart of the previous one when the eutaxy property is supposed to be uniform in the sense of Section~\ref{subsec:eutaxyunif}. In that spirit, as regards linear sequences, the analog of Theorem~\ref{thm:nxeutaxicd} is a strong result due to Kurzweil; this is the main result that we establish below, see Theorem~\ref{thm:nxunifeutaxicd}. Of course, uniform eutaxy being more restrictive than sequencewise eutaxy, the resulting metrical implications are much stronger. As shown hereunder, the associated limsup sets indeed fall into the category of fully describable sets.

\subsubsection{Preliminary results}

Our approach depends on uniformly distributed sequences, so we begin by recalling some basic definitions and results on that topic. A sequence $(x_n)_{n\geq 1}$ of points in $\R^d$ is {\em uniformly distributed modulo one} if for any points $(a_1,\ldots,a_d)$ and $(b_1,\ldots,b_d)$ in $[0,1)^d$ such that $a_i\leq b_i$ for all $i\in\{1,\ldots,d\}$,
	\[
		\lim_{N\to\infty}\frac{1}{N}\#\left\{n\in\{1,\ldots,N\}\Biggm|\{x_n\}\in\prod_{i=1}^d [a_i,b_i)\right\}
		=\prod_{i=1}^d (b_i-a_i).
	\]
	When trying to prove that a sequence is uniformly distributed modulo one, a convenient tool is a criterion due to Weyl, see~{\em e.g.}~Theorems~1.4 and~1.19 in~\cite{Drmota:1997kx}. Applying this criterion to linear sequences, one obtains the following statement.

\begin{Theorem}\label{thm:nxudmod1}
	Let us consider a point $x=(x_1,\ldots,x_d)$ in $\R^d$. Then, the sequence $(nx)_{n\geq 1}$ is uniformly distributed modulo one if and only if the real numbers $1,x_1,\ldots,x_d$ are linearly independent over $\Q$.
\end{Theorem}

If a sequence $(x_n)_{n\geq 1}$ is uniformly distributed modulo one, then the sequence $(\{x_n\})_{n\geq 1}$ is clearly dense in $[0,1)^d$. Therefore, the above theorem enables us to recover a classical result due to Kronecker concerning the density of the sequence $(\{nx\})_{n\geq 1}$. Theorem~\ref{thm:nxudmod1} is thus a measure theoretic analog of the next result.

\begin{Theorem}[Kronecker]\label{thm:Kronecker}
	Let us consider a point $x=(x_1,\ldots,x_d)$ in $\R^d$. Then, the sequence $(\{nx\})_{n\geq 1}$ is dense in the unit cube $[0,1)^d$ if and only if the real numbers $1,x_1,\ldots,x_d$ are linearly independent over $\Q$.
\end{Theorem}

The badly approximable points will play a particularly important r\^ole in our study, so it is worth pointing out now a simple connection with linear independence over the rationals. In accordance with Section~\ref{subsec:bad} where it is defined, the set of badly approximable points is denoted by $\bad_d$ in what follows.

\begin{Lemma}\label{lem:badlinindep}
	Let us consider a point $x=(x_1,\ldots,x_d)$ in $\bad_d$. Then, the real numbers $1,x_1,\ldots,x_d$ are linearly independent over $\Q$.
\end{Lemma}

Combining this result with Theorems~\ref{thm:nxudmod1} and~\ref{thm:Kronecker}, we directly deduce that when $x$ is a badly approximable point, the sequence $(nx)_{n\geq 1}$ is uniformly distributed modulo one, and the sequence $(\{nx\})_{n\geq 1}$ is dense in the unit cube $[0,1)^d$. We shall establish hereafter that the latter sequence is in fact uniformly eutaxic in the open cube $(0,1)^d$\,: this is Kurzweil's theorem, namely, Theorem~\ref{thm:nxunifeutaxicd}.

The proof of Lemma~\ref{lem:badlinindep} makes use of several notations that we now introduce. The distance to the nearest integer point given by~(\ref{eq:df:distZd}) enables us to define
\begin{equation}\label{eq:df:kappad}
\kappa(x)=\liminf_{q\to\infty} q^{1/d}\distZ{qx}
\end{equation}
for every $x$ in $\R^d$. If the point $x$ has rational coordinates, then $\kappa(x)$ clearly vanishes. Otherwise, we may use the corollary to Dirichlet's theorem, that is, Corollary~\ref{cor:Dirichlet} to prove that $\kappa(x)$ is bounded above by one. Finally, the exponent $\kappa$ characterizes the badly approximable points, namely,
\begin{equation}\label{eq:caracbadkappad}
x\in\bad_d \qquad\Longleftrightarrow\qquad \kappa(x)>0.
\end{equation}
Now that these notations are set, we may detail the proof of the lemma.

\begin{proof}[Proof of Lemma~\ref{lem:badlinindep}]
	We argue by contradiction. Let us assume the existence of integers $r,s_1,\ldots,s_d$ that do not vanish simultaneously and satisfy
	\[
		s_1 x_1+\ldots+s_d x_d=r.
	\]
	Up to rearranging the coordinates of $x$ and multiplying the above equation by minus one, we may assume that $s_d\geq 1$. Now, given $q$ in $\N$ and $p=(p_1,\ldots,p_{d-1})$ in $\Z^{d-1}$, we define $q'=s_d q$, as well as $p'_i=s_d p_i$ for $i\in\{1,\ldots,d-1\}$ and
	\[
		p'_d=rq-s_1p_1-\ldots-s_{d-1}p_{d-1}.
	\]
	If the index $i$ is different from $d$, it is clear that $q'x_i-p'_i$ is equal to $s_d(q x_i-p_i)$. Moreover, concerning the $d$-th coordinate, we have
	\[
		q'x_d-p'_d=s_1(p_1-qx_1)+\ldots+s_{d-1}(p_{d-1}-qx_{d-1}).
	\]
	Letting $|\,\cdot\,|_1$ stand as usual for the taxicab norm and letting $s$ denote the $d$-tuple $(s_1,\ldots,s_d)$, we infer that
	\[
		\max_{i\in\{1,\ldots,d\}}|q'x_i-p'_i|_\infty
		\leq |s|_1\max_{i\in\{1,\ldots,d-1\}}|q x_i-p_i|_\infty.
	\]
	Taking the infimum over all $(d-1)$-tuples $p$, we deduce that $\distZ{s_d q x}$ is bounded above by $|s|_1$ times $\distZ{q(x_1,\ldots,x_{d-1})}$, from which it follows that
	\[
		(s_d q)^{1/d}\distZ{s_d q x}
		\leq\frac{|s|_1 s_d^{1/d}}{q^{1/(d(d-1))}}\left(q^{1/(d-1)}\distZ{q(x_1,\ldots,x_{d-1})}\right).
	\]
	Since $\kappa(x_1,\ldots,x_{d-1})$ is bounded above by one, there is an infinite set of integers $q$ on which the term in parentheses in the above right-hand side is bounded. As this term is then divided by $q^{1/(d(d-1))}$, the latter upper bound implies that $\kappa(x)$ vanishes, thereby contradicting the fact that $x$ is badly approximable.
\end{proof}

We now establish two preliminary results on fractional parts of the form $\{a_n x\}$, where $a_n$ is the general term of an increasing sequence of positive integers. In the linear case, these results will straightforwardly lead to Kurzweil's theorem. Such a sequence $(a_n)_{n\geq 1}$ being given, we define its lower asymptotic density by
	\[
		\lad((a_n)_{n\geq 1})=\liminf_{N\to\infty}\frac{1}{N}\#\{n\geq 1\:|\:a_n\leq N\}.
	\]
	Moreover, we shall also use the exponent $\kappa$ defined by~(\ref{eq:df:kappad}), and accordingly work with the supremum norm, which does not alter uniform eutaxy, as mentioned in Section~\ref{subsec:eutaxyunif}. We then have the following result, established by Reversat~\cite{Reversat:1976yq}.

\begin{Proposition}\label{prp:nonbadnoneutaxic}
	Let us consider an increasing sequence $(a_n)_{n\geq 1}$ of positive integers with positive lower asymptotic density, and a point $x=(x_1,\ldots,x_d)$ in $\R^d$ such that the real numbers $1,x_1,\ldots,x_d$ are linearly independent over $\Q$. Then, for any nonempty dyadic subcube $\lambda$ of $[0,1)^d$,
	\[
		\liminf_{j\to\infty}2^{-dj}\#\Mu((\{a_n x\})_{n\geq 1};\lambda,j)
		\leq 480^d\left(\frac{\kappa(x)}{\lad((a_n)_{n\geq 1})}\right)^{d/(d+1)}.
	\]
\end{Proposition}

\begin{proof}
	If $\delta$ is positive and smaller than $\lad((a_n)_{n\geq 1})$, we have $a_n\leq n/\delta$ for any sufficiently large integer $n$. Moreover, given $\kappa>\kappa(x)$, we know that there exists an infinite set $\calQ\subseteq\N$ such that $\distZ{q x}\leq\kappa/q^{1/d}$ for all $q\in\calQ$. We now fix a nonempty dyadic cube $\lambda$ contained in $[0,1)^d$, an integer $q\in\calQ$ and an integer $j\geq 0$ satisfying
	\begin{equation}\label{eq:condqj}
	c^{d/(d+1)}\,2^{d(\gene{\lambda}+j)}\leq q\leq c^{d/(d+1)}\,2^{d(\gene{\lambda}+j+1)},
	\end{equation}
	where $c$ is a positive parameter that will be tuned up at the end of the proof.
	
	Let us consider an integer $m\leq 2^{d(\gene{\lambda}+j)}$ such that $\{a_m x\}\in\lambda$. We decompose the integer $a_m$ in the form $hq+r$ with $h\in\N_0$ and $r\in\{1,\ldots,q\}$. If the integer $q$ is sufficiently large, the integer $j$ is large as well and we may assume that
	\[
	hq\leq a_m\leq a_{2^{d(\gene{\lambda}+j)}}\leq\frac{2^{d(\gene{\lambda}+j)}}{\delta}
	\qquad\text{and}\qquad
	2^{j-1}\geq\frac{\kappa}{\delta c}.
	\]
	As a consequence,
	\[
	\distZ{rx-a_m x}=\distZ{hqx}\leq h\distZ{qx}\leq\kappa\frac{hq}{q^{1+1/d}}
	\leq\frac{\kappa}{\delta c}2^{-(\gene{\lambda}+j)}\leq 2^{-(\gene{\lambda}+1)}.
	\]
	Letting $y_\lambda$ denote the center of the cube $\lambda$, we deduce that for some point $p$ in $\Z^d$,
	\[
	|\{r x\}-p-y_\lambda|_\infty\leq|\{rx\}-\{a_m x\}-p|_\infty+|\{a_m x\}-y_\lambda|_\infty
	\leq 2^{-\gene{\lambda}}.
	\]
	We conclude that $\{rx\}$ belongs to $U(\lambda)$, the set of points $y$ in $[0,1)^d$ that are within distance $2^{-\gene{\lambda}}$ from $y_\lambda+\Z^d$. Therefore, the integer $r$ is positive, bounded above by $c^{d/(d+1)}\,2^{d(\gene{\lambda}+j+1)}$, and verifies $\{rx\}\in U(\lambda)$\,; we define $R(\lambda,j)$ as the set of all integers that satisfy these three properties.
	
	Furthermore, let $\lambda'$ be the dyadic subcube of $\lambda$ with generation $\gene{\lambda}+j$ that contains the point $\{a_m x\}$. We consider another integer $m'\leq 2^{d(\gene{\lambda}+j)}$ such that $a_{m'}$ may be written in the form $h'q+r$ for some nonnegative integer $h'$. We have
	\[
	\distZ{a_m x-a_{m'}x}=\distZ{(h-h')qx}\leq|h-h'|\distZ{qx}
	\leq\kappa\frac{\max\{hq,h'q\}}{q^{1+1/d}}
	\leq\frac{\kappa}{\delta c}2^{-(\gene{\lambda}+j)}.
	\]
	Thus, letting $y_{\lambda'}$ denote the center of the subcube $\lambda'$, we observe that there exists a point $p$ in $\Z^d$ such that
	\begin{align*}
	|\{a_{m'}x\}-p-y_{\lambda'}|_\infty
	&\leq |\{a_{m'}x\}-\{a_m x\}-p|_\infty+|\{a_m x\}-y_{\lambda'}|_\infty\\
	&\leq\left(\frac{\kappa}{\delta c}+\frac{1}{2}\right)2^{-(\gene{\lambda}+j)}.
	\end{align*}
	This means that $\{a_{m'}x\}$ belongs to a closed ball centered at $p+y_{\lambda'}$ with radius the right-hand side above, that is denoted by $\rho$. Note that the number of dyadic cubes with generation $\gene{\lambda}+j$ that are required to cover this ball is bounded above by $((2\rho)2^{\gene{\lambda}+j}+2)^d$. In addition, it is easily seen that there are at most $5^d$ possible values for $p$, because the points $\{a_{m'}x\}$ and $y_{\lambda'}$ both belong to the unit cube. We conclude that the number of dyadic  subcubes of $\lambda$ with generation $\gene{\lambda}+j$ that may contain $\{a_{m'}x\}$ is bounded above by
	\[
	5^d((2\rho)2^{\gene{\lambda}+j}+2)^d=10^d\left(\frac{3}{2}+\frac{\kappa}{\delta c}\right)^d.
	\]
	
	The upshot is that for every choice of $r$, the above value gives an upper bound on the number of dyadic subcubes of $\lambda$ with generation $\gene{\lambda}+j$ that contain at least one point of the form $\{a_m x\}$, where $m\leq 2^{d(\gene{\lambda}+j)}$ and $a_m=hq+r$ for some nonnegative integer $h$. Recalling that $r$ necessarily belongs to the set $R(\lambda,j)$ when such an integer $a_m$ exists, we deduce that
	\[
	\#\Mu(\lambda,j)\leq 10^d\left(\frac{3}{2}+\frac{\kappa}{\delta c}\right)^d\#R(\lambda,j).
	\]
	This inequality is valid for infinitely many values of $j$, namely, for every integer $j$ satisfying~(\ref{eq:condqj}) for some $q\in\calQ$. It follows that
	\begin{equation}\label{eq:majMulambdajamx}
	\liminf_{j\to\infty} 2^{-dj}\#\Mu(\lambda,j)\leq 10^d\left(\frac{3}{2}+\frac{\kappa}{\delta c}\right)^d\limsup_{j\to\infty} 2^{-dj}\#R(\lambda,j).
	\end{equation}
	
	Given that the real numbers $1,x_1,\ldots,x_d$ are linearly independent over $\Q$, we may conclude with the help of Theorem~\ref{thm:nxudmod1}. Accordingly, the sequence $(rx)_{r\geq 1}$ is uniformly distributed modulo one, so that
	\[
	\#R(\lambda,j)\sim\lfloor c^{d/(d+1)}\,2^{d(\gene{\lambda}+j+1)}\rfloor\leb^d(U(\lambda))
	\qquad\text{as}\qquad j\to\infty.
	\]
	One easily check that the set $U(\lambda)$ has Lebesgue measure at most $6^d 2^{-d\gene{\lambda}}$. Hence, the limsup in~(\ref{eq:majMulambdajamx}) is bounded above by $12^d c^{d/(d+1)}$. We deduce that
	\[
	\liminf_{j\to\infty} 2^{-dj}\#\Mu(\lambda,j)\leq 120^d c^{d/(d+1)}\left(\frac{3}{2}+\frac{\kappa}{\delta c}\right)^d.
	\]
	We conclude by choosing $c=2\kappa/\delta$, and then by letting $\delta$ and $\kappa$ go to $\lad((a_n)_{n\geq 1})$ and $\kappa(x)$, respectively.
\end{proof}

The next result is a converse to Proposition~\ref{prp:nonbadnoneutaxic}. While the latter result involves the exponent $\kappa$ defined by~(\ref{eq:df:kappad}), we consider here the exponent $\kappa_\ast$ defined by
	\[
		\kappa_\ast(x)=\inf_{q\in\N} q^{1/d}\distZ{qx}
	\]
	for all $x$ in $\R^d$. Clearly, $\kappa_\ast(x)$ is bounded above by $\kappa(x)$. Moreover, $\kappa(x)$ and $\kappa_\ast(x)$ are both positive exactly on the set of badly approximable points. Indeed, similarly to~(\ref{eq:caracbadkappad}), we have
	\begin{equation}\label{eq:caracbadkappaastd}
		x\in\bad_d \qquad\Longleftrightarrow\qquad \kappa_\ast(x)>0.
	\end{equation}
	In connection with distributions modulo one, the statement below also calls upon the limiting ratios defined by
	\begin{equation}\label{eq:df:limratmodone}
		\underline{\rho}((x_n)_{n\geq 1};\lambda)
		=\liminf_{N\to\infty}\frac{1}{N}\#\{n\in\{1,\ldots,N\}\:|\:\{x_n\}\in\lambda\}
	\end{equation}
	when $(x_n)_{n\geq 1}$ is a sequence in $\R^d$ and $\lambda$ is a nonempty dyadic subcube of $[0,1)^d$. Note that each of these limiting ratios is equal to $\leb^d(\lambda)$ when $(x_n)_{n\geq 1}$ is uniformly distributed modulo one. The following result is also due to Reversat~\cite{Reversat:1976yq}.

\begin{Proposition}\label{prp:badeutaxic}
Let $(a_n)_{n\geq 1}$ be an increasing sequence of positive integers and let $x$ be a point in $\R^d$. Then, for any nonempty dyadic subcube $\lambda$ of $[0,1)^d$,
\[
\liminf_{j\to\infty}2^{-dj}\#\Mu((\{a_n x\})_{n\geq 1};\lambda,j)
\geq\frac{\kappa_\ast(x)^d\lad((a_n)_{n\geq 1})}{2^d\leb^d(\lambda)}\,\underline{\rho}((a_n x)_{n\geq 1};\lambda).
\]
\end{Proposition}

\begin{proof}
We may obviously assume that $\kappa_\ast(x)$ and $\lad((a_n)_{n\geq 1})$ are both positive. If $\kappa$ is a positive real number smaller than $\kappa_\ast(x)$, it is clear that $\distZ{q x}>\kappa/q^{1/d}$ for all integers $q\geq 1$. Furthermore, if $\delta$ denotes a positive real number smaller than $\lad((a_n)_{n\geq 1})$, we know that the inequality $a_n\leq n/\delta$ holds for $n$ large enough. We now consider a nonempty dyadic subcube $\lambda$ of $[0,1)^d$, an integer $j\geq 0$, and a dyadic cube $\lambda'$ in the collection $\Mu(\lambda,j)$. In particular, the cube $\lambda'$ contains a point of the form $\{a_m x\}$ for some integer $m\leq 2^{d(\gene{\lambda}+j)}$. If $m'$ denotes another integer bounded above by $2^{d(\gene{\lambda}+j)}$ and for which $\{a_{m'} x\}$ belongs to $\lambda'$ as well, then
\[
|\{a_m x\}-\{a_{m'}x\}|_\infty\geq\distZ{(a_m-a_{m'})x}
>\frac{\kappa}{|a_m-a_{m'}|^{1/d}}\geq\frac{\kappa\,\delta^{1/d}}{2^{\gene{\lambda}+j}}.
\]
The last bound holds for $j$ sufficiently large, because the positive integers $a_m$ and $a_{m'}$ are then both bounded above by $2^{d(\gene{\lambda}+j)}/\delta$. We may naturally decompose the cube $\lambda'$ as the disjoint union of $\lceil 1/(\kappa\,\delta^{1/d})\rceil^d$ half-open subcubes with sidelength equal to $2^{-(\gene{\lambda}+j)}/\lceil 1/(\kappa\,\delta^{1/d})\rceil$. Moreover, if we consider any of these subcubes, the above inequalities imply that at most one integer $m\leq 2^{d(\gene{\lambda}+j)}$ can be such that the point $\{a_m x\}$ lies in the cube. So, there can be no more than $\lceil 1/(\kappa\,\delta^{1/d})\rceil^d$ integers $m\leq 2^{d(\gene{\lambda}+j)}$ for which $\{a_m x\}$ is in $\lambda'$. As a consequence,
\[
\#\{m\leq 2^{d(\gene{\lambda}+j)}\:|\:\{a_m x\}\in\lambda\}
\leq\left\lceil\frac{1}{\kappa\,\delta^{1/d}}\right\rceil^d\#\Mu(\lambda,j),
\]
from which we readily deduce that
\[
2^{-dj}\#\Mu(\lambda,j)
\geq \frac{\kappa^d\delta}{2^d\leb^d(\lambda)}\,2^{-d(\gene{\lambda}+j)}\#\{m\leq 2^{d(\gene{\lambda}+j)}\:|\:\{a_m x\}\in\lambda\}.
\]
The result follows in a straightforward manner by letting $j$ tend to infinity, and then by letting $\kappa$ and $\delta$ go to $\kappa_\ast(x)$ and $\lad((a_n)_{n\geq 1})$, respectively.
\end{proof}

\subsubsection{Linear sequences: Kurzweil's theorem}\label{subsubsec:Kurzweil}

Regarding the uniform eutaxy of the sequences $(\{nx\})_{n\geq 1}$, the main result is Theorem~\ref{thm:nxunifeutaxicd} below, which was first obtained by Kurzweil~\cite{Kurzweil:1955la} and subsequently recovered by Lesca~\cite{Lesca:1968rm}. For the sake of completeness, let us mention in addition that Kurzweil also obtained in~\cite{Kurzweil:1955la} an extension of Theorem~\ref{thm:nxunifeutaxicd} that deals with linear forms. In order to let the reader compare the next result with Theorem~\ref{thm:nxeutaxicd}, it is worth mentioning some metric properties of the set $\bad_d$ of badly approximable points defined in Section~\ref{subsec:bad}. Specifically, Proposition~\ref{prp:badlebesgued} therein shows that $\bad_d$ has Lebesgue measure zero, and Schmidt~\cite{Schmidt:1969fk} proved that this set has Hausdorff dimension $d$.

\begin{Theorem}[Kurzweil]\label{thm:nxunifeutaxicd}
	For any point $x$ in $\R^d$, the sequence $(\{nx\})_{n\geq 1}$ is uniformly eutaxic in $(0,1)^d$ if and only if $x$ is badly approximable.
\end{Theorem}

\begin{proof}
	The idea is to apply Propositions~\ref{prp:nonbadnoneutaxic} and~\ref{prp:badeutaxic} to the sequence $(n)_{n\geq 1}$, which is increasing and has lower asymptotic density equal to one. Let us first assume that the point $x$ is not badly approximable, and let $x_1,\ldots,x_d$ denote its coordinates. If the real numbers $1,x_1,\ldots,x_d$ are linearly dependent over the rationals, it follows from Kronecker's theorem, namely, Theorem~\ref{thm:Kronecker} that the sequence $(\{nx\})_{n\geq 1}$ is not dense in $[0,1)^d$. This sequence is thus clearly not eutaxic in $(0,1)^d$. Now, if the above real numbers are linearly independent over $\Q$, we may apply Proposition~\ref{prp:nonbadnoneutaxic}, thereby inferring that for any $x\in\R^d$ and any nonempty dyadic cube $\lambda\subseteq[0,1)^d$,
	\[
		\liminf_{j\to\infty}2^{-dj}\#\Mu((\{nx\})_{n\geq 1};\lambda,j)
		\leq 480^d\kappa(x)^{d/(d+1)}.
	\]
	Since $x$ is not badly approximable, the exponent $\kappa(x)$ vanishes by virtue of~(\ref{eq:caracbadkappad}). The left-hand side above thus vanishes as well, and Theorem~\ref{thm:CNeutaxy} ensures that the sequence $(\{nx\})_{n\geq 1}$ is not uniformly eutaxic in $(0,1)^d$.
	
	Conversely, let us assume that $x$ is badly approximable. Lemma~\ref{lem:badlinindep} ensures that the real numbers $1,x_1,\ldots,x_d$ are linearly independent over $\Q$. We then deduce from Theorem~\ref{thm:nxudmod1} that the sequence $(nx)_{n\geq 1}$ is uniformly distributed modulo one, so that for any nonempty dyadic subcube $\lambda$ of $[0,1)^d$, the limiting ratio $\underline{\rho}((n x)_{n\geq 1};\lambda)$ defined by~(\ref{eq:df:limratmodone}) is equal to $\leb^d(\lambda)$. Applying Proposition~\ref{prp:badeutaxic}, we thus infer that
	\[
		\liminf_{j\to\infty}2^{-dj}\#\Mu((\{nx\})_{n\geq 1};\lambda,j)
		\geq 2^{-d}\kappa_\ast(x)^d.
	\]
	Finally, in view of~(\ref{eq:caracbadkappaastd}), the exponent $\kappa_\ast(x)$ is positive, and we conclude with the help of Theorem~\ref{thm:CSeutaxy} that the sequence $(\{nx\})_{n\geq 1}$ is uniformly eutaxic in $(0,1)^d$.
\end{proof}

In the vein of Corollary~\ref{cor:dimfracpart} and the discussion that precedes its statement, an interesting application is the study of the Diophantine approximation properties of the sequence $(\{nx\})_{n\geq 1}$ when $x$ is a badly approximable point. That sequence being uniformly eutaxic, we end up with much stronger results, specifically, a complete description of the size and large intersection properties of the set
	\begin{equation}\label{eq:df:Frxlin}
		F_\rmr(x)=\left\{y\in\R^d\bigm||y-\{nx\}|<r_n\quad\text{for i.m.~}n\geq 1\right\},
	\end{equation}
	where $\rmr=(r_n)_{n\geq 1}$ is a nonincreasing sequence of positive real numbers. In fact, this set is clearly of the form~(\ref{eq:df:Fxnrn}), so Theorems~\ref{thm:desceutaxy} and~\ref{thm:nxunifeutaxicd} directly entail the next statement. Here, $\frakn_\rmr$ denotes as above the measure characterized by~(\ref{eq:df:fraknrmr}).

\begin{Theorem}\label{thm:fracpartlindesc}
	For any point $x$ in $\bad_d$ and for any nonincreasing sequence $\rmr=(r_n)_{n\geq 1}$ of positive real numbers, the following properties hold:
	\begin{enumerate}
		\item if $\sum_n r_n^d$ diverges, then $F_\rmr(x)$ has full Lebesgue measure in $(0,1)^d$\,;
		\item if $\sum_n r_n^d$ converges, then $F_\rmr(x)$ is $\frakn_\rmr$-describable in $(0,1)^d$.
	\end{enumerate}
\end{Theorem}

We may recast this result with the help of the distance to the nearest integer point defined by~(\ref{eq:df:distZd}), thus considering instead of $F_\rmr(x)$ the companion set
	\[
		F'_\rmr(x)=\left\{y\in\R^d\bigm|\distZ{y-nx}<r_n\quad\text{for i.m.~}n\geq 1\right\}.
	\]
	The resulting statement bearing on this set is the following one. The describability property is now valid on the whole space $\R^d$ instead of the mere open unit cube $(0,1)^d$\,; this is because the companion set $F'_\rmr(x)$ may basically be seen as the initial set $F_\rmr(x)$, along with its images under all translations by vectors in $\Z^d$.

\begin{Corollary}\label{cor:fracpartlindesc}
	For any point $x$ in $\bad_d$ and for any nonincreasing sequence $\rmr=(r_n)_{n\geq 1}$ of positive real numbers, the following properties hold:
	\begin{enumerate}
		\item if $\sum_n r_n^d$ diverges, then $F'_\rmr(x)$ has full Lebesgue measure in $\R^d$\,;
		\item if $\sum_n r_n^d$ converges, then $F'_\rmr(x)$ is $\frakn_\rmr$-describable in $\R^d$.
	\end{enumerate}
\end{Corollary}

\begin{proof}
	The divergence case results from Theorem~\ref{thm:fracpartlindesc} and the observation that $F'_\rmr(x)$ contains the images of $F_\rmr(x)$ under all translations by vectors in $\Z^d$, along with the subadditivity of Lebesgue measure and its translation invariance.

	Placing ourselves in the convergence case, let us consider a gauge function $g$ in $\gauge(\frakn_\rmr)$, a $d$-normalized gauge function $h$ satisfying $h\prec g_d$, and a nonempty dyadic cube $\lambda$ in the collection $\Lambda_h$ introduced in Section~\ref{subsubsec:netmrev}. We also assume that $\lambda$ has diameter at most that of the unit cube $[0,1)^d$, which is equal to one because we work with the supremum norm when considering the distance to the nearest integer point. Thus, $\lambda$ is included in the dyadic cube $k+[0,1)^d$ for some $k\in\Z^d$. Given that $F'_\rmr(x)$ contains the image of $F_\rmr(x)$ under the translation by vector $k$ and that~(\ref{eq:netmlip}) remains valid for such translations, along with the net measures associated with general gauge functions, we get
		\[
			\netm^h_\infty(F'_\rmr(x)\cap\lambda)
			\geq\netm^h_\infty(k+(F_\rmr(x)\cap(-k+\lambda)))
			\geq 3^{-d}\netm^h_\infty(F_\rmr(x)\cap(-k+\lambda)).
		\]
		In addition, the interior of $-k+\lambda$ is contained in the open unit cube $(0,1)^d$, and Theorem~\ref{thm:fracpartlindesc} implies that $F_\rmr(x)$ satisfies a large intersection property with respect to $g$ in the latter open cube. Hence,
		\[
			\netm^h_\infty(F_\rmr(x)\cap\interior{(-k+\lambda)})=\netm^h_\infty(\interior{(-k+\lambda)})=\netm^h_\infty(\lambda),
		\]
		where the last equality is due to~(\ref{eq:netmlambdagauge}). We deduce that the set $F'_\rmr(x)$ belongs to the class $\lic{g}{\R^d}$ by making use of Lemmas~10 and~12 in~\cite{Durand:2007uq}, namely, the natural extension of Lemmas~\ref{lem:netmlambdaV} and~\ref{lem:netmst} to general gauge functions. Therefore,
		\[
			\mino(F'_\rmr(x),\R^d)\supseteq\gauge(\frakn_\rmr).
		\]
		
	Conversely, we recall from the proof of Corollary~\ref{cor:dimfracpart} that the set $F'_\rmr(x)$ is invariant under the translations by vectors in $\Z^d$, and that
		\[
			F'_\rmr(x)\cap[0,1)^d\subseteq\limsup_{n\to\infty}\bigcup_{p\in\{-1,0,1\}^d}\opball_\infty(\{nx\}+p,r_n).
		\]
		We deduce from Lemma~\ref{lem:upbndlimsup} and Proposition~\ref{prp:compnormalgauge} that $F'_\rmr(x)$ has Hausdorff $g$-measure zero for any gauge function $g$ such that $\sum_n g_d(r_n)$ converges. This means that
		\[
			\majo(F'_\rmr(x),\R^d)\supseteq\gauge(\frakn_\rmr)^\complement.
		\]
		To conclude, it suffices to apply Proposition~\ref{prp:descincl}.
\end{proof}

A simple example is obtained by assuming that the sequence $\rmr$ is defined by $r_n=n^{-\sigma}$ for all $n\geq 1$, and for a fixed $\sigma>1/d$. Indeed, one then easily checks that $\gauge(\frakn_\rmr)$ coincides with $\gauge(\frakn_{1/\sigma})$, where $\frakn_{1/\sigma}$ is defined as in~(\ref{eq:df:frakns}). If $x$ is badly approximable, we deduce from Corollary~\ref{cor:fracpartlindesc} that the set of all $y\in\R^d$ such that
	\[
		\distZ{y-nx}<\frac{1}{n^\sigma}\qquad\text{for i.m.~}n\geq 1
	\]
	is $\frakn_{1/\sigma}$-describable in $\R^d$, thereby ending up with a major improvement on~(\ref{eq:dimfracpartsigma}).

Similarly to the end of Section~\ref{subsec:inhomdesc}, a typical application consists in considering the intersection of countably many sets of the form $F'_\rmr(x)$. Specifically, for each integer $n\geq 1$, let us consider a badly approximable point $x_n$ and a nonincreasing sequence $\rmr_n=(r_{n,m})_{m\geq 1}$ of positive real numbers such that $\sum_m r_{n,m}^d$ converges. Using Corollary~\ref{cor:fracpartlindesc} together with Propositions~\ref{prp:monomajomino} and~\ref{prp:cupcapmajomino}, we infer that~(\ref{eq:capmajominoinhom}) still holds when the sets $\frakQ^{\alpha_n}_{d,\psi_n}$ and $\gauge(\frakn_{d,\psi_n})$ are replaced by the sets $F'_{\rmr_n}(x_n)$ and $\gauge(\frakn_{\rmr_n})$, respectively. The intersection of the sets $F'_{\rmr_n}(x_n)$ is thus fully describable in $\R^d$. By way of illustration, if we assume in addition that $r_{n,m}=m^{-\sigma_n}$ for all $m\geq 1$ and some $\sigma_n>1/d$, then the each set $F'_{\rmr_n}(x_n)$ is $\frakn_{1/\sigma_n}$-describable in $\R^d$. According to Proposition~\ref{prp:dichofraksdesc}, we conclude that their intersection is either $\frakn_{1/\sigma_\ast}$-describable or $(1/\sigma_\ast)$-describable, depending respectively on whether or not the supremum, denoted by $\sigma_\ast$, of all parameters $\sigma_n$ is attained.

\subsubsection{Sequences with very fast growth}

The uniform analogs of Theorems~\ref{thm:eutaxicpoly} and~\ref{thm:eutaxicpower} need not be valid, because the Lebesgue null set of points $x$ on which each of these results may fail depends on the choice of the sequence $(r_n)_{n\geq 1}$, and there are of course uncountably many sequences in $\Rho_d$. In that direction, we have however the following one-dimensional statement, established by Reversat~\cite{Reversat:1976yq}.

\begin{Theorem}\label{thm:eutaxicanfast}
Let $(a_n)_{n\geq 1}$ be a sequence of positive real numbers such that the series $\sum_n a_n/a_{n+1}$ converges. Then, for Lebesgue-almost every real number $x$, the sequence $(\{a_nx\})_{n\geq 1}$ is uniformly eutaxic in $(0,1)$.
\end{Theorem}

We omit the proof from these notes. However, it is fairly parallel to that of Theorem~\ref{thm:eutaxicindepunif} below, so the reader may refer to the proof of the latter result to get a glimpse of that of Theorem~\ref{thm:eutaxicanfast}. The fundamental reason behind this similarity is that when $\sum_n a_n/a_{n+1}$ converges and $x$ is chosen according to the Lebesgue measure, the real numbers $a_n$ tend so fast to infinity that the fractional parts $\{a_nx\}$ behave fairly like independent and uniform random variables on $(0,1)$, which is precisely the situation addressed by Theorem~\ref{thm:eutaxicindepunif}.

Note that Theorem~\ref{thm:eutaxicanfast} does not apply to the case where $a_n=b^n$, which corresponds to the $b$-adic expansion of real numbers, simply because the corresponding series $\sum_n a_n/a_{n+1}$ does not converge. In fact, the hypothesis of Theorem~\ref{thm:eutaxicanfast} is satisfied if the sequence $(a_n)_{n\geq 1}$ grows superexponentially fast, such as for instance when $a_n=n^{(1+\eps)n}$ for some $\eps>0$, or when $a_n=b^{n^2}$ for some $b>1$.

Furthermore, we deduce from Theorem~\ref{thm:eutaxicanfast} metrical results similar to those obtained in Section~\ref{subsubsec:Kurzweil}. Let us fix a sequence $(a_n)_{n\geq 1}$ of positive real numbers such that the series $\sum_n a_n/a_{n+1}$ converges and a nonincreasing sequence $\rmr=(r_n)_{n\geq 1}$ of positive real numbers. The set given by~(\ref{eq:df:Frxlin}) is now replaced by
	\[
		F_\rmr(x)=\left\{y\in\R\bigm||y-\{a_n x\}|<r_n\quad\text{for i.m.~}n\geq 1\right\}.
	\]
	Applying Theorems~\ref{thm:desceutaxy} and~\ref{thm:eutaxicanfast}, and letting $\frakn_\rmr$ still denote the measure characterized by~(\ref{eq:df:fraknrmr}), we reach the following analog of Theorem~\ref{thm:fracpartlindesc}.

\begin{Theorem}\label{thm:fracpartfastdesc}
	Let $(a_n)_{n\geq 1}$ be a sequence of positive reals such that $\sum_n a_n/a_{n+1}$ converges. Then, for Lebesgue-almost every $x\in\R$ and for every nonincreasing sequence $\rmr=(r_n)_{n\geq 1}$ of positive real numbers, the following properties hold:
	\begin{enumerate}
		\item if $\sum_n r_n$ diverges, then $F_\rmr(x)$ has full Lebesgue measure in $(0,1)$\,;
		\item if $\sum_n r_n$ converges, then $F_\rmr(x)$ is $\frakn_\rmr$-describable in $(0,1)$.
	\end{enumerate}
\end{Theorem}

We now rephrase this result by means of the distance to the nearest integer point defined by~(\ref{eq:df:distZd}), thereby dealing with the companion set
	\[
		F'_\rmr(x)=\left\{y\in\R\bigm|\distZ{y-a_n x}<r_n\quad\text{for i.m.~}n\geq 1\right\}.
	\]
	The statement bearing on this set is the following counterpart to Corollary~\ref{cor:fracpartlindesc}.

\begin{Corollary}\label{cor:fracpartfastdesc}
	Let $(a_n)_{n\geq 1}$ be a sequence of positive reals such that $\sum_n a_n/a_{n+1}$ converges. Then, for Lebesgue-almost every $x\in\R$ and for every nonincreasing sequence $\rmr=(r_n)_{n\geq 1}$ of positive real numbers, the following properties hold:
	\begin{enumerate}
		\item if $\sum_n r_n$ diverges, then $F'_\rmr(x)$ has full Lebesgue measure in $\R$\,;
		\item if $\sum_n r_n$ converges, then $F'_\rmr(x)$ is $\frakn_\rmr$-describable in $\R$.
	\end{enumerate}
\end{Corollary}

The above corollary may be deduced from Theorem~\ref{thm:fracpartfastdesc} by simply adapting the arguments employed to deduce Corollary~\ref{cor:fracpartlindesc} from Theorem~\ref{thm:fracpartlindesc}. We leave the proof to the reader. Moreover, in the particular case where $r_n=n^{-\sigma}$ for all $n\geq 1$ and some fixed $\sigma>1$, we deduce from Corollary~\ref{cor:fracpartfastdesc} that for Lebesgue-almost every real number $x$ and for every $\sigma>1$, the set of all points $y\in\R$ such that
	\[
		\distZ{y-a_n x}<\frac{1}{n^\sigma}\qquad\text{for i.m.~}n\geq 1
	\]
	is $\frakn_{1/\sigma}$-describable in $\R$. In view of Corollary~\ref{cor:frakndescslip}, we thus have a set with large intersection with Hausdorff dimension equal to $1/\sigma$. Besides, we could as well consider countable intersections of such sets, similarly to the end of Section~\ref{subsubsec:Kurzweil}. Let us finally mention that a challenging problem is to understand how the Hausdorff dimension of sets of the form $F'_t(x)$ behaves when one considers their intersection with a given compact set. We do not address this problem here, and we refer to~\cite{Bugeaud:2013fk} for precise statements and motivations.


\section{Approximation by algebraic numbers}\label{sec:alg}

We now turn our attention to the examples supplied by the real algebraic numbers and the real algebraic integers. Our treatment will be somewhat brief, as for instance we shall not detail all the proofs; for further details, we refer to the seminal paper by Baker and Schmidt~\cite{Baker:1970jf}, subsequent important works by Beresnevich~\cite{Beresnevich:1999ys} and Bugeaud~\cite{Bugeaud:2002uq}, and the references therein. We shall show that the algebraic numbers and integers lead to optimal regular systems, and we shall state the metrical results obtained from subsequently applying Theorem~\ref{thm:descoptregsys}.

\subsection{Associated optimal regular system}\label{subsec:algoptregsys}

The collection of all real algebraic numbers is denoted by $\A$. The na\"\i ve height of a number $a$ in $\A$, denoted by $\Eta(a)$, is the maximum of the absolute values of the coefficients of its minimal defining polynomial over $\Z$. Moreover, the set of all real algebraic numbers with degree at most $n$ is denoted by $\A_n$. Baker and Schmidt~\cite{Baker:1970jf} proved that the set $\A_n$, endowed with the height function
	\[
	a\mapsto\frac{\Eta(a)^{n+1}}{(\max\{1,\log\Eta(a)\})^{3n^2}},
	\]
	forms a regular system. The trouble is that, due to the logarithmic term, this height function does not lead to the best possible metrical statements. However, Beresnevich proved that the height function
	\[
		H_n(a)=\frac{\Eta(a)^{n+1}}{(1+|a|)^{n(n+1)}},
	\]
	where there is no logarithmic term, is actually convenient. We shall therefore privilege the following statement when deriving metrical results underneath.

\begin{Theorem}[Beresnevich]\label{thm:algnumoptregsys}
	For any integer $n\geq 1$, the pair $(\A_n,H_n)$ is an optimal regular system in $\R$.
\end{Theorem}

It is elementary to check that $(\A_n,H_n)$ is an optimal system. Establishing the regularity is much more difficult and relies on a fine knowledge of the distribution of real algebraic numbers; we refer to~\cite{Beresnevich:1999ys} for a detailed proof. Note that $\A_1$ obviously coincides with the set $\Q$ of rational numbers. Moreover, writing an element $a$ in $\A_1$ in the form $p/q$ for two coprime integers $p$ and $q$, the latter being positive, we have
	\[
		H_1(a)=\frac{\Eta(a)^2}{(1+|a|)^2}=\frac{\max\{|p|,q\}^2}{(1+|a|)^2}=\left(\frac{\max\{1,|a|\}}{1+|a|}\right)^2 q^2,
	\]
	so that $H_1(a)$ is between $q^2/4$ and $q^2$. Hence, the height of $a$, viewed as an algebraic number with degree one, is comparable with its height when regarded as a rational point of the real line, see~(\ref{eq:df:heightQalphad}) in the homogeneous case.

\subsection{General metrical implications}\label{subsec:algmetric}

Our purpose is now to describe the size and large intersection properties of the set
	\begin{equation}\label{eq:df:Anpsi}
		\frakA_{n,\psi}=\left\{x\in\R\bigm||x-a|<\psi(\Eta(a))\quad\text{for i.m.~}a\in\A_n\right\},
	\end{equation}
	where $\psi$ denotes a positive nonincreasing continuous function $\psi$ defined on $[0,\infty)$. Our approach is parallel to that leading to Theorem~\ref{thm:inhomdesc}. As a matter of fact, we shall show that $\frakA_{n,\psi}$ is well approximated by sets of the form~(\ref{eq:df:FphAH}) when the underlying system is $(\A_n,H_n)$. That system being optimal and regular due to Theorem~\ref{thm:algnumoptregsys}, we shall therefore be able to apply Theorem~\ref{thm:descoptregsys} to reach our goal. An integral and a measure that are close to~(\ref{eq:df:Iph}) and~(\ref{eq:df:fraknph}), respectively, will naturally come into play, specifically, the integral
	\[
		I_{n,\psi}=\int_0^\infty h^n\psi(h)\,\dd h
	\]
	and the Borel measure $\frakn_{n,\psi}$ on $(0,1]$ characterized by the condition that for any nonnegative measurable function $f$ supported in $(0,1]$,
	\[
		\int_{(0,1]} f(r)\,\frakn_{n,\psi}(\dd r)=\int_0^\infty h^n f(\psi(h))\,\dd h.
	\]
	We are now in position to state and establish our main result.

\begin{Theorem}\label{thm:algdesc}
	Let $n$ be a positive integer and let $\psi$ denote a positive nonincreasing continuous function defined on $[0,\infty)$. Then, the following properties hold:
	\begin{itemize}
		\item if $I_{n,\psi}$ diverges, then $\frakA_{n,\psi}$ has full Lebesgue measure in $\R$\,;
		\item if $I_{n,\psi}$ converges, then $\frakA_{n,\psi}$ is $\frakn_{n,\psi}$-describable in $\R$.
	\end{itemize}
\end{Theorem}

\begin{proof}
	We begin by proving that $\frakA_{n,\psi}$ may be approximated by sets of the form~(\ref{eq:df:FphAH}) when the underlying system is $(\A_n,H_n)$. For any integer $k\geq 1$, let $\ph_k$ denote the function defined for all $\eta\geq 0$ by $\ph_k(\eta)=\psi(k\,\eta^{1/(n+1)})$. We then have
	\begin{equation}\label{eq:sandFphkAnpsi}
		\bigcap_{k=1}^\infty\downarrow F_{\ph_k}\subseteq\frakA_{n,\psi}\subseteq F_{\ph_1}.
	\end{equation}
	Indeed, let $x$ be in the left-hand side and let $k$ be an integer larger than or equal to $(1+|x|+\psi(0))^n$. Since $x$ is in $F_{\ph_k}$, there are infinitely many points $a$ in $\A_n$ with
	\[
		|x-a|<\ph_k(H_n(a))=\psi(k\,H_n(a)^{1/(n+1)})
	\]
	However, the function $\psi$ is nonincreasing and the integer $k$ is bounded below by $(1+|x|+\psi(0))^n$, and thus by $(1+|a|)^n$. Hence, we have
	\[
		|x-a|<\psi((1+|a|)^n\,H_n(a)^{1/(n+1)})=\psi(\Eta(a))
	\]
	for infinitely many $a$ in $\A_n$, so that $x$ is in $\frakA_{n,\psi}$. Furthermore, in that situation, since the inequality $|x-a|<\psi(\Eta(a))$ holds for infinitely many $a$ in $\A_n$, we get
	\[
		|x-a|<\psi(\Eta(a))=\psi((1+|a|)^n\,H_n(a)^{1/(n+1)})\leq\psi(H_n(a)^{1/(n+1)})=\ph_1(H_n(a)),
	\]
	again because $\psi$ is nonincreasing, so that $x$ belongs to $F_{\ph_1}$.

	Let us deal with the divergence case. Thanks to~(\ref{eq:sandFphkAnpsi}), it suffices to prove that all the sets $F_{\ph_k}$ have full Lebesgue measure in $\R$. However, a simple change of variable shows that the divergence of $I_{n,\psi}$ implies that of all the integrals $I_{\ph_k}$ defined as in~(\ref{eq:df:Iph}). We thus conclude with the help of Theorems~\ref{thm:descoptregsys} and~\ref{thm:algnumoptregsys}.

	Turning our attention to the convergence case, we first combine~(\ref{eq:sandFphkAnpsi}) with Propositions~\ref{prp:monomajomino} and~\ref{prp:cupcapmajomino} in order to write that
	\[
		\majo(\frakA_{n,\psi},\R)\supseteq\majo(F_{\ph_1},\R)
		\qquad\text{and}\qquad
		\mino(\frakA_{n,\psi},\R)\supseteq\bigcap_{k=1}^\infty\mino(F_{\ph_k},\R).
	\]
	Again, an elementary change of variable shows that the convergence of $I_{n,\psi}$ implies that of all the integrals $I_{\ph_k}$. Likewise, $\gauge(\frakn_{n,\psi})$ coincides with all the sets $\gauge(\frakn_{\ph_k})$, where the measures $\frakn_{\ph_k}$ are characterized as in~(\ref{eq:df:fraknph}). Applying Theorem~\ref{thm:descoptregsys}, we deduce that all the sets $F_{\ph_k}$ are $\frakn_{n,\psi}$-describable in $\R$. As a consequence,
	\[
		\majo(\frakA_{n,\psi},\R)\supseteq\gauge(\frakn_{n,\psi})^\complement
		\qquad\text{and}\qquad
		\mino(\frakA_{n,\psi},\R)\supseteq\gauge(\frakn_{n,\psi}),
	\]
	and we conclude thanks to Proposition~\ref{prp:descincl} and Lemma~\ref{lem:gaugerealopen}(\ref{item:lem:gaugerealopen2}).
\end{proof}

We now catalog some consequences of Theorem~\ref{thm:algdesc}. First, the statement entails that the set $\frakA_{n,\psi}$ has full Lebesgue measure in $\R$ if the integral $I_{n,\psi}$ diverges, and Lebesgue measure zero if it converges. Hence, for any nonempty open set $V\subseteq\R$,
	\[
		\leb^1(\frakA_{n,\psi}\cap V)=
		\begin{cases}
			\leb^1(V) & \text{if } \sum_h h^n \psi(h)=\infty \\[2mm]
			0 & \text{if } \sum_h h^n \psi(h)<\infty.
		\end{cases}
	\]
	In fact, the monotonicity of $\psi$ shows that the convergence of $I_{n,\psi}$ amounts to that of the above series. We thus recover a result due to Beresnevich~\cite{Beresnevich:1999ys}.

Combined with Theorem~\ref{thm:frakndescslip}, the previous result yields a complete description of the size and large intersection properties of the set $\frakA_{n,\psi}$. We focus on the case where $\frakA_{n,\psi}$ has Lebesgue measure zero, because these properties are trivial otherwise, in light of the opening discussion in Section~\ref{sec:desc}. In particular, we recover a characterization of the Hausdorff measures of $\frakA_{n,\psi}$ obtained independently by Beresnevich, Dickinson and Velani~\cite{Beresnevich:2006ve}, and by Bugeaud~\cite{Bugeaud:2002fk}. Specifically, for any gauge function $g$ and any nonempty open set $V\subseteq\R$,
	\[
		\hau^g(\frakA_{n,\psi}\cap V)=
		\begin{cases}
			\infty & \text{if } \sum_h h^n g_1(\psi(h))=\infty \\[2mm]
			0 & \text{if } \sum_h h^n g_1(\psi(h))<\infty.
		\end{cases}
	\]
	We also used here the fact that $g\in\gauge(\frakn_{n,\psi})$ if and only if the above series diverges, owing to the monotonicity of $\psi$ and that of $g_1$ near the origin. Similarly, we recover the description of the large intersection properties of $\frakA_{n,\psi}$ obtained in~\cite{Durand:2007uq}.

Finally, Corollary~\ref{cor:frakndescslip} enables us to give a more concise dimensional statement. In fact, still assuming that $\frakA_{n,\psi}$ has Lebesgue measure zero, we infer that the integral $I_{n,\psi}$ converges and that the exponent associated with $\frakn_{n,\psi}$ through~(\ref{eq:df:sfraknrad01}) is
	\begin{equation}\label{eq:df:snpsi}
		s_{n,\psi}=\limsup_{h\to\infty}\frac{(n+1)\log h}{-\log\psi(h)},
	\end{equation}
	so we eventually conclude that for any nonempty open set $V\subseteq\R$,
	\[
		\left\{\begin{array}{l}
			\Hdim (\frakA_{n,\psi}\cap V)=s_{n,\psi}\\[1mm]
			\Pdim (\frakA_{n,\psi}\cap V)=d\\[1mm]
			\frakA_{n,\psi}\in\lic{s_{n,\psi}}{V},
		\end{array}\right.
	\]
	where the last two properties are valid under the assumption that $s_{n,\psi}$ is positive.

\subsection{Koksma's classification of real transcendental numbers}

Let us now concentrate on the case in which the function $\psi$ is of the form $h\mapsto h^{-\omega-1}$ on the interval $[1,\infty)$, for some real number $\omega>-1$. In order to stress on the r\^ole of $\omega$ and ensure some coherence with Koksma's notations, the set $\frakA_{n,\psi}$ is denoted by $K^\ast_{n,\omega}$ in what follows, namely,
	\[
		K^\ast_{n,\omega}=\left\{x\in\R\bigm||x-a|<\Eta(a)^{-\omega-1}\quad\text{for i.m.~}a\in\A_n\right\}.
	\]
	Furthermore, to complete the definition of $\psi$, we suppose that it is constant equal to one on the interval $[0,1]$. It is clear that the integral $\Iota_{n,\psi}$ converges if and only if $\omega>n$ and, in that situation, the sets $\gauge(\frakn_{n,\psi})$ and $\gauge(\frakn_{(n+1)/(\omega+1)})$ coincide, where the measure $\frakn_{(n+1)/(\omega+1)}$ is again defined as in~(\ref{eq:df:frakns}). We then readily deduce the next statement from Theorem~\ref{thm:algdesc}.

\begin{Corollary}\label{cor:algdesc}
	For any integer $n\geq 1$ and any real parameter $\omega>-1$, the following properties hold:
	\begin{enumerate}
		\item if $\omega\leq n$, then $K^\ast_{n,\omega}$ has full Lebesgue measure in $\R$\,;
		\item if $\omega>n$, then $K^\ast_{n,\omega}$ is $\frakn_{(n+1)/(\omega+1)}$-describable in $\R$.
	\end{enumerate}
\end{Corollary}

We use this result to comment on a classification of real transcendental numbers due to Koksma~\cite{Koksma:1939fk}. First, it is clear that the mapping $\omega\mapsto K^\ast_{n,\omega}$ is nonincreasing; for every real number $x$, we thus naturally introduce the exponent
	\[
		\omega^\ast_n(x)=\sup\{\omega>-1\:|\:x\in K^\ast_{n,\omega}\}.
	\]

Note that when $n=1$ and $x$ is irrational, one essentially recovers the irrationality exponent defined by~(\ref{eq:df:irrexpo}). Indeed, as observed in Section~\ref{subsec:algoptregsys}, the set $\A_1$ coincides with $\Q$, and writing an element $a\in\A_1$ in the form $p/q$ for two coprime integers $p$ and $q$, the latter being positive, we have $\Eta(a)=\max\{|p|,q\}$. It is then easy to check that for all $\omega>0$,
	\[
		K^\ast_{1,\omega}\subseteq J_{1,\omega+1}\setminus\Q\subseteq\bigcap_{\eps>0}\downarrow K^\ast_{1,\omega-\eps},
	\]
	and therefore that for any irrational number $x$,
	\[
		\omega^\ast_1(x)=\tau(x)-1.
	\]

Koksma introduced a classification of the real transcendental numbers $x$ which is based on the way the exponents $\omega^\ast_n(x)$ evolve as $n$ grows. This amounts to studying how the quality with which a real number $x$ is approximated by algebraic numbers behaves when their degree is allowed to increase. Specifically, let us define
	\[
		\omega^\ast(x)=\limsup_{n\to\infty} \frac{\omega^\ast_n(x)}{n}.
	\]
	Koksma classifies the real transcendental numbers $x$ according to whether or not $\omega^\ast(x)$ is finite, see~\cite[Section~3.3]{Bugeaud:2004fk}. In the first situation, that is, if $\omega^\ast(x)$ is finite, he calls $x$ an {\em $S^\ast$-number}. Besides, let us mention that a result due to Wirsing~\cite{Wirsing:1961fk} shows that a real number $x$ is transcendental if and only if $\omega^\ast(x)$ is positive, see~\cite{Bugeaud:2004fk}.
	
	As we now explain, Corollary~\ref{cor:algdesc} entails that Lebesgue-almost every real number $x$ is an $S^\ast$-number satisfying $\omega^\ast_n(x)=n$ for every $n\geq 1$. In fact, for any real parameter $\omega>0$, let $\widehat K^\ast_{n,\omega}$ denote the set of all real numbers $x$ for which the exponent $\omega^\ast_n(x)$ is bounded below by $(n+1)\omega-1$. Observing that
	\[
		\widehat K^\ast_{n,\omega}=\bigcap_{\omega'<(n+1)\omega-1}\downarrow K^\ast_{n,\omega'},
	\]
	we deduce from Corollary~\ref{cor:algdesc} that the set $\widehat K^\ast_{n,\omega}$ has full Lebesgue measure in $\R$ when $\omega\leq 1$, and Lebesgue measure zero otherwise.
	
Our aim is now to describe the size and large intersection properties of the set $\widehat K^\ast_{n,\omega}$. As usual, we may exclude the trivial case in which this set has full Lebesgue measure, and therefore suppose that $\omega>1$. Due to the monotonicity of the mapping $\omega'\mapsto K^\ast_{n,\omega'}$, we may assume in the above intersection that $\omega'$ ranges over a sequence of real numbers strictly between $n$ and $(n+1)\omega-1$ that monotonically tends to the latter value. In view of Corollary~\ref{cor:algdesc}, we fall into the setting of Proposition~\ref{prp:dichofraksdesc} in the case where the infimum is not attained. We end up with the next result.

\begin{Corollary}\label{cor:Koksmandesc}
	For any integer $n\geq 1$ and any real parameter $\omega>1$, the set $\widehat K^\ast_{n,\omega}$ is $(1/\omega)$-describable in $\R$.
\end{Corollary}

In order to make the connection with Koksma's classification, we need to consider all the integers $n$ simultaneously. Accordingly, let us introduce the set
	\[
		\widehat K^\ast_{\omega}=\bigcap_{n=1}^\infty\widehat K^\ast_{n,\omega}.
	\]
	When $\omega\leq 1$, what precedes ensures that $\widehat K^\ast_{\omega}$ has full Lebesgue measure in $\R$, and its size and large intersection properties are trivially described. Let us assume oppositely that $\omega>1$. Combining Corollary~\ref{cor:Koksmandesc} with Propositions~\ref{prp:monomajomino} and~\ref{prp:cupcapmajomino}, we straightforwardly establish that
	\[
		\majo(\widehat K^\ast_{\omega},\R)\supseteq\gauge(1/\omega)^\complement
		\qquad\text{and}\qquad
		\mino(\widehat K^\ast_{\omega},\R)\supseteq\gauge(1/\omega).
	\]
	Applying Proposition~\ref{prp:descincl} and Lemma~\ref{lem:gaugerealopen}(\ref{item:lem:gaugerealopen2}), we eventually get the following result.

\begin{Corollary}\label{cor:Koksmadesc}
	For any $\omega>1$, the set $\widehat K^\ast_{\omega}$ is $(1/\omega)$-describable in $\R$.
\end{Corollary}

Again, combining this result with Theorem~\ref{thm:fraksdescslip}, we obtain a complete description of the size and large intersection properties of $\widehat K^\ast_{\omega}$, thereby recovering results previously established in~\cite{Bugeaud:2004wc,Durand:2007uq}. One may also use Corollary~\ref{cor:fraksdescslip} if only dimensional results are desired. In particular, we observe that the set $\widehat K^\ast_{\omega}$ has Hausdorff dimension equal to $1/\omega$. We thus recover a result established by Baker and Schmidt~\cite{Baker:1970jf}.

The connection with Koksma's classification now consists in making the obvious remark that for any real parameter $\omega>0$, the set
	\[
		\Omega^\ast_\omega=\{x\in\R\:|\:\omega^\ast(x)\geq\omega\}
	\]
	contains $\widehat K^\ast_{\omega}$. In particular, we recover the fact that $\Omega^\ast_\omega$ has full Lebesgue measure in $\R$ when $\omega\leq 1$. As regards size and large intersection properties, the opposite case is richer and is covered by the next result.

\begin{Theorem}\label{thm:Koksmafdesc}
	For any real parameter $\omega>1$, the set $\Omega^\ast_\omega$ of all real numbers $x$ such that $\omega^\ast(x)\geq\omega$ is $(1/\omega)$-describable in $\R$.
\end{Theorem}

\begin{proof}
	Since $\Omega^\ast_\omega\supseteq\widehat K^\ast_{\omega}$, we deduce from Proposition~\ref{prp:monomajomino} and Corollary~\ref{cor:Koksmadesc} that
	\[
		\mino(\Omega^\ast_\omega,\R)\supseteq\mino(\widehat K^\ast_{\omega},\R)\supseteq\gauge(1/\omega).
	\]
	Furthermore, let us consider a sequence $(\omega'_m)_{m\geq 1}$ of real numbers strictly between one and $\omega$ that monotonically tends to the latter value. We clearly have
	\[
		\Omega^\ast_\omega\subseteq
		\bigcap_{m=1}^\infty\bigcup_{n=1}^\infty K^\ast_{n,(n+1)\omega'_m-1}.
	\]
	By virtue of Propositions~\ref{prp:monomajomino} and~\ref{prp:cupcapmajomino}, and also Corollary~\ref{cor:algdesc}, this entails that
	\[
		\majo(\Omega^\ast_\omega,\R)\supseteq
		\bigcup_{m=1}^\infty\bigcap_{n=1}^\infty\majo(K^\ast_{n,(n+1)\omega'_m-1},\R)
		=\bigcup_{m=1}^\infty\gauge(\frakn_{1/\omega'_m})^\complement.
	\]
	Indeed, each set $K^\ast_{n,(n+1)\omega'_m-1}$ is $\frakn_{1/\omega'_m}$-describable in $\R$. We finally infer from~(\ref{eq:linkfraknfraksdesc}) that the right-hand side is equal to $\gauge(1/\omega)^\complement$, and we conclude thanks to Proposition~\ref{prp:descincl} and Lemma~\ref{lem:gaugerealopen}(\ref{item:lem:gaugerealopen2}).
\end{proof}

It is possible to formally let $\omega$ tend to infinity in Theorem~\ref{thm:Koksmafdesc}. This amounts to considering the intersection of the sets $\Omega^\ast_\omega$, in conjunction with the observation that the intersection of the sets $\gauge(1/\omega)$ reduces to the set $\gauge(0)$. Using the methods developed up to now, the reader should easily prove the next result.

\begin{Corollary}
	The set $\Omega^\ast_\infty$ of all real numbers $x$ such that $\omega^\ast(x)=\infty$ is $0$-describable in $\R$.
\end{Corollary}

Note that, referring to Koksma's classification, the set $\Omega^\ast_\infty$ consists of the transcendental numbers $x$ that are not $S^\ast$-numbers; they are call either {\em $T^\ast$-numbers} or {\em $U^\ast$-numbers}, depending respectively on whether $\omega^\ast_n(x)$ is finite for all $n\geq 1$, or infinite from some $n$ onwards. Let us finally mention that Koksma's classification is very close to that previously introduced by Mahler~\cite{Mahler:1932uq} and for which large intersection properties also come into play, see~\cite{Bugeaud:2004fk,Durand:2007uq} for details.

\subsection{The case of algebraic integers}

Bugeaud~\cite{Bugeaud:2002uq} obtained an analog of Theorem~\ref{thm:algnumoptregsys} for the set of real algebraic integers, that is, the real algebraic numbers whose minimal defining polynomial over $\Z$ is monic. In what follows, $\A'$ denotes the subset of $\A$ formed by the real algebraic integers, and $\A'_n$ denotes the set $\A'\cap\A_n$ of all real algebraic integers with degree at most $n$.

\begin{Theorem}[Bugeaud]\label{thm:algintoptregsys}
	For any integer $n\geq 2$, the pair $(\A'_n,H_{n-1})$ is an optimal regular system in $\R$.
\end{Theorem}

Combining Theorem~\ref{thm:algintoptregsys} with the methods of Section~\ref{subsec:algmetric}, we may describe the elementary size and large intersection properties of the set $\frakA'_{n,\psi}$ defined as that obtained when replacing $\A_n$ by $\A'_n$ in~(\ref{eq:df:Anpsi}), namely,
	\[
		\frakA'_{n,\psi}=\left\{x\in\R\bigm||x-a|<\psi(\Eta(a))\quad\text{for i.m.~}a\in\A'_n\right\}.
	\]
	Actually, adapting the proof of Theorem~\ref{thm:algdesc}, we may establish the next statement.

\begin{Theorem}
	For any integer $n\geq 2$ and any positive nonincreasing continuous function $\psi$ defined on $[0,\infty)$, the following properties hold:
	\begin{itemize}
		\item if $I_{n-1,\psi}$ diverges, then $\frakA'_{n,\psi}$ has full Lebesgue measure in $\R$\,;
		\item if $I_{n-1,\psi}$ converges, then $\frakA'_{n,\psi}$ is $\frakn_{n-1,\psi}$-describable in $\R$.
	\end{itemize}
\end{Theorem}

Next, applying Theorem~\ref{thm:frakndescslip}, we get a complete description of the size and large intersection properties of $\frakA'_{n,\psi}$. Likewise, Corollary~\ref{cor:frakndescslip} provides us with dimensional results on this set. For instance, when $I_{n-1,\psi}$ is convergent, we deduce that $\frakA'_{n,\psi}$ is a set with large intersection with Hausdorff dimension $s_{n-1,\psi}$ given by~(\ref{eq:df:snpsi}).


\section{Independent and uniform coverings}\label{sec:Dvoretzky}

We consider from now on probabilistic models where uniform eutaxy comes into play and enables us to analyze size and large intersection properties for associated limsup sets. The simplest model consists of a sequence $(X_n)_{n\geq 1}$ of points that are independently and uniformly distributed in a nonempty bounded open set $U\subseteq\R^d$.

Hence, the random points $X_n$ are stochastically independent and distributed according to the normalized Lebesgue measure $\leb^d(\,\cdot\,\cap U)/\leb^d(U)$. For any sequence $(r_n)_{n\geq 1}$ in $\Rho_d$ and any point $x$ in $U$, we have
\[
\prob(x\in\opball(X_n,r_n))=\frac{\leb^d(U\cap\opball(x,r_n))}{\leb^d(U)}=\frac{\leb^d(\opball(0,1))}{\leb^d(U)}\,r_n^d
\]
for $n$ sufficiently large. Hence, the Borel-Cantelli lemma ensures that the inequality $|x-X_n|<r_n$ holds infinitely often with probability one. By virtue of Tonelli's theorem, this implies that the sequence $(X_n)_{n\geq 1}$ is almost surely eutaxic in $U$ with respect to $(r_n)_{n\geq 1}$. However, the almost sure event on which this property holds may depend on the sequence $(r_n)_{n\geq 1}$, and we cannot yet deduce that eutaxy is uniform in the sense of Section~\ref{subsec:eutaxyunif}.

In order to show that uniform eutaxy holds, and thus establish the next result due to Reversat~\cite{Reversat:1976yq}, we need to develop more involved arguments.

\begin{Theorem}\label{thm:eutaxicindepunif}
	Let $(X_n)_{n\geq 1}$ be a sequence of random points distributed independently and uniformly in a nonempty bounded open subset $U$ of $\R^d$. Then, with probability one, the sequence $(X_n)_{n\geq 1}$ is uniformly eutaxic in $U$.
\end{Theorem}

\begin{proof}
	Let us consider a dyadic cube $\lambda\subseteq U$, a real number $\alpha\in(0,1)$ and an integer $j\geq 0$, and let us suppose that the condition
	\begin{equation}\label{eq:condrandXnleq}
		\#\Mu((X_n)_{n\geq 1};\lambda,j)\leq \alpha\,2^{dj}
	\end{equation}
	holds. Then, the first $2^{d(\gene{\lambda}+j)}$ points $X_n$ are contained in either the complement in $\R^d$ of the cube $\lambda$, or the union of $\lfloor\alpha\,2^{dj}\rfloor$ subcubes of $\lambda$ with generation $\gene{\lambda}+j$, denoted by $\lambda'_1,\ldots,\lambda'_{\lfloor\alpha\,2^{dj}\rfloor}$. Each point $X_n$ is uniformly distributed in $U$, so that
	\[
		\prob(X_n\in(\R^d\setminus\lambda)\sqcup\lambda'_1\sqcup\ldots\sqcup\lambda'_{\lfloor\alpha\,2^{dj}\rfloor})
		=1-\frac{2^{-d\gene{\lambda}}}{\leb^d(U)}+\lfloor\alpha\,2^{dj}\rfloor\frac{2^{-d(\gene{\lambda}+j)}}{\leb^d(U)}.
	\]
	Moreover, combining the fact that the points $X_n$ are independent with the obvious bound $1+z\leq\ee^z$, for $z$ in $\R$, we deduce that
	\[
		\prob(X_1,\ldots,X_{2^{d(\gene{\lambda}+j)}}\in(\R^d\setminus\lambda)\sqcup\lambda'_1\sqcup\ldots\sqcup\lambda'_{\lfloor\alpha\,2^{dj}\rfloor})
		\leq\exp\left(-\frac{1-\alpha}{\leb^d(U)}\,2^{dj}\right).
	\]
	As a consequence, taking into account all the possible choices for the subcubes $\lambda'_1,\ldots,\lambda'_{\lfloor\alpha\,2^{dj}\rfloor}$ that result from the assumption~(\ref{eq:condrandXnleq}), we conclude that
	\[
		\prob(\#\Mu((X_n)_{n\geq 1};\lambda,j)\leq \alpha\,2^{dj})
		\leq\binom{2^{dj}}{\lfloor\alpha\,2^{dj}\rfloor}\exp\left(-\frac{1-\alpha}{\leb^d(U)}\,2^{dj}\right).
	\]
	By virtue of Stirling's formula, the logarithm of the binomial coefficient above is equivalent to $\Eta(\alpha)\,2^j$ as $j$ goes to infinity, where
	\[
		\Eta(\alpha)=-\alpha\log\alpha-(1-\alpha)\log(1-\alpha).
	\]
	Note that $\Eta(\alpha)$ is the Shannon entropy of the probability vector $(\alpha,1-\alpha)$. Hence,
	\[
		\limsup_{j\to\infty}\frac{1}{2^{dj}}\log\prob(\#\Mu((X_n)_{n\geq 1};\lambda,j)\leq \alpha\,2^{dj})
		\leq\Eta(\alpha)-\frac{1-\alpha}{\leb^d(U)}.
	\]
	
	The right-hand side vanishes for a unique $\alpha\in(0,1)$, denoted by $\alpha_0$. Furthermore, the right-hand side is negative when $\alpha<\alpha_0$, and the Borel-Cantelli lemma ensures that almost surely, the condition~(\ref{eq:condrandXnleq}) is satisfied for finitely many values of $j$ only. Hence, for every dyadic cube $\lambda\subseteq U$ and every $\alpha\in(0,\alpha_0)$,
	\[
		\as\qquad\liminf_{j\to\infty} 2^{-dj}\#\Mu((X_n)_{n\geq 1};\lambda,j)\geq\alpha.
	\]
	We may let $\alpha$ tend to $\alpha_0$ along a countable sequence, and the limiting value $\alpha_0$ does not depend on the choice of the dyadic cube $\lambda$. In addition, there are countably many dyadic cubes contained in $U$. The upshot is that the sequence $(X_n)_{n\geq 1}$ verifies~(\ref{eq:infliminfpos}) with probability one. Therefore, the weaker condition~(\ref{eq:liminfpos}) is also satisfied almost surely, and we may conclude with the help of Theorem~\ref{thm:CSeutaxy}.
\end{proof}

Combining Theorems~\ref{thm:desceutaxy} and~\ref{thm:eutaxicindepunif}, we may study the size and large intersection properties of the corresponding instance of~(\ref{eq:df:Fxnrn}), namely, the random set
	\[
		F_\rmr=\left\{x\in\R^d\bigm||x-X_n|<r_n\quad\text{for i.m.~}n\geq 1\right\},
	\]
	where $\rmr=(r_n)_{n\geq 1}$ is a nonincreasing sequence of positive real numbers. In fact, letting $\frakn_\rmr$ be the measure characterized by~(\ref{eq:df:fraknrmr}), we get the next statement.

\begin{Theorem}\label{thm:Dvoretzkydesc}
	Let $(X_n)_{n\geq 1}$ be a sequence of random points distributed independently and uniformly in a nonempty bounded open subset $U$ of $\R^d$. Then, with probability one, for any nonincreasing sequence $\rmr=(r_n)_{n\geq 1}$ of positive real numbers, the following properties hold:
	\begin{itemize}
		\item if $\sum_n r_n^d$ diverges, then $F_\rmr$ has full Lebesgue measure in $U$\,;
		\item if $\sum_n r_n^d$ converges, then $F_\rmr$ is $\frakn_\rmr$-describable in $U$.
	\end{itemize}
\end{Theorem}

As usual, in combination with Theorem~\ref{thm:frakndescslip}, the above result yields a complete description of the size and large intersection properties of the random set $F_\rmr$. Such a description was first obtained in~\cite{Durand:2010fk}\,; Theorem~\ref{thm:Dvoretzkydesc} is however stronger than the description given in that paper in the sense that the almost sure event on which the statement holds does not depend on the sequence $\rmr$. Furthermore, as far as dimensional results are concerned, Corollary~\ref{cor:frakndescslip} is sufficient. By way of illustration, let us apply this result here. We assume that $\sum_n r_n^d$ converges, because the set has trivial dimensional properties otherwise, see the initial discussion in Section~\ref{sec:desc}. The exponent associated with the measure $\frakn_\rmr$ through~(\ref{eq:df:sfraknrad01}) is nothing but the critical exponent $s_\rmr$ for the convergence of the series $\sum_n r_n^s$ that is characterized by~(\ref{eq:condsumrn}). We conclude that almost surely, for any nonempty open set $V\subseteq U$,
	\[
		\left\{\begin{array}{l}
			\Hdim (F_\rmr\cap V)=s_\rmr\\[1mm]
			\Pdim (F_\rmr\cap V)=d\\[1mm]
			F_\rmr\in\lic{s_\rmr}{V},
		\end{array}\right.
	\]
	where the last two properties are valid under the additional assumption that $s_\rmr$ is positive. This can also be seen as a straightforward consequence of Corollary~\ref{cor:desceutaxy}. Finally, restricting to power functions for the approximating radii, we deduce that with probability one, for all $c>0$ and all $\sigma>1/d$,
	\[
		\Hdim\left\{x\in\R^d\Biggm||x-X_n|<\frac{c}{n^\sigma}\quad\text{for i.m.~}n\geq 1\right\}=\frac{1}{\sigma},
	\]
	thus extending a result due to Fan and Wu~\cite{Fan:2004it}, who addressed the one-dimensional case where $U$ is the open unit interval.

The above study is related with the famous problem regarding random coverings of the circle raised in~1956 by Dvoretzky~\cite{Dvoretzky:1956bs}. We now restrict our attention to the one-dimensional case. As mentioned above, the fact that a sequence $(r_n)_{n\geq 1}$ belongs to $\Rho_1$ implies, through a simple application of the Borel-Cantelli lemma and Tonelli's theorem, that with probability one, {\em Lebesgue-almost every} point $x$ of $(0,1)$ is covered by the open interval centered at $X_n$ with radius $r_n$, {\em i.e.}~satisfies $|x-X_n|<r_n$, for infinitely many integers $n\geq 1$. Dvoretzky's question can then be recast as follows: find a necessary and sufficient condition on the sequence $(r_n)_{n\geq 1}$ to ensure that with probability one, {\em every} point of the open unit interval $(0,1)$ satisfies the previous property. The problem raised the interest of many mathematicians such as Billard, Erd\H{o}s, Kahane and L\'evy, and was finally solved in 1972 by Shepp~\cite{Shepp:1972qj} who discovered that the condition is
\[
\sum_{n=1}^\infty\frac{1}{n^2}\exp(2(r_1+\ldots+r_n))=\infty.
\]
This criterion is very subtle in the sense that constants do matter: when $r_n$ is of the specific form $c/n$ with $c>0$, the condition is satisfied if and only if $c\geq 1/2$. We refer to~\cite{Durand:2010fk} and the references therein for more information on this topic.


\section{Poisson coverings}\label{sec:Poisson}

Comparable results may be obtained when the approximating points and the approximation radii are distributed according to a Poisson point measure. We begin by briefly recalling some basic facts about Poisson measures; we refer to {\em e.g.}~\cite{Kingman:1993gf,Neveu:1977mz} for additional details. The theory may be nicely developed for instance in locally compact topological spaces with a countable base. If $S$ denotes such a topological space, we call a {\em point measure} on $S$ any nonnegative measure $\varpi$ on $S$ that may be written as a sum of Dirac point masses, namely,
	\[
		\varpi=\sum_{n\in\calN}\delta_{s_n}
		\qquad\text{with}\qquad
		s_n\in S,
	\]
	and that assigns a finite mass to each compact subset of $S$. Note that the above points $s_n$ need not be distinct, but the index set $\calN$ is necessarily countable. The set of all point measures may be endowed with the $\sigma$-field generated by the mappings $\varpi\mapsto\varpi(F)$, where $F$ ranges over the Borel subsets of $S$. Naturally, a {\em random point measure} on $S$ is then a measurable mapping $\Pi$ defined on some abstract probability space and valued in the measurable space of point measures. One can show that the probability distribution of such a random point measure $\Pi$ is characterized by the distributions of all the random vectors of the form $(\Pi(E_1),\ldots,\Pi(E_n))$, where the sets $E_1,\ldots,E_n$ range over any fixed class of relatively compact Borel subsets of $S$ that is closed under finite intersections and generate the Borel $\sigma$-field on $S$. This enables us to now introduce our main definition.

\begin{Definition}
	Let $S$ be a locally compact topological space with a countable base, and let $\pi$ be a positive Radon measure thereon. There exists a random point measure $\Pi$ on $S$ such that the following two properties hold:
	\begin{itemize}
		\item for every Borel subset $E$ of $S$, the random variable $\Pi(E)$ is Poisson distributed with parameter $\pi(E)$\,;
		\item for all Borel subsets $E_1,\ldots,E_n$ of $S$ that are pairwise disjoint, the random variables $\Pi(E_1),\ldots,\Pi(E_n)$ are independent.
	\end{itemize}
	The random point measure $\Pi$ is called a {\em Poisson point measure} with intensity $\pi$, and its law is uniquely determined by the above two properties.
\end{Definition}

Note that we adopt the usual convention that a Poisson random variable with infinite parameter is almost surely equal to $\infty$. In addition to the aforementioned characterization, the distribution of a random point measure $\Pi$ is also determined by its {\em Laplace functional}, namely, the mapping defined by the formula
	\[
		\frakL_\Pi(f)=\esp\left[\exp\left(-\int_S f(s)\,\Pi(\dd s)\right)\right],
	\]
	where $f$ is any nonnegative Borel measurable function defined on $S$. Thus, $\Pi$ is a Poisson point measure with intensity $\pi$ if and only if for any such $f$,
	\[
		\frakL_\Pi(f)=\exp\left(-\int_S(1-\ee^{-f(s)})\,\pi(\dd s)\right).
	\]

Throughout the remainder of this section, we shall restrict our attention to Poisson point measures on the interval $(0,1]$, the product space $(0,1]\times\R^d$, or subsets thereof. Let $\nu$ be a measure in $\rad01$. We recall from Section~\ref{subsubsec:frakndesc} that $\nu$ is then a positive Borel measure on $(0,1]$ with infinite total mass and such that~(\ref{eq:condrad01}) holds, namely, the proper subintervals of the form $[r,1]$ all have finite mass. In addition, given a nonempty open subset $U$ of $\R^d$, we consider on the product space
	\[
		U_+=(0,1]\times U
	\]
	a Poisson point measure, denoted by $\Pi$, with intensity $\nu\otimes\leb^d(\,\cdot\,\cap U)$. This enables us to introduce the random set
	\begin{equation}\label{eq:df:Fnu}
		F_\nu=\left\{y\in\R^d\Biggm|\int_{U_+} \ind_{\{|y-x|<r\}}\,\Pi(\dd r,\dd x)=\infty\right\}.	
	\end{equation}

Although this might not be clear at first sight, this set is of the form~(\ref{eq:df:Fxnrn}). In fact, evaluating the Laplace functional of $\Pi$ at the function constant equal to one on $U_+$, and using the fact that $\nu$ has infinite total mass, we infer that with probability one, $\Pi$ has infinite total mass as well. Hence, it may almost surely be written as a countably infinite sum of Dirac point masses located at random pairs in $U_+$, denoted by $(R_n,X_n)$, for $n\geq 1$. Therefore, we also have almost surely
	\begin{equation}\label{eq:df:FnuXnRn}
		F_\nu=\left\{y\in\R^d\bigm||y-X_n|<R_n\quad\text{for i.m.~}n\geq 1\right\}.
	\end{equation}
	This shows in particular that $F_\nu$ is almost surely a $G_\delta$-subset of $\R^d$.

Our main result is the following complete description of the size and large intersection properties of $F_\nu$. We recall from~(\ref{eq:df:rad01d}) that $\nu$ is in $\rad01_d$ if and only if
	\[
		\croc{\nu}{r\mapsto r^d}=\int_{(0,1]} r^d\,\nu(\dd r)<\infty.
	\]

\begin{Theorem}\label{thm:Poissondesc}
	For any measure $\nu\in\rad01$ and a nonempty open set $U\subseteq\R^d$, the following properties hold:
	\begin{itemize}
		\item if $\nu\not\in\rad01_d$, then $F_\nu$ almost surely has full Lebesgue measure in $U$\,;
		\item if $\nu\in\rad01_d$, then $F_\nu$ is almost surely $\nu$-describable in $U$.
	\end{itemize}
\end{Theorem}

Before establishing Theorem~\ref{thm:Poissondesc}, let us make some comments. The description of the size and large intersection properties of the set $F_\nu$ follows as usual from the combination of that result with Theorem~\ref{thm:frakndescslip}. Also, we may restrict our attention to the case where $\nu$ is in $\rad01_d$, as otherwise these properties are trivial, in light of the beginning of Section~\ref{sec:desc}. Moreover, if one is only interested in dimensional results, Corollary~\ref{cor:frakndescslip} is enough, and actually implies in that situation that with probability one, for any nonempty open set $V\subseteq U$,
	\[
		\left\{\begin{array}{l}
			\Hdim (F_\nu\cap V)=s_\nu\\[1mm]
			\Pdim (F_\nu\cap V)=d\\[1mm]
			F_\nu\in\lic{s_\nu}{V},
		\end{array}\right.
	\]
	where the last two properties hold if $s_\nu$ is positive. Here, the exponent $s_\nu$ is that associated with $\nu$ through~(\ref{eq:df:sfraknrad01})\,; it is also characterized by the following condition:
	\[
		\left\{\begin{array}{lll}
			s<s_\nu & \Longrightarrow & \displaystyle{\int_{(0,1]}r^s\,\nu(\dd r)=\infty}\\[4mm]
			s>s_\nu & \Longrightarrow & \displaystyle{\int_{(0,1]}r^s\,\nu(\dd r)<\infty}.
		\end{array}
		\right.
	\]

Besides, in the spirit of Dvoretzky's covering problem briefly discussed in Section~\ref{sec:Dvoretzky}, one may ask for a necessarily and sufficient condition on the measure $\nu$ to ensure that with probability one, {\em all} the points of the open set $U$ are covered by the Poisson distributed balls, {\em i.e.}~that the set $F_\nu$ contains the whole open set $U$ almost surely. This problem was posed by Mandelbrot~\cite{Mandelbrot:1972aa} and solved by Shepp~\cite{Shepp:1972lr} in dimension $d=1$ when the open set $U$ is equal to the whole real line. We refer to~\cite{Bierme:2012aa} and the references therein for further results in that direction.

The remainder of this section is devoted to the proof of Theorem~\ref{thm:Poissondesc}. Since it is quite long, we split it into several parts.

\subsection{Preliminary lemmas}

For any nonempty bounded open set $V\subseteq U$, let
	\begin{equation}\label{eq:df:FVnu}
		F^V_\nu=\left\{y\in\R^d\Biggm|\int_{V_+} \ind_{\{|y-x|<r\}}\,\Pi(\dd r,\dd x)=\infty\right\}.
	\end{equation}
	The proof calls upon the following connection between $F^V_\nu$ and the intersection with $V$ of the initial set $F_\nu$.

\begin{Lemma}\label{lem:restricPoisson}
	Let $V$ be a nonempty bounded open subset of $U$. Then, the restriction $\Pi(\,\cdot\,\cap V_+)$ is a Poisson point measure on $V_+$ with intensity $\nu\otimes\leb^d(\,\cdot\,\cap V)$. Moreover, with probability one,
	\[
		F_\nu\cap V\subseteq F^V_\nu\subseteq F_\nu\cap\closure{V}.
	\]
\end{Lemma}

\begin{proof}
	The law of the restriction is easily obtained by computing its Laplace functional. In order to establish the inclusions, we define $V_1$ as the set of points $x$ in $U$ such that $\dist{x}{V}<1$, and we observe that for any $\rho\in(0,1]$, the random variable $\Pi([\rho,1]\times V_1)$ is Poisson distributed with parameter $\Phi_\nu(\rho)\leb^d(V_1)$. This parameter is finite by virtue of~(\ref{eq:condrad01}) and the boundedness of $V$. Therefore, $\Pi([\rho,1]\times V_1)$ is almost surely finite. However, this random variable is a monotonic function of $\rho$. We deduce that with probability one, all the values $\Pi([\rho,1]\times V_1)$, for $\rho\in(0,1]$, are finite. From now on, we assume that the corresponding almost sure event holds.
	
	Let us consider a point $y$ in $F_\nu\cap V$. Given that the set $V$ is open, it contains the open ball $\opball(y,\delta)$ for some $\delta>0$. Let us consider a pair $(r,x)$ in $U_+$ satisfying $|y-x|<r$. Then, this pair actually belongs to $V_+$ when $r<\delta$, and to $[\delta,1]\times V_1$ otherwise. As a consequence,
	\[
		\infty=\int_{U_+} \ind_{\{|y-x|<r\}}\,\Pi(\dd r,\dd x)
		\leq\int_{V_+} \ind_{\{|y-x|<r\}}\,\Pi(\dd r,\dd x)+\Pi([\delta,1]\times V_1).
	\]
	On the almost sure event that we considered, the second term in the right-hand side above is finite. Hence, the first term is infinite, {\em i.e.}~the point $y$ is in $F^V_\nu$.
	
	Conversely, let us consider a point $y$ in $F^V_\nu$. Given that $V_+\subseteq U_+$, the point $y$ is automatically in $F_\nu$. In order to show that $y$ also belongs to the closure of $V$, it suffices to consider an arbitrary real number $\delta>0$ and to prove that the ball $\opball(y,\delta)$ meets $V$. If $(r,x)$ denotes a pair in $V_+$ with $|y-x|<r$, we remark that $x$ belongs to that ball if $r<\delta$, and simply to the set $V_1$ otherwise. Accordingly,
	\[
		\infty=\int_{V_+} \ind_{\{|y-x|<r\}}\,\Pi(\dd r,\dd x)
		\leq\Pi((0,1]\times(\opball(y,\delta)\cap V))+\Pi([\delta,1]\times V_1).
	\]
	Again, the second term in the right-hand side is finite, so the first term is infinite, which means in particular that the sets $\opball(y,\delta)$ and $V$ intersect.
\end{proof}

Lemma~\ref{lem:restricPoisson} will enable us to reduce the proof of Theorem~\ref{thm:Poissondesc} to the case of bounded open sets. The advantage is that, with the help of the next lemma, we will be able to use a convenient representation of the Poisson point measure $\Pi$.

\begin{Lemma}\label{lem:prodUnifPoisson}
	Let us assume that the open set $U$ is bounded.
	\begin{enumerate}
		\item\label{item:lem:prodUnifPoisson1} Let $\Nu^U$ denote a Poisson point measure on the interval $(0,1]$ with intensity
		\[
			\nu^U=\leb^d(U)\,\nu.
		\]
		Then, there exists a nonincreasing sequence $(R_n)_{n\geq 1}$ of positive random variables that converges to zero such that with probability one,
		\begin{equation}\label{eq:PoissonDiracNu}
			\Nu^U=\sum_{n=1}^\infty\delta_{R_n}.
		\end{equation}
		\item\label{item:lem:prodUnifPoisson2} Let $(X_n)_{n\geq 1}$ be a sequence of random variables that are independently and uniformly distributed in $U$, and are also independent on $\Nu^U$. Then,
		\begin{equation}\label{eq:df:DiracNuU}
			\Nu^U_+=\sum_{n=1}^\infty\delta_{(R_n,X_n)}
		\end{equation}
		is a Poisson point measure on $U_+$ with intensity $\nu\otimes\leb^d(\,\cdot\,\cap U)$.
	\end{enumerate}
\end{Lemma}

\begin{proof}
	In order to prove~(\ref{item:lem:prodUnifPoisson1}), we begin by observing that the Poisson point measure $\Nu^U$ must have infinite total mass with probability one, because its intensity $\nu^U$ has infinite total mass too. Thus, there is a sequence $(R_n)_{n\geq 1}$ of positive random variables such that~(\ref{eq:PoissonDiracNu}) holds. However, the assumption~(\ref{eq:condrad01}) implies that
	\[
		\forall\rho>0\qquad
		\esp[\#\{n\geq 1\:|\:R_n\geq\rho\}]=\esp[\Nu^U([\rho,1])]=\Phi_{\nu^U}(\rho)<\infty.
	\]
	Thus, $(R_n)_{n\geq 1}$ converges to zero with probability one. Now, up to rearranging the terms, we can assume that this sequence is nonincreasing and still verifies~(\ref{eq:PoissonDiracNu}).
	
	The property~(\ref{item:lem:prodUnifPoisson2}) may be established by computing the Laplace functional of the random point measure $\Nu^U_+$. Let $f$ denote a nonnegative Borel measurable function defined on $U_+$. Then, we have
	\[
		\frakL_{\Nu^U_+}(f)
		=\esp\left[\exp\left(-\sum_{n=1}^\infty f(R_n,X_n)\right)\right]
		=\esp\left[\prod_{n=1}^\infty\left(\int_U\ee^{-f(R_n,x)}\frac{\dd x}{\leb^d(U)}\right)\right].
	\]
	The right-hand side may be rewritten as the Laplace functional of the random point measure $\Nu^U$ evaluated at the nonnegative Borel measurable function
	\[
		r\mapsto -\log\int_U\ee^{-f(r,x)}\frac{\dd x}{\leb^d(U)}.
	\]
	Since $\Nu^U$ is a Poisson point measure with intensity $\nu^U$, we finally deduce that for every nonnegative Borel measurable function $f$ defined on the set $U_+$, we have
	\[
		\frakL_{\Nu^U_+}(f)
		=\exp\left(-\int_{U_+} (1-\ee^{-f(r,x)})\,\nu^U(\dd r)\otimes\frac{\dd x}{\leb^d(U)}\right),
	\]
	from which we may determine the law of the random point measure $\Nu^U_+$.
\end{proof}

The representation supplied by Lemma~\ref{lem:prodUnifPoisson} calls upon a sequence of independent uniform random points. In view of Theorem~\ref{thm:eutaxicindepunif}, it thus establishes a connection with eutaxy that we shall exploit in the proof of Theorem~\ref{thm:Poissondesc}. The proof will also make a crucial use of the following result on the integrability of gauge functions with respect to Poisson random measures.

\begin{Lemma}\label{lem:caracgaugePoisson}
	Let us suppose that the measure $\nu$ is in $\rad01_d$ and that the open set $U$ is bounded, and let us consider a Poisson point measure $\Nu^U$ on the interval $(0,1]$ with intensity equal to $\leb^d(U)\,\nu$. Then, with probability one,
	\[
		\Nu^U\in\rad01_d
		\qquad\text{and}\qquad
		\gauge(\Nu^U)=\gauge(\nu).
	\]
\end{Lemma}

\begin{proof}
	Similarly to~(\ref{eq:condrad01}), for any real number $\rho\in(0,1]$, we define $\Phi_{\Nu^U}(\rho)$ and $\Phi_\nu(\rho)$ as being equal to $\Nu^U([\rho,1])$ and $\nu([\rho,1])$, respectively. The proof of the lemma relies on the crucial observation that with probability one,
	\begin{equation}\label{eq:asymptNu}
		\Phi_{\Nu^U}(\rho)\sim\leb^d(U)\,\Phi_\nu(\rho) \quad\text{as}\quad \rho\to 0,
	\end{equation}
	see~\cite[Lemma~5]{Durand:2010uq}. This shows in particular that $\Nu^U$ is almost surely in $\rad01$.

Let us consider a gauge function $g$ in $\gauge$ with $d$-normalization denoted by $g_d$ as usual. The function $g_d$ is nondecreasing and continuous near zero, but need not satisfy this property on the whole interval $[0,1]$. However, $g_d$ clearly coincides near zero with some function denoted by $\widetilde g$ which is both nondecreasing and continuous on the whole $[0,1]$. Moreover, due to~(\ref{eq:condrad01}) and the fact that $g_d$ is bounded on $(0,1]$,
	\[
		\int_{(0,1]} g_d(r)\,\nu(\dd r)=\infty
		\qquad\Longleftrightarrow\qquad
		\int_{(0,1]} \widetilde g(r)\,\nu(\dd r)=\infty.
	\]
	Likewise, the above characterization is valid if $\nu$ is replaced by $\Nu^U$, on the almost sure event on which it belongs to $\rad01$.

Now, similarly to the proof of Theorem~\ref{thm:descoptregsys}, let us introduce the Lebesgue-Stieltjes measure associated with the monotonic function $\widetilde g$. Specifically, let $\zeta$ be the premeasure satisfying $\zeta((r,r'])=\widetilde g(r')-\widetilde g(r)$ when $0\leq r\leq r'\leq 1$, and let $\zeta_\ast$ be the outer measure defined by~(\ref{eq:df:zetaastm}). Theorem~\ref{thm:bormeasm} shows that the Borel sets contained in $(0,1]$ are $\zeta_\ast$-measurable, and we may thus integrate locally bounded Borel-measurable functions with respect to $\zeta_\ast$. One may prove that $\zeta_\ast$ coincides with the premeasure $\zeta$ on the intervals where it is defined, and in particular that $\zeta_\ast((0,r])=\widetilde g(r)$ for any $r\in(0,1]$. Using Tonelli's theorem, we deduce that
	\[
		\int_{(0,1]} \widetilde g(r)\,\nu(\dd r)=\int_{(0,1]} \Phi_\nu(\rho)\,\zeta_\ast(\dd\rho),
	\]
	and that the same property holds when $\nu$ is replaced by $\Nu^U$. Placing ourselves on the almost sure event where~(\ref{eq:asymptNu}) holds, we conclude that $\Nu^U$ belongs to $\rad01_d$ because $\nu$ does, and that the sets $\gauge(\Nu^U)$ and $\gauge(\nu)$ coincide.
\end{proof}

\subsection{Reduction to the bounded case}

We now reduce the study to the case in which the ambient open set is bounded. The sets defined for $\ell\geq 1$ by
	\[
		U_\ell=\{x\in U\cap\opball(0,\ell)\:|\:\dist{x}{\R^d\setminus(U\cap\opball(0,\ell))}>1/\ell\},
	\]
	form a nondecreasing sequence of bounded open sets with union $U$. In particular, the sets $U_\ell$ are nonempty from some $\ell_0$ onwards. We have actually $\closure{U_\ell}\subseteq U_{\ell+1}$ for all $\ell\geq\ell_0$. Letting $F^{U_\ell}_\nu$ be defined as in~(\ref{eq:df:FVnu}), we infer from Lemma~\ref{lem:restricPoisson} that
	\begin{equation}\label{eq:decompFnuU}
		F_\nu\cap U=\bigcup_{\ell=\ell_0}^\infty\uparrow (F^{U_\ell}_\nu\cap U_\ell).
	\end{equation}

Let us assume that Theorem~\ref{thm:Poissondesc} holds for bounded open sets, and let us begin by supposing that the measure $\nu$ is not in $\rad01_d$. Then, for any $\ell\geq\ell_0$, with probability one, the set $F^{U_\ell}_\nu$ is Lebesgue full in $U_\ell$. We readily deduce from~(\ref{eq:decompFnuU}) and the basic properties of Lebesgue measure that $F_\nu$ is almost surely Lebesgue full in $U$.
	
	Let us now suppose that $\nu$ is in $\rad01_d$. Then, for any $\ell\geq\ell_0$, with probability one, any gauge function in $\gauge(\nu)^\complement$ is majorizing for $F^{U_\ell}_\nu$ in $U_\ell$. Hence, with probability one, any such gauge function is majorizing for $F_\nu$ in $U$\,; this is due to~(\ref{eq:decompFnuU}) and the fact that Hausdorff measures are outer measures. In other words,
	\[
		\as \qquad \majo(F_\nu,U)\supseteq\gauge(\nu)^\complement.
	\]
	Furthermore, we also know that for any $\ell\geq\ell_0$, with probability one, any gauge function in $\gauge(\nu)$ is minorizing for $F^{U_\ell}_\nu$ in $U_\ell$. Thus, with probability one, for any such gauge function $g$, each set $F^{U_\ell}_\nu$ with $\ell\geq\ell_0$ is in $\lic{g}{U_\ell}$. By Definition~\ref{df:liclocgauge}, for any $d$-normalized gauge function $h\prec g_d$ and any open set $V\subseteq U$, we get
	\[
		\netm^h_\infty(F^{U_\ell}_\nu\cap U_\ell\cap V)=\netm^h_\infty(U_\ell\cap V),
	\]
	because $U_\ell\cap V$ is then an open subset of $U_\ell$. The sets in the right-hand side are nondecreasing with respect to $\ell$ and their union is equal to $V$. Owing to~(\ref{eq:decompFnuU}), the sets in the left-hand side satisfy the same monotonicity property, with an union equal to $F_\nu\cap V$. We now use the fact that~(\ref{eq:measincunion}) holds for the outer measure $\netm^h_\infty$ even if the involved sets need not be measurable, see~\cite[Theorem~52]{Rogers:1970wb}. We get
	\[
		\netm^h_\infty(F_\nu\cap V)=\netm^h_\infty(V).
	\]
	We have thus proved that with probability one, for any gauge function $g$ in $\gauge(\nu)$, the set $F_\nu$ belongs to the large intersection class $\lic{g}{U}$. As a result,
	\[
		\as \qquad \mino(F_\nu,U)\supseteq\gauge(\nu).
	\]
	To conclude that $F_\nu$ is almost surely $\nu$-describable in $U$, it suffices to apply Proposition~\ref{prp:descincl} and Lemma~\ref{lem:gaugemeasopen}(\ref{item:lem:gaugemeasopen2}).

\subsection{Proof in the bounded case}

It remains us to establish Theorem~\ref{thm:Poissondesc} in the case where the open set $U$ is bounded. Let $\Nu^U$ denote a Poisson point measure on $(0,1]$ with intensity $\leb^d(U)\,\nu$. Lemma~\ref{lem:prodUnifPoisson}(\ref{item:lem:prodUnifPoisson1}) ensures the existence of a nonincreasing sequence $(R_n)_{n\geq 1}$ of positive random variables that converges to zero such that~(\ref{eq:PoissonDiracNu}) holds with probability one. Moreover, let $(X_n)_{n\geq 1}$ be a sequence of random variables that are independently and uniformly distributed in $U$, and are also independent on $\Nu^U$. Lemma~\ref{lem:prodUnifPoisson}(\ref{item:lem:prodUnifPoisson2}) now implies that the random point measure defined on $U_+$ by~(\ref{eq:df:DiracNuU}) is Poisson distributed with intensity $\nu\otimes\leb^d(\,\cdot\,\cap U)$. Hence, the random point measures $\Pi$ and $\Nu^U_+$ share the same law, and we may assume that $\Pi$ is replaced by $\Nu^U_+$ in the definition of the random set $F_\nu$. This enables us to write $F_\nu$ in the alternate form~(\ref{eq:df:FnuXnRn}), where the random variables $R_n$ and $X_n$ are defined as above. On top of that, Theorem~\ref{thm:eutaxicindepunif} ensures that with probability one, the sequence $(X_n)_{n\geq 1}$ is almost surely uniformly eutaxic in $U$.

Evaluating the Laplace functional of $\Nu^U$ at the function $r\mapsto r^d$, we obtain
	\[
		\esp\left[\exp\left(-\sum_{n=1}^\infty R_n^d\right)\right]
		=\exp\left(-\leb^d(U)\int_{(0,1]}(1-\ee^{-r^d})\,\nu(\dd r)\right).
	\]
Therefore, if $\nu$ is not in $\rad01_d$, the integral in the right-hand side is infinite, so that the expectation in the left-hand side vanishes. This means that $\sum_n R_n^d$ diverges almost surely, and thus that the sequence $(R_n)_{n\geq 1}$ belongs to the set $\Rho_d$. By Definition~\ref{df:eutaxicunif}, the set $F_\nu$ almost surely has full Lebesgue measure in $U$.

Lastly, if $\nu$ is in $\rad01_d$, Lemma~\ref{lem:caracgaugePoisson} entails that with probability one, $\Nu^U$ belongs to $\rad01_d$, and thus that $\sum_n R_n^d$ converges. Applying Theorem~\ref{thm:desceutaxy}, we then deduce that with probability one, the set $F_\nu$ is $\Nu^U$-describable in $U$. However, Lemma~\ref{lem:caracgaugePoisson} shows that the sets $\gauge(\Nu^U)$ and $\gauge(\nu)$ coincide almost surely. It follows that $F_\nu$ is almost surely $\nu$-describable in $U$.


\section{Singularity sets of L\'evy processes}

With that level of generality, Theorem~\ref{thm:Poissondesc} does not appear anywhere else in the literature. However, in dimension $d=1$, results of the same flavor have been obtained in~\cite{Durand:2007fk} with a view to studying the singularity sets of L\'evy processes. Similar results are used in~\cite{Durand:2010uq} to perform the multifractal analysis of multivariate extensions of L\'evy processes; this also corresponds to the case where $d=1$, but the approximating points are replaced by Poisson distributed hyperplanes. Actually,~\cite{Durand:2007fk} already gives a description of the size and large intersection the singularity sets of L\'evy processes. Thanks to Theorem~\ref{thm:Poissondesc}, we shall recast the results obtained therein in terms of describability, and also drop some inessential assumptions.

We recall from~\cite{Bertoin:1996uq,Sato:1999fk} that any $\R^{d'}$-valued L\'evy process may be written as an independent sum of a Brownian motion with drift, a compound Poisson process with jumps of magnitude larger than one, and a third process that we now define; this is the L\'evy-It\=o decomposition of the process. We also recall that the multifractal analysis of L\'evy processes was initiated by Jaffard~\cite{Jaffard:1999fg}. Starting from a Borel measure $\frakj$ with support in the set $\clball(0,1)\setminus\{0\}$ such that
	\[
		\int_{0<|z|\leq 1} |z|^2\,\frakj(\dd z)<\infty,
	\]
	we consider on the product of the latter set with the open interval $U=(0,\infty)$ a Poisson point measure, denoted by $\frakJ$, with intensity the product of the measures $\frakj$ and $\leb^1(\,\cdot\,\cap U)$. Then, for any $\delta\in(0,1)$, we may consider the process defined by
	\[
		Z^\delta_t
		=\lim_{\eps\to 0} Z^{\eps,\delta}_t
		\qquad\text{with}\qquad
		Z^{\eps,\delta}_t=\int_{\eps<|z|\leq\delta\atop 0<x\leq t} z\,\frakJ(\dd z,\dd x)-t\int_{\eps<|z|\leq\delta} z\,\frakj(\dd z)
	\]
	and convergence holds uniformly on all compact subsets of $[0,\infty)$, with probability one. The process $(Z^\delta_t)_{t\geq 0}$ is a L\'evy process with L\'evy measure $\ind_{\{|z|\leq\delta\}}\frakj(\dd z)$. The third process in the aforementioned L\'evy-It\=o decomposition is of the previous form with $\delta=1$. We restrict our attention to this process only, because the two other components are trivial from the viewpoint of multifractal analysis, see~\cite{Jaffard:1999fg}. This process is denoted by $Z=(Z_t)_{t\geq 0}$, for short, instead of $(Z^1_t)_{t\geq 0}$. To avoid another trivial situation, we also assume that the L\'evy measure $\frakj$ has infinite total mass.

In the context of multifractal analysis, the pointwise regularity of the process $Z$ is measured by means of the {\em H\"older exponent} $h_Z(t)$, {\em i.e.}~the supremum of all $h>0$ such that there exist a real number $c>0$ and a $d'$-tuple $P$ of polynomials satisfying
	\[
		|Z_{t'}-P(t'-t)|\leq c\, |t'-t|^h
	\]
	for any nonnegative real number $t'$ in a neighborhood of $t$, see~\cite{Jaffard:2004fh}. A {\em singularity set} is then a set of times at which the process is continuous and has H\"older exponent bounded above by a given value, that is, a set of the form
	\[
		S_Z(h)=\{ t\in [0,\infty)\setminus J_Z \:|\: h_Z(t)\leq h \},
	\]
	with $h\in[0,\infty]$, where $J_Z$ is the set of jump times of $Z$. Note that $Z$ is almost surely right-continuous with finite left limits at every time $t>0$, so we may define $\Delta Z_t=Z_t-Z_{t-}$. The set $J_Z$ is then formed of the times $t>0$ at which $\Delta Z_t\neq 0$.
	
The singularity sets exhibit a remarkable connection with Poisson coverings. Indeed, for any real number $\alpha>0$, let us consider the random set
	\[
	G_\alpha=\left\{t\in\R\Biggm|\int_{0<|z|\leq 1\atop x>0} \ind_{\{|t-x|<|z|^{1/\alpha}\}}\,\frakJ(\dd z,\dd x)=\infty\right\}.
	\]
	It is clear that the sets $G_\alpha$ are nondecreasing with respect to $\alpha$, and that each of them is distributed as the set $F_{\frakj_\alpha}$ defined as in~(\ref{eq:df:Fnu}), where $\frakj_\alpha$ denotes the pushforward of the L\'evy measure $\frakj$ under the mapping $z\mapsto|z|^{1/\alpha}$. Note that $\frakj_\alpha$ belongs to the collection $\rad01$. Moreover, letting $\beta$ denote the Blumenthal-Getoor exponent of the process, namely, the exponent
		\[
			\beta=
			\inf\left\{ \gamma\geq 0 \:\Biggm|\: \int_{0<|z|\leq 1} |z|^\gamma\,\frakj(\dd z)<\infty \right\}
			\in[0,2],
		\]
		we see that $\frakj_\alpha$ does not belong to $\rad01_1$ when $\alpha>1/\beta$. Due to Theorem~\ref{thm:Poissondesc}, the set $G_\alpha$ almost surely has full Lebesgue measure in the open interval $U=(0,\infty)$. In that case, one can actually show that $G_\alpha$ is almost surely equal to the whole interval $[0,\infty)$\,; this follows for instance from the result on Poisson random coverings obtained by Shepp~\cite{Shepp:1972lr} and mentioned after the statement of Theorem~\ref{thm:Poissondesc} above. The connexion with the singularity sets is then embodied by the next statement established by Jaffard~\cite{Jaffard:1999fg}, up to a minor assumption on the L\'evy measure.

\begin{Proposition}\label{prp:linkSZGalpha}
	With probability one, for any real number $h\geq 0$,
	\[
	S_Z(h)=\left(\bigcap_{\alpha>h}\downarrow G_\alpha\right)\setminus J_Z.
	\]
\end{Proposition}

\begin{proof}
	This is essentially Proposition~1 in~\cite{Jaffard:1999fg}, except that the arguments therein require an additional integrability assumption on the L\'evy measure $\frakj$, namely, the condition numbered by~(3) in~\cite{Jaffard:1999fg}. We refer however to that paper for the proof, and we content ourselves with briefly explaining how to drop the assumption.
	
	We assert that, in order to get rid of that inessential assumption, it suffices to combine Jaffard's approach with the following uniform estimate on the processes $(Z^\delta_t)_{t\geq 0}$\,: for any integer $n\geq 1$ and any real number $\alpha\in (0,1/\beta)$, with probability one, for any integer $j$ large enough,
	\[
		\sup_{0\leq t\leq t'\leq n \atop t'-t\leq 2^{-j}} |Z^{2^{-\alpha j}}_{t'}-Z^{2^{-\alpha j}}_{t}| \leq j\,2^{-\alpha j}.
	\]
	Let us establish this bound. For $\delta\in(0,1)$ and $i\in\{1,\ldots,d'\}$, let $(Z^{\delta,i}_t)_{t\geq 0}$ denote the $i$-th coordinate of the process $(Z^\delta_t)_{t\geq 0}$. We apply a Bernstein-type inequality for real-valued integrals with respect to compensated Poisson point measures, see {\em e.g.}~\cite[Lemma~4]{Durand:2010uq}. We thus infer that for any real numbers $T\geq 0$ and $\zeta>0$,
	\[
		\prob\left(\sup_{0\leq t\leq T} |Z^{\delta,i}_t|\geq\zeta\right)
		\leq 2\exp\left(-\frac{3\zeta^2}{2\delta\zeta+6T\Iota_{\delta,i}}\right)
		\quad\text{with}\quad
		\Iota_{\delta,i}=\int_{0<|z|\leq\delta} z_i^2\,\frakj(\dd z),
	\]
	where $z_i$ is the $i$-th coordinate of $z$. Furthermore, for any interval $I\subseteq [0,\infty)$, any real numbers $\delta\in(0,1)$ and $\zeta>0$, and any positive real number $\eta\leq\min\{1/2,|I|\}$,
	\[
		\prob\left( \sup_{t,t'\in I \atop |t'-t|\leq\eta} |Z^{\delta,i}_{t'}-Z^{\delta,i}_t| \geq \zeta \right)
		\leq \frac{2|I|}{\eta}\,\prob\left(\sup_{0\leq t\leq 1/q} |Z^{\delta,i}_t|\geq\frac{\zeta}{4} \right),
	\]
	where $q$ denotes the integer part of $1/\eta$. This results from the stationarity of the increments of the process, and a standard comparison with its increments on the natural lattice where consecutive points are distant from $1/q$. On top of that, for any $\alpha\in (0,1/\beta)$, it is clear that $\eta\,\Iota_{\eta^{\alpha},i}=\smallo(\eta^{2\alpha})$ as $\eta\to 0$. It follows that for any real number $c>0$, any integer $n\geq 1$, and for $\eta$ small enough,
	\[
		\prob\left(\sup_{0\leq t\leq t'\leq n \atop t'-t\leq\eta} |Z^{\eta^\alpha,i}_{t'}-Z^{\eta^\alpha,i}_t|\geq c\,\eta^\alpha\log\frac{1}{\eta}\right)\leq 4\,n\,\eta^{c/4}.
	\]
	To obtain the required bound, it suffices to apply the Borel-Cantelli lemma, and to merge all the coordinates together.
\end{proof}

We may now describe the size and large intersection properties of the singularity sets $S_Z(h)$. It follows from Proposition~\ref{prp:linkSZGalpha} and the preceding discussion that with probability one, the sets $S_Z(h)$, for $h\geq 1/\beta$, all coincide with $[0,\infty)\setminus J_Z$. However, the set $J_Z$ of jump times of the process is almost surely countable. This means that these sets have full Lebesgue measure in $(0,\infty)$.

We rule out, as trivial, this case and we assume from now on that $h<1/\beta$. Then, Proposition~\ref{prp:linkSZGalpha} shows that the singularity sets $S_Z(h)$ are based on the sets $G_\alpha$, for $\alpha<1/\beta$. In that situation, it follows from the definition of the Blumenthal-Getoor exponent that the measure $\frakj_\alpha$ belongs to the collection $\rad01_1$. Theorem~\ref{thm:Poissondesc} entails that the set $G_\alpha$ is almost surely $\frakj_\alpha$-describable in $U=(0,\infty)$. Using the monotonicity of these sets in conjunction with Propositions~\ref{prp:monomajomino} and~\ref{prp:cupcapmajomino}, we deduce that with probability one, for any real number $h\in[0,1/\beta)$,
	\[
		\mino(S_Z(h),(0,\infty))\cap\gauge^\infty=\gauge(\frakj,h)
		\qquad\text{and}\qquad
		\majo(S_Z(h),(0,\infty))\supseteq\gauge(\frakj,h)^\complement.
	\]
	We also used again the fact that $J_Z$ has Lebesgue measure zero almost surely, so that removing this set does not alter the involved minorizing and majorizing classes. Above, $\gauge(\frakj,h)$ and $\gauge(\frakj,h)^\complement$ denote the set of gauge functions defined by
	\[
		\gauge(\frakj,h)=\bigcap_{h<\alpha<1/\beta}\downarrow\gauge(\frakj_\alpha)
	\]
	and its complement in $\gauge^\infty$, respectively. This means in particular that the singularity sets $S_Z(h)$ are fully describable in $(0,\infty)$. Moreover, it is possible to show that the collection $\gauge(\frakj,h)$ is right-open, see the proof of~\cite[Proposition~5]{Durand:2007fk}. Thus, applying Proposition~\ref{prp:descincl}, we end up with the next statement.

\begin{Theorem}
	With probability one, for any real number $h\in[0,1/\beta)$,
	\[
		\mino(S_Z(h),(0,\infty))\cap\gauge^\infty=\gauge(\frakj,h)
		\qquad\text{and}\qquad
		\majo(S_Z(h),(0,\infty))=\gauge(\frakj,h)^\complement.
	\]
\end{Theorem}

Subsequently applying Corollary~\ref{cor:linkopenmet}, we deduce a thorough description of the size and large intersection properties of the singularity sets. In particular, each singularity set $S_Z(h)$ is a set with large intersection in $(0,\infty)$ with dimension $\beta h$. Finally, under explicit assumptions on the L\'evy measure $\frakj$, a more tractable expression for the sets $\gauge(\frakj,h)$ may be obtained. For instance, in the stable case, the polar representation of $\frakj$ is the product of the measure $\frakn_\beta$ defined as in~(\ref{eq:df:frakns}) with a finite measure on the unit sphere. Hence, each measure $\frakj_\alpha$ coincides with $\frakn_{\alpha\beta}$ up to some multiplicative constant; in view of~(\ref{eq:linkfraknfraksdesc}), the sets $\gauge(\frakj,h)$ and $\gauge(\beta h)$ are thus the same. We conclude that with probability one, for any real number $h\in[0,1/\beta)$, the singularity set $S_Z(h)$ is $(\beta h)$-describable in $(0,\infty)$.


\end{document}